\documentclass{amsbook}

\usepackage{amssymb,enumitem}
\usepackage[all]{xy}
\usepackage{hyperref,aliascnt}
\usepackage{mathtools,array}
\usepackage{amsmidx}
\usepackage{ragged2e}

\DeclareRobustCommand{\SkipTocEntry}[5]{}

\numberwithin{section}{chapter}
\numberwithin{equation}{chapter}

\newtheorem{lma}{Lemma}[section]

\newaliascnt{thmCt}{lma}
\newtheorem{thm}[thmCt]{Theorem}
\aliascntresetthe{thmCt}

\newaliascnt{corCt}{lma}
\newtheorem{cor}[corCt]{Corollary}
\aliascntresetthe{corCt}

\newaliascnt{prpCt}{lma}
\newtheorem{prp}[prpCt]{Proposition}
\aliascntresetthe{prpCt}

\newtheorem*{thm*}{Theorem}
\newtheorem*{cor*}{Corollary}
\newtheorem*{prp*}{Proposition}

\newtheorem{lmaChapter}{Lemma}[chapter]

\newaliascnt{thmChapterCt}{lmaChapter}
\newtheorem{thmChapter}[thmChapterCt]{Theorem}
\aliascntresetthe{thmChapterCt}

\newaliascnt{prpChapterCt}{lmaChapter}
\newtheorem{prpChapter}[prpChapterCt]{Proposition}
\aliascntresetthe{prpChapterCt}

\theoremstyle{definition}

\newaliascnt{pgrCt}{lma}
\newtheorem{pgr}[pgrCt]{}
\aliascntresetthe{pgrCt}

\newaliascnt{dfnCt}{lma}
\newtheorem{dfn}[dfnCt]{Definition}
\aliascntresetthe{dfnCt}

\newaliascnt{rmkCt}{lma}
\newtheorem{rmk}[rmkCt]{Remark}
\aliascntresetthe{rmkCt}

\newaliascnt{rmksCt}{lma}
\newtheorem{rmks}[rmksCt]{Remarks}
\aliascntresetthe{rmksCt}

\newaliascnt{prblCt}{lma}
\newtheorem{prbl}[prblCt]{Problem}
\aliascntresetthe{prblCt}

\newaliascnt{exaCt}{lma}
\newtheorem{exa}[exaCt]{Example}
\aliascntresetthe{exaCt}

\newaliascnt{exasCt}{lma}

\aliascntresetthe{exasCt}

\newaliascnt{conjCt}{lma}
\newtheorem{conj}[conjCt]{Conjecture}
\aliascntresetthe{conjCt}

\newaliascnt{pgrChapterCt}{lmaChapter}
\newtheorem{pgrChapter}[pgrChapterCt]{}
\aliascntresetthe{pgrChapterCt}

\newaliascnt{dfnChapterCt}{lmaChapter}
\newtheorem{dfnChapter}[dfnChapterCt]{Definition}
\aliascntresetthe{dfnChapterCt}

\newaliascnt{rmksChapterCt}{lmaChapter}
\newtheorem{rmksChapter}[rmksChapterCt]{Remarks}
\aliascntresetthe{rmksChapterCt}

\newaliascnt{prblChapterCt}{lmaChapter}

\aliascntresetthe{prblChapterCt}

\newaliascnt{exaChapterCt}{lmaChapter}
\newtheorem{exaChapter}[exaChapterCt]{Example}
\aliascntresetthe{exaChapterCt}

\newaliascnt{exasChapterCt}{lmaChapter}
\newtheorem{exasChapter}[exasChapterCt]{Examples}
\aliascntresetthe{exasChapterCt}

\newcommand{\ssubset}{\ensuremath{\subset \hskip-5pt \subset}}

\newcommand{\verteq}{\rotatebox{90}{\,=}}
\newcommand{\ctsRel}{\ast}
\newcommand{\elmtrySgp}[1]{E_{#1}}
\newcommand{\inDown}{\rotatebox[origin=c]{90}{$\in$}}
\newcommand{\vect}[1]{\textbf{#1}}

\newcommand{\W}{\mathrm{W}}
\newcommand{\M}{\mathrm{M}}
\newcommand{\N}{\mathbb{N}}
\newcommand{\Q}{\mathbb{Q}}
\newcommand{\R}{\mathbb{R}}
\newcommand{\C}{\mathbb{C}}
\newcommand{\Z}{\mathbb{Z}}
\newcommand{\T}{\mathbb{T}}
\newcommand{\K}{\mathrm{K}}

\newcommand{\txtFA}{\text{ for all }}

\newcommand{\dcpo}{\textbf{dcpo}}
\newcommand{\pom}{positively ordered monoid}
\newcommand{\prePom}{positively pre-ordered monoid}

\newcommand{\pog}{partially ordered group}
\newcommand{\por}{partially ordered ring}

\newcommand{\txtCanc}{{\text{canc}}}
\newcommand{\txtDir}{{\text{dir}}}
\newcommand{\txtMax}{{\mathrm{max}}}
\newcommand{\txtMin}{{\mathrm{min}}}
\newcommand{\tensMax}{\otimes_{\mathrm{max}}}
\newcommand{\tensMin}{\otimes_{\mathrm{min}}}

\newcommand{\CatMon}{\mathrm{Mon}}
\newcommand{\CatMonMor}{\CatMon}
\newcommand{\CatMonBimor}{\mathrm{Bi}\CatMon}

\newcommand{\CatCon}{\mathrm{Con}}
\newcommand{\CatGp}{\mathrm{Gp}}

\newcommand{\CatPrePom}{\mathrm{PrePOM}}
\newcommand{\CatPom}{\mathrm{POM}}
\newcommand{\CatPomMor}{\CatPom}
\newcommand{\CatPomBimor}{\mathrm{Bi}\CatPom}
\newcommand{\CatPomLim}{\CatPom\text{-}\!\varinjlim}

\newcommand{\CatPoGp}{\mathrm{POGp}}
\newcommand{\CatPoGpMor}{\CatPoGp}
\newcommand{\CatPoGpBimor}{\mathrm{Bi}\CatPoGp}

\newcommand{\CatPoRg}{\mathrm{PORg}}

\newcommand{\CatSrg}{\mathrm{Srg}}

\newcommand{\CatSet}{\mathrm{Set}}
\newcommand{\CatTop}{\mathrm{Top}}
\newcommand{\CatCGHTop}{\mathrm{CGHTop}}

\newcommand{\CatCa}{C^*}
\newcommand{\CatCaMor}{\mathrm{Hom}_{C^*}}
\newcommand{\CatCaBimor}{\mathrm{BiHom}_{C^*}}
\newcommand{\CatCaLim}{\CatCa\text{-}\!\varinjlim}
\newcommand{\CatLocCa}{\CatCa_{\mathrm{loc}}}
\newcommand{\CatLocCaLim}{\CatLocCa\text{-}\!\varinjlim}

\newcommand{\CatCu}{\ensuremath{\mathrm{Cu}}}
\newcommand{\CatCuMor}{\CatCu}
\newcommand{\CatCuBimor}{\mathrm{Bi}\CatCu}
\newcommand{\CatCuLim}{\CatCu\text{-}\!\varinjlim}
\newcommand{\CatCuGenMor}[1]{\CatCuMor[#1]}
\newcommand{\CatCuGenBimor}[1]{\CatCuBimor[#1]}

\newcommand{\CatV}{\mathcal{V}}
\newcommand{\CatVMor}{\CatV}
\newcommand{\CatC}{\mathcal{C}}
\newcommand{\CatCMor}{\CatC}
\newcommand{\CatCBimor}{\mathrm{Bi}\CatC}
\newcommand{\CatD}{\mathcal{D}}
\newcommand{\CatDMor}{\CatD}

\newcommand{\CatPreW}{\mathrm{PreW}}
\newcommand{\CatPreWLim}{\CatPreW\text{-}\!\varinjlim}
\newcommand{\CatPreWW}{\mathrm{(Pre)W}}
\newcommand{\CatW}{\mathrm{W}}
\newcommand{\CatWMor}{\CatW}
\newcommand{\CatWBimor}{\mathrm{Bi}\CatW}
\newcommand{\CatWLim}{\CatW\text{-}\!\varinjlim}
\newcommand{\CatWGenMor}[1]{\CatWMor[#1]}
\newcommand{\CatWGenBimor}[1]{\CatWBimor[#1]}

\newcommand{\preCa}{pre-$C^*$-algebra}
\newcommand{\ca}{$C^*$-al\-ge\-bra}

\newcommand{\starAlg}{${}^*$-al\-ge\-bra}
\newcommand{\starHom}{${}^\ast$-ho\-mo\-mor\-phism}
\newcommand{\axiomO}[1]{(O#1)}
\newcommand{\axiom}[1]{(#1)}
\newcommand{\axiomW}[1]{(W#1)}

\DeclareMathOperator{\Gr}{Gr}

\DeclareMathOperator{\Her}{Her}
\DeclareMathOperator{\Lat}{Lat}
\DeclareMathOperator{\Tor}{Tor}
\DeclareMathOperator{\Cu}{Cu}
\DeclareMathOperator{\Lsc}{Lsc}
\DeclareMathOperator{\ev}{ev}
\DeclareMathOperator{\supp}{supp}
\DeclareMathOperator{\id}{id}

\DeclareMathOperator{\QT}{QT}
\DeclareMathOperator{\Idl}{Idl}

\newcommand{\pnc}{{\rm{pnc}}}
\newcommand{\wpnc}{{\rm{wpnc}}}
\newcommand{\soft}{{\rm{soft}}}
\newcommand{\op}{{\rm{op}}}
\newcommand{\alg}{\mathrm{alg}}

\newcommand{\freeVar}{\_\,}

\newcolumntype{P}[1]{>{\RaggedRight\hspace{0pt}}m{#1}}

\makeatletter
\newcommand{\rmnum}[1]{\romannumeral #1}
\newcommand{\Rmnum}[1]{\expandafter\@slowromancap\romannumeral #1@}
\makeatother

\newcommand{\enumStatement}[1]{(#1)}
\newcommand{\enumCondition}[1]{(\rmnum{#1})}

\newcommand{\implStatements}[2]{\enumStatement{#1}$\Rightarrow$\enumStatement{#2}}

\SetLabelAlign{Center}{\makebox[2em]{#1}}
\newcommand{\beginEnumStatements}{\begin{enumerate}[label=(\arabic*),align=Center,leftmargin=*, widest=iiii]}
\newcommand{\beginEnumConditions}{\begin{enumerate}[label=(\roman*),align=Center, leftmargin=*, widest=iiii]}

\makeindex{terms}
\makeindex{symbols}

\begin{document}

\frontmatter
\title{Tensor products and regularity properties of Cuntz semigroups}

\author{Ramon Antoine}

\author{Francesc Perera}

\author{Hannes Thiel}

\date{October 27, 2014; Revised October 08, 2015}

\address{
R.~Antoine, Departament de Matem\`{a}tiques,
Universitat Aut\`{o}noma de Barcelona,
08193 Bellaterra, Barcelona, Spain}
\email[]{ramon@mat.uab.cat}

\address{
F.~Perera, Departament de Matem\`{a}tiques,
Universitat Aut\`{o}noma de Barcelona,
08193 Bellaterra, Barcelona, Spain}
\email[]{perera@mat.uab.cat}

\address{
H.~Thiel,
Mathematisches Institut,
Universit\"at M\"unster,
Einsteinstra{\ss}e 62, 48149 M\"{u}nster, Germany}
\email[]{hannes.thiel@uni-muenster.de}

\dedicatory{\vspace{1.5cm} {\Large Dedicated to George A.\ Elliott on the occasion of his 70th birthday} \\ \vspace{0.5cm} October 08, 2015}



%\thanks{}

%==========================================================================================
%==========================================================================================
\begin{abstract}
The Cuntz semigroup of a \ca{} is an important invariant in the structure and classification theory of \ca{s}.
It captures more information than $K$-theory but is often more delicate to handle.
We systematically study the lattice and category theoretic aspects of Cuntz semigroups.

Given a \ca{} $A$, its (concrete) Cuntz semigroup $\Cu(A)$ is an object in the category $\CatCu$ of (abstract) Cuntz semigroups, as introduced by Coward, Elliott and Ivanescu in \cite{CowEllIva08CuInv}.
To clarify the distinction between concrete and abstract Cuntz semigroups, we will call the latter $\CatCu$-semigroups.

We establish the existence of tensor products in the category $\CatCu$ and study the basic properties of this construction.
We show that $\CatCu$ is a symmetric, monoidal category and relate $\Cu(A\otimes B)$ with $\Cu(A)\otimes_\CatCu\Cu(B)$ for certain classes of \ca{s}.

As a main tool for our approach we introduce the category $\CatW$ of pre-completed Cuntz semigroups.
We show that $\CatCu$ is a full, reflective subcategory of $\CatW$.
One can then easily deduce properties of $\CatCu$ from respective properties of $\CatW$, for example the existence of tensor products and inductive limits.
The advantage is that constructions in $\CatW$ are much easier since the objects are purely algebraic.

For every (local) \ca{} $A$, the classical Cuntz semigroup $W(A)$ together with a natural auxiliary relation is an object of $\CatW$.
This defines a functor from \ca{s} to $\CatW$ which preserves inductive limits.
We deduce that the assignment $A\mapsto\Cu(A)$ defines a functor from \ca{s} to $\CatCu$ which preserves inductive limits.
This generalizes a result from \cite{CowEllIva08CuInv}.

We also develop a theory of $\CatCu$-semirings and their semimodules.
The Cuntz semigroup of a strongly self-absorbing \ca{} has a natural product giving it the structure of a $\CatCu$-semiring.
For \ca{s}, it is an important regularity property to tensorially absorb a strongly self-absorbing \ca{}.
Accordingly, it is of particular interest to analyse the tensor products of $\CatCu$-semigroups with the $\CatCu$-semiring of a strongly self-absorbing \ca{}.
This leads us to define `solid' $\CatCu$-semirings (adopting the terminology from solid rings), as those $\CatCu$-semirings $S$ for which the product induces an isomorphism between $S\otimes_\CatCu S$ and $S$.
This can be considered as an analog of being strongly self-absorbing for $\CatCu$-semirings.
As it turns out, if a strongly self-absorbing \ca{} satisfies the UCT, then its $\CatCu$-semiring is solid.
We prove a classification theorem for solid $\CatCu$-semirings.
This raises the question of whether the Cuntz semiring of every strongly self-absorbing \ca{} is solid.

If $R$ is a solid $\CatCu$-semiring, then a $\CatCu$-semigroup $S$ is a semimodule over $R$ if and only if $R\otimes_{\CatCu}S$ is isomorphic to $S$.
Thus, analogous to the case for \ca{s}, we can think of semimodules over $R$ as $\CatCu$-semigroups that tensorially absorb $R$.
We give explicit characterizations when a $\CatCu$-semigroup is such a semimodule for the cases that $R$ is the $\CatCu$-semiring of a strongly self-absorbing \ca{} satisfying the UCT.
For instance, we show that a $\CatCu$-semigroup $S$ tensorially absorbs the $\CatCu$-semiring of the Jiang-Su algebra if and only if $S$ is almost unperforated and almost divisible, thus establishing a semigroup version of the Toms-Winter conjecture.
\end{abstract}

\maketitle

\tableofcontents

\mainmatter

%==========================================================================================
%==========================================================================================
\chapter{Introduction}
\label{sec:intro}

%------------------------------------------------------------------------------------------
%==========================================================================================
\section{Background}

%------------------------------------------------------------------------------------------
This paper is concerned with a number of regularity properties of Cuntz semigroups, which are invariants naturally associated to \ca{s}.
To put our results into perspective, we first review the $C^*$-motivation behind our work, as well as the importance of these semigroups in the context of the Elliott classification program.

%------------------------------------------------------------------------------------------
%==========================================================================================
\subsection{The Elliott classification program}
\label{subsct:classification}
The Cuntz semigroup $\W(A)$ of a \ca{} $A$ is an important invariant in the structure theory of \ca{s}, particularly in connection with the classification program of simple, nuclear \ca{s} initiated by George Elliott.
In itself, nuclearity is a finite-dimensional approximation property that includes a large number of our stock-in-trade \ca{s}.

The original Elliott Conjecture asserts that simple, separable, unital, nuclear \ca{s} can be classified by an invariant $\mathrm{Ell}(\freeVar)$ of a $K$-theoretic nature.
Without going into much detail, the Elliott invariant $\mathrm{Ell}(A)$ for a \ca{} $A$ consists of the ordered $K_0$-group, the topological $K_1$-group, the trace simplex $T(A)$ and the (natural) pairing between traces and projections.

\begin{conj}[Elliott's Classification Conjecture]
\index{terms}{Elliott's Classification Conjecture}
For \ca{s} $A$ and $B$ as above, we have $\mathrm{Ell}(A)\cong\mathrm{Ell}(B)$ if and only if $A\cong B$.
\end{conj}

The Elliott program has had tremendous success in the classification of wide classes of algebras (see for example \cite{Ror92StructureUHF2} and \cite{EllTOm08Regularity}).
However, the first counterexamples to the conjecture as stated above appeared in the work of R{\o}rdam (\cite{Ror03FinInfProj}) and Toms (\cite{Tom05IndependenceK-SR}).
Both examples allowed to repair the conjecture by adding a minimal amount of information to the invariant (in this case, the real rank).
Soon after that though, Toms produced in \cite{Tom08ClassificationNuclear} two simple AH-algebras that agreed not only on the Elliott invariant, but also on a whole collection of topological invariants (among them the real and stable rank).

The distinguishing factor for the said algebras is precisely the Cuntz semigroup $W(A)$.
This is an object that was introduced by Cuntz in \cite{Cun78DimFct} as equivalence classes of positive elements in matrices over a \ca{} $A$, in very much the same way the semigroup $V(A)$ (as a precursor of $K_0(A)$) is constructed via Murray-von Neumann equivalence classes of projections in matrices over $A$. For a large class of simple \ca{s} (see \ref{subsc:ssa}), the Elliott invariant and the Cuntz semigroup, suitably interpreted, determine one another in a functorial way.

One of the key features of the Cuntz semigroup is its ordering, which is in general not algebraic.
As a matter of fact, it is one order property -- almost unperforation -- that is used to distinguish the algebras mentioned above.

Many of the classes of algebras considered in the classification program admit an inductive limit decomposition, and hence it is desirable that any addition to the original Elliott invariant behaves well with respect to inductive limits.
This is not the case of the Cuntz semigroup, when considered as an invariant from the category of \ca{s} to the category of semigroups.
This shortcoming can be remedied by passing to stable algebras and considering as a target category a suitable category $\CatCu$ of ordered semigroups (see below).
This was carried out by Coward, Elliott and Ivanescu in \cite{CowEllIva08CuInv}, where they defined $\Cu(A)$ using Hilbert modules (and showed it is naturally isomorphic to $W(A\otimes K)$).
In this way, the assignment $A\mapsto \Cu(A)$ defines a sequentially continuous functor.

To this date, there is no counterexample to the conjecture of whether the Elliott invariant, together with the Cuntz semigroup, constitutes a complete invariant for the class of unital, simple, separable, nuclear \ca{s}.
It is therefore natural to ask what is the largest possible class for which the Elliott Conjecture can be proved to hold.

It is important to point out that the Cuntz semigroup alone has become a useful tool in the classification of certain classes of nonsimple algebras.
A remarkable instance of this situation is found in the work of Robert in \cite{Rob12LimitsNCCW}, where the Cuntz semigroup is used to classify, up to approximate unitary equivalence, \starHom{s} out of an inductive limit of $1$-dimensional noncommutative CW-complexes with trivial $K_1$-groups into a stable rank one algebra.
As a consequence, Robert classifies all (not necessarily simple) inductive limits of $1$-dimensional NCCW-complexes with trivial $K_1$-groups using the Cuntz semigroup.

%------------------------------------------------------------------------------------------
%==========================================================================================
\subsection{Strongly self-absorbing \texorpdfstring{\ca{s}}{C*-algebras}. \texorpdfstring{$\mathcal Z$}{Z}-stability}
\label{subsc:ssa}

Toms and Winter, \cite[Definition~1.3]{TomWin07ssa}, termed a \ca{} $D$ strongly self-absorbing if $D\neq\C$ and if there is an isomorphism $\varphi\colon D\to D\otimes D$ that is approximately unitarily equivalent to the inclusion in the first factor (or, as it turns out, in the second).
\index{terms}{C*-algebra@\ca{}!strongly self-absorbing}
\index{terms}{strongly self-absorbing \ca{}}
\index{terms}{Jiang-Su algebra}
Such \ca{s} are automatically simple, nuclear (Effros-Rosenberg), and are either purely infinite or stably finite with a unique trace (Kirchberg).
The only known examples of strongly self-absorbing \ca{s} are:
The Cuntz algebras $\mathcal{O}_2$ and $\mathcal{O}_\infty$, every UHF-algebra of infinite type, the tensor products $\mathcal{O}_\infty\otimes U$ where $U$ is a UHF-algebra of infinite type, and the Jiang-Su algebra $\mathcal{Z}$.
All these algebras satisfy the Universal Coefficient Theorem (UCT).

The Elliott classification program predicts that they are in fact the only strongly self-absorbing \ca{s} satisfying the UCT. This has been verified very recently in \cite{TikWhiWin15}.
It remains an important open problem to determine whether there is a strongly self-absorbing \ca{} outside the UCT class, as this would provide a nuclear, non-UCT \ca.
If $D$ is strongly self-absorbing, a \ca{} $A$ is called $D$-stable provided that $A\cong A\otimes D$.
Winter showed that all strongly self-absorbing \ca{s} are $\mathcal{Z}$-stable (see \cite{Win11ssaZstable}, \cite{DadRor09ssa}).
It follows from this result that $\mathcal{Z}$ is an initial object in the category of strongly self-absorbing \ca{s}.

The Jiang-Su algebra $\mathcal{Z}$ has the same Elliott invariant as the complex numbers, and it has become prominent in the classification program.
In fact, tensoring a \ca{} with $\mathcal{Z}$ is inert at the level of $K$-theory and traces (although it may change the order of the $K_0$-group, except under some additional assumptions).
It is thus reasonable to expect that classification can be achieved within the class of simple, separable, unital, nuclear, $\mathcal{Z}$-stable algebras.
In this way, $\mathcal{Z}$-stability postulates itself as a regularity property for \ca{s}. For the said algebras that furthermore fall in the UCT class and have finite nuclear dimension, classification is now complete by the work of many authors (see \cite{GonLinNiu15}, \cite{EllGonLinNiu15}, \cite{TikWhiWin15} and the references therein).

All classes of simple, nuclear \ca{s} for which the Elliott Conjecture has been verified consist of $\mathcal{Z}$-stable \ca{s}; see \cite{TomWin08ASH}.
One may therefore wonder what role the Cuntz semigroup plays in these results, if any.
As proved in \cite{AntDadPerSan14RecoverElliott} (see also \cite{Tik11CuFunctions}, \cite{BroPerTom08CuElliottConj}), for the class of unital, simple, separable, nuclear and $\mathcal{Z}$-stable \ca{s} (this is the class alluded to above), the Elliott invariant and the Cuntz semigroup of any such algebra tensored with the circle define naturally equivalent functors.
Thus, for this class, $\mathrm{Ell}(\freeVar)$ is a classifying functor if and only if so is $\Cu(C(\T,\freeVar))$.

%------------------------------------------------------------------------------------------
%==========================================================================================
\subsection{The regularity Conjecture}

This conjecture, which is also known as the Toms-Winter conjecture (see \cite[Remarks~3.5]{TomWin09Villadsen} and \cite[Conjecture~0.1]{Win12NuclDimZstable}) links three seemingly unrelated regularity properties that a simple, separable, nuclear, nonelementary \ca{} $A$ may enjoy.
The first of these properties is that $A$ has finite nuclear dimension.
We will not define nuclear dimension here.
Instead let us just say that it is a strengthening of the definition of nuclearity that uses completely positive order-zero maps (that is, completely positive maps that preserve orthogonality of elements).

The second regularity property is $\mathcal{Z}$-stability, and the third one is strict comparison of positive elements, which may be roughly stated by saying that comparison of positive elements (modulo Cuntz subequivalence) is determined by the states on the Cuntz semigroup.
This is equivalent to saying that the Cuntz semigroup is almost unperforated.

\begin{conj}[Toms-Winter]
\index{terms}{Regularity Conjecture}\index{terms}{Toms-Winter Conjecture}
Let $A$ be a simple, separable, nuclear, nonelementary \ca.
Then the following conditions are equivalent:
\beginEnumConditions
\item
The \ca{} $A$ has finite nuclear dimension.
\item
The \ca{} $A$ is $\mathcal{Z}$-stable.
\item
The Cuntz semigroup $\W(A)$ is almost unperforated.
\end{enumerate}
\end{conj}

Important work by many authors has led to a partial confirmation of this conjecture:
Winter proved in \cite{Win12NuclDimZstable} that~(i) implies~(ii), and R{\o}rdam showed in \cite{Ror04StableRealRankZ} that~(ii) implies~(iii).
If $A$ has no tracial states then it is purely infinite and the conjecture has been confirmed in that case using \cite{KirPhi00Crelle1}, \cite{Ror04StableRealRankZ}, and \cite{MatSat14}.
It was shown that~(iii) implies~(ii) in the case that $T(A)$ is a Bauer simplex whose extreme boundary is finite-dimensional (see \cite{KirRor12arXivCentralSeq}, \cite{Sat12TraceSpace}, and \cite{TomWhiWin12arXivZStable}, and the precursor \cite{MatSat12StrComparison}).
It was proved that~(ii) implies~(i) in the case that $A$ has a unique tracial case (\cite{SatWhiWin14arXivNuclDim}), which was very recently generalized to the case that $T(A)$ is any Bauer simplex; see \cite{BosBroSatTikWhiWin14}.
It follows that the Toms-Winter conjecture is verified for all \ca{s} $A$ such that $T(A)$ is a Bauer simplex with finite-dimensional extreme boundary.

For every Elliott invariant of a simple, separable, nuclear $\mathcal{Z}$-stable \ca{} there exists a model given as an inductive limit whose building blocks are either Cuntz algebras over the circle, or subhomogeneus algebras whose primitive ideal spaces have dimension at most $2$.
Therefore, if classification and the regularity conjecture hold we would get deep insight into the structure of simple, nuclear \ca{s}.

%------------------------------------------------------------------------------------------
%==========================================================================================
\section{The categories \texorpdfstring{$\CatW$}{W} and \texorpdfstring{$\CatCu$}{Cu}}
\label{subsect:W and Cu}

%------------------------------------------------------------------------------------------
As we have mentioned in \ref{subsct:classification}, Coward, Elliott and Ivanescu introduced a category of ordered semigroups $\CatCu$ such that $\Cu(A)$ is an object in $\CatCu$ for every \ca{} $A$.
The axioms defining this category capture the continuous nature of the Cuntz semigroup.
The first axiom asks for every increasing sequence to admit an order-theoretic supremum, while the second axiom requires that every element $a$ is the supremum of a sequence $(a_n)_n$ such that $a_n\ll a_{n+1}$ for each $n$.
(See \autoref{pgr:axiomsW} for the definition of `$\ll$' and further details).
Given a \ca{} $A$, a positive element $a$ in $A$ and $\varepsilon>0$, one always has that $[(a-\varepsilon)_+]\ll [a]$ in $\Cu(A)$.
A projection $p$ in $A$ always satisfies $[p]\ll [p]$.
The elements $s$ satisfying $s\ll s$ play an important role and are termed \emph{compact}.
We may think with advantage that they are equivalence classes of projections.
The third and fourth axioms express compatibility between addition, suprema, and the relation~$\ll$.

It is then natural to ask what continuity properties are already reflected in $W(A)$ and how $\Cu(A)$ is obtained out of them.
Attempts in this direction may be found in \cite{AntBosPer11CompletionsCu}.

We introduce here a new category of semigroups $\CatW$ parallel to the category $\CatCu$ and show that $W(A)$ is an object of this category.
One of the key ingredients is that the objects in $\CatW$ are semigroups equipped with an additional relation, sufficiently compatible with addition, referred to as an auxiliary relation (\cite{GieHof+03Domains}).
We show that $W(A)$ can be endowed with such a relation; this is also the case with $\Cu(A)$, as was already noted in \cite{CowEllIva08CuInv}, where one takes $\ll$ as an auxiliary relation.
Another ingredient in our approach consists of considering the larger category $\CatLocCa$ of local \ca{s}.
Essentially, these are pre-\ca{s} that admit functional calculus on finite sets of positive elements.

We give a full picture of the exact relation between the functors $W(\freeVar)$ and $\Cu(\freeVar)$ (see \autoref{prp:limitsW}, and Theorems~\ref{prp:functorW}, \ref{prp:CureflectivePreW},~\ref{prp:commutingFctrs}).

%==========================================================================================
\begin{thm*} The following conditions hold true:
\beginEnumConditions
\item
The category $\CatW$ admits arbitrary inductive limits and the assignment $A\mapsto W(A)$ defines a continuous functor from the category $\CatLocCa$ of local \ca{s} to the category $\CatW$.
\item
The category $\CatCu$ is a full, reflective subcategory of $\CatW$.
Therefore, $\CatCu$ also admits arbitrary inductive limits.
\item
There is a diagram, that commutes up to natural isomorphisms:
\[
\xymatrix@R=15pt@M+3pt{
\CatLocCa \ar@/_1pc/[d]_{\gamma} \ar[r]^{\W}
& \CatW \ar@/^1pc/[d]^{\gamma} \\
\CatCa \ar@{^{(}->}[u] \ar[r]^{\Cu}
& \CatCu \ar@{^{(}->}[u] \\
}
\]
where $\gamma\colon\CatW\to\CatCu$ is the reflection functor and $\gamma\colon\CatLocCa\to \CatCa$ is the completion functor that assigns to a local \ca{} its completion (which is a \ca{}).
In particular, the assignment $A\mapsto\Cu(A)$ is also a continuous functor from the category of \ca{s} to the category $\CatCu$.
\end{enumerate}
\end{thm*}

Notice that condition (i) above sets up the right framework for the functor $\W$ to be continuous, by enlarging the source category to $\CatLocCa$ and identifying the range category $\CatW$.
Condition (iii) generalizes \cite[Theorem~2]{CowEllIva08CuInv} from sequential to arbitrary inductive limits.

A key concept in the proof is that of a $\Cu$-completion of a semigroup $S$ in the category $\CatW$.
This may be thought of as a pair $(T,\alpha)$, where $T\in \CatCu$ and $\alpha\colon S\to T$ is a morphism that, suitably interpreted, is an embedding with dense image.

%------------------------------------------------------------------------------------------
%==========================================================================================
\subsection{The range problem. Additional axioms}
\index{terms}{Range problem}
It is an important problem to determine which semigroups in the category $\CatCu$ come as Cuntz semigroups of \ca{s}.
For example, we know that, for any finite dimensional, compact Hausdorff space, the semigroup $\Lsc(X,\N\cup\{\infty\})$ of lower semicontinuous functions is an object of $\CatCu$ (\cite{AntPerSan11PullbacksCu}), but if fails to be the Cuntz semigroup of a \ca{} whenever the dimension of $X$ is larger than $2$ (\cite{Rob13CuSpaces2D}).

There are two additional axioms that the Cuntz semigroup of any \ca{} satisfies, and which are not derived from the original set of axioms used to define the category $\CatCu$.
The first of such axioms was established by R{\o}rdam and Winter (\cite{RorWin10ZRevisited}) and indicates how far the partial order in $\Cu(A)$ is from being algebraic.
It is usually referred to as the \emph{almost algebraic order} axiom.
Given $a',a$ and $b$, the axiom says:
\[
\text{If }a'\ll a\leq b, \text{ then there is }c\text{ such that }a'+c\leq b\leq a+c.
\]
It is worth pointing out that, if $a\ll a$, then the above implies that whenever $a\leq b$, there is an element $c$ with $a+c=b$.
Thus, this axiom is a generalization of the fact that the order among Cuntz classes of projections is algebraic.

The second axiom was established by Robert (\cite{Rob13Cone}) and is a condition of a Riesz decomposition type, usually referred to as the \emph{almost Riesz decomposition} axiom.
Given $a',a,b$ and $c$, the axiom reads as follows:
\[
\text{If }a'\ll a\leq b+c, \text{ then there are } b'\leq b,a \text{ and } c'\leq c,a \text{ such that } a'\leq b'+c'.
\]

In \autoref{dfn:addAxioms}, we introduce a strengthening of the almost algebraic order axiom, and we prove that it is satisfied by the Cuntz semigroup $\Cu(A)$ of any \ca{} $A$.
It is equivalent to the original formulation if the semigroup is weakly cancellative, that is, if $a+c\ll b+c$ implies $a\ll b$.
(As shown in \cite{RorWin10ZRevisited}, $\Cu(A)$ is weakly cancellative when $A$ has stable rank one.) With this new formulation, the axiom passes to inductive limits.

We also introduce corresponding versions of these axioms for the category $\CatW$ and show that they are satisfied by $W(A)$, for any local \ca{} $A$.
The $\Cu$-completion process, as described above, relates exactly each one of the $\CatW$-axioms with its $\CatCu$-counterpart.
All the axioms considered pass to inductive limits.

Although it may be premature to recast the category $\CatCu$ by adding the axioms of almost algebraic order and almost Riesz decomposition (as new axioms may emerge in the near future), it is quite pertinent to add them to our basket of assumptions in many results of the paper.

%------------------------------------------------------------------------------------------
%==========================================================================================
\subsection{Softness and pure noncompactness}

While compact elements in a $\CatCu$-semigroup may be thought of as `projections', the class of purely noncompact elements can be placed at the other end of the scale, that is, as far as possible from projections.
This may be phrased by saying that the element in question only becomes compact in a quotient when it is zero or properly infinite.
It was shown by Elliott, Robert, and Santiago that the purely noncompact elements in $\Cu(A)$ are, in the almost unperforated case, the ones that can be compared by traces, \cite{EllRobSan11Cone}.

It is natural to seek for a result of this nature in the framework of the category $\CatCu$ alone.
For this, given a $\CatCu$-semigroup $S$, we need to consider the set $F(S)$ of functionals on $S$, that is, extended states on $S$ that respect suprema of increasing sequences.
Note that, in this way, any element $a$ in $S$ induces a linear, lower semicontinuous, $[0,\infty]$-valued function $\hat{a}$ on $F(S)$ by evaluation.
It is to be noted that $F(\Cu(A))$ is homeomorphic to the trace simplex of non-normalized traces on $A$ (when $A$ is exact), as shown in \cite[Theorem 4.4]{EllRobSan11Cone}.
Indeed, given a trace $\tau$, its corresponding functional $d_\tau$ maps the class $[x]$ of a positive element $x$ in $A\otimes\K$ to $\lim_n\tau(x^{1/n})$.

The key notion in the abstract setting of $\CatCu$-semigroups is that of a soft element.
By definition, $a\in S$ is soft if any $a'\ll a$ satisfies $(n+1)a'\leq na$ for some $n$.
As it turns out, the subset $S_\soft$ of soft elements is a submonoid of $S$.
If $S$ is furthermore simple and stably finite, then $S_\soft$ is also a $\CatCu$-semigroup.
Our definition of softness is inspired by \cite{GooHan82Stenosis}.

For almost unperforated $\CatCu$-semigroups, soft elements are the ones whose comparison theory is completely determined by functionals.
Namely, given $a$, $b\in S$ with $a$ soft, then $a\leq b$ precisely when $\hat{a}\leq \hat{b}$.
This generalizes \cite[Theorem~6.6]{EllRobSan11Cone}.
It is worth mentioning that, in the presence of the almost algebraic order axiom, softness is equivalent to a suitable weakening of pure noncompactness.
The concept of softness is, however, easier to state and to use.

%------------------------------------------------------------------------------------------
%==========================================================================================
\subsection{Algebraic semigroups}

A particularly interesting class of $\CatCu$-se\-mi\-groups is that of algebraic semigroups.
These are $\CatCu$-semigroups where the compact elements are dense, and they are modelled after \ca{s} of real rank zero, where the structure of projections determines a great deal of the structure of the algebra.

We show that this is also the case at the semigroup level.
Of particular significance is the fact that axioms of interest have a translation into properties of the compact elements (see \autoref{prp:propertiesAlgebraic}):

%==========================================================================================
\begin{thm*}
Let $S$ be an algebraic $\CatCu$-semigroup, and let $S_c$ be the submonoid of compact elements.
Then:
\beginEnumConditions
\item
The $\CatCu$-semigroup $S$ satisfies the axiom of almost algebraic order if and only if $S_c$ is algebraically ordered.
\item
The $\CatCu$-semigroup $S$ is weakly cancellative if and only if $S_c$ is a cancellative semigroup.
\item
If $S_c$ has Riesz decomposition, then $S$ satisfies the axiom of almost Riesz decomposition.
Conversely, if $S$ satisfies the axioms of almost algebraic order and almost Riesz decomposition and is weakly cancellative, then $S_c$ has Riesz decomposition.
\end{enumerate}
\end{thm*}

%------------------------------------------------------------------------------------------
%==========================================================================================
\subsection{Near unperforation}

The notion of near unperforation allows us to analyse almost unperforation from a new perspective. A \pom{} $S$ is \emph{nearly unperforated} if $a\leq b$ whenever $2a\leq 2b$ and $3a\leq 3b$.
(This is not our original definition, but a useful restatement.)
Nearly unperforated semigroups are always almost unperforated, so this becomes a strengthened notion.
In the simple case, a converse is available (see \autoref{prp:nearUnperfSimpleCu}):

%==========================================================================================
\begin{thm*}
Let $S$ be a simple, stably finite $\CatCu$-semigroup that satisfies the almost algebraic order axiom. Then $S$ is nearly unperforated precisely when it is almost unperforated and weakly cancellative.
\end{thm*}

R\o rdam proved in \cite{Ror04StableRealRankZ} that a simple, $\mathcal{Z}$-stable \ca{} is either purely infinite or has stable rank one.
He also showed that the Cuntz semigroup of every $\mathcal{Z}$-stable \ca{} is almost unperforated. While there are examples of simple, almost unperforated $\CatCu$-semigroups that fail to be weakly cancellative, the above theorem shows that this is no longer the case for nearly unperforated semigroups.

Thus, as a corollary, if $A$ is a simple $\mathcal{Z}$-stable \ca, $\Cu(A)$ is nearly unperforated. We have also verified that the Cuntz semigroup of (not necessarily simple) $\mathcal{Z}$-stable \ca{s} is nearly unperforated in a variety of situations. For example, if further $A$ has real rank zero and stable rank one, or also if $A$ has no $K_1$-obstructions (see \autoref{sec:nearUnp}). It then remains an interesting open problem to decide whether $\Cu(A)$ is always nearly unperforated for any $\mathcal{Z}$-stable \ca{} $A$.
We conjecture this is always the case.

%------------------------------------------------------------------------------------------
%==========================================================================================
\section{Tensor products}

Tensor products with the Jiang-Su algebra or, more generally, by a strongly self-absorbing \ca, are of particular relevance in connection with the current status of the classification program.
%Besides the fact that tensor products are interesting in themselves!!!

The tensor product construction at the level of (ordered) semigroups has a long tradition (for completeness, we have included a review of the necessary definitions and results in \autoref{sec:appendixPom}).
It is therefore a very natural question to ask how $\Cu(A\otimes B)$ and $\Cu(A)\otimes\Cu(B)$ are related.
The first step towards a solution to this question resides in carrying out a construction of the tensor product within the category $\CatCu$, so as to `equip' the usual semigroup tensor product with the necessary continuous structure.

Our approach has a categorical flavor, and at the same time allows for computations of examples.
A central notion is that of a bimorphism $\varphi\colon S\times T\to R$, that is, a biadditive map that is required to satisfy certain additional conditions depending on the category where the objects $S$, $T$ and $R$ live.
Thus, for example, if we focus on the category $\CatCu$, we speak of $\CatCu$-bimorphisms and we shall be asking that $\varphi$ is continuous in each variable (that is, preserves suprema of increasing sequences) and is jointly preserving the relation $\ll$, that is, $s'\ll s$ and $t'\ll t$ imply $\varphi(s',t')\ll\varphi(s,t)$.
This requirement breaks away from the usual conventions in existing (semigroup) tensor products, and makes our construction technically more demanding.
One asks a tensor product in $\CatCu$ of $S$ and $T$ to be a pair $(Q,\varphi)$, where $Q$ is an object in $\CatCu$ and $\varphi\colon S\times T\to Q$ is a $\CatCu$-bimorphism with certain universal properties involving different types of morphisms.

We can also regard the tensor product as an object that represents the bimorphism bifunctor $\CatCuBimor(S\times T,\freeVar)$.
We use $\CatCuMor(\freeVar,\freeVar)$ and $\CatWMor(\freeVar,\freeVar)$ below to denote the corresponding morphism sets, which are positively ordered semigroups in a natural way.
We prove (see \autoref{prp:tensProdCu}):

%==========================================================================================
\begin{thm*}
Let $S$ and $T$ be $\CatCu$-semigroups.
There is a $\CatCu$-semigroup $S\otimes_{\CatCu}T$ and a $\CatCu$-bimorphism $\varphi\colon S\times T\to S\otimes_{\CatCu}T$ such that the pair ($S\otimes_{\CatCu} T$, $\varphi$) represents the bimorphism functor $\CatCuBimor(S\times T,\freeVar)$ that takes values in the category of positively ordered semigroups.
Thus, for every $\CatCu$-semigroup $R$, the $\CatCu$-bimorphism $\varphi$ induces a positively ordered semigroup isomorphism of the following (bi)morphism sets:
\[
\CatCuMor(S\otimes_{\CatCu} T,R)\to \CatCuBimor(S\times T,R).
\]
\end{thm*}

In outline, the construction of the object $S\otimes_{\CatCu}T$ in the theorem above uses the reflection functor $\gamma\colon\CatW\to \CatCu$ as described in \ref{subsect:W and Cu}, and so the tensor product in $\CatCu$ comes as a completion of the corresponding object in $\CatW$.
In fact, recalling that $\CatCu$ is a reflective subcategory of $\CatW$, we have (see \autoref{thm:tensProdCompl}):

%==========================================================================================
\begin{thm*}
Let $S$ and $T$ be semigroups in the category $\CatW$.
There is then a $\CatW$-semigroup $S\otimes_{\CatW}T$ and a $\CatW$-bimorphism that, for every $\CatCu$-semigroup $R$, induces a commutative diagram where every row and column are semigroup isomorphisms:
\[
\xymatrix{
\CatW(S\otimes_{\CatW} T,R) \ar[r]_{}^{\cong}
& \CatWBimor(S\times T,R) \\
\CatCuMor(\gamma(S\otimes_{\CatW} T),R) \ar[u]_{}_{\cong} \ar[r]^<<<<{\cong}
& \CatCuBimor(\gamma(S)\times \gamma(T),R). \ar[u]^{\cong}
}
\]
In particular, we can identify $\gamma(S)\otimes_{\CatCu}\gamma(T)$ with $\gamma(S\otimes_{\CatW}T)$.
\end{thm*}

Applied to \ca{s}, the results above yield:

%==========================================================================================
\begin{thm*}
The following conditions hold true:
\beginEnumConditions
\item
Let $A$ and $B$ be local \ca{s}.
Then
\[
\Cu(A)\otimes_{\CatCu}\Cu(B)=\gamma(W(A)\otimes_{\CatW}W(B)).
\]
\item
Let $D$ be a strongly self-absorbing \ca\ of real rank zero that satisfies the UCT.
Then
\[
\Cu(A\otimes D)\cong \Cu(A)\otimes_\CatCu\Cu(D)\,,
\]
for any \ca{} $A$.
\end{enumerate}
\end{thm*}

%------------------------------------------------------------------------------------------
%==========================================================================================
\section{Multiplicative structure of \texorpdfstring{$\CatCu$}{Cu}-semigroups. Solid \texorpdfstring{$\CatCu$}{Cu}-semirings}

As noted in \ref{subsc:ssa}, the class of $D$-stable \ca{s}, where $D$ is strongly self-absorbing, is relevant for the theory, and thus a description of their Cuntz semigroup is of particular interest.
Towards this end, we identify which semigroups should play the role of strongly self absorbing \ca{s}.
If $D$ is such an algebra, then the isomorphism $D\otimes D\cong D$ induces a $\CatCu$-bimorphism $\Cu(D)\times \Cu(D)\to \Cu(D)$, which in turn can be used to equip $\Cu(D)$ with a unital semiring structure (see \autoref{sec:poRg} for the definition of a semiring).
This is also compatible with the continuous properties of $\Cu(D)$ and leads us to introduce the notion of a $\CatCu$-semiring.
In the case of the Jiang-Su algebra $\mathcal{Z}$, its Cuntz semigroup may be identified with $Z:=\N\sqcup (0,\infty]$, where the product is the obvious one in either $\N$ or $(0,\infty]$ and mixed terms multiply into $(0,\infty]$.

In a similar vein, if $A$ is a $D$-stable \ca, there is a natural $\CatCu$-bimorphism $\Cu(D)\times \Cu(A)\to \Cu(A)$ which is moreover compatible with the multiplicative structure of $\Cu(D)$.
This leads us to define the notion of a $\CatCu$-semimodule $S$ over a $\CatCu$-semiring $R$.
We refer to this situation by saying that $S$ has an $R$-multiplication.

Of particular importance is the structure of $\CatCu$-semimodules over semirings that come from strongly self-absorbing algebras, or from the Jacelon-Razak algebra $\mathcal R$, whose Cuntz semigroup is $[0,\infty]$ (see \cite{Jac13Projectionless} and also \cite{Rob13Cone}).
As Robert points out for $\Cu(\mathcal R)$ (see \cite{Rob13Cone}), having a $\Cu(\mathcal R)$-multiplication is in fact a property of the semigroup rather than an additional structure.
Denote by $R_q$ the Cuntz semigroup of a UHF-algebra of infinite type (and supernatural number $q$).
We then prove the following (see Theorems~\ref{prp:ZModTFAE}, \ref{prp:RqModTFAE}, \ref{prp:RModTFAE}, and \ref{prp:PIModTFAE}):

%==========================================================================================
\begin{thm*}
Let $S$ be a $\CatCu$-semigroup.
Then:
\beginEnumConditions
\item
The $\CatCu$-semigroup $S$ has $Z$-multiplication if and only if $S$ is almost divisible and almost unperforated.
\item
The $\CatCu$-semigroup $S$ has $R_q$-multiplication if and only if $S$ is $p$-divisible and $p$-unperforated whenever $p$ is an integer that divides $q$.
\item
The $\CatCu$-semigroup $S$ has $[0,\infty]$-multiplication if and only if $S$ is unperforated, divisible and every element of $S$ is soft.
\item
The $\CatCu$-semigroup $S$ has $\{0,\infty\}$-multiplication if and only if $2x=x$ for every $x\in S$.
\end{enumerate}
\end{thm*}

Condition (i) above allows us to prove a semigroup version of the Toms-Winter conjecture:

%==========================================================================================
\begin{thm*}
Let $S$ be a $\CatCu$-semigroup.
Then the following are equivalent:
\beginEnumConditions
\item
We have $S\cong Z\otimes_\CatCu S$.
\item
The $\CatCu$-semigroup $S$ is almost unperforated and almost divisible.
\end{enumerate}
\end{thm*}

A key ingredient in the above theorem is the fact that $Z\otimes_\CatCu Z\cong Z$, where the isomorphism is induced by the natural product.
This naturally poses the question of which is the right notion for a `strongly self-absorbing $\CatCu$-semigroup'.
We adopt here the terminology of a solid ring, as introduced in \cite{BouKan72Core}, and call a unital $\CatCu$-semiring $R$ \emph{solid} if the multiplication induces an isomorphism $R\otimes_{\CatCu}R\cong R$.
Every such semiring is automatically simple and, in the stably finite case, has a unique normalized functional.

%==========================================================================================
\begin{thm*}
Let $D$ be a strongly self-absorbing \ca{} satisfying the UCT.
Then $\Cu(D)$ is a solid $\CatCu$-semiring, and so $\Cu(D)\otimes_\CatCu\Cu(D)\cong\Cu(D)$.
\end{thm*}

As solid $\CatCu$-semirings have good structural properties, it is natural to analyse the tensor product of a $\CatCu$-semigroup with one of these semirings.
This process may be termed a \emph{regularization}, as the final object enjoys regularity properties (for example, it absorbs the $\CatCu$-semiring $Z$ tensorially).
We explore two such constructions, closely related to \ca{s}: the rationalization and the realification of a semigroup.

The \emph{rationalization} of a $\CatCu$-semigroup $S$ is, by definition, its tensor product with a semigroup of the form $R_q$, where $R_q$ is, as mentioned above, the Cuntz semigroup of a UHF-algebra of infinite type, so that $q=(p_i)$ is a supernatural number of infinite type.
The tensor product $R_q\otimes_\CatCu S$ can be realized as the inductive limit $S_q$ constructed as $S \stackrel{\cdot p_1}{\longrightarrow} S \stackrel{\cdot p_2}{\longrightarrow}  S \stackrel{\cdot p_3}{\longrightarrow} \dots$.

Given a $\CatCu$-semigroup $S$, Robert introduced in \cite{Rob13Cone} the \emph{realification} of $S$, which is a $\CatCu$-semigroup denoted by $S_R$.
Robert indicates that his construction may be thought of as the tensor product with $[0,\infty]$.
This semigroup is, by definition, the subsemigroup of lower semicontinuous, linear, $[0,\infty]$-valued functions defined on $F(S)$ that can be obtained as pointwise suprema of functions of the type $\frac{1}{n}\hat{s}$, for $s\in S$.
Robert obtains in \cite[Theorem~3.2.1]{Rob13Cone} a more abstract characterization of $S_R$.
We make the connection of $S_R$ with the tensor product construction precise and we show that we indeed have $S_R\cong [0,\infty]\otimes_\CatCu S$.
It then follows from our results and \cite[Theorem~5.1.2]{Rob13Cone} that $\Cu(\mathcal{R}\otimes A)\cong\Cu(\mathcal{R})\otimes_\CatCu\Cu(A)$ for any \ca{} $A$.

In contrast to strongly self-absorbing \ca{s}, where a complete classification is not available yet, solid $\CatCu$-semirings do admit a very satisfactory classification, as follows (see \autoref{thm:solidSemirgClassification}):

\begin{thm*}
Let $S$ be a nonzero $\CatCu$-semiring that satisfies the almost algebraic order and the almost Riesz decomposition axioms.
Then exactly one of the following statements hold:
\beginEnumConditions
\item
We have $S\cong\{0,1,\ldots, k,\infty\}$ for some $k\geq 0$.
\item
We have $S\cong \overline{\N}$.
\item
We have $S\cong [0,\infty]$.
\item
We have $S\cong Z$.
\item
We have that $S$ is algebraic, and there exists a solid ring $R\not\cong\Z$ with non-torsion unit such that $S_c=R_+$.
\end{enumerate}
\end{thm*}

Since there is a classification theory for solid rings (see \cite{BouKan72Core}, \cite{BowSch77Rings}), it is therefore possible to give a complete list of solid $\CatCu$-semirings.

Now denote by $Q=\Q_+\sqcup (0,\infty]$ the Cuntz semigroup of the universal UHF-algebra.
As a consequence of our classification theorem, we obtain that $Z$ and $Q$ can be (uniquely) characterized as initial and final objects in the category of nonelementary, solid $\CatCu$-semirings satisfying the almost algebraic order axiom, and whose unit is compact.
Likewise, $[0,\infty]$ is the unique solid $\CatCu$-semiring that contains no nonzero compact elements.
This is an exact parallell of Winter's result that strongly self-absorbing \ca{s} are $\mathcal{Z}$-stable and in fact, our methods allow us to recover this result.

It would be interesting to know whether the Cuntz semigroup of any strongly self-absorbing \ca{} $D$ is always solid.
This is clearly the case if $D$ is purely infinite simple.
As $D$ is $\mathcal{Z}$-stable by Winter's result, \cite{Win11ssaZstable}, and monotracial in the stably finite case, it follows that $\Cu(D)$ may be identified with $V(D)\sqcup (0,\infty]$.
With our classification theorem of solid $\CatCu$-semirings at hand, this would shed light on whether there could exist a non-UCT strongly self-absorbing example.

%------------------------------------------------------------------------------------------
%==========================================================================================
\section*{Acknowledgments}

%==========================================================================================
We thank Joan Bosa, Nate Brown, Joachim Cuntz, George Elliott, Martin Engbers, Klaus Keimel, Henning Petzka, Leonel Robert, Luis Santiago, Aaron Tikuisis, Fred Wehrung, Stuart White and Wilhelm Winter for valuable comments.
We are grateful to Ken Goodearl for several discussions on tensor products of semigroups.
We also thank Snigdhayan Mahanta for pointing out the notion of a solid ring, and Eusebio Gardella for carefully reading the first chapters of a draft of this paper.

We also thank the anonymous referees for a careful reading and valuable comments that led to a substantial improvement of the manuscript.

This work was initiated at the Centre de Recerca Matem\`{a}tica (Bellaterra) during the program `The Cuntz Semigroup and the Classification of \ca{s}' in 2011.
The authors would like to thank the CRM for financial support and working conditions provided.
Part of this research was conducted while the third named author was visiting the Universitat Aut\`{o}noma de Barcelona (UAB) in 2012 and 2013, and while he was attending the Thematic Program on Abstract Harmonic Analysis, Banach and Operator Algebras at the Fields Institute in January-June 2014.
He would like to thank both Mathematics departments for their kind hospitality.

The two first named authors were partially supported by DGI MICIIN (grant No.\ MTM2011-28992-C02-01), and by the Comissionat per Universitats i Recerca de la Generalitat de Catalunya.
The third named author was partially supported by the Danish National Research Foundation through the Centre for Symmetry and Deformation (DNRF92), and by the Deutsche Forschungsgemeinschaft (SFB 878).

%------------------------------------------------------------------------------------------
%==========================================================================================
%##########################################################################################
\chapter{Pre-completed Cuntz semigroups}
\label{sec:catW}

%------------------------------------------------------------------------------------------
In the first part of this chapter, we introduce the categories $\CatPreW$ and $\CatW$ of (abstract) pre-completed Cuntz semigroups and we develop their general theory.
An object of $\CatPreW$ or $\CatW$ is a \pom\ (see \autoref{pgr:pom} for the definition) equipped with an auxiliary relation such that certain axioms are satisfied; see \autoref{pgr:axiomsW} and \autoref{dfn:catW}.
We show that $\CatW$ is a full, reflective subcategory of $\CatPreW$ (see \autoref{prp:WreflectivePreW}) and that both categories have inductive limits;
see \autoref{prp:limitsPreW} and \autoref{prp:limitsW}.

In the second part, we associate to every local \ca\ $A$ its \emph{pre-completed Cuntz semigroup} $W(A)$, which naturally belongs to the category $\CatW$.
It is given as the original definition of the Cuntz semigroup (equivalence classes of positive elements in matrices over $A$), together with a natural auxiliary relation; see \autoref{prp:WfromCAlg}.
We show that the assignment $A\mapsto W(A)$ extends to a continuous functor from local \ca{s} to the category $\CatW$; see \autoref{prp:functorW}.
This is inspired by \cite{CowEllIva08CuInv}, where the analogous results are shown for the completed Cuntz semigroup; see \autoref{sec:catCu}.

%------------------------------------------------------------------------------------------
%==========================================================================================
\section{The categories \texorpdfstring{$\CatPreW$}{PreW} and \texorpdfstring{$\CatW$}{W}}

%------------------------------------------------------------------------------------------
We refer to \autoref{sec:pom} for the basic theory of \pom{s}.

%==========================================================================================
\begin{pgr}[Axioms for the category $\CatW$]
\label{pgr:axiomsW}
\index{terms}{auxiliary relation}
\index{terms}{countably-based}
\index{terms}{basis (of a $\CatPreW$-semigroup)}
\index{terms}{compact containment}
\index{terms}{way-below \seeonly{compact containment}}
\index{terms}{relation!compact containment}
\index{terms}{relation!way-below \seeonly{compact containment}}
\index{terms}{W1-W4@\axiomW{1}-\axiomW{4}}
\index{terms}{element!compact}
\index{symbols}{$\prec$ \quad(auxiliary relation)}
\index{symbols}{$\ll$ \quad(compact containment)}
\index{symbols}{a$^\prec$@$a^\prec$ \quad(set of $\prec$-predecessor of $a$)}
Let $S$ be a \pom.
Following \cite[Definition~I-1.11, p.57]{GieHof+03Domains}, an \emph{auxiliary relation} on $S$ is a binary relation $\prec$ such that the following conditions hold:
\beginEnumConditions
\item
We have that $a\prec b$ implies $a\leq b$, for any $a,b\in S$.
\item
We have that $a\leq b\prec c\leq d$ implies $a\prec d$, for any $a,b,c,d\in S$.
\item
We have $0\prec a$, for any $a\in S$.
\end{enumerate}

Let $S$ be a \pom\ and fix an auxiliary relation $\prec$ on $S$.
We say that $S$ is \emph{countably-based} if there exists a countable subset $B\subset S$ such that for any $a',a$ in $S$ satisfying $a'\prec a$, there exists $b\in B$ such that $a'\leq b\prec a$.
A subset $B$ with these properties is called a \emph{basis} for $S$;
see \cite[Proposition~III-4.2, p.~241]{GieHof+03Domains}.

A particularly interesting auxiliary relation is the following:
Given elements $a$ and $b$ in a \pom\ $S$, we say that $a$ is \emph{compactly contained in} $b$ (or $a$ is \emph{way-below} $b$), denoted $a\ll b$, if whenever $(b_n)_{n\in\N}$ is an increasing sequence in $S$ for which the supremum $\sup_n b_n$ exists, then $b\leq \sup_n b_n$ implies that there is $k$ such that $a\leq b_k$.
If $a\in S$ satisfies $a\ll a$, we say that $a$ is \emph{compact}, and we shall denote the set of compact elements by $S_c$.

Note that the compact containment relation is usually defined by considering suprema of arbitrary upwards directed sets;
see \cite[Definition~I-1.1, p.49]{GieHof+03Domains}.
The definition given here is a sequential version.
In \autoref{rmk:catCu} we will see that both notions agree under a suitable separability assumption.

We will use the following axioms to define the objects in the categories $\CatPreW$ and $\CatW$.
Given an element $a\in S$, we use the notation $a^\prec:=\left\{ x\in S : x\prec a \right\}$.

\begin{itemize}
\item[\axiomW{1}]
For each $a\in S$, there exists a sequence $(a_k)_k$ in $a^\prec$ satisfying $a_k\prec a_{k+1}$ for each $k$ and such that for any $b\in a^\prec$ there exists $k$ with $b\prec a_k$.
\item[\axiomW{2}]
For each $a\in S$, we have $a=\sup a^\prec$.
\item[\axiomW{3}]
If $a',a,b',b\in S$ satisfy $a'\prec a$ and $b'\prec b$, then $a'+b'\prec a+b$.
\item[\axiomW{4}]
If $a,b,c\in S$ satisfy $a\prec b+c$, then there exist $b',c'\in S$ such that $a\prec b'+c'$, $b'\prec b$, and $c'\prec c$.
\end{itemize}

Axiom \axiomW{1} implies that for each $a\in S$, the set $a^\prec$ is upward directed and contains a cofinal increasing sequence with respect to $\prec$. For a general auxiliary relation $\prec$, the set $a^\prec$ need not be upward directed. For instance, consider $S=\N^2$ with pointwise order and addition. Define an auxiliary relation by   $x\prec y$ if $x\leq y$ and $x\neq y$ or $x=0$.

Given $a,b\in S$, axiom \axiomW{3} means that the set
\[
a^\prec+b^\prec = \left\{ a'+b' : a'\prec a, b'\prec b\right\},
\]
is contained in $(a+b)^\prec$.
Axiom \axiomW{4} means that $a^\prec+b^\prec$ is cofinal in $(a+b)^\prec$.
\end{pgr}

%==========================================================================================
\begin{dfn}
\label{dfn:catW}
\index{terms}{PreW-semigroup@$\CatPreW$-semigroup}
\index{terms}{W-semigroup@$\CatW$-semigroup}
\index{terms}{W-morphism@$\CatW$-morphism!generalized}
\index{terms}{category@category!$\CatPreW$}
\index{terms}{category@category!$\CatW$}
\index{symbols}{PreW@$\CatPreW$ \quad (category of $\CatPreW$-semigroups)}
\index{symbols}{W@$\CatW$ \quad (category of $\CatW$-semigroups)}
\index{symbols}{W(S,T)@$\CatWMor(S,T)$ \quad ($\CatW$-morphisms)}
\index{symbols}{W[S,T]@$\CatWGenMor{S,T}$ \quad (generalized $\CatW$-morphisms)}
A \emph{$\CatPreW$-semigroup} is a pair $(S,\prec)$, where $S$ is a \pom\ and $\prec$ is a fixed auxiliary relation on $S$, satisfying axioms \axiomW{1}, \axiomW{3} and \axiomW{4} from \autoref{pgr:axiomsW}.
If $(S,\prec)$ also satisfies axiom \axiomW{2}, then it is called a \emph{$\CatW$-semigroup}.
We often drop the reference to the auxiliary relation and simply write $S$ for a $\CatPreWW$-semigroup.

Given $\CatPreW$-semigroups $S$ and $T$, a \emph{generalized $\CatW$-morphism} $f\colon S\to T$ is a $\CatPom$-morphism that is continuous in the following sense:
\begin{itemize}
\item[\axiom{M}]
For every $a\in S$ and $b\in T$ with $b\prec f(a)$, there exists $a'\in S$ such that $a'\prec a$ and $b\leq f(a')$.
\end{itemize}
We denote the collection of all such maps by $\CatWGenMor{S,T}$.
A $\CatW$-morphism is a generalized $\CatW$-morphism that preserves the auxiliary relation and we denote the set of all such maps by $\CatWMor(S,T)$.

We let $\CatPreW$ be the category that has as objects all $\CatPreW$-semigroups, and whose morphisms are the $\CatW$-morphisms.
We let $\CatW$ be the full subcategory of $\CatPreW$ whose objects are $\CatW$-semigroups.
Note that we call the morphisms in both categories $\CatW$-morphisms.
\end{dfn}

%==========================================================================================
\begin{rmks}
\label{rmk:catW}
(1)
The order of the axioms \axiomW{1}-\axiomW{4} has been chosen so that they roughly correspond to the axioms \axiomO{1}-\axiomO{4} for $\CatCu$-semigroups; see \autoref{pgr:axiomsCu}.
Indeed, one should think of the W-axioms as a version of the O-axioms formulated in such a way that the semigroup is not required to have suprema of increasing sequences.

(2)
Let $S$ be a \pom{} with an auxiliary relation $\prec$ such that \axiomW{1} is satisfied.
It is easy to check that $S$ satisfies \axiomW{2} if and only if for every $a,b\in S$ we have that $a^\prec\subset b^\prec$ implies $a\leq b$.
Note that the converse of this statement is always true, that is, if $a\leq b$ then $a^\prec\subset b^\prec$.
This means that in the presence of \axiomW{2}, the partial order may be derived from the auxiliary relation.
For a $\CatCu$-semigroup, the converse is also true, since then the auxiliary relation is the compact containment relation which is defined in terms of the partial order.

(3)
In a $\CatW$-semigroup, the relation $\ll$ is stronger than $\prec$.
The class of $\CatW$-semigroups where $\prec$ is equal to $\ll$ was studied in \cite{AntBosPer11CompletionsCu}.

(4)
Let $f\colon S\to T$ be a $\CatPom$-morphism between $\CatW$-semigroups.
If $f$ is continuous, then for each $a\in S$ we have
\[
f(a)=\sup f(a^\prec).
\]
Indeed, by \axiomW{2} in $T$, we have $f(a)=\sup f(a)^\prec$.
Continuity of the map $f$ implies that $f(a^\prec)$ is cofinal in $f(a)^\prec$ in the sense that for every $t'\in f(a)^\prec$ there exists $a'\in a^\prec$ with $t'\leq f(a')$.
However, let us remark that we do not assume that $f$ preserves the auxiliary relation, so that we do not necessarily have $f(a^\prec)\subset f(a)^\prec$.
Nevertheless, it follows from this version of cofinality that the suprema of the sets $f(a^\prec)$ and $f(a)^\prec$ agree.

The converse of this statement in the context of $\CatCu$-semigroups is given in \autoref{prp:CuWMor}.
\end{rmks}

%==========================================================================================
\begin{pgr}[$\CatW$-completions]
\label{pgr:WreflectivePreW}
\index{terms}{W-completion@$\CatW$-completion}
We will now show that the category $\CatW$ is a reflective subcategory of $\CatPreW$.
This means that the inclusion functor $\CatW\to\CatPreW$ admits a left adjoint.
We use the idea of universal completions described in \cite[\S~2]{KeiLaw09DCompletion}.
Note that in this reference, what we call $\CatW$-completion appears under the name of universal $\CatW$-ification.
\end{pgr}

%==========================================================================================
\begin{dfn}
\label{dfn:Wification}
\index{terms}{W-completion@$\CatW$-completion}
Let $S$ be a $\CatPreW$-semigroup.
A \emph{$\CatW$-completion} of $S$ is a $\CatW$-semigroup $T$ together with a $\CatW$-morphism $\alpha\colon S\to T$ satisfying the following universal property:
For every $\CatW$-semigroup $R$ and for every $\CatW$-morphism $f\colon S\to R$, there exists a unique $\CatW$-morphism $\tilde{f}\colon T\to R$ such that $f=\tilde{f}\circ\alpha$.
\end{dfn}

%==========================================================================================
It is clear that $\CatW$-completions are unique whenever they exist.
We will now prove their existence by giving an explicit construction.

Let $(S,\prec)$ be a $\CatPreW$-semigroup.
In order to enforce \axiomW{2}, we consider the binary relation $\preceq$ on $S$ given by $a\preceq b$ if and only if $a^\prec\subset b^\prec$, for any $a,b\in S$.
It is clear that $\preceq$ is a pre-order on $S$.
By symmetrizing $\preceq$, we obtain an equivalence relation $\sim$ on $S$ such that for any $a,b\in S$ we have
\[
a\sim b \quad \text{ if and only if }\quad
a\preceq b\preceq a\quad \text{ if and only if }\quad
a^\prec=b^\prec.
\]
We let $\mu(S)=S/\!\sim$ denote the set of equivalence classes, and we denote the class of an element $a\in S$ by $[a]$.

By construction, the pre-order $\preceq$ induces a partial order on $\mu(S)$ by setting $[a]\leq[b]$ if and only if $a^\prec\subset b^\prec$, for any $a,b\in S$.
It is easy to check that the addition on $S$ induces an addition on $\mu(S)$ and that this endows $\mu(S)$ with the structure of a \pom.
We define an auxiliary relation on $\mu(S)$ by setting $[a]\prec[b]$ if and only if $[a]\leq [b']$ for some $b'\in b^\prec$.

%==========================================================================================
\begin{prp}
\label{prp:Wification}
Let $S$ be a $\CatPreW$-semigroup.
Then $\mu(S)$ with the addition, order and auxiliary relation defined above becomes a $\CatW$-semigroup.
Moreover, the map
\[
\beta\colon S\to \mu(S),\quad
a\mapsto[a],\quad
\txtFA a\in S,
\]
is a $\CatW$-morphism defining a $W$-completion of $S$.
\end{prp}
\begin{proof}
In order to prove that $\mu(S)$ is a $\CatPreW$-semigroup, the only detail that needs some verification is that $\mu(S)$ satisfies \axiomW{2}, and for this we use the second observation in \autoref{rmk:catW}.
Thus assume that $a,b\in S$ satisfy $[a]^\prec\subset [b]^\prec$.
By definition, this means that $[a']\prec [b]$ for any $a'\in S$ satisfying $[a']\prec [a]$.
This in turn implies that $a'\prec b$ whenever $a'\in S$ satisfies $a'\prec a$.
Therefore $a^\prec\subset b^\prec$, which means $[a]\leq [b]$, as desired.

It is clear that $\beta$ is a $\CatW$-morphism.
Thus, we only need to prove that $\beta$ satisfies the universal property in \autoref{dfn:Wification}.
In order to verify this, let $R$ be a $\CatW$-semigroup and let $f\colon S\to R$ be a $\CatW$-morphism.
It is clear that there is at most one map $\tilde{f}\colon\mu(S)\to R$ satisfying $f=\tilde{f}\circ\beta$.
To show existence, we define
\[
\tilde{f}\colon \mu(S)\to R,\quad
[a]\mapsto f(a),\quad
\txtFA a\in S.
\]

Let us show that $\tilde{f}$ has the desired properties.
Suppose $[a]\leq [b]$ for some $a,b\in S$.
Let $x\in R$ satisfy $x\prec f(a)$.
Since $f$ is continuous, we can choose $a'\in S$ such that $a'\prec a$ and $x\leq f(a')$.
Now, since $[a]\leq [b]$, we have $a'\prec b$ and thus
\[
x\leq f(a')\prec f(b).
\]
Thus, $f(a)^\prec\subset f(b)^\prec$, and since $R$ is a $\CatW$-semigroup, this ensures $f(a)\leq f(b)$.
It follows that $\tilde{f}$ is a well-defined, order-preserving map satisfying $f=\tilde{f}\circ\beta$.
Similarly, one proves that $\tilde{f}$ is continuous, and preserves addition and the auxiliary relation.
Hence, $\tilde{f}$ is a $\CatW$-morphism.
\end{proof}

%==========================================================================================
Now, by the arguments in \cite[\S~2]{KeiLaw09DCompletion} regarding universal $W$-ifications, the above construction of the $\CatW$-completion induces a reflection functor $\mu\colon\CatPreW\to\CatW$:

%==========================================================================================
\begin{prp}
\label{prp:WreflectivePreW}
The category $\CatW$ is a full, reflective subcategory of $\CatPreW$.
\end{prp}

%==========================================================================================
\begin{pgr}[Inductive limits in $\CatPom$]
\label{pgr:limPom}
Recall that an inductive system $\mathcal{S}$ in a category $\CatC$ consists of a directed set $I$, a family $(A_i)_{i\in I}$ of objects in $\CatC$, and a family of morphisms $(f_{i,j}\colon A_i\to A_j)_{i\leq j\in I}$ such that $f_{i,i}=\text{id}_{A_i}$ for all $i\in I$ and $f_{i,k}=f_{j,k}\circ f_{i,j}$ whenever $i\leq j\leq k$ in $I$.

A morphism from $\mathcal{S}$ to an object $X$ in $\CatC$ is a collection $(g_i\colon A_i\to X)_{i\in I}$ such that $g_j\circ f_{i,j}=g_i$ whenever $i\leq j$ in $I$.
An inductive limit of $\mathcal{S}$ is an object $X$ in $\CatC$ together with a morphism $(g_i)_{i\in I}$ from $\mathcal{S}$ to $X$ satisfying the following universal property:
For every object $Y$ and morphism $(h_i)_{i\in I}$ from $\mathcal{S}$ to $Y$, there exists a unique morphism $h\colon X\to Y$ such that $h_i=h\circ g_i$ for each $i\in I$.
If it exists, the inductive limit is unique up to isomorphism and will be denoted by $\varinjlim (A_i,f_{i,j})$, or simply $\varinjlim A_i$.

Let $((S_i)_{i\in I},(\varphi_{i,j})_{i, j\in I, i\leq j})$ be an inductive system in $\CatPom$, indexed over $I$.
We define an equivalence relation $\sim$ on the disjoint union $\bigsqcup_{i\in I} S_i$, by setting for any $a\in S_i$ and any $b\in S_j$:
\[
a\sim b\quad \text{ if and only if there exists } k\geq i,j \text{ such that } \varphi_{i,k}(a)=\varphi_{j,k}(b).
\]
The set of equivalence classes is denoted by $\CatPomLim S_i$.
We denote the equivalence class of an element $a\in S_i$ by $[a]$.
Given $a\in S_i$ and $b\in S_j$, we define
\[
[a]+[b]=[\varphi_{i,k}(a)+\varphi_{j,k}(b)],\quad
\text{ for any } k\geq i,j.
\]
It is clear that this is a well-defined addition on $\CatPomLim S_i$.
We also define a partial order by setting for $a\in S_i$ and $b\in S_j$:
\[
[a]\leq[b]\quad
\text{ if and only if there exists } k\geq i,j \text{ such that } \varphi_{i,k}(a)\leq\varphi_{j,k}(b).
\]
This gives $\CatPomLim S_i$ the structure of a \pom, and it is well-known that this is the inductive limit in the category $\CatPom$.
The $\CatPom$-morphism from $S_i$ to the inductive limit is denoted by $\varphi_{i,\infty}$.
\end{pgr}

%==========================================================================================
\begin{dfn}
%[Auxiliary relation on $\CatPomLim S_i$]
\label{dfn:limAuxRel}
Let $((S_i)_{i\in I},(\varphi_{i,j})_{i,j\in I, i\leq j})$ be an inductive system in $\CatPreW$.
We define a relation $\prec$ on the inductive limit $\CatPomLim S_i$ of the underlying \pom{s} by setting for any $a\in S_i$ and $b\in S_j$:
\[
[a]\prec[b]\quad
\text{ if and only if there exists } k\geq i,j \text{ such that } \varphi_{i,k}(a)\prec\varphi_{j,k}(b).
\]
\end{dfn}

%==========================================================================================
\begin{thm}
\label{prp:limitsPreW}
\index{terms}{inductive limits@inductive limits!in $\CatPreW$}
\index{symbols}{PreWLim@$\CatPreWLim$ \quad (inductive limit in $\CatPreW$)}
The category $\CatPreW$ has inductive limits.
More precisely, let $((S_i)_{i\in I},(\varphi_{i,j})_{i,j\in I, i\leq j})$ be an inductive system in $\CatPreW$.
Let $S=\CatPomLim S_i$ be the inductive limit of the underlying \pom{s}, together with $\CatPom$-morphisms $\varphi_{i,\infty}\colon S_i\to S$.
The relation $\prec$ on $S$ as defined in \autoref{dfn:limAuxRel} is an auxiliary relation and $(S,\prec)$ is a $\CatPreW$-semigroup, denoted by $\CatPreWLim S_i$.
Moreover, the maps $\varphi_{i,\infty}$ are $\CatW$-morphisms and $\CatPreWLim S_i$ is the inductive limit of the system $((S_i)_{i\in I},(\varphi_{i,j})_{i,j\in I, i\leq j})$ in $\CatPreW$.
\end{thm}
\begin{proof}
We first show that $\prec$ is an auxiliary relation.
Condition~(i) and~(ii) from \autoref{pgr:axiomsW} are easily verified.
To show condition~(iii), let $[a],[b],[c],[d]\in S$ satisfy $[a]\leq [b]\prec [c]\leq [d]$.
We have to show $[a]\prec [d]$.
We may assume that there is an index $i$ such that all four elements are represented in $S_i$.
Since $[a]\leq [b]$, we can choose $j\geq i$ with $\varphi_{i,j}(a)\leq\varphi_{i,j}(b)$.
Similarly, we can choose $k\geq i$ and $l\geq i$ satisfying $\varphi_{i,k}(b)\prec\varphi_{i,k}(c)$ and $\varphi_{i,l}(c)\leq\varphi_{i,l}(d)$.
Choose $n\geq j,k,l$.
Using that the connecting maps preserve the order and the auxiliary relation, we obtain
\[
\varphi_{i,n}(a)\leq\varphi_{i,n}(b)\prec\varphi_{j,n}(c)\leq\varphi_{i,n}(d),
\]
which implies $[a]\prec [d]$, as desired.

Next, we show that $(S\prec)$ is a $\CatPreW$-semigroup.
In order to verify \axiomW{1}, let $a\in S_i$ for some $i$.
Since $S_i$ satisfies \axiomW{1}, there is a cofinal $\prec$-increasing sequence $(a_k)_{k\in\N}$ in $a^{\prec}$.
Then $[a_k]\prec [a]$ for all $k$.
Further, if $[b]\prec [a]$ for some $j$ and $b\in S_j$, then we can choose $n\geq i,j$ with $\varphi_{j,n}(b)\prec\varphi_{i,n}(a)$.
Using that $\varphi_{i,n}$ is continuous, we can choose $k$ and $a'\prec a$ such that $\varphi_{j,n}(b)\leq \varphi_{i,n}(a')\prec\varphi_{i,n}(a_k)$.
Thus $[b]\prec [a_k]$.
This shows that $[a]^\prec$ is upward directed and contains a cofinal $\prec$-increasing sequence.

It is routine to check \axiomW{3}.
In order to show that $(S,\prec)$ satisfies \axiomW{4}, suppose that $[c]\prec [a]+[b]$ for some elements $[a],[b]$ and $[c]$ in $S$.
We may assume that all three elements are represented in $S_i$, for some index $i$.
Choose $j\geq i$ such that $\varphi_{i,j}(c)\prec\varphi_{i,j}(a)+\varphi_{i,j}(b)$.
Since $S_j$ satisfies \axiomW{4}, we can choose $d,e\in S_j$ with
\[
\varphi_{i,j}(c)\leq d+e,\quad
d\prec\varphi_{i,j}(a),\quad
e\prec\varphi_{i,j}(b).
\]
Using that $\varphi_{i,j}$ is continuous, we can choose $a',b'\in S_i$ such that
\[
a'\prec a,\quad
b'\prec b,\quad
d\leq\varphi_{i,j}(a'),\quad
e\leq\varphi_{i,j}(b').
\]
Then $[a']\prec [a]$, $[b']\prec [b]$, and $[c]\leq [a']+[b']$, which shows that the elements $[a']$ and $[b']$ have the desired properties to verify \axiomW{4} in $S$.

The natural maps $\varphi_{i,\infty}\colon S_i\to S$ preserve the auxiliary relation.
It is also easy to see that they are continuous using arguments similar to the proof of \axiomW{1}.

Finally, in order to show that $S$ is the inductive limit in the category $\CatPreW$, let $T$ be a $\CatPreW$-semigroup and let $\lambda_i\colon S_i\to T$ be $\CatW$-morphisms such that $\lambda_j\circ\varphi_{i,j}=\lambda_i$ whenever $i\leq j$.
Since $S$ is the limit in the category $\CatPom$, there is a unique $\CatPom$-morphism $\alpha\colon S\to T$ such that $\alpha\circ\varphi_{i,\infty}=\lambda_i$ for each $i$.

In order to show that $\alpha$ is continuous, let $x\in T$ and $y\in S$ satisfy $x\prec\alpha(y)$.
Choose $i\in I$ and $a\in S_i$ such that $y=[a]$.
Then $\alpha([a])=\lambda_i(a)$ and since $\lambda_i$ is a $\CatW$-morphism, we can choose $a'\in S_i$ such that $a'\prec a$ and $x\leq \lambda_i(a')$.
We have $[a']\prec [a]$ and $x \leq \lambda_i(a') = \alpha([a'])$, which shows that $[a']$ is the desired element to verify that $\alpha$ is continuous.

In order to show that $\alpha$ preserves the auxiliary relation, let $a\in S_i$ and $b\in S_j$ satisfy $[a]\prec [b]$.
Choose $k\geq i,j$ such that $\varphi_{i,k}(a)\prec \varphi_{j,k}(b)$ in $S_k$.
Using that $\lambda_k$ preserves the auxiliary relation at the third step, we deduce
\begin{align*}
\alpha([a])
=\alpha(\varphi_{k,\infty}\circ\varphi_{i,k}(a))
&=\lambda_k(\varphi_{i,k}(a)) \\
&\prec\lambda_k(\varphi_{j,k}(b))
=\alpha(\varphi_{k,\infty}\circ\varphi_{j,k}(b))
=\alpha([b]),
\end{align*}
as desired.
Thus, $\alpha$ preserves the auxiliary relation and is hence a $\CatW$-morphism.
\end{proof}

%------------------------------------------------------------------------------------------
Given an inductive system $((S_i)_{i\in I},(\varphi_{i,j})_{i,j\in I, i\leq j})$ in $\CatW$, the inductive limit $\CatPreWLim S_i$ from \autoref{prp:limitsPreW} need not be a $\CatW$-semigroup.
However, since $\CatW$ is a reflective subcategory of $\CatPreW$, it follows from general category theory that $\CatW$ has inductive limits and that the reflection functor $\mu\colon\CatPreW\to\CatW$ is continuous;
see for example \cite[Proposition~3.2.2, p.106]{Bor94HandbookCat1}.

%==========================================================================================
\begin{cor}
\label{prp:limitsW}
\index{terms}{inductive limits@inductive limits!in $\CatW$}
\index{symbols}{WLim@$\CatWLim$ \quad (inductive limit in $\CatW$)}
The category $\CatW$ has inductive limits.
More precisely, let $((S_i)_{i\in I},(\varphi_{i,j})_{i,j\in I, i\leq j})$ be an inductive system in $\CatW$.
Then the inductive limit in $\CatW$ is the $\CatW$-completion of the inductive limit in $\CatPreW$:
\[
\CatWLim S_i = \mu(\CatPreWLim S_i).
\]
\end{cor}

\vspace{5pt}
%------------------------------------------------------------------------------------------
%==========================================================================================
\section{The pre-completed Cuntz semigroup of a \texorpdfstring{C*-algebra}{{C}*-algebra}}

%==========================================================================================
\begin{pgr}[Local \ca{s}]
\label{pgr:localCa}
\index{terms}{pre-C*-algebra@\preCa}
\index{terms}{C*-algebra@\ca{}!pre-}
\index{terms}{local C*-algebra@local \ca}
\index{terms}{C*-algebra@\ca{}!local}
\index{symbols}{M$_\infty$@$M_\infty(A)$}
\index{symbols}{$\CatLocCa$ \quad (category of local \ca{s})}
A \emph{\preCa} $A$ is a \starAlg\ over the complex numbers together with a $C^*$-norm $\|\freeVar\|$, that is, $\|a^*a\|=\|a\|^2$ for all $a\in A$.
It is known that such a norm is automatically submultiplicative and that the involution becomes isometric;
see \cite[Theorem~9.5.14, p.956]{Pal01BAlg2}.
Every \preCa\ $A$ naturally embeds as a dense sub-\starAlgß\ in its completion $\overline{A}$, which is a \ca.

In this paper, we will say that $A$ is a \emph{local \ca} if there is a family of complete, ${}^*$-invariant subalgebras $A_i\subset A$ such that for any $i_1,i_2$ there is $i_3$ such that $A_{i_1}\cup A_{i_2}\subset A_{i_3}$ and such that $A=\bigcup_i A_i$.
Note that each $A_i$ is a \ca.
Viewing a \preCa\ $A$ inside its completion $\overline{A}$, it is a local \ca\ if and only if for any finite subset $F\subset A$, the \ca\ $C^*(F)$ generated inside $\overline{A}$ is contained in $A$.
The main point is that local \ca{s} are closed under continuous functional calculus.

We say that a pre-\ca\ $A$ is \emph{separable} if it contains a countable dense subset (equivalently, $\overline{A}$ is separable).
If $A$ is a local \ca, then so is every matrix algebra $M_k(A)$, and there is a natural dense embedding $\M_k(A)\subset M_k(\bar{A})$.
The \ca{s} $M_k(\bar{A})$ sit (as upper left corners) inside the stabilization $\bar{A}\otimes\K$, where $\K$ denotes the compact operators, and we may consider the union
\[
M_\infty(A):=\bigcup_k M_k(A)\subset \bar{A}\otimes\K.
\]
This is a dense embedding, hence $\overline{M_\infty(A)}=\bar{A}\otimes\K$, and one sees that $M_\infty(A)$ is again a local \ca.

A \starHom\ between local \ca{s} is automatically continuous and even norm-decreasing.
We let $\CatLocCa$ be the category whose objects are local \ca{s}, and whose morphisms are \starHom s.
\index{terms}{category@category!$\CatLocCa$}

We remark that there are other definitions of a local \ca\ in the literature, in particular in \cite[3.1]{Bla98KThy}, \cite{Mey99Cyclic}, \cite{CunMeyRos07TopBivarK} and \cite[Definition~I.1.1(a)]{BlaHan82DimFct}.
Some of these definitions seem to be more general than the one given here.
It is conceivable that the theory of pre-completed Cuntz semigroups can be carried out in this more general framework, but we will not pursue this here.
\end{pgr}

%==========================================================================================
\begin{pgr}[Cuntz comparison in a local \ca]
\label{pgr:CuntzComparison}
\index{terms}{Cuntz subequivalent}
\index{terms}{relation!Cuntz subequivalence}
\index{terms}{Cuntz equivalent}
\index{terms}{relation!Cuntz equivalence}
\index{symbols}{$\precsim$ \quad (Cuntz subequivalence)}
\index{symbols}{$\sim$ \quad (Cuntz equivalence)}
Let $A$ be a local \ca{}, and let $A_+$ be the subset of positive elements.
Given $x,y\in A_+$ we say that $x$ is \emph{Cuntz sub-equivalent} to $y$, in symbols $x\precsim y$, if there exists a sequence $(z_n)_n$ in $A$ such that $x=\lim_n z^*_nyz_n$.
We say $x$ is \emph{Cuntz equivalent} to $y$, in symbols $x\sim y$, if $x\precsim y$ and $y\precsim x$.
These relations were introduced in \cite{Cun78DimFct}.

R{\o}rdam's fundamental results on Cuntz comparison, \cite[Proposition~2.4]{Ror92StructureUHF2} (see also \cite[Proposition~2.17]{AraPerTom11Cu}), remain valid in local \ca{s}, that is, for any $x,y\in A_+$ the following conditions are equivalent:
\beginEnumStatements
\item
We have $x\precsim y$.
\item
For every $\varepsilon>0$ there exists $\delta>0$ such that $(x-\varepsilon)_+\precsim(y-\delta)_+$.
\item
For every $\varepsilon>0$ there exist $\delta>0$ and $r\in A$ such that $(x-\varepsilon)_+=r(y-\delta)_+r^*$.
\end{enumerate}

Here, $(a-\varepsilon)_+$ denotes the standard $\varepsilon$-cutdown of a positive element obtained as $(a-\varepsilon)_+=f_\varepsilon(a)$ where $f_\varepsilon(t)=\text{max}\{0,t-\varepsilon\}$.
\end{pgr}

%==========================================================================================
\begin{pgr}
\label{pgr:originalW}
\index{symbols}{W(A)@$\W(A)$ \quad (precompleted Cuntz semigroup)}
Given a local \ca\ $A$, the (original) definition of the Cuntz semigroup of $A$ is
\[
W(A):=M_\infty(A)_+/\!\!\sim,
\]
the set of Cuntz equivalence classes of positive elements in matrix algebras over $A$.
The equivalence class of an element $x\in M_\infty(A)_+$ is denoted by $[x]$.
Given $x,y\in M_\infty(A)_+$, we set $[x]\leq [y]$ if $x\precsim y$, and we define
$[x]+[y]=[(\begin{smallmatrix}x & 0 \\ 0 & y \end{smallmatrix})]$.
This defines a partial order and a well-defined abelian addition on $W(A)$.
The zero element in $W(A)$ is given by the class of the zero element $0\in A$.
This equips $W(A)$ with the structure of a \pom.
Next, we will endow $W(A)$ with an auxiliary relation, making it a $\CatW$-semigroup.
\end{pgr}

%==========================================================================================
\begin{dfn}
\label{dfn:WfromCAlg}
\index{terms}{Cuntz semigroup!pre-completed}
Let $A$ be a local \ca.
We define a relation $\prec$ on the \pom\ $W(A)=M_\infty(A)_+/\!\sim$ by setting for any $a,b\in M_\infty(A)_+$:
\[
[a]\prec[b]\quad
\text{ if and only if there exists } \varepsilon>0 \text{ such that } [a]\leq[(b-\varepsilon)_+].
\]
We call $W(A)=(W(A),\prec)$ the \emph{pre-completed Cuntz semigroup} of $A$.
\end{dfn}

%==========================================================================================
\begin{prp}
\label{prp:WfromCAlg}
Let $A$ be a local \ca.
Then the relation $\prec$ defined in \autoref{dfn:WfromCAlg} is an auxiliary relation and $(W(A),\prec)$ is a $\CatW$-semigroup.
If $A$ is separable, then $W(A)$ is countably-based.
\end{prp}
\begin{proof}
By abusing notation, let us define a relation $\prec$ on positive elements in $M_\infty(A)$ by setting $a\prec b$ if there exists $\varepsilon>0$ such that $a\precsim (b-\varepsilon)_+$.
R{\o}rdam's results on Cuntz comparison show that $a\precsim b$ if and only if for every $\varepsilon>0$ there exists $\delta>0$ such that $(a-\varepsilon)_+\precsim(b-\delta)_+$;
see \autoref{pgr:CuntzComparison} and \cite[Proposition~2.4]{Ror92StructureUHF2},
Thus, given any positive elements $a,b$ and $c$, we see that $a\prec b\precsim c$ implies $a\prec c$.

It follows that $\prec$ as in \autoref{dfn:WfromCAlg} is well-defined.
Given $a,b\in M_\infty(A)_+$, we have $(a-\varepsilon)_+\precsim a$ for each $\varepsilon>0$, which shows that $[a]\prec [b]$ implies $[a]\leq [b]$.
It follows from this and the considerations in the previous paragraph that $\prec$ is an auxiliary relation on $W(A)$.
The axioms \axiomW{1}-\axiomW{4} are now straightforward to check, for example using \cite[Lemma 2.5]{KirRor00PureInf}.
Finally, if $A$ is separable, then $W(A)$ countably-based; the argument can be found in the proof of \cite[Lemma~1.3]{AntPerSan11PullbacksCu}; see also \cite[Proposition~5.1.1]{Rob13Cone}.
\end{proof}

%==========================================================================================
\begin{rmks}
\label{rmk:WfromCAlg}
(1)
Usually, by $W(A)$ we denote the pre-completed Cuntz semigroup of $A$, considered as a $\CatW$-semigroup which is understood to be equipped with an auxiliary relation.
It should be clear from the context when by $W(A)$ we only mean the underlying \pom.

(2)
It is not known whether the auxiliary relation on $W(A)$ can be deduced from its structure as a \pom, but it seems unlikely that this is the case without assuming certain regularity properties on the \ca.
Thus, we consider the auxiliary relation on $W(A)$ as an additional structure, not just as a property.
This is in contrast to $\CatCu$-semigroups, where the auxiliary relation (the way-below relation) is defined in terms of the order structure; see \autoref{sec:catCu}.

The case when the auxiliary relation in $W(A)$ is just the way-below relation was studied in \cite{AntBosPer11CompletionsCu}.
The class of such semigroups was denoted by PreCu.
See also \autoref{rmk:catW}(3).
\end{rmks}

%------------------------------------------------------------------------------------------
We want to extend the assignment $A\mapsto W(A)$ to a functor from the category of local \ca{s} to the category $\CatW$.
Thus, we need to show that a \starHom\ $\varphi\colon A\to B$ induces a $\CatW$-morphism between the respective $W$-semigroups.
We will first prove a related result for c.p.c.\ order-zero maps.
As shown in \cite[Corollary~4.5]{WinZac09CpOrd0}, a c.p.c.\ order-zero map between \ca{s} induces a $\CatPom$-morphism between the respective $\CatW$-semigroups.
In the result below we establish that this map is also continuous, hence a generalized $\CatW$-morphism.

Note that a c.p.c.\ order-zero map $\varphi\colon A\to B$ naturally extends to a c.p.c.\ order-zero map between the local \ca{s} $M_\infty(A)$ and $M_\infty(B)$.
By abuse of notation, we will denote this extension also by $\varphi$.
%The same will be done with the natural extension of $\varphi$ to a c.p.c.\  map between the minimal unitalisations $\widetilde{A}$ and $\widetilde{B}$.

%==========================================================================================
\begin{prp}[{Winter, Zacharias, \cite[Corollary~4.5]{WinZac09CpOrd0}}]
\label{prp:mor_from_cpc0}
Let $A$ and $B$ be local \ca{s}.
Then every c.p.c.\ order-zero map $\varphi\colon A\to B$ naturally induces a generalized $\CatW$-morphism
\[
W(\varphi)\colon W(A)\to W(B),\quad
[x]\mapsto[\varphi(x)],\quad
\txtFA x\in M_\infty(A)_+.
\]
If $\varphi$ is a \starHom, then $W(\varphi)$ also preserves the auxiliary relation and thus is a $\CatW$-morphism.
\end{prp}
\begin{proof}
It follows from \cite[Corollary~4.5]{WinZac09CpOrd0} that $W(\varphi)$ is a well-defined map that preserves addition, order and the zero element.
Given $a\in M_\infty(A)_+$ and $\varepsilon>0$, let us first show
\begin{align}
\label{prp:mor_from_cpc0:eq1}
(\varphi(x)-\varepsilon)_+
\leq\varphi((x-\varepsilon)_+).
\end{align}
We let $\widetilde{A}$ denote the minimal unitalization of $A$ (with $\widetilde{A}=A$ if $A$ is unital), and similarly for $\widetilde{B}$.
We can extend $\varphi$ to a c.p.c.\ map $\tilde{\varphi}\colon\widetilde{A}\to\widetilde{B}$.
Choose $k$ such that $a\in M_k(A)$, and consider the amplification of $\tilde{\varphi}$ to a map $M_k(\widetilde{A})\to M_k(\widetilde{B})$, which we also denote $\tilde{\varphi}$.
We have $\tilde{\varphi}(1_{M_k(\widetilde{A})})\leq 1_{M_k(\widetilde{B})}$, which we use to deduce
\[
\tilde{\varphi}(x) - \varepsilon 1_{M_k(\widetilde{B})}
\leq \tilde{\varphi}(x) - \varepsilon\tilde{\varphi}(1_{M_k(\widetilde{A})})
=\tilde{\varphi}(x-\varepsilon 1_{M_k(\widetilde{A})})
\leq \tilde{\varphi}((x-\varepsilon 1_{M_k(\widetilde{A})})_+).
\]
Using that $\tilde{\varphi}(x) - \varepsilon 1_{M_k(\widetilde{B})}$ commutes with $\tilde{\varphi}((x-\varepsilon 1_{M_k(\widetilde{A})})_+)$ for the second step, we obtain
\[
(\varphi(x)-\varepsilon)_+
= (\tilde{\varphi}(x) - \varepsilon 1_{M_k(\widetilde{B})})_+
\leq \tilde{\varphi}((x-\varepsilon 1_{M_k(\widetilde{A})})_+)
=\varphi((x-\varepsilon)_+),
\]
as desired.

To show that $W(\varphi)$ is continuous, let $b\in W(B)$ and $a\in W(A)$ satisfy $b\prec W(\varphi)(a)$.
We need to find $a'\in W(A)$ such that $a'\prec a$ and $b\leq W(\varphi)(a')$.
Choose $x\in M_\infty(A)_+$ with $a=[x]$.
By definition of $\prec$ on $W(B)$, we can choose $\varepsilon>0$ such that $b\leq[(\varphi(x)-\varepsilon)_+]$.
Set $a':=[(x-\varepsilon)_+]$.
Then $a'\prec a$ in $W(A)$.
Using \eqref{prp:mor_from_cpc0:eq1} at the second step, we deduce
\[
b\leq[(\varphi(x)-\varepsilon)_+]\leq [\varphi((x-\varepsilon)_+)]=f(a'),
\]
as desired.

Finally, if $\varphi$ is a \starHom, then $\varphi((x-\varepsilon)_+)=(\varphi(x)-\varepsilon)_+$ for each $x\in M_\infty(A)_+$ and $\varepsilon>0$, which implies that $W(\varphi)$ preserves the auxiliary relation.
\end{proof}

%------------------------------------------------------------------------------------------
It follows that $\W$ is a functor from the category $\CatLocCa$ of local \ca{s} with \starHom s to the category $\CatW$.

%==========================================================================================
\begin{pgr}[Inductive limits in $\CatLocCa$]
\index{terms}{inductive limits@inductive limits!in $\CatLocCa$}
\label{pgr:LimCatLocCa}
\index{symbols}{C*LocLim@$\CatLocCaLim$ \quad (inductive limit in $\CatLocCa$)}
Let $((A_i)_{i\in I},(\varphi_{i,j})_{i,j\in I, i\leq j})$ be an inductive system in the category $\CatLocCa$, indexed over the directed set $I$.
To construct the inductive limit in $\CatLocCa$, we first consider $((A_i)_{i\in I},(\varphi_{i,j})_{i,j\in I, i\leq j})$ as an inductive system in the category of $^*$-algebras, and we let $A_\alg$ denote the corresponding inductive limit.
This $^*$-algebra can be equipped with the pre-norm defined by
\[
\|x\| :=\inf \left\{ \|\varphi_{i,j}(x)\| : j\in I, j\geq i \right\},
\]
for $x\in A_i$.
The set $N:=\left\{ x\in A_\alg : \|x\|=0 \right\}$ is a two-sided ${}^*$-ideal.
Set
\[
\CatLocCaLim A_i := A_\alg / N,
\]
which is a local \ca\ satisfying the universal properties of an inductive limit.

Observe that for each $i,j\in I$ satisfying $i\leq j$, the map $\varphi_{i,j}$ induces a natural \starHom\ $M_\infty(A_i)\to M_\infty(A_j)$, which we denote by $\tilde{\varphi}_{i,j}$.
The limit of the inductive system $((M_{\infty}(A_i))_{i\in I},(\varphi_{i,j})_{i,j\in I, i\leq j})$ is naturally isomorphic to $M_\infty(\CatLocCaLim A_i)$.
\end{pgr}

%==========================================================================================
\begin{thm}
\label{prp:functorW}
The functor $W\colon\CatLocCa\to\CatW$ is continuous.
\end{thm}
\begin{proof}
Let $((A_i)_{i\in I},(\varphi_{i,j})_{i,j\in I, i\leq j})$ be an inductive system in $\CatLocCa$.
We let $A_{alg}$ be the algebraic inductive limit, we set $A:=A_{alg}$ for convenience, and we let $N$ denote the ${}^*$-ideal defined in \autoref{pgr:LimCatLocCa} so that the inductive limit in $\CatLocCa$ is given by $A/N$ with \starHom s $\varphi_{i,\infty}\colon A_i\to A/N$.
As explained in \autoref{pgr:LimCatLocCa}, we may replace each $A_i$ by $M_\infty(A_i)$.
Then every Cuntz class in $W(A_i)$ is realized by a positive element in $A_i$.
It follows that also every class in $W(A/N)$ is realized by a positive element in $A/N$, that is, $W(A/N)=(A/N)_+/\!\!\sim$.

The following diagram shows the algebras and maps to be constructed.
\[
\xymatrix@R=15pt@C=15pt{
A_i \ar[r]^-{\varphi_{i,j}}
    & A_j \ar[r]
    & \ldots \ar[r]
    & A \ar@{->>}[r]
    & A/N
    \\
& & & & W(A/N)
    \\
W(A_i) \ar[r]^{\psi_{i,j}}
    & W(A_j) \ar[r] \ar[drr]_{\psi_{j,\infty}} \ar[urrr]^{W(\varphi_{j,\infty})}
    & \ldots \ar[r]
    & {\CatPreWLim W(A_i)} \ar@{}[d]|{\verteq} \ar@{->>}[r] \ar[ur]^{\omega}
    & {\CatWLim W(A_i)}\ar[u]_{\tilde{\omega}} \ar@{}[d]|{\verteq}
    \\
& & & S \ar@{->>}[r]_-{\beta}
    & \mu(S)
}
\]

For $i\leq j$, set $\psi_{i,j}:=W(\varphi_{i,j})$ and consider the induced $\CatW$-inductive system $((W(A_i))_{i\in I},(\psi_{i,j})_{i,j\in I, i\leq j})$.
Let $S$ be the inductive limit in $\CatPreW$; see \autoref{prp:limitsPreW}.
Denote the $\CatW$-morphisms into the limit by $\psi_{i,\infty}\colon W(A_i)\to S$.
Let $\omega\colon S\to W(A/N)$ be the unique $\CatW$-morphism that is induced by the maps $W(\varphi_{i,\infty})\colon W(A_i)\to W(A/N)$.
The limit in the category $\CatW$ is given as the $\CatW$-completion of the limit in $\CatPreW$; see \autoref{prp:limitsW}.
By \autoref{prp:Wification}, there exists a $\CatW$-morphism $\tilde\omega\colon \CatWLim W(A_i)\to W(A/N)$ such that $\tilde w([s])=w(s)$ for all $s\in S$.
Thus it is enough to prove that $\tilde\omega$ is a surjective order-embedding.
Since every positive element in $A/N$ is the image of a positive element in some $A_i$, we conclude that $\tilde\omega$ is surjective.

In order to show that $\tilde\omega$ is an order-embedding, let $s,t\in S$ satisfy $\omega(s)\leq\omega(t)$.
We need to show $s^\prec\subset t^\prec$.
Let $s'$ be an element in $S$ with $s'\prec s$.
Choose $i$ and elements $a',a\in W(A_i)$ such that $a'\prec a$ in $W(A_i)$ and $s'=\psi_{i,\infty}(a')$ and $s=\psi_{i,\infty}(a)$.
We may assume that $t$ is also realized by an element in $W(A_i)$ (by passing to a larger index, if necessary).
This means that we can choose $b\in W(A_i)$ with $t=\psi_{i,\infty}(b)$.

Let $x,y\in (A_i)_+$ with $a=[x]$ and $b=[y]$.
By definition of the relation $\prec$ in $W(A_i)$, we can choose $\varepsilon>0$ such that $a'\leq[(x-\varepsilon)_+]$.
Note that $\omega(s)$ and $\omega(t)$ are the Cuntz classes of $\varphi_{i,\infty}(x)$ and $\varphi_{i,\infty}(y)$ in $A/N$, respectively.
By assumption, $\varphi_{i,\infty}(x)\precsim\varphi_{i,\infty}(y)$.
Using R{\o}rdam's lemma (see \autoref{pgr:CuntzComparison}), we can choose $\delta>0$ and $r\in A/N$ such that
\[
\left( \varphi_{i,\infty}(x)-\frac{\varepsilon}{2} \right)_+
=r \left( \varphi_{i,\infty}(y)-\delta \right)_+r^*.
\]
Choose $j$ and $\bar{r}\in A_j$ such that $r=\varphi_{j,\infty}(\bar{r})$.
We may assume $j\geq i$.
Then
\[
\varphi_{i,\infty}\left( (x-\frac{\varepsilon}{2})_+ \right)
= \varphi_{j,\infty}\left( \bar{r}\varphi_{i,j}((y-\delta)_+)\bar{r}^* \right).
\]
Using the description of the limit in $\CatLocCa$ (see \autoref{pgr:LimCatLocCa}), this implies that we can choose $k\geq j$ such that
\[
\left\| \varphi_{i,k}\left( (x-\frac{\varepsilon}{2})_+ \right)
-\varphi_{j,k}\left( \bar{r}\varphi_{i,j}((y-\delta)_+)\bar{r}^* \right) \right\|
\leq \frac{\varepsilon}{2}.
\]
Using \cite[Lemma~2.2]{Ror92StructureUHF2} at the second step, we deduce
\begin{align*}
\left( \varphi_{i,k}(x)-\varepsilon \right)_+
&=\left( \varphi_{i,k}\left( (x-\frac{\varepsilon}{2})_+ \right) -\frac{\varepsilon}{2} \right)_+ \\
&\precsim \varphi_{j,k}\left( \bar{r}\varphi_{i,j}((y-\delta)_+)\bar{r}^* \right) \\
&\precsim \left( \varphi_{j,k}(y) -\delta \right)_+.
\end{align*}
Then $\psi_{i,\infty}( [ (x-\varepsilon)_+ ] ) \leq \psi_{i,\infty}( [ (y-\delta)_+ ] )$.
Using this at the third step, we obtain
\[
s'
= \psi_{i,\infty}(a')
\leq \psi_{i,\infty}( [(x-\varepsilon)_+] )
\leq \psi_{i,\infty}( [(y-\delta)_+] )
\prec \psi_{i,\infty}( [y] )
= t.
\]
Hence, $s'\prec t$ as desired.
\end{proof}

%==========================================================================================
%==========================================================================================
\chapter{Completed Cuntz semigroups}
\label{sec:catCu}

%------------------------------------------------------------------------------------------
\index{terms}{completed Cuntz semigroup}

In the first part of this chapter, we recall the definition of the category $\CatCu$ of (abstract) completed Cuntz semigroups, as introduced in \cite{CowEllIva08CuInv}.
We show that $\CatCu$ is a full, reflective subcategory of $\CatPreW$;
see \autoref{prp:CureflectivePreW}.
The reflection of a $\CatPreW$-semigroup $S$ in $\CatCu$ is called its \emph{$\Cu$-completion}.
Since $\CatPreW$ has inductive limits, the same holds for $\CatCu$.
This generalizes \cite[Theorem~2]{CowEllIva08CuInv} and it provides a new description of inductive limits in $\CatCu$;
see \autoref{prp:limitsCu}.

In the second part, we consider the functor $\Cu\colon\CatCa\to\CatCu$, as introduced in \cite{CowEllIva08CuInv}.
It associates to a \ca\ $A$ the set of Cuntz equivalence classes of positive elements in the stabilization of $A$, that is, $\Cu(A)=(A\otimes\K)_+/\!\!\sim$.
It turns out that $\Cu(A)$ is an object in $\CatCu$ and we call it the \emph{completed Cuntz semigroup} of $A$.
The main result of this chapter, \autoref{prp:commutingFctrs}, states that for every \ca\ $A$, its completed Cuntz semigroup $\Cu(A)$ is naturally isomorphic to the $\Cu$-completion of its pre-completed Cuntz semigroup $W(A)$.
Moreover, all involved functors are continuous.

%------------------------------------------------------------------------------------------
%==========================================================================================
\section{The category \texorpdfstring{\CatCu}{Cu}}

%==========================================================================================
\begin{pgr}[Axioms for the category $\CatCu$]
\label{pgr:axiomsCu}
\index{terms}{rapidly increasing sequence}
\index{terms}{O1-O4@\axiomO{1}-\axiomO{4}}
Given a \pom\ $S$, the following axioms were introduced in \cite{CowEllIva08CuInv};
see also \cite{Rob13Cone}.
Recall the definition of the compact containment relation $\ll$ from \autoref{pgr:axiomsW}.
\begin{itemize}
\item[\axiomO{1}]
Every increasing sequence $(a_n)_{n\in\N}$ in $S$ has a supremum $\sup_n a_n\in S$.
\item[\axiomO{2}]
Every element $a\in S$ is the supremum of a sequence $(a_n)_n$ such that $a_n\ll a_{n+1}$ for all $n$.
\item[\axiomO{3}]
If $a',a,b',b\in S$ satisfy $a'\ll a$ and $b'\ll b$, then $a'+b'\ll a+b$.
\item[\axiomO{4}]
If $(a_n)_n$ and $(b_n)_n$ are increasing sequences in $S$, then $\sup_n(a_n+b_n)=\sup_n a_n+\sup_n b_n$.
\end{itemize}
A sequence as in \axiomO{2} is called \emph{rapidly increasing}.
\end{pgr}

%==========================================================================================
\begin{dfn}[{The category $\CatCu$; Coward, Elliott, Ivanescu, \cite{CowEllIva08CuInv}}]
\label{dfn:catCu}
\index{terms}{Cu-semigroup@$\CatCu$-semigroup}
\index{terms}{Cu-morphism@$\CatCu$-morphism}
\index{terms}{Cu-morphism@$\CatCu$-morphism!generalized}
\index{terms}{category@category!$\CatCu$}
\index{symbols}{Cu@$\CatCu$ \quad (category of $\CatCu$-semigroups)}
\index{symbols}{Cu(S,T)@$\CatCuMor(S,T)$ \quad ($\CatCu$-morphisms)}
\index{symbols}{Cu[S,T]@$\CatCuGenMor{S,T}$ \quad (generalized $\CatCu$-morphisms)}
A \emph{$\CatCu$-semigroup} is a \pom\ that satisfies axioms \axiomO{1}-\axiomO{4} from \autoref{pgr:axiomsCu}.
Given $\CatCu$-semigroups $S$ and $T$, a \emph{$\CatCu$-morphism} $f\colon S\to T$ is a $\CatPom$-morphism that preserves compact containment and suprema of increasing sequences.
We denote the collection of such maps by $\CatCuMor(S,T)$.
We let $\CatCu$ be the category whose objects are $\Cu$-semigroups and whose morphisms are $\Cu$-morphisms.

A \emph{generalized $\CatCu$-morphism} between $\Cu$-semigroups is a $\CatCu$-morphism that does not necessarily preserve compact containment, that is, a $\CatPom$-morphism that preserves suprema of increasing sequences.
We denote the set of generalized $\CatCu$-morphisms by $\CatCuGenMor{S,T}$.
\end{dfn}

%==========================================================================================
\begin{rmks}
\label{rmk:catCu}
\index{terms}{directed complete partially ordered set (dcpo)}
\index{terms}{w-continuous@$\omega$-continuous}
(1)
In lattice theory, a partially ordered set $M$ is called a \emph{directed complete partially ordered set}, often abbreviated to \dcpo, if each upward directed set in $M$ has a supremum;
see \cite[Definition~0-2.1, p.~9]{GieHof+03Domains}.
If the existence of suprema is only required for increasing sequences, then $M$ is called an $\omega$-\dcpo.

A \dcpo\ $M$ is called \emph{continuous} if each element $a$ in $M$ is the supremum of the elements compactly contained in $a$;
see \cite[Definition~I-1.6, p.~54]{GieHof+03Domains}.
Recall from \autoref{pgr:axiomsW} that we use a sequential version of compact containment.
An $\omega$-\dcpo\ is called \emph{$\omega$-continuous} if every element $a$ is the supremum of a sequence $(a_k)_k$ where $a_k$ is sequentially compactly contained in $a_{k+1}$ for each $k$.

Thus, axioms \axiomO{1} and \axiomO{2} mean exactly that the \pom\ in question is a $\omega$-continuous $\omega$-\dcpo.

(2)
Let $S$ be a \pom, considered with the derived auxiliary relation $\ll$.
Recall from \autoref{pgr:axiomsW} that $S$ is called \emph{countably-based} if there exists a countable subset $B\subset S$ such that, whenever $a',a\in S$ satisfy $a'\ll a$, there exists $b\in B$ with $a'\leq b\ll a$.
If $S$ satisfies \axiomO{1} and \axiomO{2} from \autoref{pgr:axiomsCu}, then this is equivalent to the condition that every $a\in S$ is the supremum of a rapidly increasing sequence $(a_k)_k$ with $a_k\in B$ for each $k$.

Assume now that $S$ is a countably-based \pom\ satisfying \axiomO{1} and \axiomO{2}.
Then $S$ is a continuous \dcpo.
In order to verify that $S$ is a \dcpo, let $X$ be an upward directed subset of $S$ and let $B$ be a (countable) basis for $S$.
Set $X':=\{b\in B : b\ll x\text{ for some }x\in X\}$.
Then $X'$ is countable and upward directed.
Indeed, if $b_1$, $b_2\in X'$, find $x\in X$ satisfying $b_1,b_2\ll x$.
By \axiomO{2} and the assumption that $S$ is countably-based, we can choose a rapidly increasing sequence $(x_k)_k$ in $B$ such that $x=\sup x_k$.
Choose $n\in\N$ with $b_1,b_2\ll x_n\ll x$.
Then $x_n\in X'$.
As $X'$ is countable and upward directed, there is an increasing sequence $(c_k)_k$ in $X'$ such that $\sup_k c_k$, which exists by \axiomO{1}, is the supremum of $X'$.

Next, let us verify that $\sup X'$ is the supremum of $X$.
If $x\in X$ and $x'\in B$ satisfy $x'\ll x$, then $x'\in X'$ and thus $\sup X'\geq x'$.
Since $S$ is countably-based, it follows that $\sup X'\geq x$ for every $x\in X$.
Moreover, if $y\in S$ satisfies $y\geq x$ for every $x\in X$, then $y\geq x'$ for every $x'\in X'$, and consequently $\sup X'=\sup X$.
This shows that $S$ is a \dcpo.
It is left to the reader to verify that $S$ is also continuous.

(3) \index{terms}{Scott-topology}
A $\CatPom$-morphism between $\CatCu$-semigroups preserves suprema of increasing sequences if and only if it is sequentially continuous for the so-called \emph{Scott-topology};
see \cite[Definition~II-2.2, p.~158]{GieHof+03Domains}.
\end{rmks}

%------------------------------------------------------------------------------------------
Given a $\CatCu$-semigroup $S$, it is easily checked that the pair $(S,\ll)$ is a $\CatW$-semigroup.
The next result implies that under this identification, the notions of (generalized) $\CatCu$-morphisms and (generalized) $\CatW$-morphisms agree.

%==========================================================================================
\begin{lma}
\label{prp:CuWMor}
Let $S$ and $T$ be $\CatCu$-semigroups, and let $f\colon S\to T$ be a $\CatPom$-morphism.
Then the following are equivalent:
\beginEnumStatements
\item
The map $f\colon S\to T$ preserves suprema of increasing sequences.
\item
The map $f\colon S\to T$ is continuous in the sense of \autoref{dfn:catW}.
\item
We have $f(a)=\sup f(a^\ll)$ for each $a\in S$.
\end{enumerate}
\end{lma}
\begin{proof}
For the implication `\implStatements{1}{2}', let $a\in S$ and $b\in T$ satisfy $b\ll f(a)$.
Using \axiomO{2} for $S$, we choose a rapidly increasing sequence $(a_k)_k$ in $S$ with $a=\sup_k a_k$.
Since $f(a)=\sup_k f(a_k)$ by assumption, we can choose $k$ such that $b\leq f(a_k)$.
Then $a_k$ has the desired properties.

The implication `\implStatements{2}{3}' was shown in \autoref{rmk:catW}(4).

Finally, to show the implication `\implStatements{3}{1}', let $(a_k)_k$ be an increasing sequence in $S$.
Set $a:=\sup_k a_k$.
We clearly have $f(a)\geq\sup_k f(a_k)$.
For the converse inequality, we choose a rapidly increasing sequence $(c_n)_n$ in $S$ satisfying $a=\sup_n c_n$.
It follows from the assumption that $f(a)=\sup_n f(c_n)$.
Now, given any $b\in T$ satisfying $b\ll f(a)$, we can choose an index $n$ such that $b\leq f(c_n)$.
Since $c_n\ll a=\sup_k a_k$, we can choose an index $k$ with $c_n\leq a_k$.
Then
\[
b\leq f(c_n)\leq f(a_k)\leq \sup_k f(a_k).
\]
Thus, for every $b$ satisfying $b\ll f(a)$ we have $b\leq\sup_kf(a_k)$.
Using \axiomW{2} we have $f(a)=\sup\{b\in T : b\ll f(a)\}$.
It follows $f(a)\leq \sup_kf(a_k)$, as desired.
\end{proof}

%==========================================================================================
\begin{pgr}[$\CatCu$-completions]
\label{pgr:CufullW}
Consider the functor $\CatCu\to\CatW$ that maps a $\CatCu$-se\-mi\-group $S$ to the $\CatW$-semigroup $(S,\ll)$, and that sends a $\CatCu$-morphism $f\colon S\to T$ to the same map $f$, considered as a $\CatW$-morphism.
It follows from \autoref{prp:CuWMor} that this functor is fully faithful.
This means that we may consider $\CatCu$ as a full subcategory of $\CatW$, and we will therefore usually not distinguish between a $\CatCu$-semigroup $S$ and the associated $\CatW$-semigroup $(S,\ll)$.
\end{pgr}

%------------------------------------------------------------------------------------------
We will now show that $\CatCu$ is a reflective subcategory of $\CatW$, using again the idea of a universal completion as described in \cite[\S~2]{KeiLaw09DCompletion}.
We first show that every $\CatPreW$-semigroup $S$ can be suitably completed to a $\CatCu$-semigroup $\gamma(S)$;
see \autoref{prp:CuificationExists}.
We then show that this has the desired universal properties;
see \autoref{thm:Cuification}.

The proof of the following result is inspired by the so-called round ideal completion, which associates to a partially ordered set with an auxiliary relation a continuous \dcpo;
see \cite[Theorem~2.4]{Law97RoundIdeal}.
The construction given here is a sequential version that also takes the additive structure into account.

%==========================================================================================
\begin{prp}
\label{prp:CuificationExists}
Let $(S,\prec)$ be a $\CatPreW$-semigroup.
Then there exist a $\CatCu$-semi\-group $\gamma(S)$ and a $\CatW$-morphism $\alpha\colon S\to\gamma(S)$ satisfying the following conditions:
\beginEnumConditions
\item
The map $\alpha$ is an `embedding' in the sense that $a'\prec a$ whenever $\alpha(a')\ll\alpha(a)$, for any $a',a\in S$.
\item
The map $\alpha$ has `dense image' in the sense that for every $b',b\in \gamma(S)$ with $b'\ll b$ there exists $a\in S$ such that $b'\leq\alpha(a)\leq b$.
\end{enumerate}
In particular, if $S$ is countably-based, then so is $\gamma(S)$.
\end{prp}
\begin{proof}
To construct $\gamma(S)$, first consider the set $\overline{S}$ of $\prec$-increasing sequences in $S$.
We write such sequences as $\vect{a}=(a_k)_k=(a_1,a_2,\ldots)$.
For $\vect{a}$ and $\vect{b}$ in $\overline{S}$, we define their sum as $\vect{a}+\vect{b}=(a_k+b_k)_k$.
We define a binary relation $\subset$ on $\overline{S}$ by setting for any $\vect{a},\vect{b}\in\overline{S}$:
\[
\vect{a}\subset\vect{b}\quad
\text{ if and only if for every } k\in\N \text{ there exists } n\in\N \text{ such that } a_k\prec b_n.
\]
It is easy to check that $\subset$ is a pre-order on $\overline{S}$.
We obtain an equivalence relation by setting for any $\vect{a},\vect{b}\in\overline{S}$:
\[
\vect{a}\sim\vect{b}\quad
\text{ if and only if }\quad \vect{a}\subset\vect{b} \text{ and } \vect{b}\subset\vect{a}.
\]
We denote the set of equivalence classes by
\[
\gamma(S):=\overline{S}/\!\!\sim.
\]
For an element $\vect{a}\in\overline{S}$, we denote its class in $\gamma(S)$ by $[\vect{a}]$.
The relation $\subset$ induces a partial order $\leq$ on $\gamma(S)$ by setting $[\vect{a}]\leq[\vect{b}]$ if and only if $\vect{a}\subset\vect{b}$, for any $\vect{a},\vect{b}\in\overline{S}$.

Since $0\prec 0$ in $S$, the sequence $\vect{0}=(0,0,\ldots)$ is an element of $\overline{S}$.
We denote its class in $\gamma(S)$ by $0$.
For each $\vect{a}\in\overline{S}$, we have $\vect{0}\subset\vect{a}$, and therefore $0\leq[\vect{a}]$.
It follows from axiom \axiomW{3} for $S$ that $\vect{a}\subset\vect{b}$ implies $\vect{a}+\vect{c}\subset\vect{b}+\vect{c}$, for any $\vect{a},\vect{b},\vect{c}\in\overline{S}$.
Thus, the addition on $\overline{S}$ induces an addition on $\gamma(S)$.
Together with the partial order and the zero element, this gives $\gamma(S)$ the structure of a \pom.

We define a binary relation $\ssubset$ on $\overline{S}$ by setting for any $\vect{a},\vect{b}\in\overline{S}$:
\[
\vect{a}\ssubset\vect{b}\quad
\text{ if and only if there exists } n\in\N \text{ such that } a_k\prec b_n \text{ for all } k\in\N.
\]
It is easy to check that $\ssubset$ is an auxiliary relation on $\overline{S}$.
This induces an auxiliary relation $\prec$ on $\gamma(S)$ by setting $[\vect{a}]\prec[\vect{b}]$ if and only if $\vect{a}\ssubset\vect{b}$, for any $\vect{a},\vect{b}\in\overline{S}$.

We will now check that $\gamma(S)$ satisfies axioms \axiomO{1}-\axiomO{4}.
To show \axiomO{1}, let an increasing sequence $\vect{a}^{(1)}\subset\vect{a}^{(2)}\subset\ldots$ in $\overline{S}$ be given.
We employ a standard diagonalization argument, which is also used to show existence of suprema in the inductive limit construction in $\CatCu$;
see for example \cite[Theorem~4.34]{AraPerTom11Cu}.
Write $\vect{a}^{(k)}=(a^{(k)}_1,a^{(k)}_2,\ldots)$ for each $k$.
We inductively choose indices $n_k\in\{0,1,2,\ldots\}$ such that
\[
a^{(i)}_{n_i+j}\prec a^{(k)}_{n_k}\quad
\text{ for all }\quad i,j \text{ with } i+j\leq k.
\]
We start with $n_1:=0$.
Since $\vect{a}^{(1)}\subset\vect{a}^{(2)}$, we can choose $n_2$ such that $a^{(1)}_{n_1+1}=a^{(1)}_1\prec a^{(2)}_{n_2}$.
Now assume $n_i$ has been chosen for $i\leq k$.
Since $\vect{a}^{(1)},\ldots,\vect{a}^{(k)}\subset\vect{a}^{(k+1)}$, we can choose $n_{k+1}$ such that
\[
a^{(1)}_{n_1+k}, a^{(2)}_{n_2+k-1}, \ldots,
a^{(k-1)}_{n_{k-1}+2}, a^{(k)}_{n_k+1}
\prec a^{(k+1)}_{n_{k+1}}
\]
This completes the inductive step.
After reindexing the sequences, we may assume $n_i=i$ and therefore $a^{(i)}_{i+j}\prec a^{(i+j)}_{i+j}$ for all $i,j\geq 1$.
Set $b_j:=a^{(j)}_j$ for each $j$, and set $\vect{b}:=(b_1,b_2,\ldots)\in\overline{S}$.
It is straightforward to check that $\vect{b}$ is the supremum of the sequence $(\vect{a}^{(k)})_k$ in $\overline{S}$, that is, $\vect{a}^k\subset\vect{b}$ (for each $k$) and if $\vect{a}^k\subset \vect{b}'$ (for each $k$) then $\vect{b}\subset\vect{b}'$.
It follows
\[
[\vect{b}] = \sup_k [\vect{a}^{(k)}],
\]
in $\gamma(S)$, which verifies \axiomO{1}.
It is left to the reader to prove axiom \axiomO{4} for $\gamma(S)$.

Next, we show that $\ssubset$ induces the compact containment relation $\ll$ on $\gamma(S)$.
To that end, we first show that the analog of \axiomO{2} holds for $\ssubset$, that is, for every $\vect{a}=(a_n)_n$ in $\overline{S}$ there exist elements $\vect{a}^{(k)}\in\overline{S}$ for $k\in\N$ such that
\[
\vect{a}=\sup_k \vect{a}^{(k)},\quad
\text{ and }\quad
\vect{a}^{(1)} \ssubset\vect{a}^{(2)} \ssubset\vect{a}^{(3)} \ssubset\ldots.
\]
Given $\vect{a}\in\overline{S}$, by \axiomW{1} for $S$, for each $k$ there is a sequence $a_{[k,1]}\prec a_{[k,2]}\prec\ldots$ with $a_k\prec a_{[k,1]}$ and $a_{[k,i]}\prec a_{k+1}$ for all $i$.
Set
\[
\vect{a}^{(k)}:=(a_1,a_2,\ldots,a_k,a_{[k,1]},a_{[k,2]},\ldots).
\]
It is straightforward to check that this sequence has the desired properties.

Now let $\vect{a},\vect{b}\in\overline{S}$ satisfy $[\vect{a}]\ll[\vect{b}]$.
As explained in the previous paragraph, we can choose a sequence of elements $\vect{b}^{(k)}\in\overline{S}$ such that $\vect{b}=\sup_k\vect{b}^{(k)}$ and such that $b^{(k)}_k=b_k$ and $b^{(k)}_i\prec b_{k+1}$ for each $k,i\in\N$.
Then
\[
[\vect{a}]\ll[\vect{b}]=\sup_k[\vect{b}^{(k)}],
\]
which implies that we can choose $k$ with $[\vect{a}]\leq[\vect{b}^{(k)}]$.
This mean $\vect{a}\subset\vect{b}^{(k)}$, and it follows easily that $\vect{a}\ssubset\vect{b}$.

Conversely, let $\vect{a},\vect{b}\in\overline{S}$ satisfy $\vect{a}\ssubset\vect{b}$.
In order to show $[\vect{a}]\ll[\vect{b}]$ in $\gamma(S)$, let $(\vect{c}^{(k)})_k$ be an increasing sequence in $\overline{S}$ satisfying $[\vect{b}]\leq \sup_k[\vect{c}^{(k)}]$.
By the argument above, $\overline{S}$ is closed under suprema of increasing sequences. We can therefore define $\vect{c}=\sup_k\vect{c}^{k}$, giving $[\vect{c}]=\sup_k[\vect{c}^k]$.
Then
\[
\vect{b}\subset \vect{c}=\sup_k\vect{c}^{(k)}.
\]
After reindexing, we may assume $\vect{c}=(c^{(k)}_k)_k$.
Since $\vect{a}\ssubset\vect{b}$, we can choose an index $n$ such that $a_i\prec b_n$ for all $i$.
Since $\vect{b}\subset\vect{c}$, we can choose $m$ such that $b_n\prec c_m=c^{(m)}_m$.
It follows that $\vect{a}\subset\vect{c}^{(m)}$ and therefore $[\vect{a}]\leq[\vect{c}^{(m)}]$.
Thus, for any $\vect{a},\vect{b}\in\overline{S}$ we have shown that
\[
\vect{a}\ssubset\vect{b}\text{ in } \overline{S}\quad
\text{ if and only if }\quad
[\vect{a}]\ll[\vect{b}]\text{ in } \gamma(S).
\]
It easily follows that \axiomO{2} holds in $\gamma(S)$.

To verify \axiomO{3} for $\gamma(S)$, we first show the analog for $\overline{S}$.
In order to verify this, let $\vect{a}',\vect{a},\vect{b}',\vect{b}\in\overline{S}$ satisfy $\vect{a}'\ssubset\vect{a}$ and $\vect{b}'\ssubset\vect{b}$.
Choose indices $m$ and $n$ such that $a_k'\prec a_m$ and $b_k'\prec b_n$ for all $k$.
Set $d:=\max\{m,n\}$.
Using \axiomW{3} for $S$, we deduce
\[
a_k'+b_k'\prec a_d+b_d,
\]
for all $k$.
This shows $\vect{a}'+\vect{b}'\ssubset\vect{a}+\vect{b}$.
It easily follows that $\gamma(S)$ satisfies \axiomO{3}.
This completes the proof that $\gamma(S)$ is a $\CatCu$-semigroup.

We define the map $\alpha\colon S\to\gamma(S)$ as follows:
Given $a\in S$, we apply \axiomW{1} to choose a sequence $a_1\prec a_2\prec\ldots$ that is cofinal in $a^\prec$.
Set $\alpha(a):=[(a_1,a_2,\ldots)]$.
It is straightforward to check that $\alpha(a)$ does not depend on the choice of the cofinal sequence in $a^\prec$ and that $\alpha$ is a $\CatW$-morphism satisfying conditions (i) and (ii) of the statement.
\end{proof}

%==========================================================================================
\begin{dfn}
\label{dfn:Cuification}
\index{terms}{Cu-completion@$\CatCu$-completion}
Let $S$ be a $\CatPreW$-semigroup.
A \emph{$\CatCu$-completion} of $S$ is a $\CatCu$-semigroup $T$ together with a $\CatW$-morphism $\alpha\colon S\to T$ satisfying the following universal property:
For every $\CatCu$-semigroup $R$ and for every $\CatW$-morphism $f\colon S\to R$, there exists a unique $\CatCu$-morphism $\tilde{f}\colon T\to R$ such that $f=\tilde{f}\circ\alpha$.
\end{dfn}

%------------------------------------------------------------------------------------------
Any two $\CatCu$-completions of a $\CatW$-semigroup are isomorphic in the following sense:
If $\alpha_i\colon S\to T_i$ are two $\CatCu$-completions, then there is a unique isomorphism $\varphi\colon T_1\to T_2$ such that $\alpha_2=\varphi\circ\alpha_1$.

%==========================================================================================
\begin{thm}
\label{thm:Cuification}
Let $S$ be a $\CatPreW$-semigroup, let $T$ be a $\CatCu$-semigroup, and let $\alpha\colon S\to T$ be a $\CatW$-morphism.
Then the following are equivalent:
\beginEnumStatements
\item
The map $\alpha$ satisfies the conditions of \autoref{prp:CuificationExists}, namely:
\beginEnumConditions
\item
The map $\alpha$ is an `embedding' in the sense that $a'\prec a$ whenever $\alpha(a')\ll\alpha(a)$, for any $a',a\in S$.
\item
The map $\alpha$ has `dense image' in the sense that for every $b',b\in T$ with $b'\ll b$ there exists $a\in S$ such that $b'\leq\alpha(a)\leq b$.
\end{enumerate}
\item
The map $\alpha$ is a $\CatCu$-completion of $S$.
\item
For every $\CatCu$-semigroup $R$ and every generalized $\CatW$-morphism $f\colon S\to R$, there exists a generalized $\CatCu$-morphism $\tilde{f}\colon T\to R$ such that $f=\tilde{f}\circ\alpha$.
Moreover, $f$ is a $\CatW$-morphism if and only if $\tilde{f}$ is a $\CatCu$-morphism.
Moreover, if $g_1,g_2\colon T\to R$ are generalized $\CatCu$-morphisms such that $g_1\circ\alpha\leq g_2\circ\alpha$, then $g_1\leq g_2$.
(We consider the pointwise ordering among morphisms.)
\end{enumerate}
\end{thm}
\begin{proof}
The implication `\implStatements{3}{2}' is clear.
To show the implication `\implStatements{1}{3}', let $R$ be a $\CatCu$-semigroup, and let $f\colon S\to R$ be a generalized $\CatW$-morphism.
It follows from conditions (i) and (ii) that for every $t\in T$, the set $\left\{ a\in S : \alpha(a)\ll t \right\}$ is $\prec$-upwards directed and contains a cofinal sequence.
Thus, we may define
\[
\tilde{f}\colon T\to R,\quad
\tilde{f}(t) := \sup \left\{f(a) : \alpha(a)\ll t \right\},\quad
\txtFA t\in T.
\]
Let us show that $\tilde{f}$ is a generalized $\Cu$-morphism.
Let $t_1,t_2\in T$.
If $t_1\leq t_2$, then it follows from condition (ii) of an auxiliary relation (see \autoref{pgr:axiomsW}) that
\[
\left\{ a\in S : \alpha(a)\ll t_1 \right\}
\subset \left\{ a\in S : \alpha(a)\ll t_2 \right\},
\]
and therefore $\tilde{f}(t_1)\leq\tilde{f}(t_2)$.

Similarly, it follows from axiom \axiomW{3} for $S$ that there is an inclusion
\[
\left\{ a\in S : \alpha(a)\ll t_1 \right\}
+ \left\{ a\in S : \alpha(a)\ll t_2 \right\}
\subset \left\{ a\in S : \alpha(a)\ll t_1+t_2 \right\},
\]
which is moreover cofinal by \axiomW{4}.
It follows that $\tilde{f}(t_1+t_2)=\tilde{f}(t_1)+\tilde{f}(t_2)$.
It is easy to check that $\tilde{f}$ is continuous in the sense of \autoref{dfn:catW}.
By \autoref{prp:CuWMor}, this implies that $\tilde{f}$ preserves suprema of increasing sequences.
Thus, $\tilde{f}$ is a generalized $\CatCu$-morphism.

To show $f=\tilde{f}\circ\alpha$, let $a\in S$ be given.
Using at the second step that $\alpha$ preserves the auxiliary relations and satisfies condition (i), and at the third step that $f$ is continuous, we obtain
\begin{align*}
\tilde{f}\circ\alpha(a)
=\sup \left\{ f(a') : \alpha(a')\ll\alpha(a) \right\}
=\sup \left\{ f(a') : a'\prec a \right\}
=f(a).
\end{align*}

We claim that if $f$ is additionally assumed to preserve the auxiliary relation, then so does $\tilde{f}$.
To see this, let $t',t\in T$ satisfy $t'\ll t$.
Choose $x\in T$ such that $t'\ll x\ll t$.
By condition (ii) for $\alpha$, we can choose $a\in S$ with $x\leq\alpha(a)\leq t$.
Then $t'\ll\alpha(a)$, and since $\alpha$ is continuous, we can choose $a'\in S$ satisfying $a'\prec a$ and $t'\leq\alpha(a')$.
We deduce
\[
\tilde{f}(t')
\leq \tilde{f}(\alpha(a'))
=f(a')
\ll f(a)
= \tilde{f}(\alpha(a))
\leq \tilde{f}(t).
\]

Finally, assume $g_1,g_2\colon T\to R$ are generalized $\CatCu$-morphisms satisfying $g_1\circ\alpha\leq g_2\circ\alpha$.
Given $t\in T$, choose an increasing sequence $(a_k)_k$ in $S$ with $t=\sup_k\alpha(a_k)$.
Then
\[
g_1(t)
=g_1(\sup_k\alpha(a_k))
=\sup_k (g_1\circ\alpha)(a_k)
\leq\sup_k(g_2\circ\alpha)(a_k)
=g_2(t),
\]
which shows $g_1\leq g_2$.

Let us show the implication `\implStatements{2}{1}'.
By \autoref{prp:CuificationExists}, we can choose a $\CatCu$-semigroup $\gamma(S)$ and a $\CatW$-morphism $\tilde{\alpha}\colon S\to\gamma(S)$ satisfying (1).
We have seen that (1) implies (2).
Thus, $\tilde{\alpha}$ is a $\CatCu$-completion satisfying (1).
Since every two $\CatCu$-completions of $S$ are isomorphic, it follows that every $\Cu$-completion satisfies (1), as desired.
\end{proof}

%==========================================================================================
\begin{rmks}
\label{rmk:Cuification}
(1)
Let $S$ be a $\CatPreW$-semigroup.
By \autoref{prp:CuificationExists} and \autoref{thm:Cuification}, we can choose a $\CatCu$-completion $\alpha\colon S\to\gamma(S)$.
Given a $\CatCu$-semi\-group $R$, assigning to a (generalized) $\CatCu$-morphism $f\colon \gamma(S)\to R$ the (generalized) $\CatW$-morphism $f\circ\alpha$ is an isomorphism of the following morphism sets:
\[
\xymatrix@R=15pt@C=15pt@M+3pt{
\CatWGenMor{S,R} & \CatCuGenMor{\gamma(S),R} \ar[l]_{\cong}\\
\CatWMor(S,R) \ar@{^{(}->}[u]
& \CatCuMor(\gamma(S),R) \ar@{^{(}->}[u] \ar[l]_{\cong}.
}
\]

(2)
Given a $\CatPreW$-semigroup $S$, let $\alpha\colon S\to\gamma(S)$ be its $\CatCu$-completion, as constructed in \autoref{prp:CuificationExists}.
As remarked in \autoref{rmk:catW}(2), $S$ satisfies \axiomW{2} if and only if $a^\prec\subset b^\prec$ implies $a\leq b$.
It follows that $\alpha$ is an order-embedding if and only if $S$ is a $\CatW$-semigroup.

(3)
More generally, let $\alpha\colon S\to T$ be a $\CatW$-morphism from a $\CatW$-semigroup $S$ to a $\CatCu$-semigroup $T$.
Then $\alpha$ is a $\CatCu$-completion if and only if $\alpha$ is an order-embedding that has `dense image' in the sense of condition (ii) from \autoref{thm:Cuification}.

Necessity follows from (2) above.
For the converse, assume $\alpha$ is an order-embedding.
Let $a',a\in S$ such that $\alpha(a')\ll\alpha(a)$.
Since $\alpha$ is continuous, we can choose $x\in S$ with $x\prec a$ and $\alpha(a')\leq\alpha(x)$.
Since $\alpha$ is an order-embedding, $a'\leq x$ and then $a'\prec a$, as desired.
\end{rmks}

%------------------------------------------------------------------------------------------
By \autoref{pgr:CufullW}, the category $\CatCu$ is a full subcategory of $\CatPreW$.
Moreover, for every $\CatPreW$-semigroup $S$, there exists a $\CatCu$-completion.
As described in \cite[\S~2]{KeiLaw09DCompletion}, this induces a reflection functor $\gamma\colon\CatPreW\to\CatCu$ (see also \autoref{rmk:Cuification}(1)).
\index{symbols}{$\gamma(S)$ \quad (reflection from $\CatPreW$ in $\CatCu$)}

%==========================================================================================
\begin{thm}
\label{prp:CureflectivePreW}
The category $\CatCu$ is a full, reflective subcategory of $\CatPreW$.
\end{thm}

%------------------------------------------------------------------------------------------
As noticed before \autoref{prp:limitsW}, it follows from general category theory that the category $\CatCu$ has inductive limits and that the reflection functor $\gamma\colon\CatPreW\to\CatCu$ is continuous.
Thus, we obtain the following generalization of \cite[Theorem~2]{CowEllIva08CuInv}.

%==========================================================================================
\begin{cor}
\label{prp:limitsCu}
\index{terms}{inductive limits@inductive limits!in $\CatCu$}
\index{symbols}{CuLim@$\CatCuLim$ \quad (inductive limit in $\CatCu$)}
The category $\CatCu$ has inductive limits. More precisely, let $((S_i)_{i\in I},(\varphi_{i,j})_{i,j\in I, i\leq j})$ be an inductive system in $\CatCu$.
Then the inductive limit in $\CatCu$ is the $\CatCu$-completion of the inductive limit in $\CatPreW$:
\[
\CatCuLim S_i = \gamma(\CatPreWLim S_i).
\]
\end{cor}

%==========================================================================================
\begin{rmk}
\label{rmk:reflectors}
The reflection functors $\mu\colon\CatPreW\to\CatW$ from \autoref{pgr:WreflectivePreW} and $\gamma\colon\CatPreW\to\CatCu$ from \autoref{prp:CureflectivePreW} commute in the sense that the composed functor $\gamma\circ\mu\colon\CatPreW\to\CatCu$ is naturally isomorphic to $\gamma$.
The situation is shown in the following diagram:
\[
\xymatrix{
\CatPreW \ar[r]^{\mu} \ar@/_1.5pc/[rr]^{\gamma}
& \CatW \ar[r]^{\gamma}
& \CatCu \\
}
\]
More precisely, starting with a $\CatPreW$-semigroup $S$, let us first consider the universal $\CatW$-morphism $\beta_S\colon S\to\mu(S)$ from $S$ to its $\CatW$-completion.
There is a universal $\CatW$-morphism $\alpha_{\mu(S)}\colon\mu(S)\to\gamma(\mu(S))$ to the $\CatCu$-completion of $\mu(S)$.
The composition $\alpha_{\mu(S)}\circ\beta_S\colon S\to\gamma(\mu(S))$ is a $\CatCu$-completion of $S$.
Since any two $\CatCu$-completions are isomorphic, there is an isomorphism $\gamma(S)\cong\gamma(\mu(S))$ which intertwines $\alpha_S$ and $\alpha_{\mu(S)}\circ\beta_S$.
\end{rmk}

\vspace{5pt}
%------------------------------------------------------------------------------------------
%==========================================================================================
\section{The completed Cuntz semigroup of a \texorpdfstring{{C}*-algebra}{{C}*-algebra}}

%==========================================================================================
\begin{pgr}
\label{pgr:CufromCAlg}
Let $A$ be \ca.
In \cite{CowEllIva08CuInv}, a new definition of the Cuntz semigroup is introduced as $\Cu(A)=(A\otimes\K)_+/\!\!\sim$, the set of Cuntz equivalence classes of positive elements in the stabilization of $A$.
(See \autoref{pgr:CuntzComparison} for the definition of Cuntz equivalence.)
The relation of Cuntz subequivalence induces a partial order on $\Cu(A)$.
Using an isomorphism $M_2(\K)\cong\K$ we get an isomorphism $\psi\colon M_2(A\otimes\K)\to A\otimes\K$, which is used to obtain a well-defined addition $[x]+[y]=[\psi(\begin{smallmatrix}x & 0 \\ 0 & y \end{smallmatrix})]$.
With this structure, $\Cu(A)$ becomes a \pom.
\end{pgr}

%==========================================================================================
\begin{dfn}
\label{dfn:CufromCAlg}
\index{terms}{Cuntz semigroup!completed}
\index{symbols}{Cu(A)@$\Cu(A)$ \quad (completed Cuntz semigroup)}
Let $A$ be a \ca.
We call
\[
\Cu(A) :=(A\otimes\K)_+/\!\!\sim
\]
the \emph{completed Cuntz semigroup} of $A$.
\end{dfn}

%==========================================================================================
\begin{prp}[{\cite[Theorem~1]{CowEllIva08CuInv}}]
\label{prp:CufromCAlg}
Let $A$ be a \ca.
Then the \pom\ $\Cu(A)$ satisfies \axiomO{1}-\axiomO{4} from \autoref{pgr:axiomsCu} and is therefore a $\Cu$-semigroup.
If $A$ is separable, then $\Cu(A)$ is countably-based.
\end{prp}

%==========================================================================================
\begin{rmks}
\label{rmk:CufromCAlg}
(1)
We will show in \autoref{prp:commutingFctrs} that $\Cu(A)$ is isomorphic to the $\CatCu$-completion of $\W(A)$.
This is why we call $\Cu(A)$ the \emph{completed} Cuntz semigroup, and $W(A)$ the \emph{pre-completed} Cuntz semigroup of $A$.

(2)
Another way of looking at $\Cu(A)$ is to identify it with $W(A\otimes\K)$.
In fact, the \starHom\ $A\otimes\K\to M_\infty(A\otimes\K)$ given by embedding an element in the upper-left corner induces a bijection of Cuntz equivalence classes, respecting the given order and addition.
Thus, as a \pom, $\Cu(A)$ is nothing but $W(A\otimes\K)$.
The auxiliary relation $\prec$ on $W(A\otimes\K)$ as defined in \autoref{dfn:WfromCAlg} is precisely the compact containment relation, which is deduced from the order structure.

Indeed, it was shown in \cite{CowEllIva08CuInv} that for every $b\in (A\otimes\K)_+$ and $\varepsilon>0$ we have $[(b-\varepsilon)_+]\ll [b]$ in $W(A\otimes\K)$.
It follows in particular that $[p]\ll [p]$ for any projection $p$, so projections are a natural source of compact elements in $\Cu(A)$ (sometimes the only source; see \cite{BroCiu09IsoHilbert}).
Given $a,b\in (A\otimes\K)_+$, we have by definition that $[a]\prec [b]$ if and only if $[a]\leq [(b-\varepsilon)_+]$ for some $\varepsilon>0$.
It follows that
\[
[a]\prec [b]\text{ in } \W(A\otimes\K)\quad
\text{ if and only if}\quad
[a]\ll [b]\text{ in } \Cu(A)=W(A\otimes\K).
\]
\end{rmks}

%==========================================================================================
\begin{pgr}
\label{pgr:morphisms}
Let $\varphi\colon A\to B$ be a \starHom\ (respectively a c.p.c.\ order-zero map).
Then $\varphi$ naturally extends to a \starHom\ (respectively a c.p.c.\ order-zero map) between the stabilizations, which we denote by $\bar{\varphi}\colon A\otimes\K\to B\otimes\K$.
By \autoref{prp:mor_from_cpc0}, this induces a (generalized) $\CatW$-morphism $W(\bar{\varphi})\colon W(A\otimes\K)\to W(B\otimes\K)$.
Using the identification $\Cu(A)=W(A\otimes\K)$ (see \autoref{rmk:CufromCAlg}(2)), and \autoref{prp:CuWMor}, the map $W(\bar{\varphi})$ corresponds to a (generalized) $\CatCu$-morphism
\[
\Cu(\varphi)\colon\Cu(A)\to\Cu(B).
\]
One obtains a functor $\Cu\colon\CatCa\to\CatCu$;
see \cite[Theorem~2]{CowEllIva08CuInv}.

We have seen that c.p.c.\ order-zero maps between \ca{s} naturally induce generalized $\CatCu$-morphisms between their respective completed Cuntz semigroups.
Another source of generalized $\CatCu$-morphisms are lower-semicontinuous $2$-quasitraces.
For each \ca\ $A$, these are in natural one-to-one correspondence with the generalized $\CatCu$-morphisms from $\Cu(A)$ to the extended positive real line, $[0,\infty]$, also called the functionals on $\Cu(A)$; see \autoref{pgr:fctl}.
We refer the reader to \cite[Section~2.9]{BlaKir04PureInf} and \cite[Section~4]{EllRobSan11Cone} for details.
\end{pgr}

%==========================================================================================
\begin{pgr}
\label{pgr:connectionFunctors}
\index{symbols}{C*@$\CatCa$ \quad (category of \ca{s})}
Let us clarify the connection between the functors
\[
\W\colon\CatLocCa\to\CatW\quad
\text{ and }\quad
\Cu\colon\CatCa\to\CatCu.
\]
The category $\CatCa$ of \ca{s} with \starHom{s} is a full, reflective subcategory of $\CatLocCa$.
Indeed, assigning to a local \ca\ $A$ its completion $\overline{A}$ extends to a functor $\gamma\colon\CatLocCa\to\CatCa$ which is left adjoint to the inclusion of $\CatCa$ in $\CatLocCa$.\index{terms}{category@category!$\CatCa$}

On the other hand, we have the functor $\gamma\colon\CatW\to\CatCu$ from \autoref{prp:CureflectivePreW}.
In \autoref{prp:commutingFctrs}, we will show that the functors $\W$ and $\Cu$ are intertwined by these completion functors.
\end{pgr}

%==========================================================================================
\begin{lma}
\label{prp:CuntzCompCompletion}
Let $A$ be a local \ca, considered as a dense subalgebra of its completion $\overline{A}$.
Then:
\beginEnumStatements
\item
For $x,y\in A_+$, we have $x\precsim y$ in $A$ if and only if $x\precsim y$ in $\overline{A}$.
\item
For every $x\in \overline{A}_+$ and $\varepsilon>0$, there exists $y\in A$ such that $(x-\varepsilon)_+\precsim y\precsim x$ in $\overline{A}$.
\end{enumerate}
\end{lma}
\begin{proof}
(1):
The forward implication is obvious.
To show the converse implication, assume $x\precsim y$ in $\overline{A}$.
Then $x=\lim_k z_k^*yz_k$ for some sequence $(z_k)_k$ in $\overline{A}$.
Since $A$ is dense in $\overline{A}$, we may approximate each $z_k$ arbitrarily well by elements from $A$.
A diagonalization argument shows that $x\precsim y$ in $A$.

(2):
Let $x$ and $\varepsilon$ be as in the statement.
Since $A$ is dense in $\overline{A}$, we can choose $z\in A$ with $\|x-z\|<\tfrac{\varepsilon}{2}$.
By \cite[Proposition~2.2]{Ror92StructureUHF2}, we have $(z-\tfrac{\varepsilon}{2})_+\precsim x$ in $\overline{A}$.
Similarly, since $\|x-(z-\tfrac{\varepsilon}{2})_+\|<\varepsilon$, we obtain $(x-\varepsilon)_+\precsim (z-\tfrac{\varepsilon}{2})_+$.
Thus, the element $y:=(z-\tfrac{\varepsilon}{2})_+$ has the desired properties.
\end{proof}

%==========================================================================================
\begin{thm}
\label{prp:commutingFctrs}
The compositions $\gamma\circ\W$ and $\Cu\circ\gamma$ are naturally isomorphic as functors $\CatLocCa\to\CatCu$.
This means that the following diagram commutes (up to natural isomorphism):
\[
\xymatrix@R=15pt@M+3pt{
\CatLocCa \ar@/_1pc/[d]_{\gamma} \ar[r]^{\W}
& \CatW \ar@/^1pc/[d]^{\gamma} \\
\CatCa \ar@{^{(}->}[u]  \ar[r]^{\Cu}
& \CatCu \ar@{^{(}->}[u] \\
}
\]

In particular, if $A$ is a \ca, then its completed Cuntz semigroup $\Cu(A)$ is naturally isomorphic to the $\CatCu$-completion $\gamma(W(A))$ of its pre-completed Cuntz semigroup $W(A)$.
\end{thm}
\begin{proof}
Let $A$ be a local \ca.
Set $B:=M_\infty(A)$, which is again a local \ca.
Let $\iota\colon B\to\overline{B}$ be the natural inclusion map into the completion.
Note that there is a natural isomorphism $\overline{B}\cong\gamma(A)\otimes\K$.

We have $W(A)=B_+/\!\!\sim$ and $\Cu(\gamma(A))=\overline{B}_+/\!\!\sim$, and the \starHom\ $\iota$ induces a $\CatW$-morphism $W(\iota)\colon W(A)\to\Cu(\gamma(A))$.
By \autoref{prp:CuntzCompCompletion}, $W(\iota)$ is an order-embedding with `dense image' in the sense of  \autoref{thm:Cuification} (ii).
It then follows from \autoref{rmk:Cuification}(3) that $W(\iota)$ is a $\Cu$-completion of $W(A)$.
By uniqueness of $\Cu$-completions combined with \autoref{prp:CuificationExists}, we obtain $\Cu(\gamma(A))\cong \gamma(W(A))$.
%It is enough to show that $W(\iota)$ is the $\Cu$-completion of $W(A)$.
%By \autoref{rmk:Cuification}(3), since $\W(A)$ is a $\CatW$-semigroup, it is enough to check that $W(\iota)$ is an order-embedding satisfying condition (ii) of \autoref{thm:Cuification}.
%But this follows directly from \autoref{prp:CuntzCompCompletion}.
\end{proof}

%------------------------------------------------------------------------------------------
By \autoref{prp:functorW}, the functor $\W\colon\CatLocCa\to\CatW$ is continuous.
Since $\gamma\colon\CatW\to\CatCu$ is also continuous, we obtain from \autoref{prp:commutingFctrs} the following generalization of \cite[Theorem~2]{CowEllIva08CuInv}.

%==========================================================================================
\begin{cor}
\label{prp:functorCu}
The functor $\Cu\colon\CatCa\to\CatCu$ is continuous.
More precisely, given an inductive system of \ca{s} $((A_i)_{i\in I},(\varphi_{i,j})_{i,j\in I, i\leq j})$, there are natural isomorphisms:
\[
\Cu(\CatCaLim A_i) \cong \gamma(\CatWLim \W(A_i)) \cong \CatCuLim\Cu(A_i).
\]
\end{cor}

%------------------------------------------------------------------------------------------
%==========================================================================================
\chapter{Additional axioms}
\label{sec:addAxioms}

%------------------------------------------------------------------------------------------
In this chapter, we consider additional axioms for (pre)completed Cuntz semigroups.
For $\CatCu$-semigroups, these are denoted by \axiomO{5} and \axiomO{6}, and they are satisfied by all completed Cuntz semigroups of \ca{s};
see \autoref{prp:o5o6}.
We work with a slightly stronger version of \axiomO{5} than the one that has previously appeared in the literature.
The advantage is that the new \axiomO{5} passes to inductive limits in $\CatCu$;
see \autoref{prp:axiomsPassToLim}.

We also introduce axioms \axiomW{5} and \axiomW{6} for pre-completed Cuntz semigroups, which are the exact counterparts of \axiomO{5} and \axiomO{6}.
Indeed, a $\CatPreW$-semigroup satisfies \axiomW{5} if and only if its $\CatCu$-completion satisfies \axiomO{5}, and analogously for \axiomW{6} and \axiomO{6};
see \autoref{prp:axiomsPassToCu}.

\vspace{5pt}

%------------------------------------------------------------------------------------------
Let $A$ be a \ca.
The axiomatic description of $\Cu(A)$ as an object in $\CatCu$ had a positive impact in the study of the Cuntz semigroup as an invariant.
For instance, the structure as a $\omega$-continuous $\omega$-\dcpo\ provides $\Cu(A)$ with nice topological properties.

However, the category $\CatCu$ of (abstract) Cuntz semigroups is still far bigger than the subcategory of concrete Cuntz semigroups that are isomorphic to $\Cu(A)$ for some \ca\ $A$.
For example, it is shown in \cite[Theorem~1.3]{Rob13CuSpaces2D} that the semigroup $\Lsc(S^2,\overline{\N})$ of lower-semicontinuous functions from the sphere to $\overline{\N}=\{0,1,2,\ldots,\infty\}$ is not the Cuntz semigroup of any \ca.
In order to get a better understanding of the class of concrete Cuntz semigroups, it has been useful to determine additional axioms satisfied by Cuntz semigroups of \ca{s}.

%==========================================================================================
\begin{dfnChapter}
\label{dfn:addAxioms}
\index{terms}{O5@\axiomO{5}}
\index{terms}{almost algebraic order}
\index{terms}{O6@\axiomO{6}}
\index{terms}{almost Riesz decomposition}
\index{terms}{weak cancellation}
Let $S$ be a $\CatCu$-semigroup.
\begin{itemize}
\item[\axiomO{5}]
We say that $S$ has \emph{almost algebraic order}, or that $S$ satisfies \axiomO{5}, if for every $a',a,b',b,c\in S$ that satisfy
\[
a+b\leq c,\quad
a'\ll a,\quad
b'\ll b,
\]
there exists $x\in S$ such that
\[
a'+x\leq c\leq a+x,\quad
b'\leq x.
\]
\item[\axiomO{6}]
We say that $S$ has \emph{almost Riesz decomposition}, or that $S$ satisfies \axiomO{6}, if for every $a',a,b,c\in S$ that satisfy
\[
a'\ll a\leq b+c,
\]
there exist elements $e$ and $f$ in $S$ such that
\[
a'\leq e+f,\quad
e\leq a,b,\quad
f\leq a,c.
\]
\item[\axiom{C}]
We say that $S$ has \emph{weak cancellation}, or that $S$ is \emph{weakly cancellative}, if for every $a,b,x\in S$ we have that $a+x\ll b+x$ implies $a\ll b$.
\end{itemize}
\end{dfnChapter}

%==========================================================================================
\begin{rmksChapter}
\label{rmk:addAxioms}
(1)
The axiom \axiomO{5} of almost algebraic order was first considered in \cite[Lemma~7.1]{RorWin10ZRevisited}.
It also appeared in \cite[Corollary~4.16]{OrtRorThi11CuOpenProj} and \cite[2.1]{Rob13Cone}.
Note, however, that the version of \axiomO{5} given here is formally stronger than the original versions in the literature.
Nevertheless, \axiomO{5} is satisfied by every Cuntz semigroup coming from a \ca{}
(see \autoref{prp:o5o6}), and the proof is essentially the same as the original one in \cite[Lemma~7.1]{RorWin10ZRevisited}, with additional care in the choice of the complement.

The new \axiomO{5} has the advantage that it passes to inductive limits in the category $\CatCu$;
see \autoref{prp:axiomsPassToCu}.
This seems unlikely for the original \axiomO{5}, although we have no example where the original \axiomO{5} does not pass to an inductive limit.
In \autoref{rmk:functorToo5o6}, we show that for weakly cancellative $\CatCu$-semigroups, the new \axiomO{5} is equivalent to the original version of the axiom.

(2)
Axiom \axiomO{6} was introduced in \cite[\S 4]{Rob13Cone}.
It was shown to hold for completed Cuntz semigroups of \ca{s} in \cite[Proposition~5.1.1]{Rob13Cone}.

(3)
The axiom of weak cancellation was introduced in \cite[after Lemma~1]{RobSan10ClassificationHom}.
The definition given there is equivalent to \autoref{dfn:addAxioms}.
It is also equivalent to the property that $a+c\leq b+c'$ for $c'\ll c$ implies $a\leq b$, which was shown to hold in completed Cuntz semigroups of \ca{s} with stable rank one;
see \cite[Theorem~4.3]{RorWin10ZRevisited}.
\end{rmksChapter}

%------------------------------------------------------------------------------------------
The spirit of $\CatW$-semigroups is that the order relation $\leq$ is derived from the auxiliary relation $\prec$.
It is therefore natural to formulate versions of axioms \axiomO{5} and \axiomO{6} only in terms of the auxiliary relation, without using the partial order.
That the axioms \axiomW{5} and \axiomW{6} of \autoref{dfn:w5w6} are the `correct' analogs is justified by \autoref{prp:axiomsPassToCu}.
We also formulate the axiom of weak cancellation for $\CatPreW$-semigroups, simply by replacing the compact containment relation by an arbitrary auxiliary relation.
It should cause no confusion that we call this axiom `weak cancellation' as well.

%==========================================================================================
\begin{dfnChapter}
\label{dfn:w5w6}
\index{terms}{W5@\axiomW{5}}
\index{terms}{W6@\axiomW{6}}
Let $(S,\prec)$ be a $\CatPreW$-semigroup.
We define the axioms \axiomW{5}, \axiomW{6} and weak cancellation for $S$ as follows:
\begin{itemize}
\item[\axiomW{5}]
We say that $S$ satisfies \axiomW{5} if for every $a',a,b',b,c,\tilde{c}\in S$ that satisfy
\[
a+b\prec c,\quad
a'\prec a,\quad
b'\prec b,\quad
c\prec\tilde{c},
\]
there exist elements $x'$ and $x$ in $S$ such that:
\[
a'+x\prec \tilde{c},\quad
c\prec a+x',\quad
b'\prec x'\prec x.
\]
\item[\axiomW{6}]
We say that $S$ satisfies \axiomW{6} if for every $a',a,b,c\in S$ that satisfy
\[
a'\prec a\prec b+c,
\]
there exist elements $e$ and $f$ in $S$ such that
\[
a'\prec e+f,\quad
e\prec a,b,\quad
f\prec a,c.
\]
\item[\axiom{C}]
We say that $M$ satisfies \emph{weak cancellation}, or that $M$ is \emph{weakly cancellative}, if for every $a,b,x\in S$ we have that $a+x\prec b+x$ implies $a\prec b$.
\end{itemize}
\end{dfnChapter}

%==========================================================================================
\begin{thmChapter}
\label{prp:axiomsPassToCu}
Let $S$ be a $\CatPreW$-semigroup, and let $\gamma(S)$ be its $\CatCu$-com\-ple\-tion.
Then:
\beginEnumStatements
\item
The semigroup $S$ satisfies \axiomW{5} if and only if $\gamma(S)$ satisfies \axiomO{5}.
\item
The semigroup $S$ satisfies \axiomW{6} if and only if $\gamma(S)$ satisfies \axiomO{6}.
\item
The semigroup $S$ is weakly cancellative if and only if $\gamma(S)$ is.
\end{enumerate}
\end{thmChapter}
\begin{proof}
Given an element $s\in S$, we will denote its image in $\gamma(S)$ by $\bar{s}$.

First, let us show that weak cancellation passes from $S$ to its $\CatCu$-completion.
Let $a,b,c\in\gamma(S)$ satisfy $a+c\ll b+c$.
Using that $\alpha$ has dense image in the sense of \autoref{thm:Cuification}, we can find elements $s,t\in S$ such that $\bar{s}\ll b$, $\bar{t}\ll c$, and $a+c\ll\bar{s}+\bar{t}$.
We can choose an increasing sequence $(r_n)_n$ in $S$ such that $a=\sup_n\bar{r}_n$.
Then $\bar{r}_n+\bar{t}\ll\bar{s}+\bar{t}$ for each $n$.
Since $\alpha$ is an embedding in the sense of \autoref{thm:Cuification}, we obtain the same inequality for the pre-images in $S$.
Since $S$ is weakly cancellative, we have $r_n\prec s$ for each $n$.
Then
\[
a=\sup_n\bar{r}_n\leq\bar{s}\ll b,
\]
which verifies that $\gamma(S)$ is weakly cancellative.
The converse follows directly from the properties of $\alpha$.
Using similar methods one proves \enumStatement{2}.

In order to verify that \axiomO{5} for $\gamma(S)$ implies \axiomW{5} for $S$, let $a',a,b',b,c,\tilde{c}\in S$ satisfy
\[
a+b\prec c,\quad
a'\prec a,\quad
b'\prec b,\quad
c\prec\tilde{c}.
\]
Choose $b_0,c_0\in S$ such that $b'\prec b_0\prec b$ and $c\prec c_0\prec\tilde{c}$.
Then
\[
\bar{a}+\bar{b}\leq\bar{c}_0,\quad
\bar{a}'\ll\bar{a},\quad
\bar{b}_0\ll\bar{b}.
\]
Applying \axiomO{6}, we obtain $y\in\gamma(S)$ with
\[
\bar{a}'+y\leq \bar{c}_0\leq \bar{a}+y,\quad
\bar{b}_0\leq y.
\]
Then $\bar{c}\ll\bar{a}+y$ and $\bar{b}'\ll y$.
Therefore, we can choose $x\in S$ satisfying $\bar{x}\leq y$, $c\prec a+x$, and $b'\prec x$.
Then we can find $x'\in S$ such that $x'\prec x$, $c\prec a+x'$, and $b'\prec x'$.
It follows that $x'$ and $x$ have the desired properties to verify \axiomW{5}.

In order to prove the other implication of \enumStatement{1}, let us assume that $S$ satisfies \axiomW{5}.
We have to show that $\gamma(S)$ satisfies \axiomO{5}.
Let $a',a,b',b,c\in\gamma(S)$ satisfy
\[
a+b\leq c,\quad
a'\ll a,\quad
b'\ll b.
\]

Choose a rapidly decreasing sequence $(a_n)_n$ in $S$ such that $a'\ll \bar{a}_{n+1}\ll \bar{a}_{n}\ll a$ for all $n$.
Choose $s'$ and $s$ in $S$ such that $b'\ll\bar{s}'\ll\bar{s}\ll b$.
Finally, choose a rapidly increasing sequence $(c_n)_n$ in $S$ such that $c=\sup_n \bar{c}_n$.
We can moreover assume $\bar{a}_1+\bar{s}\ll \bar{c}_1$, and so $a_1+s\prec c_1$.

We will inductively find elements $x_n$ and $x_n'$ in $S$ satisfying:
\begin{align}
\tag{$R_n$}
x_{n-1}'\prec x_n'\prec x_n,\quad
a_{n+1}+x_n\prec c_{n+1},\quad
c_n\prec a_n+x_n'.
\end{align}

To make sense of ($R_1$), we set $x_0':=s'$.
Then:
\[
a_1+s\prec c_1,\quad
a_2\prec a_1,\quad
x_0\prec s,\quad
c_1\prec c_2.
\]
By \axiomW{5}, we can choose $x_1'$ and $x_1$ in $S$ fulfilling ($R_1$).

For the inductive step, assume $x_n'$ and $x_n$ have been constructed satisfying $(R_n)$.
Applying \axiomW{5} to $a_{n+1}+x_n\prec c_{n+1}$ and $a_{n+2}\prec a_{n+1}$, $x_n'\prec x_n$ and $c_{n+1}\prec c_{n+2}$, we can find $x_{n+1}'$ and $x_{n+1}$ fulfilling ($R_{n+1}$).

We obtain a rapidly increasing sequence $(x_n')_n$ in $S$.
Using the existence of suprema in $\gamma(S)$, we may set $x:=\sup_n \bar{x}_n'$.
For each $n$, we have:
\[
a'+\bar{x}_n'
\leq \bar{a}_{n+1}+\bar{x}_n'
\leq \bar{a}_{n+1}+\bar{x}_n
\leq \bar{c}_{n+1},\quad
\bar{c}_n\leq \bar{a}_n+\bar{x}_n'\leq a+\bar{x}_n.
\]
Therefore:
\[
a'+x
=\sup_n(a'+\bar{x}_n')
\leq\sup_n \bar{c}_{n+1}
=c
=\sup_n \bar{c}_n
\leq\sup_n(a+\bar{x}_n')
=a+x.
\]
Moreover, $x\geq \bar{x}_1'\geq b'$, which shows that $x$ has the desired properties.
\end{proof}

%==========================================================================================
\begin{thmChapter}
\label{prp:axiomsPassToLim}
Let $((S_i)_{i\in I},(\varphi_{i,j})_{i,j\in I, i\leq j})$ be an inductive system in $\CatPreW$.
If each $S_i$ satisfies \axiomW{5} (respectively \axiomW{6}, weak cancellation), then so does the inductive limit in $\CatPreW$.

Similarly, axiom \axiomO{5} (respectively \axiomO{6}, weak cancellation) passes to inductive limits in $\CatCu$.
\end{thmChapter}
\begin{proof}
Let us verify that weak cancellation passes to inductive limits in $\CatPreW$.
Set $S:=\CatPreWLim S_i$ and let $a,b,c\in S$ satisfy $a+c\prec b+c$.
We can choose an index $i$ and elements $x,y,z\in S_i$ with $a=[x],b=[y]$, and $c=[z]$.
By definition of $\prec$ on $S$ (see \autoref{dfn:limAuxRel}), we can find $j\geq i$ such that $\varphi_{i,j}(x+z)\prec\varphi_{i,j}(y+z)$ in $S_j$.
Using that $S_j$ is weakly cancellative, we deduce $\varphi_{i,j}(x)\prec\varphi_{i,j}(y)$ in $S_j$.
It follows $a\prec b$ in $S$, as desired.

It is shown analogously that the other axioms pass to inductive limits in $\CatPreW$.

Now, let $((S_i)_{i\in I},(\varphi_{i,j})_{i,j\in I, i\leq j})$ be an inductive system in $\CatCu$.
Assume that each $S_i$ is weakly cancellative.
Considering $S_i$ as a $\CatPreW$-semigroup that is isomorphic to its own $\CatCu$-completion, we obtain from \autoref{prp:axiomsPassToCu} that $S_i$ is weakly cancellative as a $\CatPreW$-semigroup.
It follows that the limit in $\CatPreW$, $S:=\CatPreWLim S_i$, is weakly cancellative.
By \autoref{prp:limitsCu}, the limit in $\CatCu$ is isomorphic to the $\CatCu$-completion of $S$.
Using the `only if' implication of part (3) of \autoref{prp:axiomsPassToCu}, we deduce that $\CatCuLim S_i$ is weakly cancellative.

The argument for \axiomO{5} and \axiomO{6} is completely analogous.
\end{proof}

%==========================================================================================
\begin{prpChapter}
\label{prp:o5o6}
Let $A$ be a \ca.
Then $\Cu(A)$ is a $\CatCu$-semigroup satisfying \axiomO{5} and \axiomO{6}.
\end{prpChapter}
\begin{proof}
By \cite[Theorem~1]{CowEllIva08CuInv}, $\Cu(A)$ is a $\CatCu$-semigroup.
Axiom \axiomO{6} is verified in \cite[Proposition~5.1.1]{Rob13Cone}.

The proof of \axiomO{5} is based on that of \cite[Lemma~7.1]{RorWin10ZRevisited}.
Let $a',a,b',b,c\in\Cu(A)$ satisfy $a+b\leq c$, $a'\ll a$, and $b'\ll b$.
We may assume that $A$ is stable.
Let $\tilde{x}, \tilde{y}, z\in A_+$, with $\tilde{x}\perp\tilde{y}$, be such that $a=[\tilde{x}]$, $b=[\tilde{y}]$, and $c=[z]$.
Choose $\varepsilon>0$ such that $a'\ll[(\tilde{x}-\varepsilon)_+]$ and $b'\ll[(\tilde{y}-\varepsilon)_+]$.
Since $\tilde{x}+\tilde{y}\precsim z$, by \cite[Section~2]{Ror92StructureUHF2} and \autoref{pgr:CuntzComparison}, we can choose $\delta>0$ and $r\in A$ such that
\[
(\tilde{x}-\varepsilon)_++(\tilde{y}-\varepsilon)_+=(\tilde{x}+\tilde{y}-\varepsilon)_+=r^*(z-\delta)_+r.
\]
Set $v:=(z-\delta)_+^{1/2}r$, set $x:=v(\tilde{x}-\varepsilon)_+v^*$ and set $y:=v(\tilde{y}-\varepsilon)_+v^*$.
Then $x\perp y$ and
\begin{align*}
(\tilde{x}-\varepsilon)_+
\sim(\tilde{x}-\varepsilon)_+^2
=(\tilde{x}-\varepsilon)_+^{1/2}v^*v(\tilde{x}-\varepsilon)_+^{1/2}
\sim v(\tilde{x}-\varepsilon)_+v^*
=x,
\end{align*}
and similarly $(\tilde{y}-\varepsilon)_+\sim y$.
We have shown that there exist $\delta>0$ and orthogonal positive elements $x,y$ in $\Her((z-\delta)_+)$, the hereditary sub-\ca{} generated by $(z-\delta)_+$, such that $a'\ll[x]\leq a$ and $b'\ll[y]\leq b$.

For $\eta>0$, let $f_\eta\colon\R^+\to[0,1]$ be the function that takes value $0$ at $0$, value $1$ on $[\eta,\infty)$ and which is linear on $[0,\eta]$.
Set $e:=f_{\delta}(z)$.
Then $e\sim z$ and $e$ acts as a unit on $x$ and $y$.
Choose $\beta>0$ such that $a'\leq[(x-\beta)_+]$.
Set $w:=e-f_\beta(x)$, which is positive since $e$ commutes with $x$ and therefore with $f_{\beta}(x)$.
Set $s:=[w]$, which we will show to have the desired properties to verify \axiomO{5}.

Then we have $w\in\Her(z)$ and $(x-\beta)_+\perp w$.
Moreover, the element $x+w$ is strictly positive in $\Her(z)$.
We deduce
\begin{align*}
a'+s
&\leq [(x-\beta)_+]+[w]
=[(x-\beta)_+ + w]
\leq [z]
=c \\
c
&=[z]
=[x+w]
\leq [x]+[w]
\leq a + s.
\end{align*}

Moreover, $x+w\geq\delta' e$ for some $\delta'>0$.
Therefore, since $x\perp y$,
\[
y
=y^{1/2}ey^{1/2}
\leq \frac{1}{\delta'} y^{1/2}(x+w)y^{1/2}
=\frac{1}{\delta'} y^{1/2}wy^{1/2}
\precsim w,
\]
and thus $b'\leq [y]\leq[w]=s$.
This shows that $s$ has the desired properties to verify \axiomO{5} for $\Cu(A)$.
\end{proof}

%==========================================================================================
\begin{prpChapter}
\label{prp:w5w6}
Let $A$ be a local \ca.
Then $W(A)$ is a $\CatW$-sem\-i\-group satisfying \axiomW{5} and \axiomW{6}.
\end{prpChapter}
\begin{proof}
By \autoref{prp:WfromCAlg}, $\W(A)$ is a $\CatW$-semigroup.
Let $\overline{A}$ denote the completion of $A$.
Since $\overline{A}$ is a \ca{}, it follows from \autoref{prp:o5o6} that $\Cu(\overline{A})$ satisfies \axiomO{5} and \axiomO{6}.
By \autoref{prp:commutingFctrs}, the $\CatCu$-completion of $\W(A)$ is isomorphic to $\Cu(\overline{A})$.
Therefore, \axiomW{5} and \axiomW{6} for $\W(A)$ follow from \autoref{prp:axiomsPassToCu}.
\end{proof}

%------------------------------------------------------------------------------------------
Though \axiomO{5} holds for all complete Cuntz semigroups coming from \ca{s}, we note that,
under the additional hypothesis of weak cancellation, it is equivalent to the original
formulation of the axiom:
\begin{itemize}
\item[\axiomO{5'}]
If $a\leq b$ and $a'\ll a$, then there is $x\in S$ such that $a'+x\leq b\leq a+x$.
\end{itemize}

%==========================================================================================
\begin{lmaChapter}
\label{prp:o5o5'}
Let $S$ be a $\CatCu$-semigroup.
If $S$ satisfies \axiomO{5}, then it also satisfies \axiomO{5'}.
The converse holds if $S$ is weakly cancellative.
\end{lmaChapter}
\begin{proof}
It is clear that \axiomO{5} implies \axiomO{5'} in general.
To show the converse, assume that $S$ is a weakly cancellative $\CatCu$-semigroup satisfying \axiomO{5'}.
By \autoref{prp:axiomsPassToCu}, it is enough to verify \axiomW{5}.
Suppose we are given $a',a,b',b,c,\tilde{c}\in S$ satisfying
\[
a+b\ll c,\quad
a'\ll a,\quad
b'\ll b,\quad
c\ll\tilde c.
\]
Choose $c^\sharp\in S$ with $c\ll c^\sharp\ll \tilde c$.
Applying \axiom{O5'} to $a'\ll a\leq c^\sharp$, we obtain $x\in S$ such that $a'+x\leq c^\sharp\leq a+x$.
Then
\[
b+a\leq c\ll c^\sharp\leq x+a,
\]
which by weak cancellation implies $b\ll x$.
Choose $x'\in S$ such that $x'\ll x$, $b'\ll x'$, and $c\ll a+x'$.
Then the elements $x$ and $x'$ have the desired properties to verify \axiomW{5}.
\end{proof}

%==========================================================================================
\begin{rmksChapter}
\label{rmk:functorToo5o6}
Let $\CatW_{5,6}$ be the full subcategory of $\CatW$ consisting of $\CatW$-sem\-i\-groups satisfying \axiomW{5} and \axiomW{6}.
It follows from \autoref{prp:axiomsPassToLim} that $\CatW_{5,6}$ is closed under inductive limits in $\CatW$ and therefore has inductive limits itself.
By \autoref{prp:w5w6}, the functor $\W$ from \autoref{prp:functorW} takes values in $\CatW_{5,6}$.

Similarly, the full subcategory $\CatCu_{5,6}$ of $\CatCu$ consisting of $\CatCu$-semigroups satisfying \axiomO{5} and \axiomO{6} is closed under inductive limits.
By \autoref{prp:o5o6}, the functor $\Cu$ takes values in $\CatCu_{5,6}$.

By \autoref{prp:axiomsPassToCu}, the reflector $\gamma\colon\CatW\to\CatCu$ maps $\CatW_{5,6}$ to $\CatCu_{5,6}$.
\end{rmksChapter}

%------------------------------------------------------------------------------------------
Given a $\CatPreW$-semigroup $(S,\prec)$, we say that an element $s\in S$ is \emph{full}\index{terms}{element!full} if, whenever there are $t',t\in S$ satisfying $t'\prec t$, then there is $n\in\N$ such that $t'\leq ns$.
We say that an element $s$ \emph{cancels from sums} if $a+s\leq b+s$ implies $a\leq b$ for any $a,b$.

%==========================================================================================
\begin{prpChapter}
\label{prp:wc}
Let $(S,\prec)$ be a $\CatPreW$-semigroup satisfying \axiomW{5}.
If $S$ contains a full element $e$ such that $e\prec e$ and $e$ cancels from sums, then $S$ has weak cancellation.

Similarly, if a $\CatCu$-semigroup satisfying \axiomO{5} contains a full compact element that cancels from sums, then it has weak cancellation.
\end{prpChapter}
\begin{proof}
Suppose $a,b,c\in S$ satisfy $a+c\prec b+c$.
Using \axiomW{4}, we can choose $b',c'\in S$ such that
\[
b'\prec b,\quad
c'\prec c,\quad
a+c\leq b'+c'.
\]
Apply \axiomW{1} to find $c''\in S$ with $c'\prec c''\prec c$.
Since $e$ is full, we can choose $n\in\N$ satisfying $c''\leq ne$.
By \axiomW{5}, applied to $c''+0\prec ne$, $c'\prec c''$, and using that $e\prec e$, we can find elements $x',x\in S$ such that $x'\prec x$ and $c'+x\prec ne\prec c''+x'$.
Now
\[
a+ne\leq a+c''+x'\leq a+c+x'\leq b'+c'+x\prec b+ne,
\]
from which we deduce $a\prec b$, as desired.

The analogous result for $\CatCu$-semigroups follows from \autoref{prp:axiomsPassToCu}.
\end{proof}

%------------------------------------------------------------------------------------------
%==========================================================================================
\chapter{Structure of \texorpdfstring{$\CatCu$}{Cu}-semigroups}
\label{sec:misc}

%------------------------------------------------------------------------------------------
This chapter contains general results about the structure of $\CatCu$-sem\-i\-groups.

%------------------------------------------------------------------------------------------
%==========================================================================================
\section{Ideals and quotients}

%------------------------------------------------------------------------------------------
In this section, we study ideals and quotients of $\CatCu$-semigroups.
We show that \axiomO{5}, \axiomO{6} and weak cancellation pass to ideals and quotients;
see \autoref{prp:axiomsQuotient}.
Given a $\CatCu$-semigroup $S$, we denote the set of ideals in $S$ by $\Lat(S)$.
We show that $\Lat(S)$ has a natural structure as a complete lattice;
see \autoref{pgr:Lat}.
The subset of singly-generated ideals forms a sublattice, denoted by $\Lat_{\mathrm{f}}(S)$.
We show that $\Lat_{\mathrm{f}}(S)$ is a $\CatCu$-semigroup;
see \autoref{prp:LatCu}.
In \autoref{sec:CuSrgSmod}, we will see that $\Lat_{\mathrm{f}}(S)$ is naturally isomorphic to the tensor product $S\otimes_\CatCu\{0,\infty\}$;
see \autoref{prp:tensWithInfty}.

Then, we consider the case of a concrete Cuntz semigroup $\Cu(A)$ of a \ca{} $A$.
We show that there is a natural isomorphism between $\Lat(\Cu(A))$ and the lattice of ideals in $A$, which we denote by $\Lat(A)$;
see \autoref{prp:LatCa}.
This isomorphism identifies the $\CatCu$-semigroup $\Lat_{\mathrm{f}}(\Cu(A))$ with the subset of $\Lat(A)$ consisting of ideals that contain a full, positive element.
In the case that $A$ is separable, every ideal in $A$ is $\sigma$-unital and hence contains a positive, full (even strictly positive) element.
It follows that in this case, $\Lat(A)$ is a $\CatCu$-semigroup;
see \autoref{prp:LatSepCa}.

%==========================================================================================
\begin{pgr}
\label{pgr:ideals}
\index{terms}{order-hereditary}
\index{terms}{ideal}
\index{terms}{order-ideal}
\index{symbols}{$\leq_I$}
\index{symbols}{$\sim_I$}
\index{symbols}{S/I@$S/I$ \quad (quotient semigroup)}
Let $M$ be a \pom.
A subset $I$ of $M$ is \emph{order-hereditary} if for every $a,b\in M$ we have that $a\leq b$ and $b\in I$ imply $a\in I$.
An \emph{ideal} (also called \emph{order-ideal}) in $M$ is a subsemigroup which is order-hereditary.
Given a $\CatCu$-semigroup $S$, we shall also require that an \emph{ideal} in $S$ is closed under suprema of increasing sequences.

Given an ideal $I$ in a $\CatCu$-semigroup $S$, we define a binary relation $\leq_I$ on $S$ as follows:
For elements $a,b\in S$, we set $a\leq_I b$ if and only if there exists $c\in I$ such that $a\leq b+c$.
By symmetrizing, we define a relation $\sim_I$ on $S$:
For elements $a,b\in S$, we set $a\sim_I b$ if and only if both conditions $a\leq_I b$ and $b\leq_I a$ are met.

It is easy to see that $\sim_I$ is a congruence relation on $S$.
Recall that a congruence is (by definition) an additive equivalent relation; see \autoref{pgr:tensorMonConstr}.
We denote the set of congruence classes by
\[
S/I := S/\!\!\sim_I.
\]
The partial order on $S$ induces a partial order on $S/I$, giving the latter the structure of a \pom.
Given an element $a\in S$, we denote its congruence class in $S/I$ by $a_I$.
In the next result, we verify that $S/I$ satisfies \axiomO{1}-\axiomO{4}.
\end{pgr}

%==========================================================================================
\begin{lma}
\label{prp:quotientCu}
Let $S$ be a $\CatCu$-semigroup, and let $I$ be an ideal in $S$.
Then $S/I$ is a $\CatCu$-semigroup.
Moreover, the map
\[
\pi_I\colon S\to S/I,\quad
a\mapsto a_I,\quad
\txtFA a\in S,
\]
is a surjective $\CatCu$-morphism.
\end{lma}
\begin{proof}
Let $S$ and $I$ be as in the statement.
As explained in \autoref{pgr:ideals}, we have that $S/I$ is a \pom.
It is also easy to see that $\pi_I$ is a surjective $\CatPom$-morphism.
The following claims are easily verified.

Claim 1:
Given $x,y\in S/I$, we have $x\leq y$ if and only if there exist representatives $a,b\in S$ such that $x=a_I$, $y=b_I$, and $a\leq b$.

Claim 2:
Given an increasing sequence $(x_k)_k$ in $S/I$, there exists an increasing sequence $(a_k)_k$ in $S$ such that $x_k=(a_k)_I$ for each $k$.

To verify \axiomO{1} for $S/I$, let $(x_k)_k$ be an increasing sequence in $S/I$.
By claim~2, we can choose an increasing sequence $(a_k)_k$ in $S$ such that for each $k$, the element $x_k$ is represented by $a_k$.
Since $S$ satisfies \axiomO{1}, the sequence $(a_k)_k$ has a supremum in $S$ which we denote by $a:=\sup_ka_k$.
We claim that $a_I$ is the supremum of the sequence $(x_k)_k$ in $S/I$.

We have $x_k\leq a_I$ for each $k$.
Conversely, let $y\in S/I$ satisfy $x_k\leq y$ for all $k$.
Choose $b\in S$ with $y=b_I$.
Then, for each $k$, we can choose $c_k\in I$ such that $a_k\leq b+c_k$.
Set
\[
c :=\sum_{k=0}^\infty c_k = \sup_n \sum_{k=0}^n c_k,
\]
which is an element in $I$.
We obtain
\[
a_k \leq b+c_k \leq b+c,
\]
for each $k$.
By definition of $a$, this implies
\[
a=\sup_ka_k\leq b+c.
\]
Since $c\in I$, we get $a\leq_I b$ and therefore $x\leq y$, as desired.
It also follows from the above argument that $\pi_I$ preserves suprema of increasing sequences.

In order to show that $\pi_I$ preserves the way-below relation, let $a,b\in S$ satisfy $a\ll b$ in $S$.
To show $a_I\ll b_I$ in $S/I$, let $(x_k)_k$ be an increasing sequence in $S/I$ satisfying $b_I\leq \sup_k x_k$.
By claim 2, we can choose an increasing sequence $(b_k)_k$ in $S$ such that $x_k=(b_k)_I$ for each $k$.
Then
\[
b_I \leq \sup_k x_k = (\sup_k b_k)_I,
\]
whence we can choose $c\in I$ such that $b\leq (\sup_k b_k) +c$.
Using that $S$ satisfies \axiomO{4}, we obtain
\[
a \ll b \leq (\sup_k b_k) +c = \sup_k (b_k +c).
\]
Therefore, there exists $n\in\N$ such that $a\leq b_n+c$, and hence $a_I\leq(b_n)_I=x_n$, as desired.

To verify \axiomO{2} for $S/I$, let $x\in S/I$.
Choose $a\in S$ such that $x=a_I$.
Since $S$ satisfies \axiomO{2}, we can choose a rapidly increasing sequence $(a_k)_k$ in $S$ such that $a=\sup_k a_k$.
For each $k$, set $x_k:=(a_k)_I$.
It follows that $(x_k)_k$ is a rapidly increasing sequence in $S/I$ with $x=\sup_k x_k$.
This finishes the proof of \axiomO{2} for $S/I$.
Finally, it is straightforward to verify the axioms \axiomO{3} and \axiomO{4} for $S/I$.
\end{proof}

%==========================================================================================
\begin{prp}
\label{prp:axiomsQuotient}
Let $S$ be a $\CatCu$-semigroup, and let $I$ be an ideal in $S$.
If $S$ satisfies \axiomO{5} (respectively \axiomO{6}, weak cancellation), the so does the ideal $I$ and the quotient $S/I$.
\end{prp}
\begin{proof}
It is easy to verify that each of the axioms passes to ideals.
To show that \axiomO{5} passes to quotients, let $S$ be a $\CatCu$-semigroup and let $I$ be an ideal in $S$.
Assume that $S$ satisfies \axiomO{5}.
To verify \axiomO{5} for $S/I$, let $a',a,b',b,c\in S/I$ satisfy
\[
a+b\leq c,\quad
a'\ll a,\quad
b'\ll b.
\]
Choose $s,t,r\in S$ such that
\[
a=s_I,\quad
b=t_I,\quad
c=r_I,\quad
s+t\leq r.
\]
Since the quotient map is continuous, we can find $s',t'\in S$ satisfying
\[
s'\ll s,\quad
t'\ll t,\quad
a'\leq (s')_I,\quad
b'\leq (t')_I.
\]
Since $S$ satisfies \axiomO{5}, we can choose $x\in S$ such that
\[
s'+x\leq r\leq s+x,\quad
t'\leq x.
\]
Then $x_I$ has the desired properties to verify \axiomO{5} for $S/I$.
The proofs that \axiomO{6} and weak cancellation pass to quotients can be obtained with the same technique and are left to the reader.
\end{proof}

%==========================================================================================
\begin{rmk}
\label{rmk:quotientsW}
It is possible to define the notion of ideals and quotients in the category $\CatPreW$.
We do not pursue this idea.
\end{rmk}

%==========================================================================================
\begin{prbl}
\label{prbl:axiomsExtension}
Let $S$ be a $\CatCu$-semigroup, and let $I$ be an ideal in $S$.
Assume that $I$ and $S/I$ satisfy \axiomO{5} (respectively \axiomO{6}, weak cancellation).
Under what assumptions does this imply that $S$ itself satisfies the respective axiom?
\end{prbl}

%==========================================================================================
\begin{pgr}
\label{pgr:Lat}
\index{terms}{ideal!lattice}
\index{symbols}{Lat(S)@$\Lat(S)$ \quad (lattice of ideals)}
Let $S$ be a $\CatCu$-semigroup.
We denote the set of all ideals in $S$ by $\Lat(S)$.
Inclusion of ideals defines a partial order on $\Lat(S)$.
Let $(I_\lambda)_\lambda\subset\Lat(S)$ be a family of ideals.
It is easy to check that the intersection $\bigcap_\lambda I_\lambda$ is again an ideal.
Clearly, this is the largest ideal contained in each $I_\lambda$.
Therefore, the family $(I_\lambda)_\lambda$ has an infimum in $\Lat(S)$, given by $\bigwedge_\lambda I_\lambda = \bigcap_\lambda I_\lambda$.

On the other hand,
the family $(I_\lambda)_\lambda$ has a supremum in $\Lat(S)$ given by
\[
\bigvee_\lambda I_\lambda
= \bigcap \left\{ J\in\Lat(S) : I_\lambda\subset J \text{ for all } \lambda \right\}.
\]
This shows that $\Lat(S)$ is a complete lattice.

Set $M:=\{a\in S : a\leq y_1+\cdots+y_n,\text{ for some }y_1,\dots,y_n\in \bigcup_\lambda I_\lambda\}$, that is, $M$ is the order-hereditary submonoid generated by $\bigcup_\lambda I_\lambda$.
We claim that the supremum of the family $(I_\lambda)_\lambda$ is also given by
\[
\bigvee_\lambda I_\lambda
= \left\{ \sup_n a_n : (a_n)_n \text{ rapidly increasing sequence in } M \right\}.
\]
To see this, let us temporarily denote the right hand side in the equation above by $P$.
Using that $S$ satisfies \axiomO{3} and \axiomO{4}, it follows easily that $P$ is closed under addition.
To show that $P$ is order-hereditary, let $a,b\in S$ satisfy $a\leq b$ and $b\in P$.
By definition of $P$, we can choose an increasing sequence $(b_n)_n$ in $M$ such that $b=\sup_nb_n$.
Since $S$ satisfies \axiomO{2}, we can choose a rapidly increasing sequence $(a_k)_k$ in $S$ such that $a=\sup_k a_k$.
For each $k$, we have
\[
a_k \ll a \leq b=\sup_nb_n,
\]
whence we can choose $n(k)\in\N$ with $a_k\leq b_{n(k)}$.
Since $M$ is order-hereditary, this implies $a_k\in M$, and hence $a\in P$, as desired.
Finally, a standard diagonalization argument shows that $P$ is closed under suprema of increasing sequences.
Thus, $P$ is an ideal of $S$ that contains $I_\lambda$ for each $\lambda$.
Since $P$ is clearly the smallest ideal with this property, we have $P=\bigvee_\lambda I_\lambda$.

It follows that an element $a\in S$ is contained in $\bigvee_\lambda I_\lambda$ if and only if for every $a'\in S$ satisfying $a'\ll a$ we have that $a'$ is contained in $M$.

For finitely many ideals $I_1,\dots,I_n$, it is easy to check that their supremum can also be described as
\[
I_1\vee\cdots\vee I_n = \left\{a\in S : a\leq y_1+\cdots y_n\text{ with }y_i\in I_i \right\}.
\]

Given $a\in S$, we denote by $\Idl(a)$ the ideal generated by $a$, that is:
\[ \index{symbols}{Idl(a)@$\Idl(a)$ \quad (ideal generated by $a$)}
\Idl(a)=\left\{ x\in S : x\leq \infty\cdot a \right\}.
\]
We claim that $\Idl(a')\ll\Idl(a)$ in $\Lat(S)$ for any $a',a\in S$ satisfying $a'\ll a$.
To prove the claim, let $(I_k)_{k\in\N}$ be an increasing sequence in $\Lat(S)$ with $\Idl(a)\subset\bigvee_k I_k$.
Then $a'\ll a\in \bigvee_k I_k$ and therefore $a'\in\bigcup_k I_k$.
Thus, there is $n\in\N$ such that $a'\in I_n$.
But this implies $\Idl(a')\subset I_n$, which proves the claim.

Let $a\in S$.
Since $S$ satisfies \axiomO{2}, we can choose a rapidly increasing sequence $(a_n)_n$ in $S$ with $a=\sup_n a_n$.
It follows that $\Idl(a)$ is the supremum of the rapidly increasing sequence $(\Idl(a_n))_n$ in $\Lat(S)$.

However, it is no longer true that general ideals in a $\CatCu$-semigroup can be written as a supremum of a rapidly increasing sequence of ideals.
We define
\[
\Lat_{\mathrm{f}}(S)
:= \left\{ \Idl(a) : a\in S \right\} \subset \Lat(S),
\]
which is the set of singly-generated ideals in $S$. Note that, for $a$, $b\in S$, we have $\Idl(a)\vee\Idl(b)=\Idl(a+b)$, so that $\Lat_{\mathrm{f}}(S)$ becomes an abelian semigroup with $\vee$ as addition.
\index{symbols}{Latf(S)@$\Lat_{\mathrm{f}}(S)$ \quad (singly-generated ideals)}

Given an ideal $I$ in $S$, the following are equivalent:
\beginEnumStatements
\item
We have $I\in\Lat_{\mathrm{f}}(S)$, that is, $I$ is generated by a single element.
\item
The ideal $I$ is generated by countably many elements.
\item
The ideal $I$ has a maximal element, denoted by $\bigvee I$, and then
\[
I = \Idl\left( \bigvee I \right) = \left\{ x\in S : x\leq \bigvee I \right\}.
\]
\end{enumerate}
It is clear that \enumStatement{1} implies \enumStatement{2}, and that \enumStatement{3} implies \enumStatement{1}.
To show that \enumStatement{2} implies \enumStatement{3}, assume that $I$ is an ideal that is generated by a countable set of elements, say $\{a_0,a_1,a_2,\ldots\}\subset S$.
Then the element
\[
s:=\infty\cdot\sum_{k=0}^\infty a_k
= \sup_n \sum_{k=0}^n na_k,
\]
is contained in $I$.
Since $a_k\leq s$ for each $k$, it is clear that $I=\Idl(s)$.
Since, moreover, $\infty\cdot s=s$, we also have $s=\bigvee I$.
\end{pgr}

%==========================================================================================
\begin{prp}
\label{prp:LatCu}
Let $S$ be a $\CatCu$-semigroup.
Then $\Lat_{\mathrm{f}}(S)$ is a $\CatCu$-semigroup satisfying \axiomO{5}.
If $S$ satisfies \axiomO{6}, then so does $\Lat_{\mathrm{f}}(S)$.
Moreover, the map
\[
S\to\Lat_{\mathrm{f}}(S),\quad
a\mapsto \Idl(a),\quad
\txtFA a\in S,
\]
is a surjective $\CatCu$-morphism.

If $S$ is countably-based, then
\[
\Lat(S)=\Lat_{\mathrm{f}}(S).
\]
\end{prp}
\begin{proof}
We denote the map $S\to\Lat_{\mathrm{f}}(S)$ from the statement by $\Idl$.
Let $a,b\in S$.
It is easy to see that $\Idl(a+b)=\Idl(a)+\Idl(b)$.
Moreover, we have $\Idl(a)\subset \Idl(b)$ if and only if $\infty\cdot a \leq \infty\cdot b$.
It follows that $\Lat_{\mathrm{f}}(S)$ is an algebraically ordered submonoid of $\Lat(S)$.
We also get that the map $\Idl$ is a $\CatPom$-morphism.

To verify \axiomO{1} for $\Lat_{\mathrm{f}}(S)$, let $(I_n)_n$ be an increasing sequence of singly-generated ideals.
The supremum $\bigvee_n I_n$ in $\Lat(S)$ is a countably-generated ideal.
As observed in \autoref{pgr:Lat}, this implies that $\bigvee_n I_n\in\Lat_{\mathrm{f}}(S)$.
It follows that $\bigvee_n I_n$ is the supremum of $(I_n)_n$ in $\Lat_{\mathrm{f}}(S)$, which verifies \axiomO{1}.

In \autoref{pgr:Lat}, we have already observed that $\Lat_{\mathrm{f}}(S)$ satisfies \axiomO{2} and that the map $\Idl$ preserves the way-below relation.
Then it is easy to check that $\Lat_{\mathrm{f}}(S)$ satisfies \axiomO{3} and \axiomO{4}, and that $\Idl$ is a surjective $\CatCu$-morphism.
Moreover, since the order on $\Lat_{\mathrm{f}}(S)$ is algebraic, \axiomO{5} holds trivially.

Next, let us assume that $S$ satisfies \axiomO{6}.
In order to show that $\Lat_{\mathrm{f}}(S)$ satisfies \axiomO{6}, let $I',I,J,K\in \Lat_{\mathrm{f}}(S)$ satisfy
\[
I'\ll I\subset J+K.
\]
Choose $a\in S$ with $I=\Idl(a)$.
Since $I'\ll I$, we can find $a'\in S$ such that
\[
a'\ll a,\quad
I'\subset \Idl(a'),\quad
I=\Idl(a).
\]
Moreover, since $I\subset J+K$, we can choose $b,c\in S$ with
\[
a\leq b+c,\quad
J=\Idl(b),\quad
K=\Idl(c).
\]
Using that $S$ satisfies \axiomO{6}, we can find $e,f\in S$ such that
\[
a'\leq e+f,\quad
e\leq a,b,\quad
f\leq a,c.
\]
It is now easy to check that the ideals $\Idl(e)$ and $\Idl(f)$ have the desired properties to verify \axiomO{6} for $\Lat_{\mathrm{f}}(S)$.

Finally, assume $S$ is countably-based.
Given an ideal $I$ in $S$, it is straightforward to check that $I$ is generated by countably many elements.
As observed in \autoref{pgr:Lat}, this implies $I\in\Lat_{\mathrm{f}}(S)$, as desired.
\end{proof}

%==========================================================================================
\begin{rmk}
\label{rmk:LatCu}
In \autoref{prp:tensWithInfty}, we will show that there is a natural isomorphism
\[
\Lat_{\mathrm{f}}(S) \cong S\otimes_\CatCu \{0,\infty\}.
\]
\end{rmk}

%==========================================================================================
\begin{pgr}
\label{pgr:idealsCa}
Let $A$ be a \ca, and let $I$ be an ideal in $A$.
(By an ideal in a \ca, we always mean a closed, two-sided ideal.)
The inclusion map $\iota\colon I\to A$ induces a $\CatCu$-morphism
\[
\Cu(\iota)\colon\Cu(I)\to\Cu(A).
\]
It is shown in \cite[Proposition~3.1.1]{Ciu08PhD} that $\Cu(\iota)$ is an order-embedding.
We may therefore identify $\Cu(I)$ with a subsemigroup of $\Cu(A)$.
(The assumption that the \ca{} is separable is not needed in the proof of \cite[Proposition~3.1.1]{Ciu08PhD}.)
In fact, the argument is not difficult and we include it for completeness.

First, we show that $\Cu(\iota)$ is an order-embedding.
We may assume that $A$ and $I$ are stable.
Let $x,y\in I_+$ such that $x$ is Cuntz-subequivalent to $y$ in $A$, and let $\varepsilon>0$.
Then, using R{\o}rdam's lemma (see \autoref{pgr:CuntzComparison}), we can find $\delta>0$ and $r\in A$ such that
\[
(x-\varepsilon)_+ = r(y-\delta)_+r^*.
\]
Let $f_\delta\colon\R\to[0,1]$ be the function that takes value $0$ on $(-\infty,\delta/2)$, that takes value $1$ on $[\delta,\infty)$, and that is linear on $[\delta/2,\delta]$.
By functional calculus, we obtain
\[
(y-\delta)_+=f_\delta(y)(y-\delta)_+f_\delta(y).
\]
This implies
\[
(x-\varepsilon)_+ = (rf_\delta(y)) (y-\delta)_+ (rf_\delta(y))^*.
\]
Since $f_\delta(y)\in I$ and since $I$ is an ideal, we have $rf_\delta(y)\in I$.
Then, using  R{\o}rdam's lemma in the other direction, it follows that $x$ is Cuntz subequivalent to $y$ in $I$, as desired.

Let us also show that $\Cu(I)$ is an ideal in $\Cu(A)$.
First, it is clear that $\Cu(I)$ is a submonoid of $\Cu(A)$.
To show that it is an order-hereditary subset, let $a,b\in\Cu(A)$ satisfy $a\leq b$ and $b\in\Cu(I)$.
Choose $x\in A_+$ and $y\in I_+$ with $a=[x]$ and $b=[y]$.
By definition, we can find a sequence $(r_k)_k$ in $A$ such that $x=\lim_kr_kyr_k^*$.
Since $I$ is an ideal, we have $r_kyr_k^*\in I$ for each $k$.
As $I$ is also closed, we get $x\in I$ and so $a\in\Cu(I)$, as desired.

Finally, in order to show that $\Cu(I)$ is closed under suprema of increasing sequences, let $(a_k)_k$ be an increasing sequence in $\Cu(I)$ with $a:=\sup_ka_k \in \Cu(A)$.
Choose representatives $x_k\in I_+$ for $k\in\N$ and $x\in A_+$ such that $a=[x]$ and $a_k=[x_k]$ for each $k$.
We need to show $a\in I$.
Let $\varepsilon>0$.
Then
\[
[(x-\varepsilon)_+]\ll[x]=a=\sup_k a_k,
\]
which implies that there exists $n\in\N$ such that $[(x-\varepsilon)_+]\leq a_k$.
We have already observed that this implies that $(x-\varepsilon)_+\in I$.
Since this holds for every $\varepsilon>0$, we get $x\in I$, and hence $a\in\Cu(I)$, as desired.

We let $\Lat(A)$ denote the collection of ideals of $A$, equipped with the partial order given by inclusion of ideals.
It is well-known that $\Lat(A)$ is a complete lattice.
We let $\Lat_{\mathrm{f}}(A)$ denote the subset of ideals in $A$ that contain a full, positive element.
We remark that every $\sigma$-unital ideal of $A$ belongs to $\Lat_{\mathrm{f}}(A)$, but the converse does not hold.
Indeed, in \cite[Lemma~2.2]{BroGreRie77StableIso} an example of a simple \ca{} without strictly positive element is given.

It is easy to see that $\Lat_{\mathrm{f}}(A)$ is a sublattice of $\Lat(A)$.
\end{pgr}

%==========================================================================================
\begin{prp}
\label{prp:LatCa}
Let $A$ be a \ca.
Then the map
\[
\Lat(A)\to\Lat(\Cu(A)),\quad
I\mapsto\Cu(I),\quad
\txtFA I\in\Lat(A),
\]
is a natural isomorphism of complete lattices.

Moreover, it maps the sublattice $\Lat_{\mathrm{f}}(A)$ of ideals in $A$ that contain a full, positive element onto the sublattice $\Lat_{\mathrm{f}}(\Cu(A))$ of singly-generated ideals in $\Cu(A)$.
\end{prp}
\begin{proof}
For the case that $A$ is a separable \ca, a proof of the statement can be found in \cite[Proposition~3.1.2]{Ciu08PhD}.
Our proof is based on the ideas given by Ciuperca, and we include it for completeness.
We may assume that $A$ is stable, so that $\Cu(A)=A_+/\!\!\sim$.
Let us denote the map of the statement by $\varphi\colon\Lat(A)\to\Lat(\Cu(A))$.
Consider the map
\[
c\colon A\to \Cu(A),\quad
x\mapsto [xx^*],\quad
\txtFA x\in A,
\]
which assigns to an element $x\in A$ the Cuntz class of $xx^*$.
Given an ideal $I$ in $A$, we easily see
\[
\varphi(I)
=\left\{ [x]\in\Cu(A) : x\in I_+ \right\}
= \left\{ [xx^*]\in\Cu(A) : x\in I \right\}
= c(I).
\]
Let us define a map, which will turn out to be the inverse of $\varphi$, as
\[
\psi\colon \Lat(\Cu(A)) \to\Lat(A),\quad
\psi(J) :+ c^{-1}(J) = \left\{ x\in A : [xx^*]\in J \right\},
\]
for all $J\in\Lat(\Cu(A)))$.

Given an ideal $J$ in $\Cu(A)$, let us check that $\psi(J)$ is an ideal of $A$.
To show that $\psi(J)$ is closed under addition, let $x,y\in\psi(J)$.
We have
\[
(x+y)(x+y)^*
\leq (x+y)(x+y)^*+(x-y)(x-y)^*
= 2xx^* + 2yy^*,
\]
and therefore $[(x+y)(x+y)^*]\leq [xx^*]+[yy^*]$.
Since $J$ is an ideal and $[xx^*],[yy^*]\in J$, we get $[(x+y)(x+y)^*]\in J$ and so $x+y\in\psi(J)$.
It is straightforward to check that $\psi(J)$ is closed under scalar multiplication.

To show that $\psi(J)$ is an ideal, let $x\in\psi(J)$ and $y\in A$.
We have
\[
(xy)(xy)^*
= x(yy^*)x^*
\precsim xx^*,\quad
(yx)(yx)^*
= y(xx^*)y^*
\precsim xx^*,
\]
which again implies $xy,yx\in\psi(J)$.
It is left to the reader to check that $\psi(J)$ is also closed.

It is clear that both $\varphi$ and $\psi$ are order-preserving.
Next, let us show that these maps are inverses of each other.
Given an ideal $J$ in $\Cu(A)$, using that $c$ is a surjective map, we easily deduce
\[
\varphi\circ\psi(J) = c(c^{-1}(J)) = J.
\]
Conversely, let $I$ be an ideal of $A$.
Then $I$ is clearly a subset of $\psi\circ\varphi(I) = c^{-1}(c(I))$.
By definition, if $x\in c^{-1}(c(I))$, then $xx^*\in c(I)$, which means that there exists $y\in I_+$ such that $xx^*\sim y$.
We have already seen that this implies $xx^*\in I$ and hence also $x\in I$, as desired.

Finally, let us see that $\varphi$ maps $\Lat_{\mathrm{f}}(A)$ onto $\Lat_{\mathrm{f}}(\Cu(A))$.
In one direction, let $I\in\Lat_{\mathrm{f}}(A)$ and choose a full, positive element $x\in I_+$.
Set $a:=[x]\in\Cu(I)$.
In order to show that $\infty\cdot a$ is the largest element of $\Cu(I)$, let $y\in I_+$, and let $\varepsilon>0$.
Since $x$ is full and $y$ is positive, we can choose $K\in\N$ and elements $r_1,\ldots,r_K \in I$ such that
\[
\left\| y - \sum_{k=1}^K r_k x r_k^* \right\|<\varepsilon.
\]
It follows
\[
[(y-\varepsilon)_+] \leq K[x] \leq \infty\cdot a.
\]
Since this holds for every $\varepsilon>0$, we get $[y]\leq\infty\cdot a$, as desired.

Conversely, assume that $J$ is a singly-generated ideal in $\Cu(A)$ and set $I:=\psi(J)$.
Then, as observed in \autoref{pgr:Lat}, there exists a largest element in $J$, which we denote by $a$.
Choose $x\in I_+$ such that $a=[x]$.
In order to show that $x$ is a full element in $I$, let $y\in I$.
Since $a$ is the largest element in $J$, we get $yy^*\precsim x$.
This implies that $yy^*$ and hence $y$ is contained in the ideal generated by $x$.
Thus, $x$ is full in $I$, as desired.
\end{proof}

%------------------------------------------------------------------------------------------
\index{terms}{C*-algebra@\ca{}!simple}
Recall that a \ca{} $A$ is called \emph{simple} if $\{0\}$ and $A$ are the only ideals of $A$.
Analogously, we define for $\CatCu$-semigroups:

%==========================================================================================
\begin{dfn}
\label{dfn:simpleCu}
\index{terms}{Cu-semigroup@$\CatCu$-semigroup!simple}
\index{terms}{simple (Cu-semigroup)@simple ($\CatCu$-semigroup)}
A $\CatCu$-semigroup $S$ is called \emph{simple} if $\{0\}$ and $S$ are the only ideals of $S$.
\end{dfn}

%==========================================================================================
\begin{cor}
\label{prp:simpleCa}
A \ca{} $A$ is simple if and only if its (completed) Cuntz semigroup $\Cu(A)$ is a simple $\CatCu$-semigroup.
\end{cor}

%==========================================================================================
\begin{cor}
\label{prp:LatSepCa}
Let $A$ be a separable \ca.
Then the ideal lattice $\Lat(A)$ is a $\CatCu$-semigroup.
\end{cor}

%==========================================================================================
\begin{rmks}
\label{rmk:LatCa}
Let $A$ be a separable \ca.

(1)
In \autoref{prp:tensCaWithInfty}, we will show that there are natural isomorphisms between the following $\CatCu$-semigroups:
\[
\Cu(\mathcal{O}_\infty\otimes A)
\cong \Lat(A)
\cong \Lat(\Cu(A))
\cong \{0,\infty\}\otimes_\CatCu\mathcal{}\Cu(A).
\]

(2)
The $\CatCu$-semigroup $\Lat(A)$ is algebraic (see \autoref{sec:algebraicSemigp}) if and only if the \ca{} $A$ has the ideal property.
\end{rmks}

%==========================================================================================
\begin{prp}[{Ciuperca, Robert, Santiago, \cite[Proposition~3.3]{CiuRobSan10CuIdealsQuot}}]
\label{prp:CuQuotientCa}
Let $I$ be an ideal in a \ca{} $A$.
Then there is a natural isomorphism
\[
\Cu(A) / \Cu(I) \cong \Cu(A / I).
\]
\end{prp}

%==========================================================================================
\begin{pgr}[Elementary semigroups]
\label{pgr:elementarySemigr}
\index{terms}{Cu-semigroup@$\CatCu$-semigroup!elementary}
\index{terms}{nonelementary Cu-semigroup@nonelementary $\CatCu$-semigroup}
\index{symbols}{Ek@$\elmtrySgp{k}$ \quad (elementary $\CatCu$-semigroup)}
We call a simple $\CatCu$-semigroup $S$ \emph{elementary} if $S\cong\{0\}$ or if $S$ contains a minimal, nonzero element.
The typical example is the semigroup of extended natural numbers
\[
\overline{\N}
=\left\{ 0,1,2,\ldots,\infty \right\}.
\]
For each $k\in\N$, we define a semigroup
\[
\elmtrySgp{k}
:=\left\{ 0,1,2,\ldots,k,\infty \right\},
\]
with the natural order and with $a+b$ defined as $\infty$ if usually one would have $a+b\geq k+1$.
For $k=0$ we obtain $\elmtrySgp{0}=\{0,\infty\}$.
It is easy to check that these are simple $\CatCu$-semigroups satisfying \axiomO{5} and \axiomO{6}, and all elements are compact.

There exist simple, elementary $\CatCu$-semigroups satisfying \axiomO{5} that are not isomorphic to $\{0\}$, to $\overline{\N}$ or to $\elmtrySgp{k}$ for some $k$;
see \autoref{exa:elmtrySemirg}.
With the assumption of \axiomO{6}, this is not possible as we will show in \autoref{prp:elmtryO6}.
\end{pgr}

%==========================================================================================
\begin{rmk}
The term elementary for a $\CatCu$-semigroup was first coined by Engbers in \cite{Eng14PrePhD} to refer only to the semigroup $\overline{\N}$.
With this terminology, he proved that a simple, separable \ca{} $A$ is elementary if and only if $\Cu(A)$ is elementary.
Our definition of elementary is motivated by Engbers', and is suitable for the general study of abstract Cuntz semigroups.
J.~Bosa proved (private communication) that $\elmtrySgp{k}$ does not come as the Cuntz semigroup of a \ca{} except for $k=0$.
It follows from this fact and from \autoref{prp:elmtryO6} below that, if $A$ is a nonzero, simple and separable \ca, then $\Cu(A)$ is elementary if and only if $A$ is elementary or $A$ is purely infinite.
\end{rmk}

%==========================================================================================
\begin{lma}
\label{prp:downwards}
Let $S$ be a simple $\CatCu$-semigroup satisfying \axiomO{6}.
Given nonzero elements $a_1,\ldots,a_n\in S$, there exists a nonzero element $x\in S$ such that $x\ll a_k$ for all $k$.
\end{lma}
\begin{proof}
It is enough to prove the case $n=2$ (and then use induction).
To verify this case, let $a$ and $b$ be nonzero elements in $S$.
We need to find a nonzero element $x\in S$ such that $x\ll a$ and $x\ll b$.

Choose nonzero elements $a'$ and $a''$ in $S$ satisfying $a''\ll a'\ll a$.
By simplicity of $S$, we can find $k\in\N$ such that $a'\leq kb$.
Considering the situation
\[
a''\ll a'\leq kb = b+b+\ldots+b,
\]
we may apply \axiomO{6} in $S$ to obtain elements $c_1,\ldots,c_k\in S$ with
\[
a''\leq c_1+\cdots+c_k,\quad
c_i\leq a',b,\quad i=1,2,\ldots,k.
\]
Since $a''$ is nonzero, there has to be an index $i_0$ with $c_{i_0}\neq 0$.
Choose a nonzero element $x\in S$ with $x\ll c_{i_0}$.
Then $x$ has the desired properties.
\end{proof}

%------------------------------------------------------------------------------------------
The following result was observed independently by Engbers, \cite{Eng14PrePhD}.
He also noted that one must exclude elementary semigroups to obtain results like Glimm Halving, \cite[Proposition~5.2.1]{Rob13Cone};
see \autoref{prp:GlimmHalving}.

%==========================================================================================
\begin{prp}
\label{prp:elmtryO6}
Let $S$ be a simple $\CatCu$-semigroup satisfying \axiomO{5} and \axiomO{6}.
Then $S$ is elementary if and only if $S$ is isomorphic to $\{0\}$, to $\overline{\N}$, or to $\elmtrySgp{k}$ for some $k\in\N$.
\end{prp}
\begin{proof}
The `if' part of the statement is clear.
To show the converse implication, assume that $S$ is an elementary $\CatCu$-semigroup with $S\neq\{0\}$.
Then there exists a minimal, nonzero element $a$ in $S$.
By \autoref{prp:downwards}, the element $a$ is compact.
We claim that $a$ is the least nonzero element.
Indeed, let $b\in S$ be an arbitrary nonzero element.
By \autoref{prp:downwards}, there exists a nonzero element $b'$ with $b'\leq a,b$.
Since $a$ is minimal, we have $b'=a$ and therefore $a\leq b$.

Now, let $b$ be an arbitrary nonzero element in $S$.
Then $a\leq b$ and since $a$ is compact and $S$ satisfies \axiomO{5}, we can choose $x\in S$ such that $a+x=b$.
If $x=0$, we have $b=a$.
Otherwise, since $a$ is the least nonzero element, we obtain $a\leq x$ and so there is $y\in S$ with $a+y=x$ and consequently $2a+y=b$.
Continuing in this way, we find that either $b=na$ for some $n\in\N$ or otherwise $na\leq b$ for all $n\in\N$.
The latter implies $b=\infty$, whence
\[
S = \{ \infty \} \cup \{ na : n\in\N \}.
\]
Now, if $na\neq ma$ for any $n,m\in\N$ with $n\neq m$, then we have $S\cong\overline{\N}$.
Otherwise, there is $k\in\N$ with $ka=(k+1)a$.
For the smallest such $k$, we have $S\cong E_k$.
\end{proof}

\vspace{5pt}
%------------------------------------------------------------------------------------------
%==========================================================================================
\section{Functionals}
\label{sec:fctl}

%------------------------------------------------------------------------------------------
In this section, we study functionals on $\CatCu$-semigroups and their connection to the order structure.
First, we show that the existence of nontrivial functionals characterizes stable finiteness of simple $\CatCu$-semigroups;
see \autoref{prp:simpleSF}.
Then, we study the relation of `stable domination' of elements in a \pom;
see \autoref{dfn:relS}.

We recall that comparison by extended states is closely related to stable domination of elements;
see \autoref{prp:relSTFAE}.
In the context of $\CatCu$-semigroups, we introduce the `regularization' of a relation;
see \autoref{dfn:regularRel}.
The main result of this section is \autoref{prp:comparisonRelations}, where we show that comparison by functionals on a $\CatCu$-semigroup is closely related to the regularization of the stable domination relation.

%==========================================================================================
\begin{pgr}
\label{pgr:fctl}
\index{terms}{state}
\index{terms}{extended state}
\index{terms}{functional}
\index{symbols}{F(S)@$F(S)$ \quad (functionals)}
Let $S$ be a \pom.
A \emph{state} on $S$ is a map $f\colon S\to [0,\infty)$ that preserves addition, order and the zero-element.
If the value $\infty$ is allowed, then we call $f$ an \emph{extended state}.
Thus, an (extended) state is a $\CatPom$-morphism from $S$ to $[0,\infty)$ (respectively from $S$ to $[0,\infty]$).

Assume now that $S$ is a $\CatCu$-semigroup.
A \emph{functional} on $S$ is a map $\lambda\colon S\to[0,\infty]$ that preserves addition, order, the zero-element, and suprema of increasing sequences.
Hence, a functional is a generalized $\CatCu$-morphism from $S$ to $[0,\infty]$.
The set of functionals on $S$ is denoted by $F(S)$.
When equipped with a suitable topology, $F(S)$ becomes a compact Hausdorff space;
see \cite[Theorem~4.8]{EllRobSan11Cone}, see also \cite{Rob13Cone}.
If $S$ is countably-based, then $F(S)$ is second-countable, hence a compact, metrizable space.

We equip $F(S)$ with pointwise addition and order, which provides it with the structure of a \pom.
If $S$ satisfies \axiomO{5} (or just the original form of the axiom, \axiomO{5'}; see \autoref{rmk:addAxioms}), then $F(S)$ is algebraically ordered;
see \cite[Proposition~2.2.3]{Rob13Cone}.

It is clear that by multiplying a functional $\lambda\in F(S)$ with a positive scalar $\theta\in(0,\infty)$, one obtains a functional $\theta\cdot\lambda$.
It was shown in the comments before Theorem~4.8 in \cite{EllRobSan11Cone} that this can be extended to a scalar multiplication
\[
[0,\infty]\times F(S) \to F(S),
\]
by defining:
\begin{align*}
(0\cdot\lambda)(a) &:= 0 \text{ if } \lambda(a')<\infty \text{ for all } a'\ll a,\quad
(0\cdot\lambda)(a) := \infty \text{ otherwise}, \\
(\infty\cdot\lambda)(a) &:= 0 \text{ if } \lambda(a)=0,\quad
(\infty\cdot\lambda)(a) := \infty \text{ otherwise},
\end{align*}
for $a\in S$.
As shown in \cite{EllRobSan11Cone}, the restricted scalar multiplication $[0,\infty)\times F(S) \to F(S)$ is jointly continuous.
However, differently than stated in \cite{EllRobSan11Cone}, the multiplication with $\infty$ is not continuous as a map $F(S) \to F(S)$.
Consider for instance $S=[0,\infty]$, the Cuntz semigroup of the Jacelon-Razak algebra;
see \autoref{rmk:Razac}.
As a \pom, $F(S)$ is isomorphic to $[0,\infty]$.
Given $t\in[0,\infty]$ we denote the corresponding functional in $F(S)$ by $\lambda_t$.
The sequence $(\lambda_{1/n})_{n\in\N}$ converges to $\lambda_0$ in the topology of $F(S)$.
However, we have $\infty\cdot\lambda_{1/n} = \lambda_\infty$ for each $n$, while $\infty\cdot \lambda_0 = \lambda_0$.

Given an element $a\in S$, we say that a functional $\lambda$ is normalized at $a$ provided $\lambda(a)=1$, and we denote the set of these functionals by $F_a(S)$.
If $S$ is simple and $a\in S$ is a compact element, then $F_a(S)$ is a closed, convex subset of $F(S)$.

We denote by $\Lsc(F(S))$ the set of functions $f\colon F(S)\to[0,\infty]$ that are lower-semicontinuous $\CatPom$-morphisms. %(that is, order-preserving monoid homomorphisms).
\index{symbols}{Lsc(F(S))@$\Lsc(F(S))$}
Given $f\in\Lsc(F(S))$ and $\theta\in(0,\infty)$, the pointwise product $\theta f$ belongs to $\Lsc(F(S))$, and we have $\theta f(\lambda)=f(\theta\cdot\lambda)$ for every $\lambda\in F(S)$.
Indeed, one may first show this for rational $\theta$ and then use lower semicontinuity of $f$ to extend to all $\theta\in(0,\infty)$.
%Given $f\in\Lsc(F(S))$ and $\lambda\in F(S)$ one might expect that $(0\cdot f)(\lambda)=f(0\cdot\lambda)$.
%Of course, one could \emph{define} a function $0\cdot f\colon F(S)\to[0,\infty]$ with this equation.
%However, $0\cdot f$ will not be given as a pointwise product of

If $F(S)$ is algebraically ordered (for example if $S$ satisfies \axiomO{5}), then a function $f\colon F(S)\to[0,\infty]$ is automatically order-preserving as soon as it is additive.

We define a binary relation $\lhd$ on $\Lsc(F(S))$ as follows:
Given $f,g\in\Lsc(F(S))$, we set $f\triangleleft g$ if and only if $f\leq (1-\varepsilon)g$ for some $\varepsilon>0$, and if moreover $f$ is continuous at each $\lambda\in F(S)$ where $g(\lambda)<\infty$;
see the paragraph after Remark~3.1.5 in \cite{Rob13Cone}.
We let $L(F(S))$ be the subset of $\Lsc(F(S))$ consisting of all $f\in\Lsc(F(S))$ for which there exists a sequence $(f_n)_n$ in $\Lsc(F(S))$ satisfying $f=\sup_nf_n$ (the pointwise supremum) and $f_n\triangleleft f_{n+1}$ for each $n$.
\index{symbols}{$\lhd$}
\index{symbols}{L(F(S))@$L(F(S))$}

Any element $a\in S$ induces a function
\[
\hat{a}\colon F(S)\to[0,\infty],\quad
\hat{a}=(\lambda\mapsto \lambda(a)),\quad
\txtFA \lambda\in F(S).
\]
The assignment $a\mapsto \hat{a}$ defines a map $S\to L(F(S))$ that preserves addition, order and suprema of increasing sequences.
(That $\hat{a}$ is an element of $L(F(S))$ follows from \cite[Proposition~3.1.6]{Rob13Cone}.)
\index{symbols}{a$\hat$@$\hat{a}$}

If $S$ is a $\CatCu$-semigroup satisfying \axiomO{5}, then it is shown in \cite{Rob13Cone} that $L(F(S))$ is also a $\CatCu$-semigroup satisfying \axiomO{5}.
\end{pgr}

%==========================================================================================
\begin{pgr}[Stable finiteness]
\label{pgr:finiteSemigr}
\index{terms}{element!finite}
\index{terms}{finite element}
\index{terms}{infinite element}
\index{terms}{element!properly infinite}
\index{terms}{properly infinite element}
\index{terms}{Cu-semigroup@$\CatCu$-semigroup!stably finite}
\index{terms}{stably finite Cu-semigroup@stably finite $\CatCu$-semigroup}
Let $S$ be a $\CatCu$-semigroup.
An element $a\in S$ is \emph{finite} if for every $b\in S$, we have that $a+b=a$ implies $b=0$.
Equivalently, we have $a<a+b$ for every nonzero $b\in S$.
We call an element \emph{infinite} if it is not finite.
An infinite element $a\in S$ is \emph{properly infinite} if $2a=a$.
We say that $S$ is \emph{stably finite} if an element $a\in S$ is finite whenever there exists $\tilde{a}\in S$ with $a\ll\tilde{a}$.
If $S$ contains a largest element, denoted by $\infty$, then the latter condition is equivalent to $a\ll\infty$.

In general, a $\CatCu$-semigroup does not contain a largest element.
There are, however, two important cases when a largest element always exists.
First, consider a simple $\CatCu$-semigroup $S$.
We may assume $S\neq\{0\}$.
Choose $a\in S$ nonzero and consider the increasing sequence $(ka)_{k\in\N}$.
By axiom \axiomO{1}, the supremum of this sequence exists and it is easy to check that it is the largest element of $S$:
\[
\infty=\sup_{k\in\N} ka.
\]

In the other case, assume that $S$ is a countably-based $\CatCu$-semigroup.
Choose a countable set $\{a_0,a_1,a_2,\ldots\}$ in $S$ that is a basis in the sense of \autoref{pgr:axiomsW}.
For each $n\in\N$, consider the $n$-th partial sum $\sum_{k=0}^na_k$.
It is straightforward to check that the supremum of this increasing sequence of partial sums is the largest element of $S$:
\[
\infty = \sup_n \sum_{k=0}^n a_k.
\]

Thus, if $S$ is a $\CatCu$-semigroup that is simple or countably-based, then $S$ is stably finite if and only if an element $a\in S$ is finite whenever $a\ll\infty$.
\end{pgr}

%==========================================================================================
The following result is the generalization of \cite[Lemma~4.1, p.61]{Goo86Poag} from the setting of partially ordered abelian groups to the setting of \pom{s}.
It is shown in \cite{BlaRor92ExtendStates} that extensions of states on \pom{s} exist.
However, since we need to control the extended state, we have to prove a refined version of \cite[Corollary~2.7]{BlaRor92ExtendStates}.
We thank the referee for suggesting this approach to fix and generalize our original \autoref{prp:existenceFctl} from the setting of simple $\CatCu$-semigroups to general $\CatCu$-semigroups.

%==========================================================================================
\begin{lma}
\label{prp:extendingStates}
Let $M$ be a \pom, let $N$ be a submonoid of $M$, let $f\colon N\to[0,\infty)$ be a state on $N$, and let $x\in M$.
Set
\begin{align*}
p &:= \sup\left\{ \tfrac{f(y_1)-f(y_2)}{m} : y_1,y_2\in N, m\in\N_+, y_1\leq y_2+mx \right\}, \\
p'&:= \sup\left\{ \tfrac{f(y_1)-f(y_2)}{m} : y_1,y_2\in N, m\in\N_+, \bar{m}\in\N, y_1+\bar{m}x\leq y_2+(m+\bar{m})x \right\}, \\
r &:= \inf\left\{ \tfrac{f(z_2)-f(z_1)}{n} : z_1,z_2\in N, n\in\N_+, z_1+nx\leq z_2 \right\}, \\
r'&:= \inf\left\{ \tfrac{f(z_2)-f(z_1)}{n} : z_1,z_2\in N, n\in\N_+, \bar{n}\in\N, z_1+(n+\bar{n})x\leq z_2+\bar{n}x \right\}.
\end{align*}
Then:
\beginEnumStatements
\item
We have $0\leq p=p'\leq r\leq\infty$ and $-\infty\leq r'\leq r\leq\infty$.
\item
If $x\leq l y$ for some $y\in N$ and $l\in\N$, then $p<\infty$ and $r'=r$.
\item
If $\tilde{f}\colon N+\N x\to[0,\infty)$ is a state extending $f$, then $p'\leq\tilde{f}(x)\leq r'$.
\item
For every real number $q$ satisfying $p'\leq q\leq r'$ there exists a (unique) state $\tilde{f}\colon N+\N x\to[0,\infty)$ that extends $f$ and such that $\tilde{f}(x)=q$.
\end{enumerate}
\end{lma}
\begin{proof}
\enumStatement{1}:
It is straightforward to check $0\leq p\leq p'$ and $-\infty\leq r'\leq r\leq\infty$.
To verify $p\geq p'$, let $y_1,y_2\in N, m\in\N_+$ and $\bar{m}\in\N$ satisfy $y_1+\bar{m}x\leq y_2+(m+\bar{m})x$.
It is enough to show $p\geq \tfrac{f(y_1)-f(y_2)}{m}$.
We have
\[
2y_1 +\bar{m}x
\leq y_1 +y_2 +mx +\bar{m}x
= (y_2+mx) +(y_1+\bar{m}x)
\leq 2y_2 +2mx +\bar{m}x.
\]
Proceeding by induction (as in the proof of \cite[Proposition~3.2]{Ror92StructureUHF2} or \cite[Lemma~2.1]{AraGooOMePar98Separative}), we obtain
\[
ky_1 +\bar{m}x
\leq ky_2 +kmx +\bar{m}x,
\]
for all $k\in\N_+$.
Then $ky_1 \leq ky_2 + (km+\bar{m})x$, which implies
\[
p \geq \tfrac{f(ky_1)-f(ky_2)}{km+\bar{m}}
= \tfrac{f(y_1)-f(y_2)}{m+\bar{m}/k}.
\]
Since this holds for all $k\in\N_+$, we obtain $p\geq \tfrac{f(y_1)-f(y_2)}{m}$, as desired.

Next, let us show $p\leq r$.
This clearly holds if $r=\infty$.
In the other case, let $y_1,y_2,z_1,z_2\in N$ and $m,n\in\N_+$ satisfy
\[
y_1\leq y_2+mx,\quad
z_1+nx \leq z_2.
\]
Multiplying the first inequality by $n$ and adding $mz_1$ on both sides, and multiplying the second inequality by $m$ and adding $ny_2$ on both sides, we obtain
\begin{align*}
ny_1 +mz_1 \leq ny_2 +mz_1 + nmx \leq mz_2 +ny_2,
\end{align*}
and hence
\[
\tfrac{f(y_1)-f(y_2)}{m} \leq \tfrac{f(z_2) - f(z_1)}{n}.
\]
Passing to the supremum and infimum, we obtain $p\leq r$, as desired.

\enumStatement{2}:
Let $y\in N$ and $l\in\N$ such that $x\leq ly$.
(This means $x\varpropto y$; see \autoref{dfn:relS}.)
It is straightforward to show $p\leq lf(y)<\infty$.
To prove $r'\geq r$, let $z_1,z_2\in N, n\in\N_+$ and $\bar{n}\in\N$ such that $z_1+(n+\bar{n})x\leq z_2+\bar{n}x$.
It is enough to show $\tfrac{f(z_2)-f(z_1)}{n} \geq r$.
As in the proof of \enumStatement{1}, we deduce by induction
\[
kz_1+knx+\bar{n}x
\leq kz_2+\bar{n}x,
\]
for all $k\in\N_+$.
Using $x\leq ly$ at the last step, we deduce
\[
kz_1 + knx
\leq kz_1+knx+\bar{n}x
\leq kz_2+\bar{n}x
\leq kz_2+\bar{n}ly,
\]
and therefore
\[
\tfrac{f(z_2)-f(z_1)}{n} + \tfrac{f(\bar{n}ly)}{kn}
= \tfrac{f(kz_2+\bar{n}ly)-f(kz_1)}{kn}
\geq r.
\]
Since this holds for every $k\in\N_+$, we obtain $\tfrac{f(z_2)-f(z_1)}{n}\geq r$, as desired.

\enumStatement{3}:
This is straightforward to show.

Before proceeding to prove (4), we first verify a helpful statement.

\textbf{Claim 1}.
Let $y_1,y_2\in N$ and $k\in\N$ satisfy $y_1+kx\leq y_2+kx$.
If $p<\infty$, then $f(y_1)\leq f(y_2)$.

To verify the claim, let $n\in\N$.
Then
\[
ny_1 + nkx \leq ny_2 + nkx \leq  ny_2 + (nk+1)x.
\]
Using the definition of $p'$ at the second step and that $p=p'$, we compute
\[
n \left( f(y_1)-f(y_2) \right)
= \tfrac{f(ny_1)-f(ny_2)}{(nk+1)-(nk)} \leq p.
\]
Since this holds for all $n\in\N$, and since $p<\infty$, we deduce $f(y_1)-f(y_2)\leq 0$.

\enumStatement{4}:
Let $q\in\R$ satisfy $p'\leq q\leq r'$.
It is clear that there is at most one state $\tilde{f}$ on $N+\N x$ that extends $f$ and that satisfies $\tilde{f}(x)=q$.
To show existence, we verify that the assignment
\[
\tilde{f}\colon N+\N x\to[0,\infty),\quad
\tilde{f}(y+kx) = f(y) + kq\quad
(y\in N, k\in\N),
\]
defines a state on $N+\N x$.
Let $y,z\in N$ and $\bar{m},\bar{n}\in\N_+$ satisfy $y+\bar{m}x \leq z+\bar{n}x$.
By distinguishing three cases, we show $\tilde{f}(y+\bar{m}x) \leq \tilde{f}(z+\bar{n}x)$.

Case~1:
Assume $\bar{m}<\bar{n}$.
Set $m:=\bar{n}-\bar{m}$ so that
\[
y+\bar{m}x \leq z + (m+\bar{m})x.
\]
It follows from the definition of $p'$ that $f(y)-f(z)\leq mp'$.
Using this at the second step, and using $p'\leq q$ at the third step, we deduce
\[
\tilde{f}(y+\bar{m}x)
= f(y) + \bar{m}q
\leq f(z) + mp' + \bar{m}q
\leq f(z) + mq + \bar{m}q
= \tilde{f}(z+\bar{n}x).
\]

Case 2:
Assume $\bar{m}=\bar{n}$.
Using Claim~1 at the second step, we deduce
\[
\tilde{f}(y+\bar{m}x)
= f(y) + \bar{m}q
\leq f(z) + \bar{m}q
= \tilde{f}(z+\bar{n}x).
\]

Case 3:
Assume $\bar{m}>\bar{n}$.
Set $n:=\bar{m}-\bar{n}$ so that
\[
y+(n+\bar{n})x \leq z + \bar{n}x.
\]
This time, it follows from the definition of $r'$ that $nr'\leq f(z)-f(y)$.
Using this at the third step, and using $q\leq r'$ at the second step, we deduce
\[
\tilde{f}(y+\bar{m}x)
= f(y) + nq + \bar{n}q
\leq f(y) + nr' + \bar{n}q
\leq f(z) + \bar{n}q
= \tilde{f}(z+\bar{n}x).
\]
This shows that $\tilde{f}$ is well-defined and order-preserving.
It follows easily that $\tilde{f}$ is a state on $M$ such that $\tilde{f}(x)=q$.
\end{proof}

%==========================================================================================
\begin{rmks}
\label{rmk:extendingStates}
(1)
As in \autoref{prp:extendingStates}, let $M$ be a \pom, let $N$ be a submonoid of $M$, let $f\colon N\to[0,\infty)$ be a state on $N$, and let $x\in M$.
Assume $x\leq ly$ for some $y\in N$ and $l\in\N$.
(This means $x\varpropto y$; see \autoref{dfn:relS}.)
Then $f$ can be extended to a state $\tilde{f}$ on $N+\N x$.
Moreover, the maximal value of $\tilde{f}(x)$ for such an extension is $r$, as defined in \autoref{prp:extendingStates}.
Thus, we can find an extension with $\tilde{f}(x)>0$ if and only if $r>0$.

(2)
Consider the semigroup $M=\overline{\N}=\{0,1,2,\ldots,\infty\}$, and the submonoid $N=\langle 1\rangle = \N$, and $x=\infty$.
Then the canonical state on $N$ cannot be extended to a state on $M$.
Indeed, we have $p=p'=r=\infty$ and $r'=-\infty$.
\end{rmks}

%==========================================================================================
\begin{lma}
\label{prp:stateEstimateR}
Let $M$ be a \pom, and let $a,x\in M$.
Set $N:=\langle a\rangle$, the submonoid of $M$ generated by $a$.
For each $m,n\in\N$, assume that $ma\leq na$ if and only if $m\leq n$.
Consider the state $f$ on $N$ given by $f(na)=n$ for each $n\in\N$.
Set
\[
\tilde{r} := \inf \left\{ \tfrac{k}{n} : k,n\in\N_+, nx\leq ka \right\}.
\]
Then $r=\tilde{r}$, for $r$ defined as in \autoref{prp:extendingStates} (with respect to $M$, $N$, $f$ and $x$).

Given $L\in\N_+$ such that $r<\tfrac{1}{L}$, there exists $k\in\N_+$ with $kLx\leq ka$.
Conversely, if there exist $L\in\N_+$ and $k\in\N_+$ satisfying $kLx\leq ka$, then $r\leq\tfrac{1}{L}$.
Thus, we have $0<r$ if and only if there exists $L\in\N_+$ such that $kLx \nleq ka$ for all $k\in\N_+$.
\end{lma}
\begin{proof}
It is straightforward to check $r\leq\tilde{r}$.
To verify $r\geq\tilde{r}$, let $z_1,z_2\in N$ and $m\in\N_+$ such that $z_1 + mx \leq z_2$.
We need to show $\tfrac{f(z_2)-f(z_1)}{m}\geq\tilde{r}$.

Choose $k_1$ and $k_2$ such that $z_1=k_1 a$ and $z_2 = k_2 a$.
It follows from the assumptions on $a$ that $k_1\leq k_2$.
Set $k:=k_2-k_1$.
Then
\[
k_1 a + mx \leq k_1 a + ka.
\]
By induction, as in the proof of \enumStatement{1} in \autoref{prp:extendingStates} above, we deduce
\[
k_1 a + lmx
\leq k_1 a + lka,
\]
for all $l\in\N_+$.
This implies $lmx \leq (k_1+lk)a$, and therefore
\[
\tfrac{k_1}{lm} + \tfrac{k}{m}
= \tfrac{k_1+lk}{lm}
\geq \tilde{r}.
\]
Since this holds for all $l\in\N_+$, we have $\tfrac{f(z_2)-f(z_1)}{m}=\tfrac{k}{m}\geq \tilde{r}$, as desired.

Next, let $L\in\N_+$ such that $\tilde{r}<\tfrac{1}{L}$.
By definition of $\tilde{r}$, there exist $k,n\in\N_+$ such that $\tfrac{k}{n}<\tfrac{1}{L}$ (and hence $kL \leq n$) and such that $nx \leq ka$.
Then $kL x \leq nx \leq ka$, as desired.
Finally, if $kL x \leq ka$ for some $k,L\in\N_+$, then clearly $r=\tilde{r}\leq\tfrac{1}{L}$.
\end{proof}

%==========================================================================================
\begin{thm}
\label{prp:existenceFctl}
Let $S$ be a $\CatCu$-semigroup, and let $a\in S$.
Then the following are equivalent:
\beginEnumStatements
\item
For every $m,n\in\N$, we have $ma\leq na$ if and only $m\leq n$.
Furthermore, there exist $l,L\in\N$ and $x\in S$ such that $x\ll la$ and such that $kLx\nleq ka$ for all $k\in\N_+$.
\item
There exist $l\in\N$ and $x\in S$ such that $x\ll la$.
Furthermore, there exists a state $f$ on the submonoid generated by $a$ and $x$ such that $f(x)>0$.
\item
There exists a functional $\lambda\in F(S)$ such that $\lambda(a)=1$.
\end{enumerate}
\end{thm}
\begin{proof}
All three conditions imply $a\neq 0$.
Let us show that \enumStatement{1} implies \enumStatement{2}.
Let $M=\langle a,x\rangle$, and let $N=\langle a\rangle$.
Let $f_0$ be the state on $N$ given by $f_0(na)=n$ for each $n\in\N$.
By \autoref{prp:stateEstimateR}, we have $r>0$, for $r$ defined as in \autoref{prp:extendingStates} (with respect to $M$, $N$, $f_0$ and $x$).
Then, by \autoref{prp:extendingStates}, there exists a state $f$ on $M$ that extends $f_0$ and such that $f(x)=r$.
This verifies \enumStatement{2}.

Next, let us show that \enumStatement{2} implies \enumStatement{3}.
Let $M=\langle a,x\rangle$.
Consider the following subsets of $S$:
\begin{align*}
T_0 &= \left\{ b\in S : b\leq ka, \text{ for some } k\in\N \right\}, \\
T &= \left\{ b\in S : b \leq \infty\cdot a \right\}.
\end{align*}
Then $T_0$ is a submonoid of $S$ with order-unit $a$, and $M$ is a submonoid of $T_0$ containing the order-unit.
By \cite[Corollary~2.7]{BlaRor92ExtendStates}, we can extend the state $f$ on $M$ to a state $\tilde{f}$ on $T_0$.
For $b',b\in T$ satisfying $b'\ll b$, we have $b'\in T_0$.
We can therefore define
\[
\lambda_0\colon T\to[0,\infty],\quad
\lambda_0(b):= \sup \left\{ \tilde{f}(b') : b'\ll b \right\},\quad
\txtFA b\in T.
\]
It is straightforward to check that $\lambda_0$ is a functional on $T$.
We may extend $\lambda_0$ to a functional $\lambda_1$ on $S$ by setting
\[
\lambda_1(s)
:= \begin{cases}
\lambda_0(s), &\text{ if } s\in T, \\
\infty, &\text{ if } s\notin T.
\end{cases}
\]
for all $s\in S$.
It follows
\[
\lambda_1(la) = \lambda_0(la) \geq \tilde{f}(x) = f(x) > 0.
\]
Thus, $\lambda_1(a)\neq 0$.
Moreover, we have $\lambda_1(a)=\lambda_0(a)\leq \tilde{f}(a)=f(a)<\infty$.
Then the functional $\lambda=\tfrac{1}{\lambda_1(a)}\cdot\lambda_1$ has the desired properties.

Finally, let us show that \enumStatement{3} implies \enumStatement{1}. From the existence of a functional $\lambda\in F(S)$ with $\lambda(a)=1$ it is clear that $ma\leq na$ if and only if $m\leq n$.
Furthermore, since $a$ is the supremum of an increasing sequence of elements that are compactly contained in $a$, and since $\lambda$ preserves suprema of increasing sequences, there exists $x\in S$ such that $x\ll a$ and $\lambda(x)\neq 0$.
Choose $L$ such that $L>\tfrac{1}{\lambda(x)}$.
Then $kLx\nleq kx$, for all $k\in\N_+$, as desired.
\end{proof}

%==========================================================================================
\begin{cor}
\label{prp:existenceFctlSimple}
Let $S$ be a simple $\CatCu$-semigroup, and let $a\in S$ such that $a\ll\infty$.
Then the following are equivalent:
\beginEnumStatements
\item
The element $a$ is nonzero, and for every $n\in\N$ the element $na$ is finite.
\item
For every $m,n\in\N$, we have $ma\leq na$ if and only if $m\leq n$.
\item
There exists a functional $\lambda\in F(S)$ such that $\lambda(a)=1$.
\end{enumerate}
\end{cor}
\begin{proof}
One easily checks that \enumStatement{1} implies \enumStatement{2}.
Let us show that \enumStatement{3} implies \enumStatement{1}.
Assume that there exists a functional $\lambda$ with $\lambda(a)=1$.
This clearly implies that $a$ is nonzero.
Let $k\in\N$.
In order to show that $ka$ is finite, let $b\in S$ be an element such that $ka+b=ka$.
This implies that $\lambda(b)=0$.
Note that every nonzero functional $\mu$ on a simple $\CatCu$-semigroup is faithful in the sense that $\mu(x)\neq 0$ for every nonzero element $x$.
Thus, it follows from $\lambda(b)=0$ that we have $b=0$, as desired.

Finally, let us show that \enumStatement{2} implies \enumStatement{3}.
We verify \enumStatement{1} of \autoref{prp:existenceFctl}.
We have $a\neq 0$.
Choose $x\in S$ with $x\ll a$ and $x\neq 0$.
Since $S$ is simple and $x\neq 0$,
there exists $N\in\N$ such that $a\leq Nx$. Put $L=2N$.
It follows that for every $k\in \N_+$, $2ka\leq 2kNx=kLx$ and hence $kLx\nleq ka$, as desired.
\end{proof}

%==========================================================================================
\begin{lma}
\label{prp:elementInSimple}
Let $S$ be a simple $\CatCu$-semigroup, and let $a\in S$.
Then the following are equivalent:
\beginEnumStatements
\item
The element $a$ is infinite.
\item
The element $a$ is properly infinite.
\item
We have $a=\infty$.
\end{enumerate}
\end{lma}
\begin{proof}
Observe that in a general nonzero $\CatCu$-semigroup the largest element is properly infinite (if it exists), and that every properly infinite element is infinite.
Finally, to show that \enumStatement{1} implies \enumStatement{3}, let us assume that $S$ is a simple $\CatCu$-semigroup, and let $a\in S$ be infinite.
By definition, we can choose $b\in S$ nonzero with $a=a+b$.
Then $a=a+2b$, and inductively $a=a+kb$ for every $k\in\N$.
Therefore,
\[
\infty = \sup_k kb \leq \sup_k (a+kb) = \sup_k a = a \leq \infty,
\]
which shows $a=\infty$, as desired.
\end{proof}

%==========================================================================================
\begin{rmk}
\label{rmk:elementInSimple}
Let $S$ be a simple $\CatCu$-semigroup, and let $a\in S$.
Then a multiple of $a$ can be infinite even if $a$ itself is finite.
This happens for instance in the elementary semigroups $\elmtrySgp{k}$ from \autoref{pgr:elementarySemigr}.
\end{rmk}

%------------------------------------------------------------------------------------------
For the next result, we call a functional nontrivial if it does not only take the values $0$ and $\infty$.

%==========================================================================================
\begin{prp}
\label{prp:simpleSF}
Let $S$ be a simple $\CatCu$-semigroup with $S\neq\{0\}$.
Then the following are equivalent:
\beginEnumStatements
\item
The semigroup $S$ is stably finite.
\item
Every compact element in $S$ is finite.
\item
The largest element $\infty$ is not compact.
\item
There exists a nontrivial functional $\lambda\in F(S)$.
\item
There exists a nonzero element $a\ll\infty$ such that $na$ is finite for all $n\in\N$.
\end{enumerate}
\end{prp}
\begin{proof}
In general, every compact element in a stably finite $\CatCu$-semigroup is finite.
Moreover, the largest element $\infty$ is never finite.
It follows that \enumStatement{1} implies \enumStatement{2}, and that \enumStatement{2} implies \enumStatement{3}.

Let us show that \enumStatement{1} implies \enumStatement{4}.
Choose a nonzero element $a\in S$ satisfying $a\ll\infty$.
Then $na\ll\infty$ for all $n\in\N$.
Since $S$ is stably finite, we obtain that $na$ is finite for every $n\in\N$.
By \autoref{prp:existenceFctlSimple}, there exists a functional $\lambda\in F(S)$ with $\lambda(a)=1$.
This functional is nontrivial, as desired.

Let us show that \enumStatement{4} implies \enumStatement{3}.
Choose a nontrivial functional $\lambda\in F(S)$.
By rescaling if necessary, we may assume that there exists $a\in S$ with $\lambda(a)=1$.
In order to show \enumStatement{3}, assume that $\infty$ is compact.
Then, since $\infty=\sup_k(ka)$, there exists $n\in\N$ with $\infty\leq na$, and hence $\infty=na$.
This implies
\[
n=\lambda(na)=\lambda(\infty)=\lambda(2\infty)=\lambda(2na)=2n,
\]
which clearly is a contradiction.
Hence, $\infty$ is not compact, which shows \enumStatement{3}.

Finally, let us show that \enumStatement{3} implies \enumStatement{1}.
To verify the contraposition, assume that $S$ is not stably finite.
Then there exists a nonzero, infinite element $a\in S$ satisfying $a\ll\infty$.
By \autoref{prp:elementInSimple}, every infinite element in $S$ is equal to the largest element $\infty$.
It follows $\infty=a\ll\infty$, and so $\infty$ is compact.
The equivalence between \enumStatement{4} and \enumStatement{5} follows from \autoref{prp:existenceFctlSimple}.
\end{proof}

%==========================================================================================
\begin{rmk}
\label{rmk:simpleSF}
The equivalence of \enumStatement{1} and \enumStatement{4} in \autoref{prp:simpleSF} is well-known, especially for Cuntz semigroups of (simple) \ca{s}.
It is used to show that every unital, simple, stably finite \ca\ has a $2$-quasitrace.
In fact, the correspondence between $2$-quasitraces on a \ca\ and functionals on its Cuntz semigroup was one of the original motivations for Cuntz to introduce the semigroups named after him;
see \cite{Cun78DimFct}, \cite{BlaHan82DimFct}.
\end{rmk}

%------------------------------------------------------------------------------------------
In the next part of this section, we will study the connection between the order structure of a \pom\ and the set of its functionals.
We first recall a notion that has appeared many times in the literature.
The notation chosen here follows Definition~2.2 in~\cite{OrtPerRor12CoronaStability}.

%==========================================================================================
\begin{dfn}
\label{dfn:relS}
\index{terms}{stably dominated}
\index{terms}{relation!stable domination}
\index{symbols}{$<_s$ \quad (stable domination relation)}
\index{symbols}{$\varpropto$}
Let $M$ be a \pom, and let $a,b\in M$.
We will write $a\varpropto b$ if there exists $k\in\N$ such that $a\leq kb$.

We say that $a$ is \emph{stably dominated} by $b$, denoted by $a<_sb$, if there exists $k\in\N$ such that $(k+1)a\leq kb$.
\end{dfn}

%------------------------------------------------------------------------------------------
The following result provides useful characterizations of the relation $<_s$.
Several versions of this results have appeared in the literature (see for example \cite[Proposition~2.1]{OrtPerRor12CoronaStability}), and most are based on \cite[Lemma~4.1]{GooHan76RankFct}.

%==========================================================================================
\begin{prp}
\label{prp:relSTFAE}
Let $M$ be a \pom, and let $a,b\in M$.
Then the following are equivalent:
\beginEnumStatements
\item
We have $a<_sb$, that is, there exists $k\in\N$ such that $(k+1)a\leq kb$.
\item
There exists $k_0\in\N$ such that $(k+1)a\leq kb$ for all $k\geq k_0$.
\item
Given $n\in\N_{+}$, there exists $k\in\N$ such that $(k+n)a\leq kb$.
\item
Given $n\in\N_{+}$, there exists $k_0\in\N$ such that $(k+n)a\leq kb$ for all $k\geq k_0$.
\item
We have $a\varpropto b$, and $f(a)<f(b)$ for every extended state on $S$ that is normalized at $b$.
\end{enumerate}
If $b$ is an order-unit for $M$, then the above statements are also equivalent to:
\beginEnumStatements
\setcounter{enumi}{5}
\item
We have $f(a)<f(b)$ for every state on $S$ that is normalized at $b$.
\end{enumerate}
\end{prp}
\begin{proof}
It is clear that \enumStatement{4} implies \enumStatement{3} and \enumStatement{2}, and that \enumStatement{3} implies \enumStatement{1}, and that \enumStatement{2} implies  \enumStatement{1}.
It is also easy to see that \enumStatement{1} implies \enumStatement{3}.
Indeed, if $(k+1)a\leq kb$ for some $k\in\N$, then $(kn+n)a\leq knb$, as desired.

In order to show that \enumStatement{3} implies \enumStatement{4}, let $n\in\N_{+}$ be given.
By assumption, we can find $d\in\N$ such that $(d+n)a\leq db$.
Set $k_0:=d(d+1)$, which we claim has the desired properties.
To verify this, let $k\in\N$ satisfy $k\geq k_0$.
Then there are $x,y\in\N$ with $k=(d+1)x+y$ and $x\geq d$ and $y\leq d$.
Then
\[
(k+n)a
=[(d+1)x]a+(y+n)a
\leq [(d+n)x]a+(d+n)a
\leq (dx)b+db
\leq kb,
\]
as desired.

Finally, the equivalence between \enumStatement{1} and \enumStatement{5}
is shown in \cite[Proposition~2.1]{OrtPerRor12CoronaStability}.
If $b$ is an order-unit, it is easy to verify that \enumStatement{5} and \enumStatement{6} are equivalent.
\end{proof}

%==========================================================================================
\begin{dfn}
\label{dfn:regularRel}
\index{terms}{regularization of a relation}
\index{terms}{regular relation}
\index{symbols}{R*@$R^\ctsRel$ \quad (regularization of a relation)}
Let $S$ be a $\CatCu$-semigroup, and let $R\subset S\times S$ be a binary relation.
The \emph{regularization} of $R$, denoted by $R^\ctsRel$, is the binary relation defined as follows:
For any $a,b\in S$, we set $a R^\ctsRel b$ if and only if $a' R b$ for every $a'\in S$ satisfying $a'\ll a$.
A relation $R$ is \emph{regular} if it is equal to its own regularization.
\end{dfn}

%==========================================================================================
\begin{exa}
\index{symbols}{$<_s^\ctsRel$}
\index{symbols}{$\varpropto^\ctsRel$}
Let $S$ be a $\CatCu$-semigroup.

(1)
The usual order relation $\leq$ on $S$ is regular.
Indeed, given $a,b\in S$, it is clear that $a\leq b$ implies $a\leq^\ctsRel b$.
The converse follows from axiom \axiomO{2} for $S$.

(2)
The way-below relation $\ll$ on $S$ is not regular.
In fact, it is straightforward to check that the regularization of $\ll$ is nothing but the order relation $\leq$.

(3)
\index{terms}{element!soft}
\index{terms}{soft element}
The stable domination relation $<_s$ from \autoref{dfn:relS} is not regular.
However, we will show in \autoref{prp:comparisonRelations} that the regularization of $<_s$ is closely related to comparison by functionals.
In \autoref{sec:soft}, we will study elements $a\in S$ satisfying $a<_s^\ctsRel a$.
(We call such elements `soft'.)

(4)
The relation $\varpropto$ is not regular.
However, its regularization determines exactly which ideal an element of $S$ generates.
More precisely, given $a,b\in S$, we have $a\varpropto^\ctsRel b$ if and only if $a\leq\infty\cdot b$, and if and only if $\Idl(a)\subset \Idl(b)$;
see \autoref{pgr:Lat}.
\end{exa}

%==========================================================================================
\begin{dfn}[{R{\o}rdam, \cite[Section~3]{Ror92StructureUHF2}}]
\label{dfn:almUnp}
\index{terms}{positively ordered monoid!almost unperforated}
\index{terms}{almost unperforated}
\index{terms}{unperforated!almost}
A \pom\ $M$ is \emph{almost unperforated} if for every $a,b\in M$, we have that $a<_sb$ implies $a\leq b$.
\end{dfn}

%------------------------------------------------------------------------------------------
The following result is straightforward to verify.

%==========================================================================================
\begin{lma}
\label{prp:almUnpCu}
Let $S$ be a $\CatCu$-semigroup.
Then $S$ is almost unperforated if and only if for every $a,b\in S$, we have that $a<_s^\ctsRel b$ implies $a\leq b$.
\end{lma}

%==========================================================================================
\begin{thm}
\label{prp:comparisonRelations}
Let $S$ be a $\CatCu$-semigroup, and let $a,b\in S$.
Consider the following statements:
\beginEnumStatements
\item
We have $a<_sb$, that is, there exists $k\in\N$ such that $(k+1)a\leq kb$.
\item
We have $\hat{a}<_s\hat{b}$, that is, there exists $k\in\N$ such that $(k+1)\hat{a}\leq k\hat{b}$.
\item
We have $a\varpropto^\ctsRel b$, and $\lambda(a)<\lambda(b)$ for every $\lambda\in F(S)$ satisfying $\lambda(b)=1$.
\item
We have $a<_s^\ctsRel b$, that is, we have $a'<_s b$ for every $a'\in S$ satisfying $a'\ll a$.
\item
We have $\hat{a}\leq\hat{b}$.
\end{enumerate}
Then the following implications hold:
'\enumStatement{1} $\Rightarrow$ \enumStatement{2} $\Rightarrow$ \enumStatement{3} $\Rightarrow$ \enumStatement{4} $\Rightarrow$ \enumStatement{5}'.

If $a$ is compact, then \enumStatement{4} implies \enumStatement{1}.
If the element $a$ satisfies $a<_s^\ctsRel a$ (such elements will be called `soft'; see \autoref{dfn:softElement}), then \enumStatement{5} implies \enumStatement{4}.
If $S$ is almost unperforated, then \enumStatement{4} implies $a\leq b$.
The different implications are shown in the following diagram:
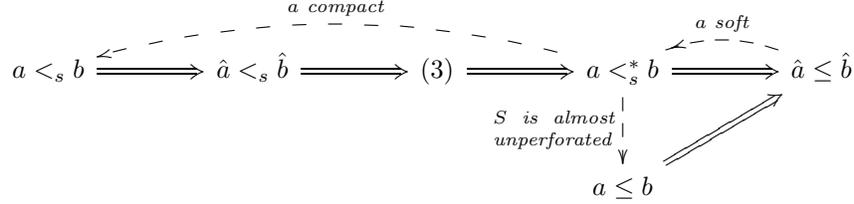
\[
\xymatrix@R+=15pt@C+=40pt@M+=5pt{
{ a<_sb } \ar@{=>}[r]
& { \hat{a}<_s\hat{b} } \ar@{=>}[r]
& {(3)}
\ar@{=>}[r]
& { a<_s^\ctsRel b } \ar@{=>}[r]\ar@/_1.5pc/@{-->}[lll]_{\text{$a$ compact}}
\ar@{-->}[d]_{\parbox{46pt}{\scriptsize $S$ is almost unperforated}}
& { \hat{a}\leq\hat{b} } \ar@/_1pc/@{-->}[l]_{\text{$a$ soft}} \\
& & &  { a\leq b } \ar@{=>}[ur]
}
\]
\end{thm}
\begin{proof}
It is clear that \enumStatement{1} implies \enumStatement{2}, and it is straightforward to check that \enumStatement{4} implies \enumStatement{5}.
To see that \enumStatement{2} implies \enumStatement{3}, assume $\hat{a}<_s\hat{b}$.
This clearly implies $\lambda(a)<1$ for every $\lambda\in F_b(S)$.
Thus, it remains to show $a\varpropto^\ctsRel b$.
Let $I$ be the ideal generated by $b$, that is, $I=\left\{ x\in S : x\leq\infty\cdot b \right\}$.
Consider the following map
\[
\lambda_I\colon S\to[0,\infty],\quad
\lambda_I(x):=\begin{cases}
0, &\text{ if } x\in I \\
\infty, &\text{ if } x\notin I \\
\end{cases},\quad
\txtFA x\in S.
\]
It is easy to check that $\lambda_I$ is a functional.
Since $\lambda_I(b)=0$ and $\hat{a}<_s\hat{b}$, it follows $\lambda_I(a)=0$ and therefore $a\leq\infty\cdot b$, as desired.

Let us show that \enumStatement{3} implies \enumStatement{4}.
Assume $a$ and $b$ satisfy the statement of \enumStatement{3}, and let us show $a<_s^\ctsRel b$.
Let $a'\in S$ satisfy $a'\ll a$.
We want to verify \enumStatement{5} of \autoref{prp:relSTFAE} to show $a'<_sb$.
The argument is similar to the one in the proof of \cite[Proposition~6.2]{EllRobSan11Cone} and \cite[Proposition~2.2.6]{Rob13Cone}.
Since $a'\ll a\varpropto^\ctsRel b$, we have $a'\varpropto b$.

Now, let $f\colon S\to[0,\infty]$ be an extended state with $f(b)=1$.
We want to show $f(a')<f(b)$.
Consider the map
\[
\tilde{f}\colon S\to[0,\infty],\quad
\tilde{f}(x)=\sup\left\{ f(x') : x'\ll x \right\},\quad
\txtFA x\in S.
\]
It is easy to see that $\tilde{f}$ is a functional on $S$.
In the literature, the map $\tilde{f}$ is sometimes called the regularization of $f$.
Since $f(b')\leq f(b)$ for every $b'\ll b$, it follows from the definition of $\tilde{f}$ that $\tilde{f}(b)\leq f(b)$.
We distinguish two cases:

In the first case, assume $\tilde{f}(b)=0$.
Since $a$ is in the ideal generated by $b$, it follows $\tilde{f}(a)=0$.
Using the definition of $\tilde{f}$ at the first step, we deduce
\[
f(a') \leq \tilde{f}(a)=0 < 1=f(b).
\]

In the second case, assume $\tilde{f}(b)>0$.
Using the definition of $\tilde{f}$ at the first and last step, and the assumption \enumStatement{3} in the middle step, we obtain
\[
f(a') \leq \tilde{f}(a) < \tilde{f}(b) \leq f(b).
\]
Thus, in either case, we have $f(a')<f(b)$.
Applying \autoref{prp:relSTFAE}, we obtain $a'<_sb$, as desired.

Finally, if $a$ is compact, it is clear that \enumStatement{4} implies \enumStatement{1}.
Moreover, as observed in \autoref{prp:almUnpCu}, if $S$ is almost unperforated then $a<_s^\ctsRel b$ implies $a\leq b$.
It remains to show that $\hat{a}\leq\hat{b}$ implies $a<_s^\ctsRel b$ if $a$ satisfies $a<_s^\ctsRel a$.
Let $x\in S$ satisfy $x\ll a$.
Choose $y$ such that $x\ll y\ll a$.
By assumption, this implies $y<_s a$.
It follows
\[
\hat{y} <_s\hat{a} \leq\hat{b}.
\]
Using that \enumStatement{2} implies \enumStatement{4}, we get $y<_s^\ctsRel b$.
Since $x\ll y$, we obtain $x<_s b$.
\end{proof}

%------------------------------------------------------------------------------------------
The next result describes which information about the order structure of a $\CatCu$-semigroup is recorded by its functionals.
It has appeared in \cite[Proposition~2.2.6]{Rob13Cone}, under the additional assumption that the $\CatCu$-semigroup satisfies \axiomO{5}.
However, an inspection of the proof of \cite[Proposition~2.2.6]{Rob13Cone} shows that \axiomO{5} is not needed.

%==========================================================================================
\begin{prp}
\label{prp:comparison_hat}
Let $S$ be a $\CatCu$-semigroup, and let $a,b\in S$.
Then the following are equivalent:
\beginEnumStatements
\item
We have $\hat{a}\leq\hat{b}$.
\item
For each $n\in\N$, we have $na<_s^\ctsRel(n+1)b$.
\item
For every $a'\in S$ satisfying $a'\ll a$ and every $\varepsilon>0$, there exist $k,n\in\N$ such that $(1-\varepsilon)<\tfrac{k}{n}$ and $ka'\leq nb$.
\end{enumerate}
If $S$ is almost unperforated, then these conditions are also equivalent to:
\beginEnumStatements
\setcounter{enumi}{3}
\item
For each $n\in\N$, we have $na\leq(n+1)b$.
\end{enumerate}
\end{prp}
\begin{proof}
Let $a,b\in S$.
First, in order to show that \enumStatement{1} implies \enumStatement{2}, assume $\hat{a}\leq\hat{b}$.
This clearly implies $n\hat{a}<_s(n+1)\hat{b}$ for each $n\in\N$.
Then \enumStatement{2} follows from \autoref{prp:comparisonRelations}.

It is straightforward to check that \enumStatement{2} implies \enumStatement{3}, and that \enumStatement{3} implies \enumStatement{1}.
Finally, \enumStatement{4} implies \enumStatement{1} in general.
Conversely, it is clear that \enumStatement{2} implies \enumStatement{4} if $S$ is almost unperforated.
\end{proof}

%------------------------------------------------------------------------------------------
The equivalence of statements \enumStatement{1} and \enumStatement{2} in the next result follows immediately from \autoref{prp:relSTFAE} and was first obtained by R{\o}rdam, \cite[Proposition~3.2]{Ror04StableRealRankZ} (see also \cite[Proposition~3.1]{Ror92StructureUHF2}).
The equivalence with condition \enumStatement{3} follows easily from \autoref{prp:comparisonRelations} and was first shown in \cite[Proposition~6.2]{EllRobSan11Cone}.

%==========================================================================================
\begin{prp}
\label{prp:almUnp}
Let $S$ be a \pom.
Then the following are equivalent:
\beginEnumStatements
\item
The semigroup $S$ is almost unperforated.
\item
For all $a,b\in S$ we have that $a\leq b$ whenever $a\varpropto b$ and $f(a)<f(b)$ for every extended state $f$ on $S$ that is normalized at $b$.
\end{enumerate}
If, moreover, $S$ is a $\CatCu$-semigroup, then these conditions are also equivalent to:
\beginEnumStatements
\setcounter{enumi}{2}
\item
For all $a,b\in S$ we have that $a\leq b$ whenever $a\varpropto^\ctsRel b$ and $\lambda(a)<\lambda(b)$ for every \emph{functional} $\lambda$ on $S$ that is normalized at $b$.
\end{enumerate}
\end{prp}

%------------------------------------------------------------------------------------------
For the next result, recall that a $2$-quasitrace on a \ca{} $A$ is a map
\[
\tau\colon (A\otimes\K)_+\to[0,\infty],
\]
such that $\tau(0)=0$, such that $\tau(xx^*)=\tau(x^*x)$ for all $x\in A\otimes\K$, and such that $\tau(x+y)=\tau(x)+\tau(y)$ for all $x,y\in (A\otimes\K)_+$ that commute.
The set of lower-semicontinuous $2$-quasitraces on $A$ is denoted by $\QT_2(A)$.
Its structure (for example as a lattice and noncancellative cone) has been thoroughly studied in \cite{EllRobSan11Cone}.

Given a $2$-quasitrace $\tau$ on $A$, consider the map
\[
d_\tau\colon (A\otimes\K)_+\to[0,\infty],\quad
d_\tau(x)=\lim_k\tau(x^{1/k}),\quad
\txtFA x\in(A\otimes\K)_+.
\]
If $x$ is Cuntz-subequivalent to $y$, then $d_\tau(x)\leq d_\tau(y)$.
It follows that a $2$-quasitrace $\tau$ on $A$ induces a map
\[
d_\tau\colon\Cu(A)\to[0,\infty],\quad
d_\tau(x)=\lim_k\tau(x^{1/k}),\quad
\txtFA x\in(A\otimes\K)_+,
\]
which is an extended state on $\Cu(A)$.
If $\tau$ is a lower-semicontinuous, then $d_\tau$ is a functional on $\Cu(A)$.

%==========================================================================================
\begin{prp}[{Elliott, Robert, Santiago, \cite[Theorem~4.4]{EllRobSan11Cone}}]
\label{prp:FctlCa}
Let $A$ be a \ca{}.
Then the map
\[
\QT_2(A) \to F(\Cu(A)),\quad
\tau\mapsto d_\tau,\quad
\txtFA \tau\in\QT_2(A),
\]
is a bijection.
When $\QT_2(A)$ and $F(\Cu(A)$ are equipped with suitable natural topologies and order structures, then this map becomes a homeomorphic order isomorphism.
\end{prp}

%==========================================================================================
\begin{cor}
\label{prp:UniqueFctlFromCa}
Let $A$ be a simple, unital \ca{} with a unique $2$-quasi\-trace $\tau$ that satisfies $\tau(1_A)=1$.
Then the Cuntz semigroup $\Cu(A)$ has a unique functional $\lambda$ that satisfies $\lambda([1_A])=1$.
\end{cor}

\vspace{5pt}
%------------------------------------------------------------------------------------------
%==========================================================================================
\section{Soft and purely noncompact elements}
\label{sec:soft}

%------------------------------------------------------------------------------------------
In this section, we first introduce the notion of `softness' for elements in a $\CatCu$-semigroup;
see \autoref{dfn:softElement}.
This concept is closely related to that of `pure noncompactness', which was introduced in the Definition before~6.4 in \cite{EllRobSan11Cone}.
In fact, we will slightly generalize their definition to that of `weak pure noncompactness', which for elements in a $\CatCu$-semigroup $S$ satisfying \axiomO{5} is equivalent to softnes;
see \autoref{prp:soft_wpnc}.
In \autoref{prp:softTFAE}, we will show that under the additional assumption that $S$ is almost unperforated or residually stably finite, an element $a\in S$ is soft if and only if it is purely noncompact.

The set of soft elements in a $\CatCu$-semigroup $S$ forms a submonoid that is closed under suprema of increasing sequences and that is absorbing in a suitable sense;
see \autoref{prp:softAbsorbing}.
The main result of this section is \autoref{prp:soft_comparison}, where we show that the order among soft elements in an almost unperforated $\CatCu$-semigroup $S$ is determined solely by the functionals of $S$.
This generalizes \cite[Theorem~6.6]{EllRobSan11Cone}, where the analogous result is shown for the comparison of purely noncompact elements in the Cuntz semigroup of a \ca.
We point out that we obtain our result without using \axiomO{5}, by considering soft elements
instead of (weakly) purely noncompact elements;
see \autoref{rmk:soft_comparison}.

\vspace{5pt}

%------------------------------------------------------------------------------------------
Let $M$ be a \pom.
An \emph{interval}\index{terms}{interval} in $M$ is a subset $I\subset M$ that is upward directed and order-hereditary.
An interval $I$ is \emph{soft}\index{terms}{interval!soft}\index{terms}{soft interval} if for every $x\in I$ there exist $y\in I$ and $n\in\N$ such that $(n+1)x\leq ny$.
This notion was introduced by Goodearl and Handelman;
see the Definition before Lemma~7.4 in \cite{GooHan82Stenosis}.
It was also studied in \cite{Goo96KMultiplier} and \cite{Per01IdealsMultiplier}.

Using the relation $<_s$ from \autoref{dfn:relS}, an interval $I$ is soft if and only if for every $x\in I$ there exists $y\in I$ such that $x<_s y$.

Next, we introduce the notion of `softness' for elements in $\CatCu$-semigroups.
In \autoref{prp:softElementTFAE}, we will show that it is equivalent to softness of the interval of compactly contained elements.

%==========================================================================================
\begin{dfn}
\label{dfn:softElement}
\index{terms}{element!soft}
\index{terms}{soft}
\index{symbols}{S$_\soft$@$S_\soft$ \quad (soft elements)}
Let $S$ be a $\CatCu$-semigroup.
An element $a\in S$ is \emph{soft} if for every $a'\in S$ satisfying $a'\ll a$ there exists $n\in\N$ such that $(n+1)a'\leq na$.

We denote by $S_\soft$ the subset of soft elements in $S$.
\end{dfn}

%==========================================================================================
\begin{rmk}
\label{rmk:softElement}
Let $S$ be a $\CatCu$-semigroup, and let $a\in S$.

(1)
Consider the set of compactly-contained elements $a^{\ll}:=\{ x\in S : x\ll a \}$.
We have that $a$ is soft if and only if for every $x\in a^{\ll}$ we have $x<_s a$.
Thus, an element is soft if and only if it stably dominates every compactly-contained element.

(2)
Recall that $<_s^\ctsRel$ denotes the regularization of the stable domination relation $<_s$;
see Definitions~\ref{dfn:relS} and~\ref{dfn:regularRel}.
Then $a$ is soft if and only if $a <_s^\ctsRel a$.
In \autoref{prp:comparisonRelations}, we have seen that for soft elements there is a close connection between the order comparison in the $\CatCu$-semigroup and the comparison by functionals.
In the case that the $\CatCu$-semigroup is almost unperforated, we even have that the functionals record the complete information about comparison between soft elements;
see \autoref{prp:soft_comparison}.
\end{rmk}

%==========================================================================================
\begin{prp}
\label{prp:softElementTFAE}
Let $S$ be a $\CatCu$-semigroup, and let $a\in S$.
Then the following are equivalent:
\beginEnumStatements
\item
The element $a$ is soft.
\item
The set of compactly-contained elements, $a^{\ll}$, is a soft interval.
\item
For every $b\in S$ satisfying $b\ll a$, we have $\hat{b}<_s\hat{a}$.
\end{enumerate}
\end{prp}
\begin{proof}
To see that \enumStatement{1} implies \enumStatement{2}, assume $a$ is soft, and let $x\in a^{\ll}$.
Choose $\tilde{x}\in S$ satisfying $x\ll \tilde{x}\ll a$.
Since $a$ is soft, we can find $n\in\N$ such that $(n+1)\tilde{x}\leq na$.
Then
\[
(n+1)x \ll (n+1)\tilde{x}\leq na.
\]
It follows that we can find $y\in S$ with $y\ll a$ and $(n+1)x\leq ny$.
Thus, $x<_s y$ and $y\in a^{\ll}$, which shows that $a^{\ll}$ is a soft interval.

It is easy to see that \enumStatement{2} implies \enumStatement{3}.
In order to show that \enumStatement{3} implies \enumStatement{1}, let $a'\in S$ satisfy $a'\ll a$.
Choose $b\in S$ with $a'\ll b\ll a$.
By assumption, we have $\hat{b}<_s\hat{a}$.
By \autoref{prp:comparisonRelations}, this implies $b<_s^\ctsRel a$.
Since $a'\ll b$, we get $a'<_sa$, as desired.
\end{proof}

%------------------------------------------------------------------------------------------
For the next definition, recall from \autoref{pgr:ideals} that given an ideal $I$ in a $\CatCu$-semigroup $S$, we denote the image of an element $a\in S$ in the quotient $S/I$ by $a_I$.
To notion of `pure noncompactness' was introduced in the Definition before~6.4 in \cite{EllRobSan11Cone}.
We will recall their definition and also generalize it to the concept of `weak pure noncompactness', which is more closely connected to softness;
see \autoref{prp:soft_wpnc}.

%==========================================================================================
\begin{dfn}
\label{dfn:pnc}
\index{terms}{purely noncompact}
\index{terms}{element!purely noncompact}
\index{terms}{weakly purely noncompact}
\index{terms}{element!weakly purely noncompact}
Let $S$ be a $\CatCu$-semigroup.
An element $a\in S$ is \emph{purely noncompact} if for every ideal $I$ in $S$, we have that $a_I\ll a_I$ implies $2a_I=a_I$.

An element $a\in S$ is \emph{weakly purely noncompact} if for every ideal $I$ in $S$, we have that $a_I\ll a_I$ implies $(k+1)a_I=ka_I$ for some $k\in\N$.
\end{dfn}

%------------------------------------------------------------------------------------------
Thus, if $a$ is a (weakly) purely noncompact element, and if $I$ is an ideal such that $a_I$ is compact, then either $a_I=0$ or (a multiple of) $a_I$ is properly infinite.

The following result clarifies the connection between softness and weak pure noncompactness.
In the context of Cuntz semigroups of \ca{s}, the following result has partially been obtained in \cite[Proposition~6.4]{EllRobSan11Cone}.

%==========================================================================================
\begin{prp}
\label{prp:soft_wpnc}
Let $S$ be a $\CatCu$-semigroup, and let $a\in S$.
Consider the following statements:
\beginEnumStatements
\item
The element $a$ is soft.
\item
The element $a$ is weakly purely noncompact.
\item
For every $a',x\in S$ satisfying $a'\ll a\leq a'+x$, there is $k\in\N$ such that $(k+1)a\leq ka'+\infty\cdot x$.
\item
For every $a',x\in S$ satisfying $a'\ll a\leq a'+x$, we have $a'<_sa'+x$.
\end{enumerate}
Then the following implications hold:
`\enumStatement{1}$\Rightarrow$\enumStatement{2}$\Rightarrow$\enumStatement{3}$\Rightarrow$\enumStatement{4}'.

If $S$ satisfies \axiomO{5}, then \enumStatement{4} implies \enumStatement{1}, and so all $4$ statements are equivalent in that case.
Moreover, if $a$ is purely noncompact, then \enumStatement{3} holds for $k=1$, and the converse holds if $S$ satisfies \axiomO{5}.
\end{prp}
\begin{proof}
In order to show that \enumStatement{1} implies \enumStatement{2}, let $I$ be an ideal of $S$ and assume $a_I\ll a_I$.
We have to show that a multiple of $a_I$ is properly infinite.
Choose a rapidly increasing sequence $(a_n)_n$ in $S$ such that $a=\sup_n a_n$.
Then, in the quotient $S/I$, we have $a_I=\sup_n (a_n)_I$.

Using that $a_I$ is compact, we can choose $n\in\N$ satisfying $a_I\leq (a_n)_I$.
Since $a_n\ll a$ and since $a$ is soft by assumption, we obtain $a_n<_s a$.
Choose $k\in\N$ such that $(k+1)a_n\leq ka$.
It follows $(k+1)a_I=ka_I$, as desired.

Next, to show that \enumStatement{2} implies \enumStatement{3}, assume that $a$ is weakly purely noncompact, and let $a',x\in S$ satisfy $a'\ll a\leq a'+x$.
Let $I:=\left\{ b\in S :  b\leq \infty\cdot x\right\}$ be the ideal of $S$ generated by $x$.
Then, in the quotient $S/I$, we have $a_I\leq a'_I\ll a_I$, whence by assumption we can find $k\in\N$ with $ka_I=(k+1)a_I=ka'_I$.
This implies $(k+1)a\leq ka'+\infty\cdot x$, as desired.

To show that \enumStatement{3} implies \enumStatement{4}, let $a',x\in S$ satisfy $a'\ll a\leq a'+x$.
By assumption, we can find $k\in\N$ such that $(k+1)a\leq ka'+\infty\cdot x$.
Then
\[
(k+1)a'\ll (k+1)a\leq ka'+\infty\cdot x = \sup_{n\in\N} (ka'+nx).
\]
It follows that we can choose $n\in\N$ with $(k+1)a'\leq ka'+nx$.
Let $m\in\N$ be the maximum of $k$ and $n$.
We get
\[
(k+1)a'\leq ka'+ mx.
\]
Adding $(m-k)a'$ on both sides, we obtain
\[
(m+1)a'\leq ma'+ mx = m(a'+x),
\]
and hence $a'<_sa'+x$, as desired.

Finally, let us show that \enumStatement{4} implies \enumStatement{1} under the assumption that $S$ satisfies \axiomO{5}.
By \autoref{prp:softElementTFAE}, it is enough to show $\hat{b}<_s\hat{a}$ for every $b\in S$ satisfying $b\ll a$.
Let such $b$ be given.
Choose $c\in S$ with $b\ll c\ll a$.
Since $S$ satisfies \axiomO{5}, we can find $x\in S$ such that
\[
b+x \leq a \leq c+x.
\]
By assumption, we get $c<_sc+x$.
This means that we can find $k\in\N$ satisfying
\begin{align}
\label{prp:soft_wpnc:eq1}
(k+1)c\leq kc+kx.
\end{align}
In order to show $(k+1)\hat{b}\leq k\hat{a}$, let $\lambda\in F(S)$.
If $\lambda(a)=\infty$, then there is nothing to show.
Thus, we may assume $\lambda(a)<\infty$.
Since $c\leq a$, it follows $\lambda(c)<\infty$.
Applying $\lambda$ to the inequality \eqref{prp:soft_wpnc:eq1}, we obtain
\[
(k+1)\lambda(c)\leq k\lambda(c)+k\lambda(x).
\]
Since $\lambda(c)<\infty$, we may cancel $k$ summands of $\lambda(c)$ on both sides to get
\[
\lambda(c)\leq k\lambda(x).
\]
Then, using $b\leq c$ at the first step, and using $b+x\leq a$ at the last step, we deduce
\[
(k+1)\lambda(b)
\leq k\lambda(b)+\lambda(c)
\leq k\lambda(b)+k\lambda(x)
\leq k\lambda(a).
\]
This shows $\hat{b}<_s\hat{a}$, as desired.

The implications concerning a purely noncompact element are obtained analogously.
\end{proof}

%==========================================================================================
\begin{pgr}
\label{pgr:RQQ}
\index{terms}{property@property!$(RQQ)$}
\index{terms}{property@property!$(QQ)$}
\index{symbols}{S$_\wpnc$@$S_\wpnc$ \quad (weakly purely noncompact elements)}
\index{symbols}{S$_\pnc$@$S_\pnc$ \quad (purely noncompact elements)}
Let $S$ be a $\CatCu$-semigroup.
Let us denote the subsets of (weakly) purely noncompact elements in $S$ by $S_\wpnc$ and $S_\pnc$.
We clearly have $S_\pnc\subset S_\wpnc$, but the converse might fail.
Indeed, if $S$ is the elementary semigroup $\elmtrySgp{k}=\left\{ 0,1,2,\ldots,k,\infty \right\}$ as considered in \autoref{pgr:elementarySemigr}, then $S_\wpnc=S$ and $S_\pnc=\left\{ 0,\infty \right\}$.

Let us say that $S$ satisfies condition $(RQQ)$ if in every quotient of $S$, an element is properly infinite whenever a multiple of it is properly infinite.
This property is a residual version (meaning to hold in all quotients) of property $(QQ)$ as introduced in \cite[Remark~2.15]{OrtPerRor12CoronaStability}, where it is also shown that $(QQ)$ is connected to the (strong) Corona factorization property.

It is easy to see that a $\CatCu$-semigroup $S$ with $(RQQ)$ satisfies $S_\pnc=S_\wpnc$.
\end{pgr}

%==========================================================================================
\begin{lma}
\label{prp:RQQ}
Let $S$ be an almost unperforated $\CatCu$-semigroup.
Then $S$ satisfies $(RQQ)$.
\end{lma}
\begin{proof}
Let $I$ be an ideal of $S$, and let $a\in S$ such that a multiple of $a_I$ is properly infinite.
Choose $k\in\N$ with $(k+1)a_I=ka_I$.
We need to show $2a_I=a_I$.

Since $S$ is almost unperforated, it is straightforward to check that the quotient $S/I$ is also almost unperforated.
It follows from $(k+1)a_I=ka_I$ that $(k+n)a_I=ka_I$ for every $n\in\N$.
In particular, we have $(k+1)2a_I=ka_I$.
By almost unperforation, it follows $2a_I\leq a_I$.
The converse inequality always holds, which shows $2a_I=a_I$, as desired.
\end{proof}

%------------------------------------------------------------------------------------------
In the next results, we will say that a $\CatCu$-semigroup $S$ is \emph{residually stably finite} \index{terms}{Cu-semigroup@$\CatCu$-semigroup!residually stably finite} if for every ideal $I$ in $S$, the quotient $\CatCu$-semigroup $S/I$ is stably finite.
This is in accordance with the terminology used in \ca{} theory;
see for example \cite[Definition~V.2.1.3]{Bla06OpAlgs}.

%==========================================================================================
\begin{lma}
\label{prp:softRSF}
Let $S$ be a residually stably finite $\CatCu$-semigroup satisfying \axiomO{5}, and let $a\in S$.
Then the following statements are equivalent:
\beginEnumStatements
\item
The element $a$ is soft.
\item
For every $a'\in S$ satisfying $a'\ll a$, there exists $x\in S$ such that $a'+x\leq a$ and $a\leq\infty\cdot x$.
\item
For every $a',x\in S$ satisfying $a'\ll a\leq a'+x$, we have $a\leq \infty\cdot x$.
\end{enumerate}
\end{lma}
\begin{proof}
The proof is similar to that of \autoref{prp:soft_wpnc}.
First, in order to show that \enumStatement{1} implies \enumStatement{2}, let $a'\in S$ satisfy $a'\ll a$.
Choose $b\in S$ with $a'\ll b\ll a$.
By \axiomO{5} in $S$, we can choose $x\in S$ such that $a'+x\leq a\leq b+x$.
Consider the ideal $J$ of $S$ generated by $x$.
Then $a_J$ is compact.

By \autoref{prp:soft_wpnc}, the element $a$ is weakly purely noncompact.
Thus, a multiple of $a_J$ is properly infinite.
Since $S/J$ is stably finite, this implies that $a_J$ is zero.
Therefore $a\leq \infty\cdot x$, as desired.

Conversely, to prove that \enumStatement{2} implies \enumStatement{1}, let $a'\in S$ satisfy $a'\ll a$.
We have to show $a'<_sa$.
By assumption, we can find $x\in S$ such that
\[
a'+x\leq a\leq\infty\cdot x.
\]
Since $a'\ll a$, we can choose $n\in\N$ satisfying $a'\leq nx$.
Then
\[
(n+1)a'
\leq na'+nx
\leq na,
\]
which shows $a'<_sa$, as desired.

Next, to show that \enumStatement{1} implies \enumStatement{3}, let $a',x\in S$ satisfy $a'\ll a\leq a'+x$.
Consider the ideal $J$ of $S$ generated by $x$.
Then $a_J$ is compact.
As in the first part of the proof, we obtain $a_J=0$ and hence $a\leq \infty\cdot x$, as desired.

Finally, statement \enumStatement{3} is stronger than statement \enumStatement{3} of \autoref{prp:soft_wpnc}, which shows that it implies that $a$ is soft.
\end{proof}

%==========================================================================================
\begin{rmk}
\label{rmk:softRSF}
Let $S$ be a $\CatCu$-semigroup satisfying \axiomO{5}, and let $a\in S$.
Consider the statements \enumStatement{2} and \enumStatement{3} from \autoref{prp:softRSF}.
Even if $S$ is not necessarily residually stably finite, then these imply that $a$ is soft.

Conversely, if $S$ is not residually stably finite, and if $a$ is a soft element in $S$, then \enumStatement{2} and \enumStatement{3} of \autoref{prp:softRSF} might fail.
Consider for example the elementary $\CatCu$-semigroup $\elmtrySgp{1}=\left\{ 0,1,\infty \right\}$.
The element $1$ in $\elmtrySgp{1}$ is soft and compact.
Then \enumStatement{3} fails by considering $a'=a=1$ and $x=0$.

To show that \enumStatement{2} fails, consider $a'=a=1$.
If $x$ is an element in $\elmtrySgp{1}$ such that $1+x\leq 1$, then $x=0$.
But then $1\nleq\infty\cdot x$.
We thank H.~Petzka for pointing out this example.
\end{rmk}

%==========================================================================================
\begin{cor}
\label{prp:softTFAE}
Let $S$ be a $\CatCu$-semigroup satisfying \axiomO{5}, and let $a\in S$.
Assume that $S$ satisfies $(RQQ)$ (for example, $S$ is almost unperforated) or that $S$ is residually stably finite.
Then the following statements are equivalent:
\beginEnumStatements
\item
The element $a$ is soft.
\item
The element $a$ is weakly purely noncompact.
\item
The element $a$ is purely noncompact.
\end{enumerate}
\end{cor}
\begin{proof}
Since $S$ satisfies axiom \axiomO{5}, statements \enumStatement{1} and \enumStatement{2} are equivalent by \autoref{prp:soft_wpnc}.
If $S$ satisfies $(RQQ)$, then statements \enumStatement{2} and \enumStatement{3} are equivalent, as observed in \autoref{pgr:RQQ}.
Finally, if $S$ is residually stably finite, it follows easily from \autoref{prp:softRSF} that \enumStatement{2} and \enumStatement{3} are equivalent.
\end{proof}

%==========================================================================================
\begin{thm}
\label{prp:softAbsorbing}
Let $S$ be a $\CatCu$-semigroup.
Then:
\beginEnumStatements
\item
The set $S_\soft$ of soft elements of $S$ is a subsemigroup of $S$ that is closed under passing to suprema of increasing sequences.
\item
The set $S_\soft$ is absorbing in the sense that for any $a,b\in S$ with $b\varpropto^\ctsRel a$, we have that $a+b$ is soft whenever $a$ is.
\end{enumerate}
\end{thm}
\begin{proof}
To prove \enumStatement{1}, let us first show that $S_\soft$ is closed under addition.
Let $a,b\in S_\soft$, and let $x\in S$ such that $x\ll a+b$.
We need to show $x<_sa+b$.

Choose $a',b'\in S$ such that
\[
x\leq a'+b',\quad
a'\ll a,\quad
b'\ll b.
\]
Since $a$ and $b$ are soft, it follows $a'<_sa$ and $b'<_sb$.
By \autoref{prp:relSTFAE}, we can find $k_0,l_0\in\N$ satisfying $(k+1)a'\leq ka$ for all $k\geq k_0$ and such that $(l+1)b'\leq lb$ for all $l\geq l_0$.
Let $n\in\N$ be the maximum of $k_0$ and $l_0$.
Then
\[
(n+1)x
\leq (n+1)(a'+b')
\leq n(a+b),
\]
which shows $x<_sa+b$, as desired.

Next, let us show that $S_\soft$ is closed under suprema of increasing sequences.
Let $(a_n)_{n\in\N}$ be an increasing sequence in $S_\soft$ and set $a:=\sup_n a_n$.
Let $x\in S$ satisfy $x\ll a$.
We need to show $x<_sa$.
Choose $\tilde{x}\in S$ such that $x\ll\tilde{x}\ll a$.

By definition of the way-below relation, we can choose $n\in\N$ such that $\tilde{x}\leq a_n$.
Then $x\ll a_n$.
Since $a_n$ is soft, it follows $x<_sa_n$, and therefore $x<_sa$, as desired.

To prove \enumStatement{2}, let $a\in S_\soft$ and $b\in S$ satisfy $b\varpropto^\ctsRel a$.
To show that $a+b$ is soft, let $x\in S$ such that $x\ll a+b$.
We need to show $x<_sa+b$.
Choose elements $a',b'\in S$ with
\[
x\leq a'+b',\quad
a'\ll a,\quad
b'\ll b.
\]
Since $b\varpropto^\ctsRel a$, we can find $n\in\N$ such that $b'\leq na$.
Moreover, since $a$ is soft, we get $a'<_sa$.
By \autoref{prp:relSTFAE}, we can choose $k\in\N$ with $(k+n+1)a'\leq ka$.
Then
\begin{align*}
(k+n+1)x
&\leq (k+n+1)(a'+b') \\
&= (k+n+1)a'+(k+n)b'+b' \\
&\leq ka+(k+n)b+na
= (k+n)(a+b),
\end{align*}
which shows $x<_sa+b$, as desired.
\end{proof}

%==========================================================================================
\begin{thm}
\label{prp:soft_comparison}
Let $S$ be an almost unperforated $\CatCu$-semigroup, and let $a,b\in S$.
Assume $a$ is soft.
Then $a\leq b$ if and only if $\hat{a}\leq\hat{b}$.
\end{thm}
\begin{proof}
It is clear that $a\leq b$ implies $\hat{a}\leq\hat{b}$.
In order to show the converse implication, assume $\hat{a}\leq\hat{b}$.
It is enough to verify $x\leq b$ for every $x\in S$ satisfying $x\ll a$.
Let $x\in S$ such that $x\ll a$.

Since $a$ is soft, we get $x<_sa$.
Then $\hat{x}<_s\hat{a}$, and together with the assumption we obtain $\hat{x}<_s\hat{b}$.
By \autoref{prp:comparisonRelations}, it follows $x<_s^\ctsRel b$.
Since $S$ is almost unperforated, this implies $x\leq b$, as desired.
\end{proof}

%==========================================================================================
\begin{rmk}
\label{rmk:soft_comparison}
Theorems \ref{prp:softAbsorbing} and \ref{prp:soft_comparison} are inspired by Proposition~6.4 and Theorem~6.6 in \cite{EllRobSan11Cone}.
Their results are concerned with purely noncompact elements in Cuntz semigroups of \ca{s}, and their proofs use \ca ic methods.

We generalize the mentioned results in \cite{EllRobSan11Cone} in two ways.
First, we consider abstract $\CatCu$-semigroups instead of concrete Cuntz semigroups coming from \ca{s}.
Therefore, our proofs are necessarily purely algebraic.

Second, we do not assume \axiomO{5}, which is implicitly used to prove the results in \cite{EllRobSan11Cone}.
Note that axiom \axiomO{5} automatically holds for Cuntz semigroups of \ca{s} (see \autoref{prp:o5o6}), whence it is not an unreasonable assumption.
We were able to obtain our results without using \axiomO{5}, by considering soft elements instead of (weakly) purely noncompact elements.

It seems that soft elements form the right class to prove desirable results like Theorems \ref{prp:softAbsorbing} and \ref{prp:soft_comparison}.
In the absence of \axiomO{5}, it is unclear whether the same results hold for the class of (weakly) purely noncompact elements.
Moreover, as shown in \autoref{prp:soft_wpnc}, under the assumption of \axiomO{5} the class of soft and weakly purely noncompact elements coincide, so that then the results for (weakly) purely noncompact elements follow from that for soft elements.
\end{rmk}

%==========================================================================================
\begin{prbl}
\label{prbl:pncInCu}
Given a $\CatCu$-semigroup $S$, is the subsemigroup $S_\soft$ of soft elements again a $\CatCu$-semigroup?
Does this hold under the additional assumption that $S$ satisfies \axiomO{5}?
If so, does then $S_\soft$ satisfy \axiomO{5} as well?
\end{prbl}

%==========================================================================================
\begin{rmk}
By \autoref{prp:softAbsorbing}, $S_\soft$ is a subsemigroup of $S$.
It therefore inherits a natural structure as a \pom.
Moreover, axiom \axiomO{1} is satisfied.
It is not clear, whether axiom \axiomO{2} holds.
(If so, axioms \axiomO{3} and \axiomO{4} should follow immediately.)

The answer to \autoref{prbl:pncInCu} is not even clear for Cuntz semigroups of \ca{s}.
In \autoref{prp:softInCuZmult}, we will provide a positive answer for semigroups with $Z$-multiplication, which includes in particular the Cuntz semigroups of $\mathcal{Z}$-stable \ca{s}.

We end this section by showing that \autoref{prbl:pncInCu} also has a positive answer for simple, stably finite $\CatCu$-semigroups satisfying \axiomO{5} and \axiomO{6}.
We first observe that, in this case, every noncompact element is automatically soft.
This should be compared to \cite[Lemma~3.4]{Per01IdealsMultiplier}.
\end{rmk}

%==========================================================================================
\begin{prp}
\label{prp:softNoncpctSimple}
Let $S$ be a simple, stably finite $\CatCu$-semigroup satisfying \axiomO{5}.
Then a nonzero element in $S$ is soft if and only if it is not compact.
\end{prp}
\begin{proof}
In general, a finite compact element is never soft.
To show the converse, let $a\in S$ be nonzero and noncompact.
We need to show that $a$ is soft.
By \autoref{prp:softRSF} and \autoref{rmk:softRSF}, it is enough to show that for every $a',x\in S$ satisfying $a'\ll a\leq a'+x$, we have $a\leq\infty\cdot x$.
Given such $a'$ and $x$, since $a'\ll a$ and $a$ is not compact, we get $a'\neq a$.
Therefore, $x$ is nonzero.
Since $S$ is simple, this implies $a\leq\infty\cdot x$, as desired.
\end{proof}

%==========================================================================================
\begin{lma}
\label{lem:density}
Let $S$ be a $\CatCu$-semigroup, and let $B\subset S$ be a submonoid.
Assume that for every $b\in B$, there exists a sequence $(b_k)_{k\in\N}$ in $B$ such that $b=\sup_k b_k$ and such that $b_k\ll b_{k+1}$ in $S$ for each $k$.
Then
\[
\overline{B} = \left\{ \sup_k b_k : (b_k)_{k\in\N} \text{ increasing sequence in } B \right\}
\]
is a submonoid of $S$ that is closed under passing to suprema of increasing sequences.
Moreover, $\overline{B}$ is a $\CatCu$-semigroup such that for any $a,b\in\overline{B}$ we have $a\ll b$ in $\overline{B}$ if and only if $a\ll b$ in $S$.
\end{lma}
\begin{proof}
We view $\overline{B}$ is a subset of $S$.
It is easy to see that $B$ is a subset of $\overline{B}$ and that $\overline{B}$ is a submonoid of $S$.
Thus, endowed with the partial order induced by $S$, we have that $\overline{B}$ is a \pom.
Given an increasing sequence $(c_n)_n$ in $\overline{B}$, let us show that the supremum $\sup_n c_n$ is an element of $\overline{B}$.
For each $n\in\N$ we can, by assumption, choose a sequence $(c_{n,k})_k$ in $B$ that is rapidly increasing in $S$ and such that $c_n=\sup_k c_{n,k}$.
As in the proof of \autoref{prp:CuificationExists}, we can inductively choose indices $k_n$ for $n\in\N$ such that
\[
c_{1,k_1+n-1}, c_{2,k_2+n-2},\ldots,c_{n,k_n} \leq c_{n+1,k_{n+1}}.
\]
Then $(c_{n,k_n})_n$ is an increasing sequence of elements in $B$ such that $\sup_n c_n = \sup_n c_{n,k_n}$.
By definition, the element $\sup_n c_{n,k_n}$ belongs to $\overline{B}$.
Thus, $\overline{B}$ is closed under passing to suprema of increasing sequences.
Then, axioms \axiomO{1} and \axiomO{4} for $\overline{B}$ follow easily from their counterparts in $S$.

For clarity, let us denote the compact-containment relation with respect to $S$ by $\ll_S$, and similarly for $\ll_{\overline{B}}$.
Given $a,b\in\overline{B}$ satisfying $a\ll_S b$, let us prove $a\ll_{\overline{B}} b$.
Since $\overline{B}$ satisfies \axiomO{1}, we need to show that for every increasing sequence $(b_k)_k$ in $\overline{B}$ satisfying $b\leq \sup_k b_k$ there exists $n\in\N$ such that $a\leq b_n$.
Since the sequence $(b_k)_k$ has the same supremum in $S$ as in $\overline{B}$, this follows directly from the assumption that $a\ll_S b$.

It follows that every element in $B$ is the supremum of a sequence of elements in $B$ that is rapidly increasing in $\overline{B}$.
A diagonalization argument shows that the same holds for every element in $\overline{B}$.
This verifies \axiomO{2} for $\overline{B}$.

Finally, given $a,b\in\overline{B}$ satisfying $a\ll_{\overline{B}} b$, let us show $a\ll_S b$.
Choose a sequence $(b_k)_k$ in $B$ with $b=\sup_k b_k$ and such that $b_k\ll_Sb_{k+1}$ for each $k$.
Since the supremum in $S$ and in $\overline{B}$ agree and since $a\ll_{\overline{B}}b$, we obtain $a\leq b_k$ for some $k$.
Then $a\leq b_k\ll_Sb_{k+1}\leq b$, so that $a\ll_Sb$.
Therefore, axiom \axiomO{3} for $\overline{B}$ follows since $S$ satisfies \axiomO{3}.
\end{proof}

%==========================================================================================
\begin{prp}
\label{prp:softSimple}
Let $S$ be a simple, stably finite $\CatCu$-semigroup satisfying \axiomO{5} and \axiomO{6}.
Then the subsemigroup $S_\soft$ is a $\CatCu$-semigroup satisfying \axiomO{5} and \axiomO{6}.
\end{prp}
\begin{proof}
Since the statement clearly holds if $S$ is elementary, we may assume from now on that $S$ is nonelementary.
We want to apply \autoref{lem:density} with $B=S_\soft$.
By \autoref{prp:softAbsorbing}, we have that $S_\soft$ is closed under passing to suprema of increasing sequences.
This implies that $S_\soft$ satisfies \axiomO{1} and that $\overline{S_\soft}=S_\soft$.

Claim 1:
For every nonzero $a\in S$ there exists a nonzero $b\in S_\soft$ satisfying $b\leq a$.
To prove this claim, we inductively construct nonzero elements $a_n\in S$ for $n\in\N$ such that $2a_{n+1}\leq a_n$ for each $n$.
We start by setting $a_0:=a$.
Assuming that we have constructed $a_k$ for all $k\leq n$, we apply \cite[Proposition~5.2.1]{Rob13Cone} (see \autoref{prp:GlimmHalving}), to obtain a nonzero element $a_{n+1}\in S$ such that $2a_{n+1}\leq a_n$.
Then set
\[
b :=\sum_{n\in\N} a_n
=\sup_{n\in\N} \sum_{k=0}^na_k.
\]
We have $b\leq a$.
Since $S$ is stably finite, the element $b$ cannot be compact.
Therefore, by \autoref{prp:softNoncpctSimple}, it is a soft element.
This proves the claim.

Claim 2:
For every $a\in S$ and $b\in S_\soft$ satisfying $a\ll b$, there exists $s\in S_\soft$ with $a+s\ll b$.
Note that by \autoref{prp:softAbsorbing}\enumStatement{2} this implies that $a+s$ is soft.
To prove this claim, we first choose $b'\in S$ such that $a\ll b'\ll b$.
By \autoref{prp:softRSF}, we can find $x\in S$ with $b'+x\leq b$ and $b\leq\infty\cdot x$.
In particular, $x$ is nonzero.
Choose $x'\in S$ nonzero such that $x'\ll x$.
By claim~1, we can find $s\in S_\soft$ nonzero with $s\leq x'$.
This implies $s\ll x$.
Moreover, we get
\[
a+s\ll b'+x\leq b,
\]
which proves the claim.

Claim 3:
For every $a\in S_\soft$, there exists a sequence $(a_k)_k$ in $S_\soft$ that is rapidly increasing in $S$ and such that $a=\sup_k a_k$.
Since $S$ satisfies \axiomO{2}, it is enough to show that for every $a'\in S$ satisfying $a'\ll a$, there exists $b\in S_\soft$ such that $a'\leq b\ll a$.
This follows directly from claim 2 and \autoref{prp:softAbsorbing}(2).

We can now apply \autoref{lem:density} to deduce that $S_\soft$ is a $\CatCu$-semigroup.
To verify \axiomO{5} for $S_\soft$, let $a',a,b',b,c\in S_\soft$ satisfy
\[
a+b \leq c,\quad
a'\ll a,\quad
b'\ll b.
\]
Applying claim~2 for $a'\ll a$, choose $s\in S_\soft$ with $a'+s\ll a$.
Then, using \axiomO{5} in $S$, we obtain an element $x\in S$ such that
\[
(a'+s)+x \leq c \leq a+x,\quad
b' \leq x.
\]
Set $d:=s+x$, which is soft by \autoref{prp:softAbsorbing}\enumStatement{2}.
Then
\[
a'+d
=a'+s+x
\leq c
\leq a+x
\leq a+d,\quad
b'\leq x\leq d,
\]
which show that $d$ has the desired properties to verify \axiomO{5} for $S_\soft$.

Finally, to verify \axiomO{6} for $S_\soft$, let $a',a,b,c\in S_\soft$ satisfy
\[
a'\ll a\leq b+c.
\]
Without loss of generality, we may assume that the elements $a',a,b$ and $c$ are nonzero.
Using \axiomO{6} for $S$, choose $e,f\in S$ such that
\[
a'\leq e+f,\quad
e\leq a,b,\quad
f\leq a,c.
\]
If $e$ and $f$ are soft, then these elements have the desired properties to verify \axiomO{6} for $S_\soft$.
Let us assume that $e$ is not soft.
By \autoref{prp:softNoncpctSimple}, this implies that $e$ is compact.
Using \axiomO{5} for $S$, this implies that we can choose elements $x_1,x_2\in S$ such that
\[
e+x_1 = a,\quad
e+x_2 = b.
\]
Since $a$ and $b$ are not compact, $x_1$ and $x_2$ are nonzero elements.
By \autoref{prp:downwards}, we can find $\tilde{x}\in S$ nonzero satisfying $\tilde{x}\leq x_1, x_2$.
Then, by claim~1, we can find $x\in S_\soft$ nonzero such that $x\leq \tilde{x}$ and hence $x\leq x_1,x_2$.
By \autoref{prp:softAbsorbing}, the element $e+x$ is soft.
Moreover, we get
\[
e+x \leq e+x_1 = a,\quad
e+x \leq e+x_2 = b.
\]
An analogous argument works in the case that $f$ is not soft.
\end{proof}

\vspace{5pt}
%------------------------------------------------------------------------------------------
%==========================================================================================
\section{Predecessors, after Engbers}
\label{sec:predecessors}

%------------------------------------------------------------------------------------------
In \cite{Eng14PrePhD}, Engbers develops a theory of \emph{predecessors} of compact elements in Cuntz semigroups of simple, stably finite \ca{s}.
His work is based on results of a problem session at a workshop on Cuntz semigroups held at the American Institute of Mathematics (AIM) in 2009.
Engbers attributes the problem to J.~Cuntz and mentions that several participants contributed to the solution, notably N.~C.~Phillips.
Using algebraic methods, we obtain a weaker version of his results;
see \autoref{thm:predecessors}.
First, we recall the following Glimm-halving result of Robert:

%==========================================================================================
\begin{prp}[{Robert, \cite[Proposition~5.2.1]{Rob13Cone}}]
\label{prp:GlimmHalving}
Let $S$ be a simple, nonelemen\-tary $\CatCu$-semigroup satisfying \axiomO{5} and \axiomO{6}.
Then, for every nonzero $a\in S$, there exists $b\in S$ nozero with $2b\leq a$.
\end{prp}

%==========================================================================================
\begin{prp}
\label{prp:smallerElement}
Let $S$ be a simple, countably-based $\CatCu$-semigroup satisfying \axiomO{5} and \axiomO{6}.
Then there exists a sequence $(g_n)_n$ in $S$ of nonzero elements with the following properties:
\beginEnumStatements
\item
The sequence is rapidly decreasing, that is, $g_n\gg g_{n+1}$ for each $n$.
\item
The sequence is cofinal among all nonzero elements, that is, for every nonzero $a\in S$, there exists $n\in\N$ such that $g_n\leq a$.
\end{enumerate}
\end{prp}
\begin{proof}
Since $S$ is countably-based, we can choose a countable set of nonzero elements $\{a_n\}_{n\in\N} \subset S$ that is dense in the sense of \autoref{pgr:axiomsW}.
We inductively construct the sequence $(g_n)_n$ such that for each $n$ we have
\[
g_{n+1} \ll g_n,a_0,a_1,\dots,a_{n+1}.
\]
We start by letting $g_0\in S$ be any nonzero element satisfying $g_0\ll a_0$.
Assume we have constructed $g_k$ for all $k\leq n$.
By \autoref{prp:downwards}, we can find $g_{n+1}$ with the desired properties.
By construction, the sequence $(g_n)_n$ is rapidly decreasing.
Finally, let $a\in S$ be nonzero.
Since $\{a_n\}_n$ is a basis, we can find $n$ with $a_n\leq a$.
It follows $g_n\leq a$, as desired.
\end{proof}

%------------------------------------------------------------------------------------------
In \cite{Eng14PrePhD}, Engbers introduced the notion of a \emph{predecessor} \index{terms}{predecessor} of a compact element $p$ in a simple $\CatCu$-semigroup $S$.
It is defined as
\[
\gamma(p) =\max \left\{ x\in S : x<p \right\},
\]
provided this maximum exists.
Engbers shows the existence of predecessors for Cuntz semigroups of separable, simple and stably finite \ca{s}, and for these semigroups he proves the following properties:
\beginEnumStatements
\item
For every nonzero $z\in S$, we have $p\leq \gamma(p)+z$.
\item
For every noncompact $y\in S$, we have $\gamma(p)+y=p+y$.
\item
For every $\lambda\in F(S)$, we have $\lambda(\gamma(p))=\lambda(p)$.
\end{enumerate}

Using the axioms \axiomO{5} and \axiomO{6}, we can almost recover this result in the algebraic setting, by showing that elements in $\left\{ x\in S : x<p \right\}$ with these properties do exist.
However, we only get existence of the maximum (and thus uniqueness of predecessors) in the presence of weak
cancellation or almost unperforation.

Since for a compact element $p$, the induced map $\hat{p}$ in $\Lsc(F(S))$ is continuous, the result gives us a \emph{noncompact} element with the same property.

%==========================================================================================
\begin{lma}
\label{lma:predecessors}
Let $S$ be a simple, countably-based, nonelementary, stably finite $\CatCu$-semigroup satisfying \axiomO{5} and \axiomO{6}.
Then, for every nonzero compact $p\in S$, there exists a noncompact $c\in S$ with $c<p$ and such that $p\leq c+z$ for every nonzero $z\in S$.
\end{lma}
\begin{proof}
Using \autoref{prp:smallerElement}, we can choose a rapidly decreasing sequence $(g_n)_n$ in $S$ that is cofinal among the nonzero elements of $S$.
By reindexing, if necessary, we may assume $g_0 \leq p$.
We will inductively construct $c_n,x_n,\tilde{c}_n\in S$ satisfying
\[
c_n\ll x_n\ll \tilde{c}_n,\quad
c_n\leq c_{n+1},\quad
x_n\leq \tilde{c}_{n+1},\quad
g_{n+1}+\tilde{c}_n\leq p \leq g_{n}+c_n,
\]
for each $n\in\N$.
First, we have
\[
g_1 \ll g_0 \leq p.
\]
Applying axiom \axiomO{5}, we find $\tilde{c}_0$ such that
\[
g_1+\tilde{c}_0\leq p \leq g_0+\tilde{c}_0.
\]
Since $p$ is compact, we can choose $c_0\in S$ with $c_0\ll \tilde{c}_0$ and $p \leq g_0+c_0$.
Find $x_0\in S$ satisfying $c_0\ll x_0\ll \tilde{c}_0$.

For the induction, assume that we have constructed $c_k,x_k$ and $\tilde{c}_k$ for $k\leq n$.
Thus, we have
\[
g_{n+1}+\tilde{c}_n\leq p,\quad
g_{n+2}\ll g_{n+1},\quad
x_n\ll\tilde{c}_n.
\]
Then, by applying axiom \axiomO{5}, we can choose $\tilde{c}_{n+1}\in S$ such that
\[
g_{n+2}+\tilde{c}_{n+1}\leq p \leq g_{n+1}+\tilde{c}_{n+1},\quad
x_n\leq \tilde{c}_{n+1}.
\]
Then $c_n\ll \tilde{c}_{n+1}$.
Using also that $p$ is a compact element, we can find $c_{n+1}\in S$ with $c_n\ll c_{n+1}\ll\tilde{c}_{n+1}$ and $p \leq g_{n+1}+c_{n+1}$.
Choose $x_{n+1}\in S$ such that $c_{n+1}\ll x_{n+1}\ll \tilde{c}_{n+1}$.

Note that the sequence $(c_n)_n$ is increasing.
Therefore, we may set
\[
c:=\sup_n c_n.
\]
Let us show that $c$ has the desired properties.

We first observe that $c<p$.
Indeed, it is clear that $c\leq p$.
To obtain a contradiction, assume $c=p$.
Then, since $p$ is compact and $c=\sup_nc_n$, we would have $p=c_n$ for some $n$.
But we have $g_{n+1}+c_n\leq p$ and $g_{n+1}$ is nonzero.
Thus, $p$ would be an infinite compact element, which is not possible since $S$ is stably finite.

Next, let $z\in S$ be nonzero.
Choose $n\in\N$ such that $z\geq g_n$.
It follows
\[
c+z\geq c+g_n\geq c_n+g_n\geq p,
\]
as desired.

Finally, let us show that $c$ is not compact.
Indeed, if $c$ were compact, then by \axiomO{5} there could find $y\in S$ satisfying $c+y=p$.
Since $c<p$, the element $y$ is nonzero.
By \autoref{prp:GlimmHalving}, we can choose $z\in S$ nonzero such that $2z\leq y$.
As shown above, this implies $p\leq c+z$.
But then
\[
p+z \leq c+z+z \leq p,
\]
which is impossible since $S$ is stably finite.
\end{proof}

%==========================================================================================
\begin{prp}
\label{prp:predecessors}
Let $S$ be a simple, nonelementary, stably finite $\CatCu$-sem\-i\-group satisfying \axiomO{5} and \axiomO{6}.
Let $p\in S$ be compact, and let $c\in S$ be nonzero such that $c<p$.
Consider the following conditions:
\beginEnumStatements
\item
For every nonzero $z\in S$, we have $p\leq c+z$.
\item
The element $c$ is noncompact, and for every noncompact $y\in S$, we have $c+y=p+y$.
\item
For every $\lambda\in F(S)$, we have $\lambda(c)=\lambda(p)$.
\end{enumerate}
Then the following implications hold:
`\enumStatement{1}$\Leftrightarrow$\enumStatement{2}$\Rightarrow$\enumStatement{3}'.

Moreover, if $S$ has weak cancellation, then an element $c$ satisfying \enumStatement{1} or \enumStatement{2} is equal to the maximum of the set $\left\{ x : x<p \right\}$, and hence it is uniquely determined.

If $S$ is almost unperforated, then all three conditions are equivalent and the element $c$ satisfying \enumStatement{1}-\enumStatement{3} is uniquely determined.
\end{prp}
\begin{proof}
In order to show that \enumStatement{1} implies \enumStatement{2}, let $c\in S$ satisfy the statement of \enumStatement{1}.
As shown at the end of the proof of \autoref{lma:predecessors}, we have that $c$ is necessarily noncompact.

Let $y\in S$ be noncompact.
To show $c+y=p+y$, we follow an argument similar to the one in \cite[Theorem~5.7]{Eng14PrePhD}, which we include for completeness.
We clearly have $c+y\leq p+y$.
For the converse inequality, it is enough to show $p+y'\leq c+y$ for every $y'\in S$ satisfying $y'\ll y$.
Given such $y'$, choose $y''\in S$ satisfying $y'\ll y''\ll y$.
Applying \axiomO{5} in $S$, we choose $z\in S$ such that
\[
y'+z \leq y \leq y''+z.
\]
Notice that $z$ is nonzero, as otherwise $y$ would be compact.
Using the assumption at the first step, we deduce
\[
y'+p\leq y'+c+z\leq y+c,
\]
as desired.

Next, to show that \enumStatement{2} implies \enumStatement{1}, let $z\in S$ be nonzero.
We need to show $p\leq c+z$.
By assumption, this is clear if $z$ is noncompact.
We may therefore assume that $z$ is compact.
Choose a noncompact, nonzero element $y\in S$ with $y<z$.
Then
\[
p\leq p+y=c+y\leq c+z,
\]
as desired.

Next, to show that \enumStatement{2} implies \enumStatement{3}, let $\lambda\in F(S)$.
We distinguish two cases.
In the first case, we assume $\lambda(p)<\infty$.
Then $\lambda(c)<\infty$.
Applying $\lambda$ to the equality $p+c=2c$, we get
\[
\lambda(p)+\lambda(c)=\lambda(c)+\lambda(c).
\]
Then we can cancel $\lambda(c)$ on both sides and obtain $\lambda(p)=\lambda(c)$, as desired.

If the other case, we assume $\lambda(p)=\infty$.
It follows that $\lambda$ is equal to $\lambda_{\infty}$, the functional that takes value $\infty$ everywhere except at $0$.
Then $\lambda(c)=\infty=\lambda(p)$, as desired.

Suppose now that $S$ has weak cancellation, and that $c\in S$ satisfies \enumStatement{1}-\enumStatement{2}.
Assume $x\in S$ satisfies $x<p$.
We need to show $x\leq c$.
For this, it is enough to show $x'\leq c$ for every $x'\in S$ satisfying $x'\ll x$.
Let $x'\in S$ satisfy $x'\ll x$.
By \axiomO{5}, we can choose $t\in S$ such that
\[
x'+t \leq p \leq x+t.
\]
Then $t\neq 0$ as $S$ is stably finite.
Therefore, we get
\[
x'+t \leq p\ll p\leq c+t.
\]
Applying weak cancellation, we obtain $x'\leq c$, as desired.

Finally, assume that $S$ is almost unperforated.
In order to show that \enumStatement{3} implies \enumStatement{1}, let $z\in S$ be nonzero.
By assumption, we have $\hat{p}=\hat{c}$.
Since $S$ is simple and $z$ is nonzero, it is straightforward to check
\[
\hat{p}<_s \widehat{c+z}.
\]
By \autoref{prp:comparisonRelations}, we get $p<_s^\ctsRel c+z$.
Since $p$ is compact and $S$ is almost unperforated, it follows $p\leq c+z$.
Moreover, the element is uniquely determined by \autoref{prp:soft_comparison} and \autoref{prp:softNoncpctSimple}.
\end{proof}

%==========================================================================================
\begin{thm}[{Engbers, \cite[Theorem~5.15]{Eng14PrePhD}}]
\label{thm:predecessors}
Let $S$ be a countably-based, simple, nonelementary $\CatCu$-semigroup satisfying \axiomO{5}, \axiomO{6} and weak cancellation.
Then every compact $p\in S$ has a predecessor $\gamma(p)$, uniquely determined by the property that $p\leq\gamma(p)+z$ for every nonzero $z\in S$.
\end{thm}
\begin{proof}
This follows from \autoref{lma:predecessors} and \autoref{prp:predecessors}.
\end{proof}

\vspace{5pt}
%------------------------------------------------------------------------------------------
%==========================================================================================
\section{Algebraic semigroups}
\label{sec:algebraicSemigp}

%------------------------------------------------------------------------------------------
In this section, we study $\CatCu$-semigroup that have a basis of compact elements.
Such semigroups are called `algebraic'.
Important examples are given by Cuntz semigroups of \ca{s} with real rank zero.

Given a \pom\ $M$, we show how to construct an algebraic $\CatCu$-semigroup $\Cu(M)$ such that the semigroup of compact elements in $\Cu(M)$ can be naturally identified with $M$;
see \autoref{prp:algebraicSemigp}.
This establishes an equivalence between the category $\CatPom$ of \pom{s}, and the full subcategory of $\CatCu$ consisting of algebraic $\CatCu$-semigroups;
see \autoref{pgr:equivalenceAlgCu}.
In \autoref{prp:propertiesAlgebraic}, we will see how certain properties of $M$ translate to properties of $\Cu(M)$.
Then we provide a version of the Effros-Handelman-Shen theorem by showing that a $\CatCu$-semigroup is an inductive limit of simplicial $\CatCu$-semigroups if and only if it is weakly cancellative, unperforated, algebraic and satisfies \axiomO{5} and \axiomO{6};
see \autoref{prp:EffHandShen_CuVersion}.
This also characterizes the Cuntz semigroups of separable AF-algebras.

%==========================================================================================
\begin{dfn}
\label{dfn:algebraicSemigp}
\index{terms}{Cu-semigroup@$\CatCu$-semigroup!algebraic}
\index{terms}{algebraic Cu-semigroup@algebraic $\CatCu$-semigroup}

A $\CatCu$-semigroup $S$ is \emph{algebraic} if every element in $S$ is the supremum of an increasing sequence of compact elements of $S$.
\end{dfn}

%==========================================================================================
\begin{rmks}
\label{rmk:algebraicSemigp}
(1)
\autoref{dfn:algebraicSemigp} is following the convention of domain theory to call a continuous partially ordered set \emph{algebraic} if its compact elements form a basis;
see \cite[Definition~I-4.2, p.~115]{GieHof+03Domains}.

(2)
If $A$ is a \ca\ with real rank zero, then $\Cu(A)$ is algebraic.
For \ca{s} with stable rank one, the converse holds;
see \cite[Corollary~5]{CowEllIva08CuInv}.
\end{rmks}

%==========================================================================================
\begin{pgr}
\label{pgr:algebraicSemigp}
Given a \pom\ $M$, it is easy to see that the partial order $\leq$ is an auxiliary relation in the sense of \autoref{pgr:axiomsW}.
In fact, it is the strongest auxiliary relation on $M$.
Moreover, it is straightforward to check that $(M,\leq)$ is a $\CatW$-semigroup.
We denote its $\CatCu$-completion by $\Cu(M)$.
In \autoref{prp:algebraicSemigp}, we will see that $\Cu(M)$ is algebraic and that every algebraic $\CatCu$-semigroup arises this way.

Every $\CatPom$-morphism $f\colon M\to N$ between \pom{s} induces a $\CatCu$-morphism $\Cu(f)\colon\Cu(M)\to\Cu(N)$.
Thus, we obtain a functor
\[
\Cu\colon \CatPom\to\CatCu,\quad
M\mapsto\Cu(M),\quad
\txtFA M\in\CatPom.
\]

Conversely, given a $\CatCu$-semigroup $S$, we let $S_c$ denote the set of compact elements in $S$.
It is easy to see that $S_c$ is a submonoid of $S$ and we equip it with the order induced by $S$.
It follows that $S_c$ is a \pom.
Moreover, every $\CatCu$-morphism $f\colon S\to T$ between $\CatCu$-semigroups restricts to a $\CatPom$-morphism from $S_c$ to $T_c$.
Hence, we obtain a functor
\[
\CatCu\to\CatPom,\quad
S\mapsto S_c,\quad
\txtFA S\in\CatCu.
\]
\end{pgr}

%==========================================================================================
\begin{prp}
\label{prp:algebraicSemigp}
\beginEnumStatements
\item
Let $M$ be a \pom.
Then $\Cu(M)$ as introduced in \autoref{pgr:algebraicSemigp} is an algebraic $\CatCu$-semigroup.
Moreover, there is a natural identification of $M$ with the \pom\ of compact elements in $\Cu(M)$.
\item
Let $S$ be an algebraic $\CatCu$-semigroup.
Consider the \pom\ $S_c$ of compact elements in $S$.
Then there is a natural isomorphism $S\cong\Cu(S_c)$.
\end{enumerate}
\end{prp}
\begin{proof}
Let us show \enumStatement{1}.
Consider the natural map $\alpha\colon M\to\Cu(M)$ from $M$ to its $\CatCu$-completion.
Since $(M,\leq)$ is a $\CatW$-semigroup, the map $\alpha$ is an order-embedding;
see \autoref{rmk:Cuification}(2).

Given $a\in M$, we have $a\leq a$ and therefore $\alpha(a)\ll\alpha(a)$, showing that $\alpha$ maps $M$ to the compact elements of $S$.
On the other hand, let $s\in S$ be a compact element.
By \autoref{thm:Cuification}(1,ii), there exists $a\in M$ such that $s\leq \alpha(a)\leq s$, and hence $s=\alpha(a)$.
This shows that $\alpha$ is an order-embedding that maps $M$ onto $S_c$.
It also follows from \autoref{thm:Cuification}(1,ii) that every element in $S$ is the supremum of an increasing sequence of compact elements, showing that $S$ is an algebraic $\CatCu$-semigroup.

We leave the proof of \enumStatement{2} to the reader.
\end{proof}

%==========================================================================================
\begin{prp}
\label{pgr:equivalenceAlgCu}
The two functors from \autoref{pgr:algebraicSemigp}, assigning to a \pom\ $M$ its $\CatCu$-completion $\Cu(M)$, and assigning to an algebraic $\CatCu$-semigroup its \pom\ of compact elements, establish an equivalence of the following categories:
\begin{enumerate}
\item
The category $\CatPom$ of \pom{s};
see \autoref{pgr:pom}.
\item
The full subcategory of $\CatCu$ consisting of algebraic $\CatCu$-semigroups.
\end{enumerate}
\end{prp}

%==========================================================================================
\begin{rmk}
Let $M$ be a \pom.
The $\CatCu$-completion $\Cu(M)$ has appeared in the literature before using different but equivalent constructions.
First, recall that an interval in $M$ is a nonempty, upwards directed, order-hereditary subset of $M$.
In the literature, intervals are often called ideals or sometimes round ideals;
see \cite[Definition~0-1.3, p.~3]{GieHof+03Domains} and \cite[Definition~2.1]{Law97RoundIdeal}.

An interval $I$ in $M$ is \emph{countably generated} \index{terms}{interval!countably generated} if there exists a countable cofinal subset for $I$.
This is equivalent to saying that there is an increasing sequence $(a_n)_n$ in $I$ such that
\[
I= \left\{ a\in M : a\leq a_n\text{ for some }n \right\}.
\]
Countably generated intervals in $M$ form a \pom\ $\Lambda_{\sigma}(M)$, where addition of intervals $I$ and $J$ is the interval generated by $I+J$, and order is given by set inclusion;
see \cite{Weh96TensorInterpol}, and also \cite{Per97StructurePositive}.

Let us define a map $\Lambda_{\sigma}(M)\to\Cu(M)$.
Given a countably generated interval $I\in\Lambda_\sigma(M)$, let $(a_n)_n$ be a cofinal subsequence of $I$.
Considering $M$ as a submonoid of $\Cu(M)$, we may assign to $I$ the element $\sup_na_n$ in $\Cu(M)$.
This induces a natural isomorphism $\Lambda_{\sigma}(M)\cong\Cu(M)$.

Similarly, if $S$ is a $\CatCu$-semigroup, we may consider the natural map $\Phi\colon S\to \Lambda_{\sigma}(S_c)$, which sends an element $a\in S$ to the interval
\[
\Phi(a) = \left\{ x\in S_c : x\leq a \right\}.
\]
If $S$ is algebraic, then $\Phi$ is an isomorphism of \pom{s};
see \cite[Theorem~6.4]{AntBosPer11CompletionsCu}.
\end{rmk}

%------------------------------------------------------------------------------------------
We will now study how properties of a \pom\ relate to properties of its $\CatCu$-completion.
The results in \autoref{prp:propertiesAlgebraic} should be compared to \autoref{prp:axiomsPassToCu}.

%==========================================================================================
\begin{dfn}
\label{dfn:RieszProperties}
\index{terms}{Riesz refinement property}
\index{terms}{Riesz decomposition property}
\index{terms}{Riesz interpolation property}
\index{terms}{positively ordered monoid!cancellative}
Let $M$ be a \pom.
\beginEnumStatements
\item
We say that $M$ has the \emph{Riesz refinement property} if whenever there are $a_1,a_2,b_1,b_2\in M$ with $a_1+a_2=b_1+b_2$, then there exist $x_{i,j}\in M$ for $i=1,2$ such that $a_i=x_{i,1}+x_{i,2}$ for $i=1,2$ and $b_j=x_{1,j}+x_{2,j}$ for $j=1,2$.
\item
We say that $M$ has the \emph{Riesz decomposition property} if whenever there are $a,b,c\in M$ with $a\leq b+c$, then there exist $b',c'\in M$ such that $a=b'+c'$, $b'\leq b$ and $c'\leq c$.
\item
We say that $M$ has the \emph{Riesz interpolation property} if whenever there are $a_1,a_2,b_1,b_2\in S$ such that $a_i\leq b_j$ for $i,j=1,2$, then there exists $c\in S$ such that $a_1,a_2\leq c\leq b_1,b_2$.
\item
We say $M$ has \emph{cancellation} (or that $M$ is \emph{cancellative}) if for any $a,b,x\in M$, $a+x\leq b+x$ implies $a\leq b$.
\end{enumerate}
\end{dfn}

%------------------------------------------------------------------------------------------
The three Riesz properties are closely related but not equivalent in general.
If $M$ is algebraically ordered, then Riesz refinement implies Riesz decomposition.
If $M$ is cancellative and algebraically ordered, then all three Riesz properties are equivalent (see for example \cite[Proposition 2.1]{Goo86Poag})

%==========================================================================================
\begin{thm}
\label{prp:propertiesAlgebraic}
Let $M$ be a \pom.
Then:
\beginEnumStatements
\item
The monoid $M$ is algebraically ordered if and only if $\Cu(M)$ satisfies \axiomO{5}.
\item
The monoid $M$ is cancellative if and only if $\Cu(M)$ is weakly cancellative.
\item
The monoid $M$ has Riesz interpolation if and only if $\Cu(M)$ does.
\item
If $M$ satisfies the Riesz decomposition property, then $\Cu(M)$ satisfies \axiomO{6}.
Conversely, if $\Cu(M)$ satisfies \axiomO{5}, \axiomO{6} and weak cancellation, then $M$ satisfies the Riesz decomposition property.
\end{enumerate}
\end{thm}
\begin{proof}
Let us show \enumStatement{1}.
By \autoref{prp:axiomsPassToCu}(1), $\Cu(M)$ satisfies \axiomO{5} if and only if $M$, considered as a $\CatW$-semigroup $M=(M,\leq)$, satisfies \axiomW{5}.
First, assume that $M$ is algebraically ordered.
To show that $M$ satisfies \axiomW{5}, let $a',a,b',b,c,\tilde{c}\in M$ satisfy
\[
a+b\leq c,\quad
a'\leq a,\quad
b'\leq b,\quad
c\leq\tilde{c},
\]
Since $M$ is algebraically ordered, we can find $y\in M$ with $a+b+y=c$.
Set $x':=x=b+y$.
One checks that $x'$ and $x$ have the desired properties to verify \axiomW{5} for $M=(M,\leq)$.

Conversely, assume that $M$ satisfies \axiomW{5}.
To show that $M$ is algebraically ordered, let $a,c\in M$ satisfy $a\leq c$.
Set $a':=a$, $b':=b=0$ and set $\tilde{c}:=c$.
Since $M$ satisfies \axiomW{5}, we can find $x',x\in M$ such that
\[
a'+x\leq\tilde{c},\quad
c\leq a+x',\quad
x'\leq x.
\]
Then
\[
a+x = a'+x \leq \tilde c = c \leq a+x' \leq a+x,
\]
which shows $a+x=c$.
Thus, $M$ is algebraically ordered.

Statement \enumStatement{2} follows directly from \autoref{prp:axiomsPassToCu}(3).
Statement \enumStatement{3} follows from the equivalence between conditions (1) and (3) in \cite[Proposition~2.12]{Per97StructurePositive}.

Finally, let us show \enumStatement{4}.
First, assume that $M$ satisfies the Riesz decomposition property.
By \autoref{prp:axiomsPassToCu}(2), it is enough to verify \axiomW{6} for $M$.
Let $a',a,b,c\in M$ satisfy
\[
a'\leq a\leq b+c.
\]
By assumption, we can choose $b',c'\in M$ such that
\[
a'=b'+c',\quad
b'\leq b,\quad
c'\leq c.
\]
Then $b'\leq a$ and $c'\leq a$, showing that $b'$ and $c'$ have the desired properties to verify \axiomW{6} for the $\CatW$-semigroup $(M,\leq)$.

Conversely, assume that $S$ satisfies \axiomO{5}, \axiomO{6} and weak cancellation.
By statements \enumStatement{1} and \enumStatement{2} and \autoref{prp:axiomsPassToCu}(2), we have that $M$ is algebraically ordered, cancellative and satisfies \axiomW{6}.
To show that $M$ has Riesz decomposition, let $a,b,c\in M$ satisfy $a\leq b+c$.
Set $a':=a$.
Since $M$ satisfies \axiomW{6}, we can find $e,f\in M$ satisfying
\[
a'=a\leq e+f,\quad
e\leq a,b,\quad
f\leq a,c.
\]
Since $M$ is algebraically ordered, we can choose $x,y,z\in M$ such that
\[
a+x=e+f,\quad
e+y=a,\quad
f+z=a.
\]
Then
\[
a+x+y+z
=e+f+y+z
=2a.
\]
Since $M$ is cancellative, we obtain $a=x+y+z$.
It follows
\[
y+[x+z] = a \leq a+x = f+[x+z],
\]
which implies $y\leq f$.
Thus, we have $a=e+y$ with $e\leq b$ and $y\leq c$, as desired.
\end{proof}

%==========================================================================================
\begin{rmk}
We are thankful to the referee, who spotted a gap in the proof of \cite[Proposition~2.12]{Per97StructurePositive}.
This, however, does not affect the result that, for a partially ordered monoid $M$, the conditions (1) $M$ has Riesz interpolation; (2) The monoid $\Lambda(M)$ of all intervals in $M$ has Riesz interpolation, and (3) The monoid $\Lambda_\sigma(M)$ of all countably generated intervals has Riesz interpolation, are equivalent.
(This is proved following the arguments in \cite[Proposition 2.12]{Per97StructurePositive}.)
\end{rmk}

%==========================================================================================
\begin{cor}
\label{prp:algebraicRieszTFAE}
Let $S$ be an algebraic $\CatCu$-semigroup satisfying \axiomO{5} and weak cancellation.
Then the following are equivalent:
\beginEnumStatements
\item
The $\CatCu$-semigroup $S$ satisfies \axiomO{6}.
\item
The $\CatCu$-semigroup $S$ has Riesz refinement.
\item
The $\CatCu$-semigroup $S$ has Riesz decomposition.
\item
The $\CatCu$-semigroup $S$ has Riesz interpolation.
\item
The monoid of compact elements, $S_c$, has Riesz refinement (or equivalently, $S_c$ has Riesz decomposition, or $S_c$ has Riesz interpolation).
\end{enumerate}
\end{cor}
\begin{proof}
Let $S$ be an algebraic $\CatCu$-semigroup satisfying \axiomO{5} and weak cancellation.
Let $M$ be the \pom\ of compact elements in $S$.
As shown in \autoref{prp:algebraicSemigp}, we have that $S$ is isomorphic to the $\CatCu$-completion of $M$.
By \autoref{prp:propertiesAlgebraic}(1) and (2), $M$ is algebraically ordered and cancellative.
It follows that the Riesz properties stated in \enumStatement{5} are equivalent for $M$.

By \autoref{prp:propertiesAlgebraic}(3), we have that $M$ has Riesz interpolation if and only $S$ does.
This shows the equivalence between \enumStatement{4} and \enumStatement{5}.
Similarly, we obtain the equivalence between \enumStatement{1} and \enumStatement{5} from \autoref{prp:propertiesAlgebraic}(4).

It is easy to check that \enumStatement{3} implies \enumStatement{1}.
To see that \enumStatement{2} implies \enumStatement{1}, let $a',a,b,c\in S$ satisfy $a'\ll a\leq b+c$.
Since $S$ is algebraic, we can choose $x\in S$ compact with $a'\leq x\leq a$.
Since $S$ satisfies \axiomO{5}, we can find $y\in S$ satisfying $x+y=b+c$.
Using Riesz refinement, we choose $r_{i,j}\in S$ for $i,j=1,2$ such that
\[
x=r_{1,1}+r_{1,2},\quad
y=r_{2,1}+r_{2,2},\quad
b=r_{1,1}+r_{2,1},\quad
c=r_{1,2}+r_{2,2}.
\]
Set $e:=r_{1,1}$ and $f:=r_{1,2}$.
Then $e$ and $f$ have the desired properties to verify \axiomO{6} for $S$.
Finally, it follows from Lemma~2.6(a) and Proposition~2.5 in \cite{Goo96KMultiplier} that \enumStatement{5} implies \enumStatement{2} and \enumStatement{3}.
\end{proof}

%------------------------------------------------------------------------------------------
We will now consider the class of algebraic $\CatCu$-semigroups that are $\CatCu$-comple\-tions of dimension groups.
We first recall some definitions.

%==========================================================================================
\begin{dfn}
\label{dfn:dimMonoid}
\index{terms}{positively ordered monoid!simplicial}
\index{terms}{simplicial monoid}
\index{terms}{dimension monoid}
Let $M$ be a \pom.
\beginEnumStatements
\item
We call $M$ a \emph{simplicial monoid} if it is isomorphic to the algebraically ordered monoid $\N^r$, for some $r\in\N_+$.
\item
We call $M$ a \emph{dimension monoid} if it is isomorphic to the inductive limit in $\CatPom$ of simplicial monoids.
\end{enumerate}
\end{dfn}

%------------------------------------------------------------------------------------------
Let $M$ be a \pom.
Recall that $M$ is \emph{unperforated} if for every $a,b\in M$ we have $a\leq b$ whenever $na\leq nb$ for some $n\in\N_+$.
\index{terms}{positively ordered monoid!unperforated} \index{terms}{unperforated}
Every simplicial monoid is algebraically ordered, cancellative, unperforated and satisfies the Riesz refinement property.
It is easy to see that all these properties pass to inductive limits, whence they are satisfied by all dimension monoids.
The converse is known as the Effros-Handelman-Shen theorem, \cite{EffHanShe80DimGps}, which is formulated for partially ordered groups.
The version given here for a \pom\ $M$ follows by passing to the Grothendieck completion $G$, from which $M$ can be recovered as $M=G_+$.
It is clear that for every separable AF-algebra $A$, the Murray-von Neumann semigroup $V(A)$ is a dimension monoid.
The converse can for instance be found in \cite[Proposition~1.4.2, p.20]{Ror02Classification}.

%==========================================================================================
\begin{thm}[Effros, Handelman, Shen]
\label{prp:EffHanShen}
Let $M$ be a countable, \pom.
Then the following are equivalent:
\beginEnumStatements
\item
The monoid $M$ is a dimension monoid.
\item
The monoid $M$ is algebraically ordered, cancellative, unperforated and satisfies the Riesz refinement property.
\item
There is a separable AF-algebra $A$ such that $M\cong V(A)$.
\end{enumerate}
\end{thm}

%------------------------------------------------------------------------------------------
In order to formulate the analog of the Effros-Handelman-Shen theorem for $\CatCu$-semigroups, we will call a $\CatCu$-semigroup $S$ a \emph{simplicial $\CatCu$-semigroup} if it is isomorphic to the $\CatCu$-completion of a simplicial monoid, that is, if $S\cong\overline{\N}^r$ with the algebraic order, for some $r\in\N_+$.
\index{terms}{Cu-semigroup@$\CatCu$-semigroup!simplicial}

%==========================================================================================
\begin{cor}
\label{prp:EffHandShen_CuVersion}
Let $S$ be a countably-based $\CatCu$-semigroup.
Then the following are equivalent:
\beginEnumStatements
\item
The semigroup $S$ is isomorphic to an inductive limit of simplicial $\CatCu$-semi\-groups.
\item
There is a dimension monoid $M$ such that $S\cong\Cu(M)$.
\item
The semigroup $S$ is weakly cancellative, unperforated, algebraic and satisfies \axiomO{5} and \axiomO{6}.
\item
There is a separable AF-algebra $A$ such that $S\cong\Cu(A)$.
\end{enumerate}
\end{cor}
\begin{proof}
Using \autoref{prp:limitsCu}, it is easy to see that \enumStatement{1} and \enumStatement{2} are equivalent.
It is also easy to check that \enumStatement{3} implies \enumStatement{2}.

In order to show that \enumStatement{2} implies \enumStatement{3}, let $M$ be a dimension monoid such that $S\cong\Cu(M)$.
It follows directly from Theorems~\ref{prp:EffHanShen} and~\ref{prp:propertiesAlgebraic} that $S$ is weakly cancellative, algebraic and satisfies \axiomO{5} and \axiomO{6}.
To verify that $S$ is unperforated, let $a,b\in S$ and assume $na\leq nb$ for some $n\in\N_+$.
Since $S$ is algebraic, we can choose increasing sequences $(a_k)_k$ and $(b_k)_k$ of compact elements in $S$, such that $a=\sup_k a_k$ and $b=\sup_k b_k$.
For each $k$ we have
\[
na_k\ll na\leq nb=\sup_{l\in\N} nb_l.
\]
Choose $l\in\N$ with $na_k\leq na_l$.
Since $M$ is unperforated and the natural map from $M$ to $S$ is an order-embedding, this implies $a_k\leq b_l$.
Thus, we have $a_k\leq b$ for each $k$, and therefore $a\leq b$, as desired.

Finally, let us show that \enumStatement{2} and \enumStatement{4} are equivalent.
Given a separable AF-algebra $A$, the Cuntz semigroup of $A$ is isomorphic to the $\CatCu$-completion of $V(A)$.
Therefore, the desired equivalence follows from \autoref{prp:EffHanShen}.
\end{proof}

\vspace{5pt}
%------------------------------------------------------------------------------------------
%==========================================================================================
\section{Nearly unperforated semigroups}
\label{sec:nearUnp}

%------------------------------------------------------------------------------------------
In this section, we introduce the notion of `near unperforation' for \pom{s};
see \autoref{dfn:nearUnperf}.
We study how this concept is connected to other notions like almost unperforation, separativity and cancellation properties.
The main result of this section is \autoref{prp:nearUnperfSimpleCu} where we show that a simple, stably finite $\CatCu$-semigroup that satisfies \axiomO{5} is nearly unperforated if and only if it is weakly cancellative and almost unperforated.

In \cite{JiaSu99Projectionless}, the famous Jiang-Su algebra $\mathcal{Z}$ was introduced.
Recall that it is a unital, separable, simple, nonelementary, nuclear \ca{} with stable rank one and unique tracial state.
It is strongly self-absorbing and $KK$-equivalent to the complex numbers, which means $K_0(\mathcal{Z})\cong\Z$ and $K_1(\mathcal{Z})=0$.
Therefore, tensoring with $\mathcal{Z}$ has no effect on the $K$-theory of a \ca{}, although it can change the ordering on the $K_0$-group (see for example \cite{GonJiaSu2000}).
In the Elliott classification program, the Jiang-Su algebra is considered as the stably finite analog of the Cuntz algebra $\mathcal{O}_\infty$, which plays a central role in the classification of purely infinite \ca{s}.
\index{terms}{Jiang-Su algebra}
\index{symbols}{Z@$\mathcal{Z}$ \quad (Jiang-Su algebra)}

Given a \ca{} $A$ that tensorially absorbs the Jiang-Su algebra $\mathcal{Z}$, it is well-known that the Cuntz semigroup $\Cu(A)$ is almost unperforated, \cite[Theorem~4.5]{Ror04StableRealRankZ}.
Under the additional assumption that $A$ is simple or that $A$ has real rank zero and stable rank one, we obtain that $\Cu(A)$ is even nearly unperforated;
see \autoref{prp:nearUnpCaZstable}.
We conjecture that the Cuntz semigroup of every $\mathcal{Z}$-stable \ca{} is nearly unperforated;
see \autoref{conj:nearUnpCaZstable}.

%==========================================================================================
\begin{dfn}
\label{dfn:nearUnperf}
\index{terms}{positively ordered monoid!nearly unperforated}
\index{terms}{unperforated!nearly}
\index{terms}{nearly unperforated}
\index{symbols}{$\leq_p$}
Let $M$ be a \pom.
We define a binary relation $\leq_p$ on $M$ by setting $a\leq_p b$ for $a,b\in M$ if and only if there exists $k_0\in\N$ such that $ka\leq kb$ for all $k\in\N$ satisfying $k\geq k_0$.

We say that $M$ is \emph{nearly unperforated} if for all $a,b\in M$ we have that $a\leq_p b$ implies $a\leq b$.
\end{dfn}

%------------------------------------------------------------------------------------------
Note that $a\leq_p b$ if and only if there exists $k\in\N$ such that $ka\leq kb$ and $(k+1)a\leq(k+1)b$.

%==========================================================================================
\begin{lma}
\label{prp:nearUnperfTFAE}
Let $M$ be a \pom.
Then the following are equivalent:
\beginEnumStatements
\item
The monoid $M$ is nearly unperforated.
\item
For all $a,b\in M$, we have that $2a\leq 2b$ and $3a\leq 3b$ imply $a\leq b$.
\end{enumerate}
\end{lma}
\begin{proof}
It is easy to see that \enumStatement{1} implies \enumStatement{2}.
For the converse implication, let $a,b\in M$ satisfy $a\leq_p b$.
Let $n\in\N$ be the smallest integer such that $ka\leq kb$ for all $k\geq n$.
Arguing as in \cite[Lemma~2.1]{AraGooOMePar98Separative}, we will show that $(n-1)a\leq(n-1)b$ if $n\geq 2$.
This shows that $n=1$, and so $a\leq b$.

Assuming $n\geq 2$, we have $2(n-1)\geq n$ and $3(n-1)\geq n$.
It follows
\[
2(n-1)a\leq 2(n-1)b,\quad
3(n-1)a\leq 3(n-1)b.
\]
By assumption, this implies $(n-1)a\leq(n-1)b$, as desired.
\end{proof}

%------------------------------------------------------------------------------------------
Let $M$ be a \pom.
Recall that $M$ is \emph{unperforated} if for all elements $a,b\in M$ we have that $na\leq nb$ for some $n\in\N_+$ implies $a\leq b$.
Recall from \autoref{dfn:almUnp} that $M$ is \emph{almost unperforated} if $a<_s b$ implies $a\leq b$.

Let us say that $M$ is \emph{weakly separative} if for all elements $a$ and $b$ we have that $2a\leq a+b\leq 2b$ implies $a\leq b$.
\index{terms}{positively ordered monoid!weakly separative}
We warn the reader that different definitions of `separativity' for (partially ordered) semigroups appear in the literature.
However, in most of the recent literature, the notion of `separativity' has been used for a concept which is stronger than the condition above;
see for example \cite[Definition~1.2]{Weh94SeparativePOM}.
That is why we call the above condition `weak separativity'.

%==========================================================================================
\begin{prp}
\label{prp:nearUnperfImplications}
Let $M$ be a \pom.
Then the following implications hold:
\[
\xymatrix@C=15pt{
{M \text{ is unperforated } } \ar@{=>}[r]
& { M \text{ is nearly unperforated } } \ar@{=>}[r] \ar@{=>}[d]
& { M \text{ is almost unperforated } } \\
& { M \text{ is weakly separative } }.
}
\]
\end{prp}
\begin{proof}
It follows from \autoref{prp:relSTFAE} that the relation $a<_s b$ is stronger than the relation $\leq_p$.
Therefore, for any $a,b\in M$, the following implications hold:
\[
\xymatrix@M+=10pt{
{ na\leq nb \text{ for some $n\in\N_+$ } }
& { a\leq_p b } \ar@{=>}[l]
& { a<_s b } \ar@{=>}[l].
}
\]
This implies the horizontal implications of the diagram.

Finally, let us assume that $M$ is nearly unperforated.
In order to verify that $M$ is weakly separative, let $a,b\in M$ satisfy $2a\leq a+b\leq 2b$.
Then
\[
2a\leq 2b,\quad
3a \leq a+2b = (a+b) + b \leq 3b.
\]
By \autoref{prp:nearUnperfTFAE}, this implies $a\leq b$, as desired.
\end{proof}

%==========================================================================================
\begin{lma}
\label{prp:comparisonLeadsToP}
Let $M$ be a \pom, and let $a,b\in M$ and $k,l\in\N$.
If $a+ka\leq b+ka$ and $a+lb\leq b+lb$, then $a\leq_p b$.
\end{lma}
\begin{proof}
Let $a,b$ and $k,l$ be as in the statement.
Arguing as in \cite[Lemma~2.1]{AraGooOMePar98Separative}, it follows
\[
2a+ka
=a+[a+ka]
\leq a+[b+ka]
\leq 2b+ka.
\]
Inductively, we get $ra+ka\leq rb+ka$ for all $r\in\N$.
Analogously, we obtain $lb+sa\leq lb+sb$ for all $s\in\N$.
Then, for any $n\in\N$, we get
\[
(k+l+n)a \leq ka+(l+n)b =[ka+lb]+nb \leq [kb+lb]+nb=(k+l+n)b,
\]
which shows $a\leq_p b$, as desired.
\end{proof}

%==========================================================================================
\begin{dfn}
\label{dfn:preminSimpleSFPrePom}
\index{terms}{positively ordered monoid!preminimally ordered}
\index{terms}{preminimally ordered}
\index{terms}{positively ordered monoid!simple}
\index{terms}{simple (positively ordered monoid)}
\index{terms}{positively ordered monoid!stably finite}
\index{terms}{stably finite positively pre-ordered monoid}
Let $M$ be a \prePom.
We say that $M$ is \emph{preminimally ordered} if for all elements $a,b,x,y\in M$, we have that $a+x\leq b+x$ and $x\leq y$ imply $a+y\leq b+y$.

We say that $M$ is \emph{simple} if for all elements $a,b\in M$ with $b$ nonzero, we have $a\varpropto b$, that is, there exists $n\in\N$ such that $a\leq nb$.

An element $a$ in $M$ is \emph{finite} if $a<a+x$ for every nonzero element $x\in M$.
We say that $M$ is \emph{stably finite} if each of its elements is finite.
\end{dfn}

%==========================================================================================
\begin{rmks}
\label{rmk:preminSimpleSFPrePom}
(1)
The notion of being `preminimally ordered' was introduced in \cite[Definition~1.2]{Weh94SeparativePOM}.
This concept is closely related to what has been called `well-behaved' and `strictly well-behaved' in \cite[Definition~2.2.1]{Bla90Rational}.

(2)
Let $M$ be a \prePom.
If $M$ is cancellative, then it is stably finite.
Indeed, given elements $a$ and $x$ in $M$ we always have $a\leq a+x$ and if $a=a+x$ then cancellation implies that $x=0$.

(3)
Let $M$ be a conical monoid equipped with its algebraic pre-order.
Then $M$ is a \prePom.
If $M$ is stably finite, then its algebraic pre-order is antisymmetric and hence $M$, with its algebraic order, becomes a partially ordered monoid.

(4)
The notions of simplicity and stable finiteness have already been defined for $\CatCu$-semigroups;
see \autoref{pgr:finiteSemigr} and \autoref{dfn:simpleCu}.
However, we warn the reader that for a $\CatCu$-semigroup $S$, theses notions do not coincide when considering $S$ as a \prePom.
For instance, a nonzero $\CatCu$-semigroup always contains elements that are not finite.
Moreover, a nonzero, simple $\CatCu$-semigroup is not simple as a \pom, since for a nonzero element $a\in S$ we need not have $\infty\varpropto a$, but only $\infty\varpropto^\ctsRel a$.

One can, however, obtain a close connection as follows.
Given a $\CatCu$-semigroup $S$, consider
\[
S_0 := \left\{ x\in S : x\ll\tilde{x} \text{ for some } \tilde{x}\in S \right\}.
\]
Then $S$ is simple (respectively stably finite) as a $\CatCu$-semigroup if and only if $S_0$ is simple (respectively stably finite) as a \pom.
\end{rmks}

%------------------------------------------------------------------------------------------
The next result shows that for $\CatCu$-semigroups, the axiom~\axiomO{5} of almost algebraic order implies a suitable version of preminimality:

%==========================================================================================
\begin{lma}
\label{prp:preminO5}
Let $S$ be a $\CatCu$-semigroup satisfying \axiomO{5}, and let $a,b,x,y\in S$.
If $a+x\ll b+x$ and $x\leq y$, then $a+y\leq b+y$.
\end{lma}
\begin{proof}
Let $a,b,x$ and $y$ be as in the statement.
Choose $x'\in S$ such that $x'\ll x$ and $a+x\ll b+x'$.
Applying \axiomO{5} to the inequality $x'\ll x\leq y$, choose $d\in S$ satisfying $x'+d\leq y\leq x+d$.
Then
\[
a+y \leq a+x+d \leq b+x'+d \leq b+y,
\]
as desired.
\end{proof}

%==========================================================================================
\begin{prp}
\label{prp:cancUpToP}
(1)
Let $M$ be a preminimally \pom, and let $a,b,x\in M$.
If $a+x\leq b+x$ and $x\varpropto a,b$, then $a\leq_p b$.

(2)
Let $S$ be a $\CatCu$-semigroup satisfying \axiomO{5}, and let $a,b,x\in S$.
If $a+x\ll b+x$ and $x\varpropto^\ctsRel a,b$, then $a\leq_p b$.
\end{prp}
\begin{proof}
To show \enumStatement{2}, let $S$ be as in the statement, and let $a,b,x\in S$ satisfy $a+x\ll b+x$ and $x\varpropto^\ctsRel a,b$.
Choose $x'\in S$ with $x'\ll x$ and $a+x\ll b+x'$.
Then
\[
a+x' \leq a+x \ll b+x'.
\]
Moreover, we have $x'\varpropto a,b$, whence we can find $k,l\in\N$ such that $x\leq ka$ and $x\leq lb$.
By \autoref{prp:preminO5}, we obtain
\[
a+ka\leq b+ka,\quad
b+lb\leq b+lb.
\]
Then, by \autoref{prp:comparisonLeadsToP}, it follows $a\leq_p b$.
The proof of \enumStatement{1} is similar (and easier).
\end{proof}

%------------------------------------------------------------------------------------------
The following result should be compared to \cite[Theorem~2.2.6]{Bla90Rational}.

%==========================================================================================
\begin{cor}
\label{prp:cancUpToP_simple}
(1)
Let $M$ be a simple, stably finite, preminimally \pom, and let $a,b,x\in M$.
If $a+x\leq b+x$, then $a\leq_p b$.

(2)
Let $S$ be a simple, stably finite $\CatCu$-semigroup satisfying \axiomO{5}, and let $a,b,x\in S$.
If $a+x\ll b+x$, then $a\leq_p b$.
\end{cor}
\begin{proof}
The argument for both statements is analogous.
Let the elements $a,b$ and $x$ be as in the statements.
The conclusion is clearly true if $a$ is zero.

We may therefore assume that $a$ is nonzero.
Then, using stable finiteness in both cases, we get that $b$ cannot be zero.
Thus, we may assume that both $a$ and $b$ are nonzero.
By simplicity, this implies $x\varpropto a,b$ or $x\varpropto^\ctsRel a,b$, respectively.
Then the conclusion follows from \autoref{prp:cancUpToP}.
\end{proof}

%------------------------------------------------------------------------------------------
Recall from \autoref{dfn:addAxioms} that a $\CatCu$-semigroup $S$ is \emph{weakly cancellative} if $a+x\ll b+x$ implies $a\leq b$, for any $a,b,x\in S$.
It follows from the previous result that $S$ is weakly cancellative whenever it is simple, stably finite, nearly unperforated and satisfies \axiomO{5}.
We remark that a simple, almost unperforated $\CatCu$-semigroup need not be weakly cancellative;
see \autoref{sec:openproblems}\eqref{listPrbls:Zprime}.
See also \autoref{sec:openproblems}\eqref{listPrbls:RangeZmultNotWkCanc} where we ask if this phenomenon is also possible for Cuntz semigroups of (simple) \ca{s}.

By \autoref{prp:nearUnperfImplications}, near unperforation implies almost unperforation in general.
The following result provides a converse.

%==========================================================================================
\begin{thm}
\label{prp:nearUnperfSimpleCu}
Let $S$ be a simple $\CatCu$-semigroup satisfying \axiomO{5}.
Then $S$ is stably finite and nearly unperforated if and only if $S$ is weakly cancellative and almost unperforated.
\end{thm}
\begin{proof}
In order to verify the remaining `if' part of the statement, assume that $S$ is weakly cancellative and almost unperforated.
It is clear that weak cancellation implies that $S$ is stably finite.
Let $a,b\in S$ satisfy $a\leq_p b$.
By \autoref{prp:softNoncpctSimple}, an element in a simple, stably finite $\CatCu$-semigroup satisfying \axiomO{5} is either compact or nonzero and soft.
We may therefore distinguish the following cases.

Case 1:
Assume $a$ is soft.
Let $a'\in S$ satisfy $a'\ll a$.
Since $a$ is soft, it follows $a'<_s a$ and therefore $a'<_s b$.
Using that $S$ is almost unperforated, we get $a'\leq b$.
Thus, we have shown $a'\leq b$ for every $a'\in S$ satisfying $a'\ll a$, whence $a\leq b$.

Case 2:
Assume $b$ is soft.
Since $a\leq_p b$, we can find $n\in\N$ such that $na\leq nb$.
Note that $nb$ is also soft.

Let $a'\in S$ satisfy $a'\ll a$.
Then $na'\ll nb$, and since $nb$ is soft, it follows $na'<_s nb$.
This implies $a'<_s b$, and hence $a'\leq b$ by almost unperforation.
Again, as this holds for every $a'\in S$ satisfying $a'\ll a$, we get $a\leq b$.

Case 3:
Assume that $a$ and $b$ are compact.
If there is $n\in\N$ such that $na<nb$, then using \axiomO{5} for $S$, there exists a nonzero element $x\in S$ such that $na+x=nb$.
Since $S$ is simple and $a$ is compact, we can find $k\in\N$ such that $a\leq kx$.
Then
\[
(kn+1)a
\leq kna + kx
= knb,
\]
which shows $a<_sb$.
Since $S$ is almost unperforated, we get $a\leq b$.

In the other case, there is $n\in\N$ with $na=nb$ and $(n+1)a=(n+1)b$.
Let $x=na=nb$.
Then $a+x\ll b+x$.
It follows from weak cancellation that $a\ll b$.

Thus, in all cases, it follows $a\leq b$.
This shows that $S$ is nearly unperforated.
\end{proof}

%==========================================================================================
\begin{prp}
\label{prp:nearUnperfSimplePom}
Let $M$ be a simple, stably finite, algebraically ordered monoid.
Then $M$ is nearly unperforated if and only if $M$ is cancellative and almost unperforated.
\end{prp}
\begin{proof}
It follows from \autoref{prp:cancUpToP_simple} that every simple, stably finite, nearly unperforated, preminimally \pom\ is cancellative.
Moreover, by \autoref{prp:nearUnperfImplications}, near unperforation implies almost unperforation.
This shows the `only if' part of the statement.

For the converse, assume that $M$ is a cancellative, almost unperforated, simple, algebraically ordered monoid.
Let $a,b\in M$ satisfy $a\leq_p b$.
If there is $n\in\N$ such that $na<nb$, then since $M$ is algebraically ordered, there is a nonzero element $x\in M$ such that $na+x=nb$.
As in case 3 in the proof of \autoref{prp:nearUnperfSimpleCu}, this implies $a\leq b$.

In the other case, there exists $n\in\N$ with $na=nb$ and $(n+1)a=(n+1)b$.
By cancellation, it follows $a\leq b$, as desired.
\end{proof}

%==========================================================================================
\begin{prp}
\label{prp:nearUnperfCuAlgCanc}
Let $S$ be a $\CatCu$-semigroup satisfying \axiomO{5}.
Assume $S$ is algebraic, almost unperforated and weakly cancellative.
Then $S$ is nearly unperforated.
\end{prp}
\begin{proof}
Let $S$ be as in the statement.
To show that $S$ is nearly unperforated, let $a,b\in S$ satisfy $a\leq_p b$.
Since $S$ is algebraic, we may assume without loss of generality that $a$ and $b$ are compact.
Choose $n\in\N$ such that $na\leq nb$ and $(n+1)a\leq(n+1)b$.
Since $S$ satisfies \axiomO{5}, we can find $x,y\in S$ such that
\[
na+x=nb,\quad
(n+1)a+y=(n+1)b.
\]
Multiplying the first equation by $(n+1)$, and multiplying the second equation by $n$, we obtain
\[
n(n+1)a+(n+1)x=n(n+1)b,\quad
n(n+1)a+ny=n(n+1)b.
\]
Using that $b$ is compact and that $S$ is weakly cancellative, it follows
\[
(n+1)x = ny.
\]
Then, since $S$ is almost unperforated, we get $x\leq y$.
Using this at the second step, we get
\[
a+nb
= a+na+x
\leq a+na+y
= b + nb.
\]
Then, using that $b$ is compact and that $S$ has weak cancellation, it follows $a\leq b$, as desired.
\end{proof}

%==========================================================================================
\begin{prbl}
\label{prbl:nearUnpFromAlmUnp}
Let $S$ be an almost unperforated $\CatCu$-semigroup.
Which conditions are necessary and sufficient for $S$ to be nearly unperforated?
In particular, is it sufficient to assume that $S$ satisfies weak cancellation and \axiomO{5}?
\end{prbl}

%------------------------------------------------------------------------------------------
Concerning the second part of this problem, let $S$ be an almost unperforated, weakly cancellative $\CatCu$-semigroup satisfying \axiomO{5}.
Then $S$ is nearly unperforated if we additionally assume that $S$ is simple or algebraic;
see \autoref{prp:nearUnperfSimpleCu} and \autoref{prp:nearUnperfCuAlgCanc}.

Let us draw some conclusions for Cuntz semigroups of \ca{s}.

%==========================================================================================
\begin{cor}
\label{prp:nearUnpCaSr1}
Let $A$ be a \ca{} with stable rank one.
Assume that $A$ is either simple or has real rank zero.
Then $\Cu(A)$ is nearly unperforated whenever it is almost unperforated.
\end{cor}
\begin{proof}
By \cite[Theorem~4.3]{RorWin10ZRevisited}, $\Cu(A)$ has weak cancellation.
If $A$ is simple, then so is $\Cu(A)$;
see \autoref{prp:simpleCa}.
If $A$ has real rank zero, then $\Cu(A)$ is algebraic.
Then the statement follows from \autoref{prp:nearUnperfSimpleCu} (if $A$ is simple) and \autoref{prp:nearUnperfCuAlgCanc} (if $A$ has real rank zero).
\end{proof}

%==========================================================================================
\begin{cor}
\label{prp:nearUnpCaZstable}
Let $A$ be a $\mathcal{Z}$-stable \ca.
Then $\Cu(A)$ is nearly unperforated if $A$ is simple or has real rank zero and stable rank one.
\end{cor}
\begin{proof}
Since $A$ is $\mathcal{Z}$-stable, it follows from \cite[Theorem~4.5]{Ror04StableRealRankZ} that $\Cu(A)$ is almost unperforated.

We first assume that $A$ is simple.
Without loss of generality, we have $A\neq\{0\}$.
Since $A$ is $\mathcal{Z}$-stable, we can distinguish two cases.
If $A$ is purely infinite, then $\Cu(A)=\{0,\infty\}$, which is nearly unperforated.
In the other case, $A$ is stably finite, which by \cite[Theorem~6.7]{Ror04StableRealRankZ} implies that $A$ has stable rank one.
Then it follows from \autoref{prp:nearUnpCaSr1} that $\Cu(A)$ is nearly unperforated.

If $A$ has real rank zero and stable rank one, then it follows also directly from \autoref{prp:nearUnpCaSr1} that $\Cu(A)$ is nearly unperforated.
\end{proof}

%==========================================================================================
\begin{lma}
\label{lem:limnearlyunperf}
(1)
Let $S$ be a nearly unperforated $\CatPreW$-semigroup.
Then its $\CatCu$-completion $\gamma(S)$ is nearly unperforated.

(2)
Let $(S_i,\varphi_i)$ be an inductive system of nearly unperforated semigroups in $\CatPom$, $\CatPreW$ or $\CatCu$.
Then $S=\lim S_i$ is nearly unperforated.
\end{lma}
\begin{proof}
Let us show \enumStatement{1}.
Given $s\in S$, we denote by $\bar{s}$ its image in $\gamma(S)$.
Let $a,b\in\gamma(S)$ such that $2a\leq 2b$ and $3a\leq 3b$.
By properties of the $\CatCu$-completion (see \autoref{thm:Cuification}), we can choose rapidly increasing sequences $(a_n)_n$ and $(b_n)_n$ in $S$ such that $a=\sup_n \bar{a}_n$ and $b=\sup_n \bar{b}_n$.

Fix $n\in\N$.
Then
\[
2\bar{a}_n\ll 2b=\sup_k 2\bar{b}_k,\quad
3\bar{a}_n\ll 3b=\sup_k 3\bar{b}_k.
\]
Thus, we can find indices $k$ and $l$ such that $2\bar{a}_n\leq 2\bar{b}_k$ and $3\bar{a}_n\leq 3\bar{b}_l$.
Set $m:=\max\{k,l\}+1$.
Then
\[
2\bar{a}_n\ll 2\bar{b}_m,\quad
3\bar{a}_n\ll 3\bar{b}_m.
\]
By properties of the $\CatCu$-completion, this implies $2a_n\prec 2b_m$ and $3a_n\prec 3b_m$.
Using that $S$ is nearly unperforated, we obtain $a_n\leq b_m$, and thus $\bar{a}_n\leq \bar{b}_m\leq b$.
It follows $a\leq b$, as desired.

Next, let us show \enumStatement{2}.
It is straightforward to check the statement for limits in $\CatPom$.
Using that the limit in $\CatPreW$ has the same order structure as the limit in $\CatPom$, the result follows for limits in $\CatPreW$.
\autoref{prp:limitsCu} shows that the limit of an inductive system in $\CatCu$ is the $\CatCu$-completion of the limit of the same system considered in $\CatPreW$.
Therefore, the statement for $\CatCu$ follows from \enumStatement{1}.
\end{proof}

%------------------------------------------------------------------------------------------
For the next result, recall that we say that a \ca{} $A$ has \emph{no $K_1$-ob\-struc\-tions}, \index{terms}{no $K_1$-obstructions} if it has stable rank one and if $K_1(I)=\{0\}$ for any closed two-sided ideals $I$ of $A$;
see \cite{AntBosPer13CuFields} and \cite{AntBosPerPet14GeomDimFct}.

%==========================================================================================
\begin{thm}
\label{prp:NearUnpNoK1}
Let $A$ be a separable $\mathcal Z$-stable \ca{} that has no $K_1$-obstructions.
Then $\Cu(A)$ is nearly unperforated.
\end{thm}
\begin{proof}
Recall from \cite{RorWin10ZRevisited}, that $\mathcal{Z}$ is isomorphic to a sequential inductive limit where each algebra in the inductive system is equal to the fixed generalized dimension drop algebra
\[
Z_{2^\infty,3^\infty}
:=\left\{ f\in C([0,1],M_{2^\infty}\!\otimes\! M_{3^\infty}) :
f(0)\in M_{2^\infty}\!\otimes\! 1,\
f(1)\in 1\!\otimes\! M_{3^\infty} \right\}.
\]
Since $A$ is $\mathcal{Z}$-stable we have $A\cong \varinjlim_k A\!\otimes\! Z_{2^\infty,3^\infty}$.
By \autoref{prp:functorCu}, we have
\[
\Cu(A) \cong \varinjlim_k \Cu(A\!\otimes\! Z_{2^\infty,3^\infty}).
\]
Therefore, by \autoref{lem:limnearlyunperf}, it is enough to prove that $\Cu(A\!\otimes\! Z_{2^\infty,3^\infty})$ is nearly unperforated.
We remark that so far the argument applies for every \ca{} $A$.

We identify $A\!\otimes\! Z_{2^\infty,3^\infty}$ with the \ca{} of continuous maps $f$ from $[0,1]$ to $A\!\otimes\! M_{2^\infty}\!\otimes\! M_{3^\infty}$ such that
\[
f(0)\in A\!\otimes\! M_{2^\infty}\!\otimes\! 1,\quad
f(1)\in A\!\otimes\! 1\!\otimes\! M_{3^\infty}.
\]
We use $\ev_0$ and $\ev_1$ to denote the evaluation at the endpoints $0$ and $1$ of $[0,1]$, respectively.
Then we have a commutative pullback diagram:
\[
\xymatrix{
A\!\otimes\!  Z_{2^\infty,3^\infty}\ar@{-->}[r]^-{\ev_0\oplus\ev_1} \ar@{-->}[d]
& { A\!\otimes\! M_{2^\infty}\!\otimes\! 1 \oplus A\!\otimes\! 1\!\otimes\! M_{3^\infty} } \ar[d]^{\iota_3,\iota_2} \\
C([0,1], A\!\otimes\!  M_{2^\infty}\!\otimes\! M_{3^\infty}) \ar@{->>}[r]^-{\ev_0\oplus\ev_1}
& { A\!\otimes\! M_{2^\infty}\!\otimes\! M_{3^\infty} \oplus  A\!\otimes\! M_{2^\infty}\!\otimes\! M_{3^\infty} },
}
\]
where $\iota_3$ and $\iota_2$ denote the natural inclusion maps
\[
\iota_3\colon A\!\otimes\! M_{2^\infty}\!\otimes\! 1
\to A\!\otimes\! M_{2^\infty}\!\otimes\! M_{3^\infty},\quad
\iota_2\colon A\!\otimes\! 1\!\otimes\! M_{3^\infty}
\to A\!\otimes\! M_{2^\infty}\!\otimes\! M_{3^\infty}.
\]
We identify $M_{6^\infty}$ with $M_{2^\infty}\!\otimes\! M_{3^\infty}$, and $A\!\otimes\! M_{2^\infty}\!\otimes\! 1$ with $ A\!\otimes\! M_{2^\infty}$, and $ A\!\otimes\! 1\!\otimes\! M_{3^\infty}$ with $A\!\otimes\! M_{3^\infty}$.
Then, since $A$ has no $K_1$-obstructions, we can apply \cite[Theorem~3.5]{AntPerSan11PullbacksCu} to compute $\Cu(A\!\otimes\! Z_{2^\infty,3^\infty})$ as the pullback semigroup
\[
\xymatrix{
\Cu(A\!\otimes\!  Z_{2^\infty,3^\infty}) \ar@{-->}[r]^-{\ev_0\oplus\ev_1} \ar@{-->}[d]
& { \Cu( A\!\otimes\! M_{2^\infty}) \oplus \Cu(A\!\otimes\! M_{3^\infty}) } \ar[d]^{\Cu(\iota_3),\Cu(\iota_2)} \\
\Cu( C([0,1], A\!\otimes\! M_{6^\infty}) ) \ar@{->>}[r]^-{\ev_0\oplus\ev_1}
& { \Cu( A\!\otimes\! M_{6^\infty}) \oplus \Cu(A\!\otimes\! M_{6^\infty}) }.
}
\]
Given a $\CatCu$-semigroup $S$, we denote by $\Lsc([0,1],S)$ the semigroup of lower-semi\-con\-tin\-u\-ous functions from $[0,1]$ to $S$ with pointwise order and addition.
Again, using that $A$ has no $K_1$-obstructions, by \cite[Corollary~2.7]{AntPerSan11PullbacksCu} we have
\[
\Cu(C([0,1],A\!\otimes\! M_{6^\infty}))
\cong \Lsc([0,1],\Cu(A\!\otimes\! M_{6^\infty})).
\]

Now, let $a,b\in \Cu(A\otimes Z_{2^\infty,3^\infty})$ satisfy $a\leq_p b$.
Using the pullback description above, we can choose $f,g\in\Lsc([0,1],\Cu(A\!\otimes\! M_{6^\infty}))$, and $x,u\in\Cu(A\!\otimes\! M_{2^\infty})$, and $y,v\in \Cu(A\!\otimes\! M_{3^\infty})$ such that
\[
f(0) = x,\quad
f(1) = y,\quad
g(0) = u,\quad
g(1) = v,
\]
and so that $a,b$ are identified as
\[
a=(f,x,y),\quad b=(g,u,v).
\]
Then $f\leq_p g$, $x\leq_p u$, and $y\leq_p v$.
By \autoref{prp:UHFStableNearUnp}, the Cuntz semigroups $\Cu(C([0,1],A\!\otimes\! M_{6^\infty}))$, $\Cu(A\!\otimes\! M_{2^\infty})$ and $\Cu(A\!\otimes\! M_{3^\infty})$ are nearly unperforated.
Therefore, we obtain $f\leq g$, $x\leq u$, and $y\leq v$.
Hence $a\leq b$, as desired.
\end{proof}

%------------------------------------------------------------------------------------------
Inspired by the previous results, we make the following conjecture.

%==========================================================================================
\begin{conj}
\index{terms}{nearly unperforated Conjecture}
\label{conj:nearUnpCaZstable}
Let $A$ be a $\mathcal{Z}$-stable \ca.
Then $\Cu(A)$ is nearly unperforated.
\end{conj}

%==========================================================================================
\begin{pgr}
\label{pgr:AnswersNearUnpCaZstable}
We have verified \autoref{conj:nearUnpCaZstable} for several classes of \ca{s}.
Let $A$ be a $\mathcal{Z}$-stable \ca.
Then $\Cu(A)$ is nearly unperforated in the following cases:
\beginEnumStatements
\item
If $A$ is simple;
see \autoref{prp:nearUnpCaZstable}.
\item
If $A$ has real rank zero and stable rank one;
see \autoref{prp:nearUnpCaZstable}.
\item
If $A$ is UHF-stable;
see \autoref{prp:UHFStableNearUnp}.
\item
If $A$ is purely infinite (not necessarily simple);
see \autoref{prp:pureInfCanearUnp}.
\item
If $A$ has no $K_1$-obstructions;
see \autoref{prp:NearUnpNoK1}.
\end{enumerate}
\end{pgr}

%------------------------------------------------------------------------------------------
%==========================================================================================
%##########################################################################################
\chapter{Bimorphisms and tensor products}
\label{sec:tensProd}

%------------------------------------------------------------------------------------------
In this chapter, we first present a framework for a theory of tensor products in enriched categories.
We focus on the categories $\CatPreW$ and $\CatCu$, which are both enriched over the category $\CatPom$; see \autoref{prp:enrichmentW-Cu}.

In \autoref{sec:tensProdPreW}, we construct tensor products in $\CatPreW$.
Given $\CatPreW$-semigroup{s} $S$ and $T$, we consider the tensor product $S\otimes_\CatPom T$ of the underlying \pom{s}, as constructed in \autoref{sec:pom}, and we equip it with a natural auxiliary relation $\prec$;
see \autoref{dfn:tensProdAuxRel}.
We show that the pair
\[
(S\otimes_\CatPom T,\prec),
\]
which we abbreviate $S\otimes_\CatPreW T$, is a $\CatPreW$-semigroup{} that has the universal properties of a tensor product;
see \autoref{prp:tensProdPreW}.
We then show that this gives $\CatPreW$ the structure of a symmetric, monoidal category;
see \autoref{pgr:monoidalPreW}.

In \autoref{sec:tensProdCu}, we show the existence of tensor products in $\CatCu$ by combining the result for $\CatPreW$ with the fact that $\CatCu$ is a reflective subcategory of $\CatPreW$.
More precisely, given $\CatCu$-semigroup{s} $S$ and $T$, their tensor product in $\CatCu$ is given as
\[
S\otimes_\CatCu T = \gamma(S\otimes_\CatPreW T),
\]
which is the $\CatCu$-completion of $S\otimes_\CatPreW T$;
see \autoref{prp:tensProdCu}.

Given \ca{s} $A$ and $B$, there is a natural $\CatCu$-morphism
\[
\tau_{A,B}^{\txtMax}\colon\Cu(A)\otimes_{\CatCu}\Cu(B) \to \Cu(A\tensMax B).
\]
It is natural to ask when this map is an isomorphism.
In \autoref{prp:tensProdAF}, we provide a positive answer if one of the \ca{s} is an AF-algebra.
The crucial observation is that tensor products in $\CatPreW$ and $\CatCu$ are continuous functors in each variable;
see \autoref{prp:tensLim}.

In \autoref{prp:tensWithInfty}, we show that for every $\CatCu$-semigroup{} $S$, the tensor product of $S$ with $\{0,\infty\}$ is naturally isomorphic to the $\CatCu$-semigroup{}  $\Lat_{\mathrm{f}}(S)$ of singly-generated ideals of $S$ as considered in \autoref{prp:LatCu}.
It follows that given $\CatCu$-semigroup{s} $S$ and $T$, there is a natural isomorphism
\[
\Lat_{\mathrm{f}}(S\otimes_{\CatCu} T) \cong \Lat_{\mathrm{f}}(S)\otimes_{\CatCu} \Lat_{\mathrm{f}}(T).
\]
In \autoref{prp:tensCaWithInfty}, we apply these results for the Cuntz semigroup of a separable \ca{} $A$ and deduce that there are natural isomorphisms
\[
\Cu(A\otimes\mathcal{O}_2)
\cong \Lat(A)
\cong \Cu(A)\otimes_\CatCu\{0,\infty\}
\cong \Cu(A)\otimes_\CatCu\Cu(\mathcal{O}_2),
\]
where $\mathcal{O}_2$ denotes the Cuntz algebra generated by two isometries with range projections adding up to the unit.
The same result holds when $\mathcal{O}_2$ is replaced by any simple, purely infinite \ca.

%------------------------------------------------------------------------------------------
%==========================================================================================
\section{Tensor product as representing object}
\label{sec:abstractTensProd}

%------------------------------------------------------------------------------------------
In this section, we give a general categorical setup for tensor products, which is in part inspired by the approach in \cite{BanNel76Bimor}.
When constructing the tensor product of objects with a certain structure, the notion of a bimorphism is a crucial ingredient.

In some categories a bimorphism from a pair of objects $(X,Y)$ to a third object $Z$ is simply a set function $X\times Y\to Z$ that is a morphism in each variable; for instance in the categories $\CatMon$ and $\CatPom$ (see Paragraphs~\ref{pgr:mon} and~\ref{pgr:pom}).
In other cases, a bimorphism is a set function $X\times Y\to Z$ that is only required to be a `generalized' morphism in each variable, but additionally satisfies a condition taking both variables into account. This is the case, for instance, in the category of \ca{s} $\CatCa$ (see \autoref{exa:bimorphismFctr}), and the categories $\CatPreW$ and $\CatCu$ (see Definitions~\ref{dfn:bimorPreW} and~\ref{dfn:bimorCu}; see also Lemmas~\ref{prp:bimorPreW} and~\ref{prp:bimorCu}).

With the notion of bimorphisms at hand, the tensor product of two objects $X$ and $Y$ can be defined as an object that represents the functor $Z\mapsto\operatorname{Bimor}(X\times Y, Z)$.
This means that the tensor product $X\otimes Y$ satisfies
\[
\operatorname{Bimor}(X\times Y, Z) \cong \operatorname{Mor}(X\otimes Y,Z),
\]
for all objects $Z$.
Of course, whether the functor $\operatorname{Bimor}(X\times Y,\freeVar)$ is representable or not depends heavily on the categories considered and the objects $X$ and $Y$.

If the (bi)morphism sets carry additional structure, we may demand that it be preserved by the above identification.
This can be made precise using the language of enriched categories and functors.
The basic theory of monoidal and enriched categories can be found in \autoref{sec:appendixMonoidal}.
For details, we refer the reader to \cite{MacLan71Categories} and \cite{Kel05EnrichedCat}.

In this section, $\CatV$ will always denote a concrete, locally small, closed symmetric monoidal category, and $I$ will denote the unit object in $\CatV$.

%==========================================================================================
\begin{pgr}[Representable functor]
\label{pgr:representableFctr}
\index{terms}{representable functor}
Let $\CatC$ be a category that is enriched over $\CatV$.
Each object $X$ in $\CatC$ defines a $\CatV$-functor
\[
\CatCMor(X,\freeVar)\colon \CatC\to\CatV
\]
as follows:
An object $Z$ in $\CatC$ is sent to the object $\CatCMor(X,Z)$ in $\CatV$.
Further, given objects $Z$ and $Z'$ in $\CatC$, the $\CatV_0$-morphism
\[
C(X,\freeVar)_{Z,Z'} \colon\CatCMor(Z,Z')\to \CatCMor(X,Z')^{\CatCMor(X,Z)}
\]
is the one corresponding to $M_{X,Z,Z'}$ (defining the composition of morphisms in $\CatC$) under the identification
\[
\CatVMor_0\left( \CatCMor(Z,Z')\otimes\CatCMor(X,Z),\CatCMor(X,Z') \vphantom{\CatC(X)^{\CatC(X)}} \right)
\cong\CatVMor_0\left( \CatCMor(Z,Z'),\CatCMor(X,Z')^{\CatCMor(X,Z)} \right).
\]
The $\CatV$-functor $\CatCMor(X,\freeVar)$ is called the \emph{representable functor} corresponding to $X$.
\end{pgr}

%==========================================================================================
\begin{pgr}[Bimorphism functor]
\label{pgr:bimorphismFctr}
Let $\CatC$ be a category that is enriched over $\CatV$.
Given objects $X$ and $Y$ in $\CatC$, we assume that there is a $\CatV$-functor
\[
\CatCBimor(X\times Y,\freeVar)\colon\CatC\to\CatV.
\]
This means that for each object $Z$ in $\CatC$ there is an object $\CatCBimor(X\times Y,Z)$ in $\CatV$, representing the bimorphisms from $X\times Y$ to $Z$.
Moreover, given objects $Z$ and $Z'$ in $\CatC$, there is a $\CatV_0$-morphism
\[
\CatCBimor(X\times Y,\freeVar)_{Z,Z'}\colon\CatCMor(Z,Z')
\to \CatCBimor(X\times Y,Z')^{\CatCBimor(X\times Y,Z)}.
\]

We remark that the notation `$X\times Y$' appearing in the bimorphism functor does not refer to the product of the objects $X$ and $Y$.
In general, we do not assume that the considered category has products.
The notation is chosen since for the concrete cases considered in this paper, a bimorphism is a set function from the cartesian product of the underlying sets of $X$ and $Y$ to the underlying set of $Z$.
\end{pgr}

%==========================================================================================
\begin{exa}
\label{exa:bimorphismFctr}
Let $\CatC$ be one of the concrete categories considered in this paper (for example, $\CatCu$), and let $X,Y, Z$ be objects in $\CatC$.
Then $X,Y$ and $Z$ are sets with additional structure, and a $\CatC$-morphism from $X$ to $Z$ is just a set function $X\to Z$ preserving this structure.
Similarly, a $\CatC$-bimorphism from $X\times Y$ to $Z$ is a set function $X\times Y\to Z$ satisfying certain conditions.

Consider for example the category $\CatCa$ of \ca{s}, which is enriched over $\CatCGHTop$; see \autoref{exa:closedCat}.
Given \ca{s} $A$ and $B$ we denote the set of \starHom{s} from $A$ to $B$ by $\CatCaMor(A,B)$ (we do not use $\CatCa(A,B)$ to avoid confusion with the \ca\ generated by $A$ and $B$). $\CatCaMor(A,B)$ has a natural topology giving it the structure of a compactly generated, Hausdorff space; see \cite{DubPor71ConvenientCatTopAlg}.
The representable functor
\[
\CatCaMor(A,\freeVar)\colon \CatCa \to \CatCGHTop,
\]
sends a \ca{} $C$ to $\CatCaMor(A,C)$ in $\CatCGHTop$;
and for any pair $(C,C')$ of \ca{s}, the $\CatCGHTop$-morphism
\[
\CatCaMor(A,\freeVar)_{C,C'}\colon \CatCaMor(C,C') \to \CatCaMor(A,C')^{\CatCaMor(A,C)}
\]
is given by
\[
\CatCaMor(A,\freeVar)_{C,C'}(\alpha)(\varphi):=\alpha\circ\varphi,
\]
for all $\alpha\in\CatCaMor(C,C')$ and $\varphi\in\CatCaMor(A,C)$.

Given \ca{s} $A,B$ and $C$, a $\CatCa$-bimorphism from $A\times B$ to $C$ is a set function $\varphi\colon A\times B\to C$ satisfying the following conditions:
\beginEnumConditions
\item
The function $\varphi$ is bounded and linear in each variable.
\item
We have $\varphi(a^*,b^*)=\varphi(a,b)^*$ for each $a\in A$ and $b\in B$.
\item
We have $\varphi(a_1a_2,b_1b_2)=\varphi(a_1,b_1)\varphi(a_2,b_2)$ for each $a_1,a_2\in A$ and $b_1,b_2\in B$.
\end{enumerate}
We equip the set $\CatCaBimor(A\times B,C)$ of all $\CatCa$-bimorphisms from $A\times B$ to $C$ with the topology of point-norm convergence.

Given \ca{s} $A$ and $B$, we define the bimorphism functor
\[
\CatCaBimor(A\times B,\freeVar)\colon \CatCa \to \CatCGHTop
\]
as follows:
A \ca{} $C$ is sent to $\CatCaBimor(A\times B,C)$ in $\CatCGHTop$;
and for any pair $(C,C')$ of \ca{s}, the $\CatCGHTop$-morphism
\[
\CatCaBimor(A\times B,\freeVar)_{C,C'}\colon \CatCaMor(C,C')
\to \CatCaBimor(A\times B,C')^{\CatCaBimor(A\times B,C)}
\]
is given by
\[
\CatCaBimor(A\times B,\freeVar,)_{C,C'}(\alpha)(\varphi):=\alpha\circ\varphi,
\]
for all $\alpha\in\CatCaMor(C,C')$ and $\varphi\in\CatCaBimor(A\times B,C)$.
\end{exa}

%==========================================================================================
\begin{pgr}
\label{pgr:abstractTensProd}
Let $\CatC$ be a category that is enriched over $\CatV$.
Assume that for any objects $X$ and $Y$ in $\CatC$ there is a bimorphism $\CatV$-functor $\CatCBimor(X\times Y,\freeVar)$.
Let $V$ be an object in $\CatC$, and let $\varphi$ be an element of $\CatCBimor(X\times Y,V)$.
Let us show that this induces a $\CatV$-natural transformation
\[
\Phi\colon \ \CatCMor(V,\freeVar) \ \Rightarrow\ \CatCBimor(X\times Y,\freeVar).
\]
For each object $Z$ in $\CatC$, we need to define a $\CatV_0$-morphism $\Phi_Z$ from $\CatCMor(V,Z)$ to $\CatCBimor(X\times Y,Z)$.
To define $\Phi_Z$, we use the $\CatV_0$-morphism defining the bimorphism functor
\[
\CatCBimor(X\times Y,\freeVar)_{V,Z} \colon \CatCMor(V,Z)
\to \CatCBimor(X\times Y,Z)^{\CatCBimor(X\times Y,V)},
\]
which naturally corresponds to a $\CatV_0$-morphism
\[
G_{V,Z}\colon\CatCMor(V,Z)\otimes\CatCBimor(X\times Y,V)\to\CatCBimor(X\times Y,Z).
\]
Then $\Phi_Z$ is the $\CatV_0$-morphism given as the following composition:
\[
\CatCMor(V,Z) \xrightarrow{\cong}
\CatCMor(V,Z)\otimes I \xrightarrow{\id\otimes\varphi}
\CatCMor(V,Z)\otimes\CatCBimor(X\times Y,V) \xrightarrow{G_{V,Z}}
\CatCBimor(X\times Y,Z).
\]
\end{pgr}

%==========================================================================================
\begin{dfn}
\label{dfn:abstractTensProd}
\index{terms}{tensor product}
With the notation from \autoref{pgr:abstractTensProd}, we say that the pair $(V,\varphi)$ is a \emph{tensor product} of $X$ and $Y$ in the enriched category $\CatC$ if $\Phi$ is a natural isomorphism, that is, if
\[
\Phi_Z\colon\CatCMor(V,Z)\to\CatCBimor(X\times Y,Z)
\]
is a $\CatV_0$-isomorphism for each object $Z$ in $\CatC$.
\end{dfn}

%==========================================================================================
\begin{rmk}
\label{rmk:abstractTensProd}
We retain the notation from \autoref{pgr:abstractTensProd}.
Recall that a representation of a $\CatV$-functor $F\colon\CatC\to\CatV$ is an object $V$ in $\CatC$ together with a natural isomorphism $\Phi$ from the representable functor $\CatCMor(V,\freeVar)$ to $F$.
Thus, a tensor product $(V,\varphi)$ for $X$ and $Y$ induces a representation $(V,\Phi)$ of the $\CatV$-functor $\CatCBimor(X\times Y,\freeVar)$.

Conversely, assume that the $\CatV$-functor $\CatCBimor(X\times Y,\freeVar)$ is represented by the object $V$ and the natural isomorphism $\Phi$.
Then $\Phi_V$ is a $\CatV_0$-isomorphism
\[
\Phi_V\colon\CatCMor(V,V)\xrightarrow{\cong}\CatCBimor(X\times Y,V).
\]
Under this isomorphism, the identity element $\id_V\in\CatCMor(V,V)$ corresponds to an element $\varphi\in\CatCBimor(X\times Y,V)$.
It is straightforward to check that $(V,\varphi)$ induces the $\CatV$-natural isomorphism $(V,\Phi)$.
Thus, $(V,\varphi)$ is a tensor product of $X$ and $Y$.

To summarize, we have a natural correspondence between the following classes:
\beginEnumStatements
\item
Concrete tensor products $(V,\varphi)$ of $X$ and $Y$, where $V$ is an object in $\CatC$, and where $\varphi$ is an element in $\CatCBimor(X\times Y,V)$.
\item
Representations of the $\CatV$-functor $\CatCBimor(X\times Y,\freeVar)$.
\end{enumerate}
Any two tensor products of $X$ and $Y$ are isomorphic, and hence the notation $X\otimes Y$ is unambiguous.
We also write $X\otimes_{\CatC}Y$ if we need to specify the category where the tensor product is taken.
\end{rmk}

%==========================================================================================
\begin{exa}
\label{exa:tensProdCa}
Consider the category $\CatCa$ of \ca{s}, which is enriched over $\CatCGHTop$.
Let $A$ and $B$ be \ca{s}, and consider the $\CatCa$-bimorphism functor
\[
\CatCaBimor(A\times B,\freeVar)\colon \CatCa \to \CatCGHTop
\]
from \autoref{exa:bimorphismFctr}.
Let $A\tensMax B$ denote the maximal tensor product of $A$ and $B$;
we refer the reader to \cite[\S~II.9]{Bla06OpAlgs} for an introduction and details of the rich theory of tensor products of \ca{s}.
Consider the map
\[
\varphi_{A,B}\colon A\times B \to A\tensMax B,\quad
(a,b)\mapsto a\otimes b,\quad
\txtFA a\in A, b\in B.
\]
It is easy to see that $\varphi_{A,B}$ is a $\CatCa$-bimorphism.
For each \ca{} $C$, the assignment
\[
\CatCaMor(A\tensMax B,C) \to \CatCaBimor(A\times B,C),
\]
defined by mapping $\tau\in\CatCaMor(A\tensMax B,C)$ to $\tau\circ\varphi_{A,B}$,
is a homeomorphism, that is, an isomorphism in $\CatCGHTop$.
This means that the maximal tensor product of \ca{s} represents the $\CatCa$-bimorphism functor.
\end{exa}

%==========================================================================================
\begin{pgr}
\label{pgr:functorialityTensor}
Let $\CatC$ be a category that is enriched over $\CatV$.
Assume that $\CatC$ has a bimorphism functor that is also functorial in the first two variables.
This means there is a $\CatV$-multifunctor
\[
\CatCBimor(\freeVar\times\freeVar,\freeVar)\colon\CatC^\op\times\CatC^\op\times\CatC\to\CatV.
\]
Let us also assume that $\CatC$ has tensor products, that is, for any objects $X$ and $Y$ in $\CatC$ there are an object $X\otimes Y$ in $\CatC$ and a universal bimorphism
\[
\varphi_{X,Y}\in\CatCBimor(X\times Y,X\otimes Y).
\]

Functoriality of $\CatCBimor(\freeVar\times\freeVar,\freeVar)$ in the first two variables induces a $\CatV$-bifunctor
\[
\otimes\colon\CatC\times\CatC\to\CatC.
\]
A pair $(X,Y)$ in $\CatC\times\CatC$ is sent to the object $X\otimes Y$ in $\CatC$.
Given objects $X,X',Y$, and $Y'$ in $\CatC$, the required $\CatV_0$-morphism
\[
\CatCMor(X,X')\otimes\CatCMor(Y,Y')\to\CatCMor(X\otimes Y,X'\otimes Y')
\]
is obtained as the following composition, where the first morphism is obtained using that $\CatCBimor(\freeVar\times\freeVar,X'\otimes Y')$ is a (contravariant) $\CatV$-bifunctor, the second morphism is obtained by applying the $\CatV_0$-morphism $\varphi_{X',Y'}\colon I \to \CatCBimor(X'\times Y', X'\otimes Y')$ in the first variable to the internal hom-bifunctor, the first isomorphism is a natural isomorphism for closed monoidal categories and the second and last isomorphism is obtained using that $X\otimes Y$ represents the functor $\CatCBimor(X\times Y,\freeVar)$:
\begin{align*}
\CatCMor(X,X')\otimes\CatCMor(Y,Y')
&\to
\CatCBimor(X\times Y,X'\otimes Y')^{\CatCBimor(X'\times Y',X'\otimes Y')} \\
&\to
\CatCBimor(X\times Y,X'\otimes Y')^{I} \\
&\cong \CatCBimor(X\times Y,X'\otimes Y')
\cong \CatCMor(X\otimes Y,X'\otimes Y').
\end{align*}
\end{pgr}

\vspace{5pt}
%------------------------------------------------------------------------------------------
%==========================================================================================
\section{The tensor product in \texorpdfstring{$\CatPreW$}{PreW}}
\label{sec:tensProdPreW}

%==========================================================================================
\begin{pgr}
\label{pgr:enrichmentPreW}
Let us show that $\CatPreW$ is enriched over the closed, monoidal category $\CatPom$.
Given $\CatPreW$-semigroup{s} $S$ and $T$, recall that we denote by $\CatWMor(S,T)$ the set of $\CatW$-morphisms from $S$ to $T$.
Equipped with pointwise order and addition, $\CatWMor(S,T)$ has the natural structure of a \pom.

Given $\CatPreW$-semigroup{s} $S,T$ and $R$, it is easy to see that the composition of morphisms
\[
C_{S,T,R}\colon\CatWMor(T,R)\times\CatWMor(S,T)\to\CatWMor(S,R),\quad
(g,f)\mapsto g\circ f,%\quad
%\txtFA g\in\CatWMor(T,R), f\in\CatWMor(S,T),
\]
is a $\CatPom$-bimorphism.
By \autoref{prp:tensorPom}, $C_{S,T,R}$ factors through the $\CatPom$-tensor product.
This means that there exists a $\CatPom$-morphism
\[
M_{S,T,R}\colon\CatWMor(T,R)\otimes_{\CatPom}\CatWMor(S,T)\to\CatWMor(S,R)
\]
such that $g\circ f=M_{S,T,R}(g\otimes f)$ for every $g\in\CatWMor(T,R)$ and $f\in\CatWMor(S,T)$.

One can prove that this structure defines an enrichment of $\CatPreW$ over $\CatPom$.
Since the categories $\CatW$ and $\CatCu$ are full subcategories of $\CatPreW$, they inherit the enrichment over $\CatPom$.
\end{pgr}

%==========================================================================================
\begin{prp}
\label{prp:enrichmentW-Cu}
The categories $\CatPreW$, $\CatW$ and $\CatCu$ are enriched over the category $\CatPom$.
Moreover, the reflection functors $\mu\colon\CatPreW\to\CatW$
and $\gamma\colon\CatPreW\to\CatCu$, from \autoref{pgr:WreflectivePreW}  and \autoref{prp:CureflectivePreW}, are $\CatPom$-functors.
\end{prp}
\begin{proof}
We have already observed in \autoref{pgr:enrichmentPreW} that the three categories are enriched over $\CatPom$.
In order to show that the reflection functor $\gamma\colon\CatPreW\to\CatCu$ is compatible with the enrichment, let $S$ and $T$ be $\CatPreW$-semigroup{s}. We need to define a $\CatPom$-morphism
\[
\gamma_{S,T}\colon\CatWMor(S,T) \to \CatCuMor(\gamma(S),\gamma(T)).
\]
Let $\alpha_T\colon T\to\gamma(T)$ be the $\CatCu$-completion of $T$; see \autoref{dfn:Cuification}.
Given $f$ in $\CatWMor(S,T)$, we consider the composition $\alpha_T\circ f\colon S\to\gamma(T)$.
Using this assignment at the first step, and using \autoref{thm:Cuification} to obtain the natural identification at the second step, we obtain the following composition:
\[
\CatWMor(S,T) \to \CatWMor(S,\gamma(T))
\xrightarrow{\cong} \CatCuMor(\gamma(S),\gamma(T)).
\]
It is easy to see that these maps respect the $\CatPom$-structure of the involved morphism sets.
It is then straightforward to check that $\gamma$ is a $\CatPom$-functor.

Analogously, one shows that $\mu$ preserves is a $\CatPom$-functor.
\end{proof}

%==========================================================================================
\begin{dfn}
\label{dfn:bimorPreW}
\index{terms}{W-bimorphism@$\CatW$-bimorphism}
\index{terms}{W-bimorphism@$\CatW$-bimorphism!generalized}
Let $S$, $T$ and $R$ be $\CatPreW$-semigroup{s}, and let $f\colon S\times T\to R$ be a $\CatPom$-bimorphism.
We say that $f$ is a \emph{$\CatW$-bimorphism} if it satisfies the following conditions:
\beginEnumConditions
\item
The map $f$ is continuous in the following sense:
For every $a\in S$, $b\in T$, and $r\in R$ satisfying $r\prec f(a,b)$, there exist $a'\in S$ and $b'\in T$ such that $a'\prec a$, $b'\prec b$, and $r\leq f(a',b')$.
\item
If $a'\prec a$ and $b'\prec b$, then $f(a',b')\prec f(a,b)$, for any $a',a\in S$ and $b',b\in T$.
\end{enumerate}
We denote the set of all $\CatW$-bimorphisms by $\CatWBimor(S\times T,R)$.

If the $\CatPom$-bimorphism $f$ is only required to satisfy condition \enumCondition{1}, then we call it a \emph{generalized $\CatW$-bi\-morphism}.
We denote the collection of all generalized $\CatW$-bimorphisms by $\CatWGenBimor{S\times T,R}$.
\end{dfn}

%==========================================================================================
\begin{lma}
\label{prp:bimorPreW}
Let $S$, $T$ and $R$ be $\CatPreW$-semigroup{s}, and let $f\colon S\times T\to R$ be a $\CatPom$-bimorphism.
Then the following conditions are equivalent:
\beginEnumStatements
\item
The map $f$ is a generalized $\CatW$-bimorphism.
\item
In each variable, $f$ is a generalized $\CatW$-morphism.
\end{enumerate}
\end{lma}
\begin{proof}
To show that \enumStatement{1} implies \enumStatement{2}, fix an element $b\in T$ and consider the map $f_b\colon S\to R$ given by $f_b(a)=f(a,b)$
for $a\in S$.
This map is clearly a $\CatPom$-morphism.
In order to show that it is continuous, let $a\in S$ and $r\in R$ satisfy $r\prec f(a,b)$.
We need to find $a'\in S$ such that $a'\prec a$ and $r\leq f(a',b)$.

By assumption we can choose $a'\in S$ and $b'\in T$ such that
\[
a'\prec a,\quad
b'\prec b,\quad
r\leq f(a',b').
\]
Since $f(a',b')\leq f(a',b)$, we see that $a'$ has the desired properties.
The analogous result holds in the second variable.

To show that \enumStatement{2} implies \enumStatement{1}, let $a\in S$, $b\in T$, and $r\in R$ satisfy $r\prec f(a,b)$.
Since $R$ satisfies \axiomW{1}, we can choose $\tilde{r}\in R$ such that
\[
r\prec\tilde{r}\prec f(a,b).
\]
Since $f$ is continuous in the first variable, we can find $a'\in S$ with $a'\prec a$ and $\tilde{r}\leq f(a',b)$.
It follows $r\prec f(a',b)$.
Using that $f$ is continuous in the second variable, we obtain an element $b'\in T$ such that $b'\prec b$ and $r\leq f(a',b')$, as desired.
\end{proof}

%==========================================================================================
\begin{pgr}
\label{pgr:bimorPreW}
Let $S$ and $T$ be $\CatPreW$-semigroup{s}.
Let us show that there is a $\CatPom$-functor
\[
\CatWBimor(S\times T,\freeVar)\colon\CatPreW\to\CatPom.
\]
Given a $\CatPreW$-semigroup{} $R$, the set $\CatWBimor(S\times T,R)$ has a natural structure of a \pom{} when endowed with pointwise addition and order.
This defines an assignment from the objects in $\CatPreW$ to the objects in $\CatPom$.
Moreover, given $\CatPreW$-semigroup{s} $R$ and $R'$, we define a $\CatPom$-morphism
\[
\CatWBimor(S\times T,\freeVar)_{R,R'}\colon\CatWMor(R,R')\to \CatPomMor\left( \vphantom{A^A} \CatWBimor(S\times T,R),\CatWBimor(S\times T,R') \right),
\]
as follows:
A $\CatW$-morphism $f\in\CatWMor(R,R')$ is sent to the $\CatPom$-morphism
\[
\CatWBimor(S\times T,R)\to\CatWBimor(S\times T,R'),\quad
\tau\mapsto f\circ\tau,\quad
\txtFA \tau\in\CatWBimor(S\times T, R).
\]
It is straightforward to check that this defines a $\CatPom$-functor.
In \autoref{prp:tensProdPreW}, we show that the bimorphism functor is representable.
\end{pgr}

%==========================================================================================
\begin{pgr}[Auxiliary relation on (bi)morphism sets]
\index{terms}{auxiliary relation!on (bi)morphism sets}
\label{pgr:auxiliaryMor}
Let $S$, $T$ and $R$ be $\CatPreW$-semigroup{s}.
We define an auxiliary relation $\prec$ on the set $\CatWGenMor{S,T}$ of generalized $\CatW$-morphisms as follows:
Given $f,g\in\CatWGenMor{S,T}$, we set $f\prec g$ if and only if $a'\prec a$ implies $f(a')\prec g(a)$, for any $a',a\in S$.

Similarly, we define an auxiliary relation $\prec$ on the set $\CatWGenBimor{S\times T,R}$ of generalized $\CatW$-bimorphisms as follows:
Given $f,g\in\CatWGenBimor{S\times T,R}$, we set $f\prec g$ if and only if $a'\prec a$ and $b'\prec b$ implies $f(a',b')\prec f(a,b)$, for any $a',a\in S$ and $b',b\in T$.
We have
\[
\CatWMor(S,T)=\left\{ f\in\CatWGenMor{S,T} : f\prec f \right\},
\]
and
\[
\CatWBimor(S\times T,R)=\left\{ f\in \CatWGenBimor{S\times T,R} : f\prec f \right\}.
\]
In this way, we can think of the $\CatW$-(bi)morphisms as the `compact' generalized $\CatW$-(bi)morphisms.

It is clear that the auxiliary relation $\prec$ on $\CatWGenMor{S,T}$ satisfies \axiomW{3}.
In some cases, $\prec$ also satisfies \axiomW{1} and \axiomW{4}, but this seems not to be the case in general.
Thus, we warn the reader that the pair $(\CatWGenMor{S,T},\prec)$ is not necessarily a $\CatPreW$-semigroup.
The same remark applies to $(\CatWGenBimor{S\times T,R},\prec)$.
\end{pgr}

%------------------------------------------------------------------------------------------
The following \autoref{dfn:tensProdAuxRelHelping} and \autoref{lma:tensProdAuxRelHelping} are rather technical, but they contain the necessary details to define an auxiliary relation on the $\CatPom$-tensor product $S\otimes_{\CatPom}T$ of two $\CatPreW$-semigroups $S$ and $T$.
The idea is that for simple tensors in $S\otimes_{\CatPom}T$ we have $a'\otimes b'\prec a\otimes b$ whenever $a'\prec a$ and $b'\prec b$.

For the next definition, we need to recall some notation for the construction of tensor products in $\CatPom$ from \autoref{prp:tensorPom}.
Let $S$ and $T$ be $\CatPom$-semigroups.
We denote by $S^\times$ the submonoid of $S$ consisting of nonzero elements.

We consider the congruence relation $\cong$ from \autoref{pgr:tensorMonConstr} on the free abelian monoid $F:=\N[S^\times \times T^\times]$.
Recall that a congruence is (by definition) an additive equivalent relation.
Then $S\otimes_\CatMon T := F/\!\!\cong$ is the tensor product of the underlying monoids.
Further, recall the binary relation $\leq'$ on $F$ from \autoref{pgr:tensorPrePomConstr}.
Let $\leq$ be the relation on $F$ generated by $\cong$ and $\leq'$.
Then $\leq$ is a pre-order on $F$.

Recall from \autoref{pgr:tensorMonConstr} that for elements $a\in S^\times$ and $b\in T^\times$, we write $a\odot b$ for the generator in $F$ indexed by $(a,b)$.
For every $f\in F$, there exist a finite set $I$ and pairs $(a_i,b_i)\in S^\times\times T^\times$, for $i\in I$, satisfying $f=\sum_{i\in I}a_i\odot b_i$.
Note that we do not require that $a_i$ and $a_j$ are distinct for different indices $i$ and $j$.

%==========================================================================================
\begin{dfn}
\label{dfn:tensProdAuxRelHelping}
Let $S$ and $T$ be $\CatPreW$-semigroup{s}.
We define a relation $\lessdot$ on the free abelian monoid $F:=\N[S^\times\times T^\times]$ as follows:
For $f$ and $g=\sum_{j\in J} a_j\odot b_j$ in $F$ we set $f\lessdot g$ if and only if there exist a subset $J'\subset J$, and elements $a_j'\in S^\times$ and $b_j'\in T^\times$ for $j\in J'$ such that
\[
f \leq \sum_{j\in J'} a_j'\odot b_j',\quad\text{ and }\quad
a_j'\prec a_j,\, b_j'\prec b_j,
\]
for every $j\in J'$.
\end{dfn}

%==========================================================================================
\begin{lma}
\label{lma:tensProdAuxRelHelping}
Let $S$ and $T$ be $\CatPreW$-semigroups, and let $f,f',g,g',h\in F=\N[S^\times\times T^\times]$ .
Then:
\beginEnumStatements
\item
If $f\lessdot g$, then $f\leq g$.
\item \label{part2}
If $f'\leq f\lessdot g$, then $f'\lessdot g$.
\item
If $f\lessdot g\leq g'$, then $f\lessdot g'$.
\item
We have $0\lessdot f$.
\item
For each $g$ in $F$, there exists a sequence $(g_k)_{k\in \N}$ in $F$ such that $g_k\lessdot g_{k+1}\lessdot g$ for each $k\in \N$ and such that, for any $f\in F$ satisfying $f\lessdot g$, there exists an index $k\in \N$ with $f\leq g_k$.
\item
If $f\lessdot g$ and $f'\lessdot g'$, then $f+f'\lessdot g+g'$.
\item
If $f,g$ and $h$ in $F$ satisfy $h\lessdot f+g$, then there exist $f_0$ and $g_0$ in $F$ such that $h\leq f_0+g_0$, $f_0\lessdot f$, and $g_0\lessdot g$.
\end{enumerate}
\end{lma}
\begin{proof}
Statements \enumStatement{1}, \enumStatement{2}, \enumStatement{4} and \enumStatement{6} are straightforward to show.
To prove \enumStatement{7}, let $f,g,h\in F$ satisfy $h\lessdot f+g$.
Choose finite disjoint index sets $I_1$ and $I_2$, and elements $a_i\in S^\times$ and $b_i\in T^\times$ for $i\in I_1\cup I_2$ such that
\[
f= \sum_{i\in I_1} a_i\odot b_i,\quad
g= \sum_{i\in I_2} a_i\odot b_i.
\]
Since $h\lessdot f+g$, we can choose a subset $I'\subset I_1\cup I_2$, and elements $a_i'\in S^\times$ and $b_i'\in T^\times$ for $i\in I'$ such that
\[
h \leq \sum_{i\in I'} a_i'\odot b_i',\quad\text { and }\quad
a_i'\prec a_i,\,  b_i'\prec b_i
\]
for each $i\in I'$.
Set
\[
f_0 := \sum_{i\in I_1\cap I'} a_i'\odot b_i',\quad
g_0 := \sum_{i\in I_2\cap I'} a_i'\odot b_i'.
\]
Then it is easy to check that
\[
h\leq f_0+g_0,\quad
f_0\lessdot f,\quad
g_0\lessdot g,
\]
as desired.

Next, let us show \enumStatement{3}. To this end, let $f,g,g'\in F$ satisfy $f\lessdot g\leq g'$.
Since the relation $\leq$ is the transitive closure of the relation generated by $\rightarrow$, $\leftarrow$ and $\leq'$, it is enough to consider the cases where $g\rightarrow g'$, or $g\leftarrow g'$, or $g\leq' g'$.
We may assume that $f,g$ and $g'$ are nonzero.
Using statements \enumStatement{6} and \enumStatement{7}, it is furthermore enough to consider the following cases (recall that $a\odot b+c\odot b$ is not being identified with $(a+b)\odot c$):

Case 1:
Assume $g\rightarrow^0 g'$.
Choose elements $a\in S^\times$, $b\in T^\times$,  nonempty finite index sets $J$ and $K$, elements $a_j\in S^\times$, for $j\in J$, and elements $b_k\in T^\times$, for $k\in K$, such that
\[
a =\sum_{j\in J}a_j,\quad
b =\sum_{k\in K}b_k,\quad
g = a\odot b,\quad
g' = \sum_{j\in J, k\in K} a_j\odot b_k.
\]
Since $f\lessdot g$ we can choose elements $a'\in S^\times$ and $b'\in T^\times$ such that
\[
f \leq a'\odot b',\quad
a'\prec a,\quad
b'\prec b.
\]
Using that $S$ satisfies \axiomW{4} and that $a'\prec a = \sum_{j\in J}a_j$, we obtain elements $a_j'\in S$, for $j\in J$, such that $a'\leq \sum_{j\in J}a_j'$, and $a_j'\prec a_j$ for each $j\in J$.
Similarly, we obtain elements $b_k'\in T$, for $k\in K$, such that
$b'\leq \sum_{k\in K}b_k'$  and $b_k'\prec b_k$ for each $k\in K$.

Set $J':=\{j\in J : a_j'\neq 0\}$ and $K':=\{k\in K : b_k'\neq 0\}$.
Then it is easy to check that
\[
f \leq  \sum_{j\in J',k\in K'}a_j'\odot b_k'.
\]
This shows that $f\lessdot g'$, as desired.

Case 2:
Assume $g\prescript{0}{}{\leftarrow} g'$.
Choose elements $a\in S^\times$, $b\in T^\times$, nonempty finite index sets $J$ and $K$, elements $a_j\in S^\times$, for $j\in J$, and elements $b_k\in T^\times$, for $k\in K$, such that
\[
a =\sum_{j\in J}a_j,\quad
b =\sum_{k\in K}b_k,\quad
g = \sum_{j\in J, k\in K} a_j\odot b_k,\quad
g' = a\odot b.
\]
Since $f\lessdot g$ we can choose a subset $L\subset J\times K$, elements $a_{j,k}'\in S^\times$ and $b_{j,k}'\in T^\times$, for $(j,k)\in L$, such that
\[
f \leq \sum_{(j,k)\in L} a_{j,k}'\odot b_{j,k}'\]
and $a_{j,k}'\prec a_j$, $b_{j,k}'\prec b_k$ for each $(j,k)\in L$.

For each $j^\sharp\in J$ and $k^\sharp\in K$ set
\[
K_{j^\sharp} := \left\{ k\in K : (j^\sharp,k)\in L \right\},\quad
J_{k^\sharp} := \left\{ j\in J : (j,k^\sharp)\in L \right\}.
\]
Moreover, set
\[
J' := \left\{ j^\sharp\in J : K_{j^\sharp}\neq\emptyset \right\},\quad
K' := \left\{ k^\sharp\in K : J_{k^\sharp}\neq\emptyset \right\}.
\]
Let $j^\sharp\in J'$.
For each $k\in K_{j^\sharp}$, we have $a_{j^\sharp,k}'\prec a_{j^\sharp}$.
Using that $S$ satisfies \axiomW{1}, choose an element $a_{j^\sharp}'\in S$ such that $a_{j^\sharp}'\prec a_{j^\sharp}$ and  $a_{j^\sharp,k}'\leq a_{j^\sharp}'$ for each $k\in K_{j^\sharp}$.
Note that $a_{j^\sharp}'$ is necessarily nonzero.
Similarly, for each $k^\sharp\in K'$ choose an element $b_{k^\sharp}'\in T^\times$ such that $b_{k^\sharp}'\prec b_{k^\sharp}$ and $b_{j,k^\sharp}'\leq b_{k^\sharp}'$ for each $j\in J_{k^\sharp}$.

It follows
\[
f \leq  \sum_{(j,k)\in L}a_{j,k}'\odot b_{j,k}'
\leq' \sum_{(j,k)\in L}a_j'\odot b_k'.
\]
Set $a' := \sum_{j\in J'} a_j'$ and $b' := \sum_{k\in K'} b_k'$.
Since $S$ and $T$ satisfy \axiomW{3}, we get $a'\prec a$ and $b'\prec b$.
Then
\[
f \leq \sum_{(j,k)\in L}a_j'\odot b_k'
\leq \sum_{(j,k)\in J'\times K'}a_j'\odot b_k'
\cong a'\odot b'
\lessdot a\odot b=g'.
\]

By part \ref{part2} of the present lemma, we conclude that  $f\lessdot g'$, as desired.

Case 3:
Assume $g\leq^0 g'$.
Choose elements $a,\tilde{a}\in S^\times$ and $b,\tilde{b}\in T^\times$ with
\[
a\leq \tilde{a},\quad
b\leq \tilde{b},\quad
g = a\odot b,\quad
g' = \tilde{a}\odot \tilde{b}.
\]
Since $f\lessdot g$, we can choose elements $a'\in S^\times$ and $b'\in T^\times$ such that
\[
f \leq a'\odot b',\quad
a'\prec a,\quad
b'\prec b.
\]
Since $\prec$ is an auxiliary relation for $S$ and $T$, we deduce $a'\prec\tilde{a}$ and $b'\prec\tilde{b}$.
Therefore, we immediately get $f\lessdot g'$, as desired.

Finally, let us prove \enumStatement{5}.
Let $g=\sum_{i\in I} a_i\odot b_i\in F$.
Given $a\in S$, we write $a^\prec$ for the set $\left\{x\in S : x\prec a \right\}$, and similarly for elements in $T$.
Since $S$ satisfies \axiomW{1}, for each $i\in I$ we can choose a sequence $(a_{i,k})_{k\in\N}$ in $S$ that is cofinal in $a_i^\prec$ and such that $a_{i,k}\prec a_{i,k+1}$ for each $k$.
Similarly, for each $i$ we choose a $\prec$-increasing sequence $(b_{i,k})_{k\in\N}$ in $T$ that is cofinal in $b_i^\prec$.
Set
\[
I' := \left\{ i\in I : a_i^\prec\neq\{0\}, b_i^\prec\neq\{0\} \right\}.
\]
For $i\in I\setminus I'$ we have $a_{i,k}=0$ for each $k\in\N$ or $b_{i,k}=0$ for each $k\in\N$.
For $i\in I'$, we may assume that $a_{i,k}$ and $b_{i,k}$ are nonzero for each $k\in\N$.
Then we set
\[
g_k := \sum_{i\in I'} a_{i,k}\odot b_{i,k},
\]
which is an element of $F$.
We clearly have $g_k\lessdot g_{k+1}$, for each $k$.

Let $f\in F$ satisfy $f\lessdot g$.
We need to show that there is $n\in\N$ such that $f\leq g_n$.
Since $f\lessdot g$, we can choose a subset $J\subset I$, and elements $a_j'\in S^\times$ and $b_j'\in T^\times$ for $j\in J$ such that
\[
f \leq \sum_{j\in J} a_j'\odot b_j',\quad\text{ and }\quad
a_j'\prec a_j,\, b_j'\prec b_j,
\]
for each $j\in J$.
Note that $J$ is necessarily a subset of $I'$.

Since $S$ and $T$ satisfy \axiomW{1}, for each $j\in J$ we we choose indices $k(j),l(j)\in\N$ such that
\[
a_j'\leq a_{j,k(j)},\quad
b_j'\leq b_{j,l(j)}.
\]
Set
\[
n := \max\left\{ k(j), l(j) : j\in J \right\}.
\]
Then
\[
f \leq \sum_{j\in J} a_j'\odot b_j'
\leq \sum_{j\in J} a_{j,n}\odot b_{j,n}
\leq \sum_{i\in I'} a_{i,n}\odot b_{i,n}
= g_n,
\]
as desired.
\end{proof}

%------------------------------------------------------------------------------------------
For the next definition, recall that for $\CatPreW$-semigroups $S$ and $T$, and for an element $f\in F:=\N[S^\times\times T^\times]$, we denote the congruence class of $f$ in $S\otimes_\CatPom T=F/\!\!\cong$ by $[f]$.

%==========================================================================================
\begin{dfn}[Auxiliary relation on $S\otimes_{\CatPom}T$]
\index{terms}{auxiliary relation!on $S\otimes_{\CatPom}T$}
\label{dfn:tensProdAuxRel}
Let $S$ and $T$ be $\CatPreW$-semi\-groups, and let $\lessdot$ be the relation on $\N[S^\times\times T^\times]$ introduced in \autoref{dfn:tensProdAuxRelHelping}.

We let $\prec$ be the binary relation on the tensor product $S\otimes_\CatPom T$ of the underlying \pom{s} that is induced by $\lessdot$.
That is, given $x,y\in S\otimes_\CatPom T$ we set $x\prec y$ if and only if there exist representatives $f,g\in \N[S^\times\times T^\times]$ such that $x=[f]$, $y=[g]$, and $f\lessdot g$.
\end{dfn}

%==========================================================================================
\index{terms}{tensor product!in $\CatPreW$}
\begin{thm}
\label{prp:tensProdPreW}
\index{symbols}{$\otimes_{\CatPreW}$}
Let $S$ and $T$ be $\CatPreW$-semigroup{s}.
Let
\[
\omega\colon S\times T\to S\otimes_\CatPom T
\]
be the tensor product of the underlying \pom{s}, as constructed in \autoref{prp:tensorPom}.
Then the relation $\prec$ on $S\otimes_\CatPom T$ from \autoref{dfn:tensProdAuxRel} is an auxiliary relation and $(S\otimes_\CatPom T,\prec)$ is a $\CatPreW$-semigroup{}, denoted by $S\otimes_{\CatPreW} T$.
Moreover, the map $\omega$ is a $\CatW$-bimorphism.

Furthermore, for every $\CatPreW$-semigroup{} $R$, the following universal properties hold:
\beginEnumStatements
\item
For every (generalized) $\CatW$-bimorphism $f\colon S\times T\to R$, there exists a (generalized) $\CatW$-morphism $\tilde{f}\colon S\otimes_{\CatPreW} T\to R$ such that $f=\tilde{f}\circ\omega$.
\item
We have $g_1\circ\omega\leq g_2\circ\omega$ if and only if $g_1\leq g_2$, for any generalized $\CatW$-morphisms $g_1,g_2\colon S\otimes_{\CatPreW} T \to R$.
\item
We have $g_1\circ\omega\prec g_2\circ\omega$ if and only if $g_1\prec g_2$, for any generalized $\CatW$-morphisms $g_1,g_2\colon S\otimes_{\CatPreW} T \to R$.
\end{enumerate}
Thus, for every $\CatPreW$-semigroup $R$, we obtain the commutative diagram below.
In the top row, the assignment $w^\ast\colon f\mapsto f\circ\omega$ is an isomorphism of the sets of generalized $\CatW$-(bi)morphisms with their structure as \pom{s} and with their additional auxiliary relations from \autoref{pgr:auxiliaryMor}.
When restricting to $\CatW$-(bi)mor\-phisms, as in the bottom row of the diagram, the map $\omega$ induces a $\CatPom$-isomorphism $\omega^*$ of the respective $\CatW$-(bi)morphisms sets.
\[
\xymatrix@R=15pt@M+3pt{
\CatWGenMor{S\otimes_{\CatPreW} T,R} \ar[r]^{w^*}
& \CatWGenBimor{S\times T,R} \\
\CatWMor(S\otimes_{\CatPreW} T,R) \ar@{^{(}->}[u] \ar[r]^<<<<{w^*}
& \CatWBimor(S\times T,R). \ar@{^{(}->}[u]
}
\]
In particular, the pair ($S\otimes_{\CatPreW} T$, $\omega$) represents the bimorphism $\CatPom$-functor $\CatWBimor(S\times T,\freeVar)$.
\end{thm}
\begin{proof}
Set $F:=\N[S^\times\times T^\times]$, and let $\leq$ be the pre-order on $F$ introduced in \autoref{pgr:tensorPrePomConstr}.
Let $\sim$ be the binary relation on $F$ defined by $f\sim g$ if and only if $f\leq g$ and $g\leq f$.
Then $\sim$ is a congruence relation on $F$ and $S\otimes_\CatPom T = F/\!\!\sim$.
Given $f\in F$, we denote by $[f]$ the congruence class of $f$ in $S\otimes_\CatPom T$.
Let $\lessdot$ be the relation on $F$ from \autoref{dfn:tensProdAuxRelHelping}.

It follows from statements \enumStatement{2} and \enumStatement{3} in \autoref{lma:tensProdAuxRelHelping} that $\lessdot$ only depends on the $\sim$-equivalence class of elements in $F$.
Thus, for every $x,y\in S\otimes_\CatPom T$ the following are equivalent:
\beginEnumStatements
\item
We have $x\prec y$ in the sense of \autoref{dfn:tensProdAuxRel}, that is, there are $f,g\in F$ such that $x=[f]$, $y=[g]$ and $f\lessdot g$.
\item
For each $f,g\in F$ satisfying $x=[f]$ and $y=[g]$, we have $f\lessdot g$.
\end{enumerate}

It follows easily from \enumStatement{1}-\enumStatement{4} in \autoref{lma:tensProdAuxRelHelping} that $\prec$ is an auxiliary relation on $S\otimes_\CatPom T$.
Moreover, statements \enumStatement{5}-\enumStatement{7} in \autoref{lma:tensProdAuxRelHelping} imply that $(S\otimes_\CatPom T,\prec)$ satisfies \axiomW{1}, \axiomW{3} and \axiomW{4}, showing that it is a $\CatPreW$-semigroup{}, denoted by $S\otimes_{\CatPreW} T$.

Let us show that the map
\[
\omega\colon S\times T\to S\otimes_\CatPreW T
\]
is a $\CatW$-bimorphism.
It is clear that $\omega$ is a $\CatPom$-bimorphism respecting the auxiliary relations.
Thus, it remains to show that $\omega$ is continuous in the sense of \autoref{dfn:bimorPreW}.
To this end, let $a\in S$, $b\in T$, and $x\in S\otimes_\CatPreW T$ satisfy $x\prec \omega(a,b)$.
We may assume that $a,b$, and $x$ are nonzero.
Then $f= a\odot b$ is an element in $F$ such that $\omega(a,b) = [f]$.
By the equivalent conditions stated above, and since $x\prec [f]$, we can choose $f'\in F$ such that $x\leq [f']$ and $f'\lessdot f$.
Therefore, by definition of the relation $\lessdot$ and since $x$ is nonzero, we can choose elements $a'\in S^\times$ and $b'\in T^\times$ such that
\[
f'\leq a'\odot b',\quad
a'\prec a,\quad
b'\prec b.
\]
Then
\[
x\leq [f']\leq [a'\odot b'] = \omega(a',b'),
\]
showing that $a'$ and $b'$ have the desired properties to verify the continuity of $\omega$.

Let $\mathfrak{F}\colon\CatPreW\to\CatPom$ be the forgetful functor, which associates to a $\CatPreW$-semigroup{} $X$ the underlying \pom\ (also denoted by $X$, by abuse of notation).
It is clear that  $\mathfrak{F}$ is faithful.

Given $\CatPreW$-semigroup{s} $X,Y$, and $Z$, the functor $\mathfrak{F}$ induces maps
\begin{align*}
\mathfrak{F}_{X,Y} &\colon \CatWGenMor{X,Y} \to \CatPomMor(X,Y), \\
\mathfrak{F}_{X\times Y,Z} &\colon \CatWGenBimor{X\times Y, Z} \to \CatPomBimor(X\times Y, Z),
\end{align*}
by mapping a generalized $\CatW$-(bi)morphism to the same map considered as a $\CatPom$-(bi)morphism. It is clear that $\mathfrak{F}_{X,Y}$ and $\mathfrak{F}_{X\times Y,Z}$ are order embeddings when the respective (bi)morphism sets are equipped with their natural structure as \pom{s}.

To check the universal properties, let $R$ be a $\CatPreW$-semigroup{}.
Consider the map
\[
\Omega_R\colon\CatPomMor(S\otimes_{\CatPom} T,R)\to\CatPomBimor(S\times T,R),\quad
f\mapsto f\circ\omega.
\]
Since $S\otimes_{\CatPom} T$ and $\omega$ have the universal property of a tensor product in $\CatPom$, the map $\Omega_R$ is an isomorphism of the (bi)morphism sets with their structure as objects in $\CatPom$;
see \autoref{prp:tensorPom}.
In particular, $\Omega_R$ is an order-embedding.

Since $\omega$ is also a $\CatW$-bimorphism, the same assignment maps (generalized) $\CatW$-morphisms to (generalized) $\CatW$-bimorphisms.
We denote this map by
\[
\Phi_R\colon \CatWGenMor{S\otimes_{\CatPreW} T,R} \to\CatWGenBimor{S\times T,R},\quad
f\mapsto f\circ\omega.
\]

We have a commutative diagram of $\CatPom$-morphisms:
\[
\xymatrix@R=15pt@M+3pt{
\CatPomMor(S\otimes_{\CatPom} T,R) \ar[r]^{\Omega_R}
& \CatPomBimor(S\times T,R) \\
\CatWGenMor{S\otimes_{\CatPreW} T,R}  \ar[r]^{\Phi_R} \ar@{^{(}->}[u]^{\mathfrak{F}_{S\otimes T,R}}
& \CatWGenBimor{S\times T,R} \ar@{^{(}->}[u]_{\mathfrak{F}_{S\times T,R}} \\
}
\]
Since $\mathfrak{F}_{S\otimes T,R}$ and $\mathfrak{F}_{S\times T,R}$ are order-embeddings, and since the map $\Omega_R$ is a $\CatPom$-iso\-morphism, it follows that $\Phi_R$ is an order-embedding.
This shows the universal property \enumStatement{2}.

For \enumStatement{1}, we need to show that $\Phi_R$ is surjective.
Thus, let
\[
f\colon S\times T \to R
\]
be a generalized $\CatW$-bimorphism.
Considering $f$ as a $\CatPom$-bimorphism and using that $\Omega_R$ is an isomorphism, there exists a unique $\CatPom$-morphism
\[
\tilde{f}\colon S\otimes_{\CatPom}T\to R
\]
such that $f=\tilde{f}\circ\omega$.
We need to show that $\tilde{f}$ is continuous.

Let $x\in S\otimes_{\CatPreW} T$ and $r\in R$ satisfy $r\prec \tilde{f}(x)$.
We need to find $x'\in S\otimes_{\CatPreW} T$ such that $x'\prec x$ and $r\leq \tilde{f}(x')$.
Choose a finite index set $I$, and elements $a_i\in S$ and $b_i\in T$, for $i\in I$, such that
$x=\sum_{i\in I} a_i\otimes b_i$.
Then
\[
r\prec \tilde{f}(x)
=\sum_{i\in I} \tilde{f}(a_i\otimes b_i).
\]
Using that $R$ satisfies \axiomW{4}, choose elements $r_i$ in $R$, for $i\in I$, such that
\[
r\leq\sum_{i\in I} r_i,\quad\text{ and }\quad
r_i\prec \tilde{f}(a_i\otimes b_i)=f(a_i,b_i)\]
for each $i\in I$.

Since $f$ is continuous, for each $i\in I$ we can choose $a_i'\in S$ and $b_i'\in T$ with
\[
a_i'\prec a_i,\quad
b_i'\prec b_i,\quad
r_i\leq f(a_i',b_i').
\]
Set $x':=\sum_{i\in I} a_i'\otimes b_i'$.
Then $x'\prec x$ and
\[
r\leq\sum_{i\in I} r_i
\leq\sum_{i\in I} f(a_i',b_i')
=\tilde{f}(x'),
\]
as desired.

Finally, to prove \enumStatement{3}, we need to show that for any generalized $\CatW$-morphisms $f$ and $g$ in $\CatWGenMor{S\otimes_\CatPreW T,R}$, we have $f\prec g$ if and only if $\Phi_R(f)\prec\Phi_R(g)$.
This is left to the reader.
\end{proof}

%==========================================================================================
\begin{pgr}[$\CatPreW$ is a symmetric, monoidal category]
\label{pgr:monoidalPreW}
It is straightforward to check that the bimorphism-functor on $\CatPreW$ from \autoref{pgr:bimorPreW} is also functorial in the first entries.
Thus, we have a $\CatPom$-multifunctor
\[
\CatWBimor(\freeVar\times\freeVar,\freeVar)\colon\CatPreW^\op\times\CatPreW^\op\times\CatPreW\to\CatPom.
\]
By \autoref{prp:tensProdPreW}, the tensor product of two $\CatPreW$-semigroup{s} exists.
Therefore, as explained in \autoref{pgr:functorialityTensor}, it follows that the tensor product in $\CatPreW$ induces a $\CatPom$-bifunctor
\[
\otimes\colon \CatPreW\times\CatPreW\to\CatPreW.
\]
We use this to define a monoidal structure on $\CatPreW$.

Recall that the unit object of $\CatPom$ is given by $\N$ with its usual structure as an algebraically ordered monoid.
We equip $\N$ with the auxiliary relation that is equal to the partial order.
Then $\N$ is a $\CatPreW$-semigroup{}.

Let $S$ be a $\CatPreW$-semigroup{}.
Since $\N$ is the unit object of $\CatPom$, there are natural isomorphisms:
\[
\N \otimes_\CatPom S \cong S \cong S\otimes_\CatPom \N.
\]
It is straightforward to check that these isomorphisms preserve the auxiliary relations and are therefore isomorphisms in $\CatPreW$.
Thus, $\N$ is the unit object in $\CatPreW$.
In the same way, associativity and symmetry of the tensor product in $\CatPreW$ follow from the respective properties in $\CatPom$.
\end{pgr}

\vspace{5pt}
%------------------------------------------------------------------------------------------
%==========================================================================================
\section{The tensor product in \texorpdfstring{$\CatCu$}{Cu}}
\label{sec:tensProdCu}

%------------------------------------------------------------------------------------------
In this section, we will use the construction of tensor products in $\CatPreW$ and the fact that $\CatCu$ is a reflective subcategory of $\CatPreW$ to show that the category $\CatCu$ has a symmetric, monoidal structure.

Before we make this concrete in \autoref{prp:tensProdCu}, let us consider the natural notion of bimorphisms in the category $\Cu$; see \cite[Definition~4.3]{AntBosPer13CuFields}.

%==========================================================================================
\begin{dfn}
\label{dfn:bimorCu}
\index{terms}{Cu-bimorphism@$\CatCu$-bimorphism}
\index{terms}{Cu-bimorphism@$\CatCu$-bimorphism!generalized}
Let $S, T$, and $R$ be $\CatCu$-semigroup{s}, and let $f\colon S\times T\to R$ be a $\CatPom$-bimorphism.
We say that $f$ is a \emph{$\CatCu$-bimorphism} if it satisfies the following conditions:
\beginEnumConditions
\item
We have $\sup_k f(a_k,b_k) = f(\sup_k a_k, \sup_k b_k)$, for any increasing sequences $(a_k)_k$ in $S$ and $(b_k)_k$ in $T$.
\item
If $a'\ll a$ and $b'\ll b$, then $f(a',b')\ll f(a,b)$, for any $a',a\in S$ and $b',b\in T$.
\end{enumerate}
We denote the set of all $\CatCu$-bimorphisms by $\CatCuBimor(S\times T,R)$.

If $f$ is only required to satisfy condition \enumCondition{1} then we call it a \emph{generalized $\CatCu$-bimorphism}.
We denote the collection of all generalized $\CatCu$-bimorphisms by $\CatCuGenBimor{S\times T,R}$.
\end{dfn}

%------------------------------------------------------------------------------------------
The next result is the analog of \autoref{prp:CuWMor} for (generalized) $\CatCu$-bimorphisms.
It shows that for $\CatCu$-semigroups, the notions of (generalized) $\CatW$-bimorphism and (generalized) $\CatCu$-bimorphism agree.

%==========================================================================================
\begin{lma}
\label{prp:bimorCu}
Let $S$, $T$, and $R$ be $\CatCu$-semigroups, and let $f\colon S\times T\to R$ be a $\CatPom$-bimorphism.
Then the following are equivalent:
\beginEnumStatements
\item
The map $f$ is a (generalized) $\CatCu$-bimorphism.
\item
In each variable, $f$ is a (generalized) $\CatCu$-morphism.
\item
The map $f$ is a (generalized) $\CatW$-bimorphism.
\end{enumerate}
\end{lma}
\begin{proof}
The equivalence of \enumStatement{2} and \enumStatement{3} follows by combining \autoref{prp:bimorPreW}, \autoref{prp:CuWMor} and the fact that $\CatCu$ is a full subcategory of $\CatW$ (see \autoref{pgr:CufullW}).
Moreover, it is clear that \enumStatement{1} implies \enumStatement{2}, and the converse is straightforward to show.
\end{proof}

%------------------------------------------------------------------------------------------
Given $\CatCu$-semigroup{s} $S$ and $T$, the $\CatCu$-bimorphisms from $S\times T$ to a $\CatCu$-semigroup $R$ form a \pom{} under pointwise order and addition. Thus there is a $\CatPom$-functor
\[
\CatCuBimor(S\times T,\freeVar)\colon \CatCu\to\CatPom.
\]
By \autoref{prp:bimorCu}, this functor is just the restriction of the $\CatPom$-functor $\CatWBimor(S\times T,\freeVar)$ from \autoref{pgr:bimorPreW} to the full subcategory $\CatCu$ of $\CatPreW$.
In the next result, we will show that this bimorphism functor is representable.

Given $\CatCu$-semigroups $S,T$, and $R$, we equip the sets $\CatCuGenMor{S,R}$ and $\CatCuGenBimor{S\times T,R}$ with the same auxiliary relation as defined in \autoref{pgr:auxiliaryMor}.
For example, for $f,g\in\CatCuGenMor{S,R}$, we set $f\prec g$ if and only if $a'\ll a$ implies $f(a')\ll g(a)$, for any $a',a\in S$.

%==========================================================================================
\begin{thm}
\label{prp:tensProdCu}
\index{terms}{tensor product!in $\CatCu$}
\index{symbols}{$\otimes_{\CatCu}$}
Let $S$ and $T$ be $\CatCu$-semigroups.
Consider the tensor product
\[
\omega\colon S\times T\to S\otimes_{\CatPreW}T,
\]
as constructed in \autoref{prp:tensProdPreW}.
By applying the completion functor $\gamma$ from \autoref{prp:CuificationExists}, we obtain a $\CatCu$-semigroup{} $\gamma(S\otimes_{\CatPreW}T)$, which we denote by $S\otimes_{\CatCu}T$, and a universal $\CatW$-morphism $\alpha\colon S\otimes_{\CatPreW}T\to S\otimes_{\CatCu}T$.
Then the composed map
\[
\varphi:=\alpha\circ\omega\colon S\times T \xrightarrow{\omega} S\otimes_{\CatPreW}T
\xrightarrow{\alpha} \gamma(S\otimes_{\CatPreW}T)
= S\otimes_{\CatCu}T
\]
is a $\CatCu$-bimorphism.
For every $\CatCu$-semigroup{} $R$, it satisfies the following universal properties:
\beginEnumStatements
\item
For every (generalized) $\CatCu$-bimorphism $f\colon S\times T\to R$, there exists a (generalized) $\CatCu$-morphism $\tilde{f}\colon S\otimes_{\CatCu} T\to R$ such that $f=\tilde{f}\circ\varphi$.
\item
We have $g_1\circ\varphi\leq g_2\circ\varphi$ if and only if $g_1\leq g_2$, for any generalized $\CatCu$-morphisms $g_1,g_2\colon S\otimes_{\CatCu} T \to R$.
\item
We have $g_1\circ\varphi\prec g_2\circ\varphi$ if and only if $g_1\prec g_2$, for any generalized $\CatCu$-morphisms $g_1,g_2\colon S\otimes_{\CatCu} T \to R$.
\end{enumerate}
Thus, for every $\CatCu$-semigroup $R$, we obtain the commutative diagram below, which is analogous to that in \autoref{prp:tensProdPreW}.
\[
\xymatrix@R=15pt@M+3pt{
\CatCuGenMor{S\otimes_{\CatCu} T,R} \ar[r]^{\varphi^*}
& \CatCuGenBimor{S\times T,R} \\
\CatCuMor(S\otimes_{\CatCu} T,R) \ar@{^{(}->}[u] \ar[r]^<<<<{\varphi^*}
& \CatCuBimor(S\times T,R), \ar@{^{(}->}[u]
}
\]
where $\varphi^*$ is an isomorphism of the respective (bi)morphism sets, given by $\varphi^*(f)=f\circ\varphi$.
In particular, the pair ($S\otimes_{\CatCu} T$, $\varphi$) represents the bimorphism $\CatPom$-functor $\CatCuBimor(S\times T,\freeVar)$.
\end{thm}
\begin{proof}
It is clear that $\varphi$ is a $\CatW$-bimorphism. Therefore, it follows from \autoref{prp:bimorCu} that it is a $\CatCu$-bimorphism.
To check the universal properties, let $R$ be a $\CatCu$-semigroup{}.
In the diagram below, the horizontal maps on the left (given by composing on the right with $\alpha$) are $\CatPom$-isomorphisms by \autoref{thm:Cuification}  and the horizontal maps on the right (given by composing on the right with $\omega$) are $\CatPom$-isomorphisms by \autoref{prp:tensProdPreW}.
\[
\xymatrix@R=15pt@M+3pt{
\CatCuGenMor{S\otimes_{\CatCu} T,R} \ar[r]^{\alpha^*}
& \CatWGenMor{S\otimes_{\CatPreW}T,R} \ar[r]^-{\omega^*}
& \CatWGenBimor{S\times T,R} \\
\CatCuMor(S\otimes_{\CatCu} T,R) \ar@{^{(}->}[u] \ar[r]^-{\alpha^*}
& \CatWMor(S\otimes_{\CatPreW}T,R) \ar@{^{(}->}[u] \ar[r]^-{\omega^*}
& \CatWBimor(S\times T,R). \ar@{^{(}->}[u]
}
\]
This establishes the universal properties \enumStatement{1} and \enumStatement{2}.
It is also straightforward to check that the isomorphism between $\CatCuGenMor{S\otimes_{\CatCu}T,R}$ and $\CatCuGenBimor{S\times T, R}$ preserves the auxiliary relations, which establishes \enumStatement{3}.
\end{proof}

%==========================================================================================
\begin{rmk}[Tensor product in $\CatW$]
\index{terms}{tensor product!in $\CatW$}
\label{rmk:tensProdW}
Analogous to the above \autoref{prp:tensProdCu}, one can construct tensor products in the category $\CatW$.
Given $\CatW$-semigroup{s} $S$ and $T$, one first considers the tensor product  $S\otimes_{\CatPreW}T$ in $\CatPreW$.
Then one uses the reflection $\CatPreW\to\CatW$ from \autoref{pgr:WreflectivePreW} to obtain the tensor product in $\CatW$.
\end{rmk}

%==========================================================================================
\begin{thm}
\label{thm:tensProdCompl}
Let $S$ and $T$ be $\CatPreW$-semigroups.
Then there is a natural $\CatCu$-isomorphism
\[
\gamma(S)\otimes_{\CatCu}\gamma(T)\cong\gamma(S\otimes_{\CatPreW}T).
\]
\end{thm}
\begin{proof}
Let $R$ be a $\CatCu$-semigroup{}.
Using \autoref{prp:tensProdPreW} at the first step, and using that $\CatCu$ is a reflective subcategory of $\CatPreW$ at the second step, we have natural isomorphisms of the following (bi)morphism sets
\[
\CatWBimor(S\times T,R)\cong \CatWMor(S\otimes_{\CatPreW} T,R)\cong \CatCuMor(\gamma(S\otimes_{\CatPreW} T),R).
\]
On the other hand, using \autoref{prp:bimorCu} at the first step, and \autoref{prp:tensProdCu} at the second step, we obtain natural isomorphisms
\[
\CatWBimor(S\times T,R)
\cong\CatCuBimor(\gamma(S)\times\gamma(T),R)
\cong\CatCuMor(\gamma(S)\otimes_{\CatCu}\gamma(T),R).
\]
Hence, the $\CatCu$-semigroups $\gamma(S)\otimes_{\CatCu}\gamma(T)$ and $\gamma(S\otimes_{\CatPreW} T)$ both represent the same functor, which implies that they are naturally isomorphic.
\end{proof}

%==========================================================================================
\index{terms}{tensor product!in $\CatCu$!associative}
\begin{cor}
\label{prp:tensProdAssoc}
Let $S,T$, and $R$ be $\CatCu$-semigroup{s}.
Then there is a natural isomorphism:
\[
S\otimes_{\CatCu} (T\otimes_{\CatCu} R)\ \cong\ (S\otimes_{\CatCu} T)\otimes_{\CatCu} R.
\]
\end{cor}
\begin{proof}
Using \autoref{thm:tensProdCompl} at the second and last step, and that $\otimes_{\CatPreW}$ is associative (see \autoref{pgr:monoidalPreW}) at the third step, we obtain
\begin{align*}
S\otimes_{\CatCu} (T\otimes_{\CatCu} R)
&\cong\gamma(S)\otimes_{\CatCu}\gamma(T\otimes_{\CatPreW}R) \\
&\cong\gamma(S\otimes_{\CatPreW}(T\otimes_{\CatPreW}R)) \\
&\cong\gamma((S\otimes_{\CatPreW}T)\otimes_{\CatPreW}R)
\cong (S\otimes_{\CatCu}T)\otimes_{\CatCu}R,
\end{align*}
and all isomorphisms are natural.
\end{proof}

%==========================================================================================
\begin{pgr}[$\CatCu$ is a symmetric, monoidal category]
\label{pgr:monoidalCu}
Similar as in \autoref{pgr:monoidalPreW}, it follows that tensor product in $\CatCu$ extends to a bifunctor
\[
\otimes\colon \CatCu\times\CatCu \to \CatCu.
\]
We showed in \autoref{prp:tensProdAssoc} that this functor is associative.
Let us show that the $\CatCu$-semigroup $\overline{\N}$ is a unit object for $\CatCu$.
Note that $\overline{\N}$ is the reflection in $\CatCu$ of the unit object $\N$ of $\CatPreW$.
Let $S$ be a $\CatCu$-semigroup $S$.
Using \autoref{thm:tensProdCompl} at the first step, we obtain natural isomorphisms
\[
S\otimes_{\CatCu}\overline{\N}
\cong\gamma(S\otimes_{\CatPreW}\N)
\cong\gamma(S)\cong S,
\]
and analogously $\overline{\N}\otimes_{\CatCu}S\cong S$.
Similarly, symmetry of the tensor product in $\CatCu$ follows from symmetry of the tensor product in $\CatPreW$.
Thus, the category $\CatCu$ has a symmetric, monoidal structure.
\end{pgr}

\vspace{5pt}
%------------------------------------------------------------------------------------------
%==========================================================================================
\section{Examples and Applications}
\label{sec:bimorAppl}

%------------------------------------------------------------------------------------------
In this subsection, we are mainly concerned with the following problems:
Under which conditions do the axioms \axiomO{5}, \axiomO{6} and weak cancellation pass to tensor products of $\CatCu$-semigroups; see \autoref{prbl:AxiomsTensProd}.
Secondly, for \ca{s} $A$ and $B$, what can we say about the natural $\CatCu$-morphism from $\Cu(A)\otimes_{\CatCu}\Cu(B)$ to $\Cu(A\tensMax B)$; see \autoref{prbl:tensCa}.

The following result is a useful tool to solve particular cases of both problems.

%==========================================================================================
\index{terms}{tensor product!in $\CatCu$!continuous}
\index{terms}{tensor product!in $\CatPreW$!continuous}
\begin{prp}
\label{prp:tensLim}
The tensor products in $\CatPreW$ and $\CatCu$ are continuous in each variable.
More precisely, let $((S_i)_{i\in I},(\varphi_{i,j})_{i,j\in I,i\leq j})$ be an inductive system in $\CatCu$, and let $T$ be a $\CatCu$-semigroup{}.
Then there is a natural isomorphism
\[
\CatCuLim (S_i\otimes_{\CatCu} T)
\cong (\CatCuLim S_i)\otimes_{\CatCu} T.
\]
The analogous statement holds for the second variable and for the tensor product in $\CatPreW$.
\end{prp}
\begin{proof}
We first note that the tensor product in $\CatPom$ is a continuous functor in each variable, since
 $\freeVar\otimes_\CatPom T$ is left adjoint to $\CatPomMor(T,\freeVar)$.
To simplify notation, in the first part of this proof we will write $\otimes$ for $\otimes_{\CatPreW}$ and $\varinjlim$ for $\CatPreWLim$.

Let $((S_i)_{i\in I},(\varphi_{i,j})_{i,j\in I, i\leq j})$ be an inductive system in $\CatPreW$ indexed over the directed set $I$, and let $T$ be a $\CatPreW$-semigroup{}.
This induces an inductive system $((S_i\otimes T)_{i\in I}, (\varphi_{i,j}\otimes\id_T)_{i,j\in I, i\leq j})$.
For $j\in I$, we denote the respective $\CatW$-morphisms into the inductive limits by
\[
\varphi_{j,\infty}\colon S_j\to \varinjlim S_i,\text{ and }
\lambda_{j,\infty}\colon S_j\otimes T\to \varinjlim (S_i\otimes T).
\]
The $\CatW$-morphisms $\varphi_{j,\infty}\otimes\id_T$
induce a $\CatW$-morphism
\[
\psi\colon\varinjlim (S_i\otimes T) \to (\varinjlim S_i)\otimes T
\]
that satisfies $(\psi\circ\lambda_{j,\infty})(s\otimes t)=\varphi_{j,\infty}(s)\otimes t$ for every $j\in I$, $s\in S_j$, and $t\in T$.
These maps are shown in the following commutative diagram:
\[
\xymatrix{
S_j\otimes T \ar[d]_{\lambda_{j,\infty}} \ar[r]^-{\varphi_{j,\infty}\otimes\id_T}
& { (\varinjlim S_i)\otimes T } \\
{ \varinjlim (S_i\otimes T) } \ar[ur]_{\psi}
}
\]

The inductive limit of $\CatPreW$-semigroup{s} is simply the inductive limit in $\CatPom$ of the underlying \pom{s} equipped with a natural auxiliary relation; see \autoref{prp:limitsPreW}.
Similarly, the tensor product of two $\CatPreW$-semigroup{s} is the $\CatPom$-tensor product of the underlying \pom{s} equipped with a natural auxiliary relation; see \autoref{prp:tensProdPreW}.

Thus, since the tensor product in $\CatPom$ is continuous in each variable, the map $\psi$ is a $\CatPom$-isomorphism.
Moreover, $\psi$ preserves the auxiliary relation since it is a $\CatW$-morphism.
Hence, to show that $\psi$ is a $\CatW$-isomorphism, it remains to prove that $x\prec y$ whenever $\psi(x)\prec\psi(y)$ for any $x$ and $y$ in the domain of $\psi$.

Given such $x$ and $y$, we can choose an index $i\in I$ and $n\in\N$ and elements $s_k\in S_i$ and $t_k\in T$ for $k=1,\ldots,n$ such that $y=\lambda_{i,\infty}(\sum_{k=1}^n s_k\otimes t_k)$.
Then
\[
\psi(x)\prec\psi(y)=\sum_{k=1}^n\varphi_{i,\infty}(s_k)\otimes t_k.
\]
By definition of the auxiliary relation for tensor products in $\CatPreW$, there are elements $a_k'\in\varinjlim S_i$ and $t_k'\in T$ with $\psi(x)\leq\sum_{k=1}^n a_k'\otimes t_k'$, satisfying
\[
a_k'\prec\varphi_{i,\infty}(s_k), \text{ and } t_k'\prec t_k
\]
for each $k=1,\ldots,n$.
It follows from the definition of the auxiliary relation for inductive limits in $\CatPreW$, that there is an index $j\geq i$ and elements $s_k'\in S_j$ for $k=1,\ldots, n$ such that
\[
a_k'=\varphi_{j,\infty}(s_k'),\text{ and }
s_k'\prec\varphi_{i,j}(s_k)
\]
for each $k=1,\ldots,n$.
Set $y':=\lambda_{j,\infty}(\sum_{k=1}^n s_k'\otimes t_k')$.
Then
\[
\psi(x)
\leq\sum_{k=1}^n a_k'\otimes t_k'
=\sum_{k=1}^n\varphi_{j,\infty}(s_k')\otimes t_k'
=\psi(y').
\]
Since $\psi$ is an order-embedding, we have $x\leq y'$.
It easily follows $y'\prec y$, and thus $x\prec y$, as desired.

Continuity in the second variable is proven analogously.

The result for tensor products in $\CatCu$ follows from that for $\CatPreW$.
More precisely, let $((S_i)_{i\in I},(\varphi_{i,j})_{i,j\in I, i\leq j})$ be an inductive system in $\CatCu$, and let $T$ be a $\CatCu$-semigroup{}.
Using \autoref{prp:limitsCu} at the first and last steps, using \autoref{thm:tensProdCompl} at the second to last step, and using the result for $\CatPreW$ at the second step, we obtain natural $\CatCu$-isomorphisms:
\begin{align*}
\CatCuLim (S_i\otimes_{\CatCu} T)
&\cong \gamma\left( \CatPreWLim (S_i\otimes_{\CatPreW} T) \right) \\
&\cong \gamma\left( (\CatPreWLim S_i)\otimes_{\CatPreW} T \right) \\
&\cong \gamma\left( \CatPreWLim S_i \right) \otimes_{\CatCu} \gamma(T)
\cong (\CatCuLim S_i) \otimes_{\CatCu} T.
\end{align*}
This finishes the proof.
\end{proof}

%==========================================================================================
\begin{prbl}
\label{prbl:AxiomsTensProd}
Given $\CatCu$-semigroups $S$ and $T$ that satisfy \axiomO{5} (respectively \axiomO{6}, weak cancellation).
When does $S\otimes_\CatCu T$ satisfy \axiomO{5} (respectively \axiomO{6}, weak cancellation)?
\end{prbl}

%==========================================================================================
\begin{pgr}
\label{pgr:answersAxiomsTensProd}
In general, \axiomO{5} does not pass to tensor products; see \autoref{prp:tensornotO5}.
However, for given $\CatCu$-semigroups $S$ and $T$, we obtain the following partial positive answers to \autoref{prbl:AxiomsTensProd}:
\beginEnumStatements
\item
If $S$ or $T$ is an inductive limit of simplicial $\CatCu$-semigroups, then each of the axioms  \axiomO{5}, \axiomO{6} and weak cancellation pass from $S$ and $T$ to $S\otimes_{\CatCu}T$; see \autoref{prp:tensLimSimplicial}.
\item
If $S$ and $T$ are algebraic $\CatCu$-semigroups, then axiom \axiomO{5} passes from $S$ and $T$ to $S\otimes_{\CatCu}T$; see \autoref{prp:axiomsTensProdAlgebraic}.
\end{enumerate}

This suggests the following refined version of \autoref{prbl:AxiomsTensProd}:
Does \axiomO{5} pass to tensor products where one of the $\CatCu$-semigroups is algebraic?
Do the axioms pass to tensor products of simple $\CatCu$-semigroups?
\end{pgr}

%------------------------------------------------------------------------------------------
For the next result, we use $Z$ to denote the semigroup $\N\sqcup(0,\infty]$. Note that $Z$ is the Cuntz semigroup of the Jiang-Su algebra $\mathcal{Z}$; see \autoref{pgr:Z}.
We let $\Lsc\left( [0,1],\overline{\N} \right)$ denote the set of lower-semicontinuous functions from $[0,1]$ to $\overline{\N}$, which is known to be isomorphic to the Cuntz semigroup of the \ca{} $C([0,1])$;
see \cite{Rob13CuSpaces2D}.
It follows from \autoref{prp:o5o6} that $Z$ and $\Lsc\left( [0,1],\overline{\N} \right)$ satisfy \axiomO{5}, but this is also easy to see directly.
The next result shows that \axiomO{5} does in general not pass to tensor products. It requires the use of $Z$-multiplication, a notion that will be defined in \autoref{sec:solidCuSrg} (see also \autoref{sec:ZMod}).
\index{symbols}{Z@$Z$ \quad (Cuntz semigroup of Jiang-Su algebra)}

%==========================================================================================
\begin{prp}
\label{prp:tensornotO5}
The $\CatCu$-semigroup $Z\otimes_{\CatCu}\Lsc\left( [0,1],\overline{\N} \right)$ does not satisfy axiom \axiomO{5}.
\end{prp}
\begin{proof}
Consider the $\CatCu$-semigroup $S=\Lsc([0,1],Z)$ which clearly has $Z$-mul\-ti\-pli\-ca\-tion.
Let $S_{00}$ be the smallest submonoid of $S$ containing all elements of the form $z\cdot f$ for $z\in Z$ and $f\in\Lsc\left( [0,1],\overline{\N} \right)$.
Let us verify that the assumptions of \autoref{lem:density} are satisfied for $S_{00}$.
Given $z\in Z$ and $f\in\Lsc\left( [0,1],\overline{\N} \right)$, we choose rapidly increasing sequences $(z_n)_n$ in $Z$ and $(f_n)_n$ in $\Lsc\left( [0,1],\overline{\N} \right)$ such that $z=\sup_n z_n$ and $f=\sup_n f_n$.
Then, as $z_n\ll z_{n+1}$ and $f_n\ll f_{n+1}$ for all $n$, the sequence $(z_n\cdot f_n)_n$ is rapidly increasing and
\[
z\cdot f =\sup_n (z_n\cdot f_n).
\]
Applying \autoref{lem:density}, we obtain that the sup-closure $S_0$ of $S_{00}$ in $S$, denoted by
\[
S_0=\overline{\text{span}}\left\{ z\cdot f : z\in Z, f\in\Lsc\left( [0,1],\overline{\N} \right) \right\},
\]
is a $\CatCu$-semigroup.
We first show that $S_0$ does not satisfy \axiomO{5}.

Given an open set $U$ in $[0,1]$, we denote by $1_U$ the indicator function of $U$, which is an element in $\Lsc\left( [0,1],\overline{\N} \right)\subset S$.
Given $a\in S$, we let $\supp_\CatSet(a)$ denote the set-theoretic support of $a$, which is open since $a$ is lower semicontinuous.
Set
\[
a':=\tfrac{1}{4}\cdot 1_{(3/4,1]},\quad
a:=\tfrac{1}{2}\cdot 1_{(1/2,1]},\quad
b:=1_{[0,1]}.
\]
These are all elements in $S_0$, and it is clear that $a'\ll a\leq b$.

We claim that there does not exist $c\in S_0$ such that
\begin{align}
\label{prp:tensornotO5:eq1}
a'+c\leq b \leq a+c.
\end{align}
(Note that an element $c$ with this property can easily be found in $S$; in fact, its existence is guaranteed since $S$ is the Cuntz semigroup of a \ca{} and therefore satisfies \axiomO{5}.)

In order to obtain a contradiction, suppose that an element $c\in S_0$ with the above property does exist.
Choose a sequence $(c_n)_n$ in $S_{00}$ such that $c=\sup_n c_n$.
By evaluating the inequality \eqref{prp:tensornotO5:eq1} at each point in $[0,1]$, it is clear that $\supp_\CatSet(c)=[0,1]$.
Since $\supp_\CatSet(c)=\cup_n \supp_\CatSet(c_n)$, there exists $N\in\N$ such that $\supp_\CatSet(c_{N})=[0,1]$.
Without loss of generality, we may assume that $N=0$.

By evaluating \eqref{prp:tensornotO5:eq1} at $\tfrac{1}{2}$, we get $c\left(\tfrac{1}{2}\right)=1$.
Since the element $1$ in $Z$ is compact, it follows from $c\left(\tfrac{1}{2}\right)=\sup_n c_n\left(\tfrac{1}{2}\right)$ that there exists $N\in\N$ such that $c_{N}\left(\tfrac{1}{2}\right)=1$.
Again, without loss of generality, we may assume that $N=0$.

Note that every $f$ in $\Lsc([0,1],\overline{\N})$ has a (canonical) decomposition as $f=\sum_{k=0}^\infty 1_{f^{-1}((k,\infty])}$, where each $f^{-1}((k,\infty])$ is an open set.
Applying this to the element $c_0$, and using that $c_0\leq 1$, we can choose a finite index set $I$, nonempty open subsets $U_{i}\subset [0,1]$ and nonzero elements $z_i\in Z$ for $i\in I$ such that
\[
c_0= \sum_{i\in I} z_{i}\cdot 1_{U_{i}}.
\]
Then $c_0\left(\tfrac{1}{2}\right)$ is the sum of the elements $z_i$ for which $\tfrac{1}{2}\in U_i$.
Since $1=c_0\left(\tfrac{1}{2}\right)$ is a minimal compact element in $Z$, and since the noncompact elements in $Z$ are soft and therefore absorbing, we deduce that $\tfrac{1}{2}$ belongs to exactly one of the sets $U_i$.
Let $i_0$ be the unique index in $I$ such that $\tfrac{1}{2}\in U_{i_0}$.
We necessarily have $z_{i_0} = 1$.
Set
\[
V:=U_{i_0},\quad
W:=\bigcup_{i\neq i_0} U_i.
\]
Then $V$ and $W$ are open subset of $[0,1]$.
Since $\supp_\CatSet(c_0)=[0,1]$, it follows that $V\cup W=[0,1]$.
Since $c_0$ is strictly less than $1_{[0,1]}$, the set $V$ is a proper subset of $[0,1]$.
Therefore, the intersection $V\cap W$ is nonempty.
For each $t\in V\cap W$ we have $c_0(t)=1+\sum_{i\neq i_0}z_i1_{U_i}(t)$, with $z_i\neq 0$ for each $i\neq i_0$ and $t\in U_i$ for at least one $i\neq i_0$. Thus $c_0(t)>1$, which clearly is a contradiction. This proves the claim.

Next, consider the map
\[
\tau\colon Z\times \Lsc\left([0,1],\overline{\N}\right)
\to S_0\subset\Lsc([0,1],Z),
\]
defined by $\tau((z,f))= z\cdot f$. It is easy to see that $\tau$ is a $\Cu$-bimorphism.
By \autoref{prp:tensProdCu}, there exists a $\Cu$-morphism
\[
\tilde{\tau} \colon Z\otimes_{\CatCu} \Lsc\left([0,1],\overline{\N}\right)\to S_0,
\]
such that $\tilde{\tau}(z\otimes f)=\tau(z,f)$ for each $z\in Z$ and $f\in\Lsc\left([0,1],\overline{\N}\right)$.

Since $1_{(3/4,1]}\ll 1_{(1/4,1]}\leq 1$ in $\Lsc\left([0,1],\overline{\N}\right)$, and
$\frac{1}{4}\ll \frac{1}{2}\leq 1$ in $Z$, we have
\[
\tfrac{1}{4}\otimes 1_{(3/4,1]}
\ll \tfrac{1}{2}\otimes 1_{(1/2,1]}
\leq 1\otimes 1_{[0,1]},
\]
in $Z\otimes_{\CatCu} \Lsc\left([0,1],\overline{\N}\right)$.
Note that
\[
a' = \tilde{\tau}\left(\tfrac{1}{4}\otimes 1_{(3/4,1]}\right),\quad
a = \tilde{\tau}\left(\tfrac{1}{2}\otimes 1_{(1/2,1]}\right),\quad
b = \tilde{\tau}\left(1\otimes 1_{[0,1]}\right).
\]
Thus, if there existed $d\in Z\otimes_{\Cu}\Lsc\left([0,1],\overline{\N}\right)$ such that
\[
\tfrac{1}{4}\otimes 1_{(3/4,1]} + d
\leq 1\otimes 1_{[0,1]}
\leq \tfrac{1}{2}\otimes 1_{(1/2,1]} +d,
\]
then the element $c=\tilde{\tau}(d)$ would satisfy $a'+c \leq b \leq a+c$, which is not possible by the first part of the proof.
Therefore, $Z\otimes_{\CatCu} \Lsc\left([0,1],\overline{\N}\right)$ does not satisfy \axiomO{5}.
\end{proof}

%==========================================================================================
\begin{cor}
\label{cor:tensornotlsc}
We have $\Lsc([0,1],Z)\ncong \Lsc\left([0,1],\overline{\N}\right)\otimes_{\CatCu}Z$.
\end{cor}
\begin{proof}
By \autoref{prp:tensornotO5}, $Z\otimes_{\CatCu}\Lsc\left([0,1],\overline{\N}\right)$ does not satisfy axiom~\axiomO{5}.
On the other hand,
\[
\Lsc([0,1],Z)\cong\Cu(C([0,1],\mathcal Z)),
\]
by \cite[Theorem~3.4]{AntPerSan11PullbacksCu}, which in combination with \autoref{prp:o5o6} shows that $\Lsc([0,1],Z)$ satisfies \axiomO{5}.
\end{proof}

%------------------------------------------------------------------------------------------
For the next result, recall from \autoref{dfn:dimMonoid} that a $\CatCu$-semigroup is \emph{simplicial} if it is isomorphic to the algebraically ordered $\CatCu$-semigroup $\overline{\N}^r$ for some $r\in\N$.
In \autoref{prp:EffHandShen_CuVersion}, we have seen that a countably-based $\CatCu$-semigroup $S$ is an inductive limit of simplicial $\CatCu$-semigroups if and only if there exists a separable AF-algebra $A$ such that $S\cong\Cu(A)$.

%==========================================================================================
\begin{prp}
\label{prp:tensLimSimplicial}
Let $S$ be an inductive limit of simplicial $\CatCu$-semigroup{s}.
Then taking the tensor product with $S$ preserves \axiomO{5}, \axiomO{6} and weak cancellation.
\end{prp}
\begin{proof}
Let $((S_i)_{i\in I},(\varphi_{i,j})_{i, j\in I, i\leq j})$ be an inductive system of simplicial $\CatCu$-semigroups, indexed over a directed set $I$.
Then there are numbers $r_i\in\N$ such that $S_i\cong\overline{\N}^{r_i}$ for each $i\in I$.
Let $T$ be a $\CatCu$-semigroup{}.
For $r\in\N$, let $T^r$ be the set of $r$-tuples with entries in $T$, equipped with pointwise addition and order.
It is easily checked that $T^r$ is a $\CatCu$-semigroup{} and that there is a natural isomorphism $S_i\otimes_{\CatCu} T\cong T^{r_i}$ for each $i$.
It follows from \autoref{prp:tensLim} that $S\otimes_{\CatCu}T\cong\CatCuLim_i T^{r_i}$.

Assume now that $T$ satisfies \axiomO{5}.
It follows easily that $T^r$ satisfies \axiomO{5} for each $r\in\N$.
Then $S\otimes_{\CatCu}T$ satisfies \axiomO{5} by \autoref{prp:axiomsPassToLim}.
It is proved analogously that weak cancellation and \axiomO{6} pass from $T$ to $S\otimes_{\CatCu}T$.
\end{proof}

%------------------------------------------------------------------------------------------
For the next result, recall that for a \pom{} $M$, we denote by $\Cu(M)$ the $\CatCu$-completion of the $\CatPreW$-semigroup $(M,\leq)$;
see \autoref{pgr:algebraicSemigp}.

%==========================================================================================
\begin{prp}
\label{prp:tensorAlgebraic}
Let $M$ and $N$ be \pom{s}.
Then there is a canonical isomorphism
\[
\Cu(M)\otimes_{\CatCu}\Cu(N)
\cong \Cu(M\otimes_{\CatPom}N).
\]
\end{prp}
\begin{proof}
We will write $(M,\leq)$ and $(N,\leq)$ for the $\CatW$-semigroup{s} associated to $M$ and $N$;
see \autoref{prp:algebraicSemigp}.
It follows easily from the construction of the tensor product in $\CatPreW$ that
\[
(M,\leq)\otimes_{\CatPreW}(N,\leq)\cong(M\otimes_{\CatPom} N,\leq).
\]
Now the result follows from \autoref{thm:tensProdCompl}.
\end{proof}

%==========================================================================================
\begin{cor}
\label{prp:algebraicPassesToTensProd}
If $S$ and $T$ are algebraic $\CatCu$-semigroup{s}, then so is $S\otimes_{\CatCu}T$.
\end{cor}

%==========================================================================================
\begin{cor}
\label{prp:axiomsTensProdAlgebraic}
Let $S$ and $T$ be algebraic $\CatCu$-semigroup{s}.
If $S$ and $T$ satisfy axiom~\axiomO{5}, then so does $S\otimes_{\CatCu}T$.
\end{cor}
\begin{proof}
Let $S_c$ and $T_c$ denote the \pom{} of compact elements in $S$ and $T$, respectively.
By \autoref{prp:algebraicSemigp} and \autoref{prp:tensorAlgebraic}, there are natural isomorphisms
\[
S\cong \Cu(S_c),\quad
T\cong \Cu(T_c),\text{ and }
S\otimes_{\CatCu} T \cong \Cu (S_c\otimes_{\CatPom} T_c).
\]

Assume now that $S$ and $T$ satisfy \axiomO{5}.
By \autoref{prp:propertiesAlgebraic}, this implies that $S_c$ and $T_c$ are algebraically ordered.
It follows from \autoref{prp:tensPomPreservesAlgOrd} that $S_c\otimes_{\CatPom}T_c$ is algebraically ordered.
Using \autoref{prp:propertiesAlgebraic} again, we deduce that $S\otimes_{\CatCu} T$ satisfies \axiomO{5}.
\end{proof}

%==========================================================================================
\begin{pgr}
\label{pgr:tensCa}
\index{symbols}{$\tau_{A,B}^\txtMax$}
\index{symbols}{$\tau_{A,B}^\txtMin$}
Let $A$ and $B$ be \ca{s}, and let $A\tensMax B$ be their maximal tensor product.
Given $k,l\in\N$, $x_1,x_2\in M_k(A)_+$ and $y_1,y_2\in M_l(B)_+$, the simple tensors $x_1\otimes y_1$ and $x_2\otimes y_2$ are positive elements in $M_k(A)\tensMax M_l(B)$.
We have:
\begin{enumerate}
\item
If $x_1\precsim x_2$ and $y_1\precsim y_2$, then $x_1\otimes y_1\precsim x_2\otimes y_2$ in $M_k(A)\tensMax M_l(B)$;
see \cite[Lemma~4.1]{Ror04StableRealRankZ}.
\item
If $x_1\precsim (x_2-\varepsilon)_+$ and $y_1\precsim (y_2-\varepsilon)_+$ for some $\varepsilon>0$, then there exists $\delta>0$ such that $x_1\otimes y_1\precsim (x_2\otimes y_2-\delta)_+$ in $M_k(A)\tensMax M_l(B)$;
see the proof of Proposition~4.5 in~\cite{AntBosPer13CuFields}.
\end{enumerate}
Let $\varphi\colon M_\infty(A)\tensMax M_\infty(B) \to M_\infty(A\tensMax B)$ be an isomorphism that identifies the natural copies of $A\tensMax B$ in $M_\infty(A)\tensMax M_\infty(B)$ and $M_\infty(A\tensMax B)$.
Such an isomorphism is unique up to approximate unitary equivalence.
We define a map $W(A)\times W(B)\to W(A\tensMax B)$ by $([x],[y]) \mapsto [\varphi(x\otimes y)]$ for $x\in M_\infty(A)_+$ and $y\in M_\infty(B)_+$,
Using (1) and (2) above, it is easily checked that this map is a $\CatW$-bimorphism.
It is independent of the choice of $\varphi$.
We therefore obtain a natural $\CatW$-morphism
\[
W(A)\otimes_{\CatW} W(B)\to W(A\tensMax B).
\]

Similarly, choosing a natural isomorphism $\psi$ from $(A\otimes\K)\tensMax(B\otimes\K)$ to $(A\tensMax B)\otimes\K$, we obtain a natural $\CatCu$-bimorphism
\[
\tau_{A,B}^\txtMax\colon\Cu(A)\otimes_{\CatCu}\Cu(B) \to \CatCu(A\tensMax B),
\]
such that $\tau_{A,B}^\txtMax([x]\otimes[y])=[\psi(x\otimes y)]$ for every $x\in (A\otimes\K)_+$ and $y\in(B\otimes\K)_+$.

The natural quotient \starHom{} from $A\tensMax B$ to $A\tensMin B$ induces a surjective $\CatCu$-morphism
\[
\Cu(A\tensMax B) \to \Cu(A\tensMin B).
\]
By composing the map $\tau_{A,B}^\txtMax$ with this $\CatCu$-morphism, we obtain a natural $\CatCu$-morphism
\[
\tau_{A,B}^\txtMin\colon\Cu(A)\otimes_{\CatCu}\Cu(B) \to \Cu(A\tensMin B).
\]
\end{pgr}

%==========================================================================================
\begin{prbl}
\label{prbl:tensCa}
Let $A$ and $B$ be \ca{s}.
When is the map $\tau_{A,B}^\txtMax$ a $\CatCu$-isomorphism?
When is it surjective?
When is it an order embedding?
Similarly, when is the map $\tau_{A,B}^\txtMin$ an isomorphism? When is it surjective? When is it an order-embedding?
\end{prbl}

%==========================================================================================
\begin{pgr}
\label{pgr:AnswerTensCa}
Let $A$ and $B$ be \ca{s}.
It is clear that $\tau_{A,B}^\txtMax$ is an order-embedding whenever $\tau_{A,B}^\txtMin$ is.
Similarly, if $\tau_{A,B}^\txtMax$ is surjective, then so is $\tau_{A,B}^\txtMin$.

If $A$ or $B$ is nuclear, then the natural map from $A\tensMax B$ to $A\tensMin B$ is an isomorphism.
In that case, the maps $\tau_{A,B}^\txtMin$ and $\tau_{A,B}^\txtMax$ are equal, and we simply write $\tau_{A,B}$ for this map.

It is easy to find examples of \ca{s} $A$ and $B$ for which the natural map $\tau_{A,B}^\txtMax$ is not surjective.
For instance, this is the case for $A=C([0,1])$ and $B=\mathcal{Z}$, as shown in \autoref{prp:tensornotO5}; see also \autoref{cor:tensornotlsc}.

Other (counter)examples can be found using $K$-theory.
If $A$ is a unital, simple, stably finite \ca{}, then $K_0(A)$ is determined by the Cuntz semigroup of $A$ via the formula
\begin{align}
\label{pgr:AnswerTensCa:eq1}
K_0(A) = \Gr( \Cu(A)_c ),
\end{align}
where $\Cu(A)_c$ denotes the submonoid of compact elements in $\Cu(A)$, and where $\Gr$ denotes the Grothendieck completion (see, for example, \cite{AraPerTom11Cu}, or also \cite{BroCiu09IsoHilbert}).
Now, let $A$ and $B$ be unital, simple, stably finite \ca{s}.
Assume that $A$ is nuclear, whence we can unambiguously write $\otimes$ instead of $\tensMax$ for tensor products with $A$.
Then the tensor product $A\otimes B$ is also a unital, simple, stably finite \ca{}.
Assume that the map
\[
\tau_{A,B}\colon\Cu(A)\otimes_{\CatCu}\Cu(B) \to \Cu(A\otimes B)
\]
is surjective.
Let us show that this implies that the natural map
\[
\Cu(A)_c\otimes_{\CatPom}\Cu(B)_c \to \Cu(A\otimes B)_c
\]
is surjective as well.
Note that $A,B$ and $A\otimes B$ are simple and stably finite.

More generally, let $S,T$ and $R$ be simple, stably finite $\CatCu$-semigroups satisfying \axiomO{5}, let $\varphi\colon S\otimes_\CatCu T\to R$ be a surjective $\CatCu$-morphism, and let $c\in R$ be compact.
Choose $a_k\in S$ and $b_k\in T$ for $k=1,\ldots,n$ such that $c=\varphi( \sum_k a_k\otimes b_k )$.
By \autoref{prp:softNoncpctSimple}, a nonzero element in $R$ (or $S$, $T$ respectively) is either compact or soft.
By \autoref{prp:softAbsorbing}, and using that $R$ is simple, a sum of elements in $R$ is only compact if each summand is.
Thus, $\varphi( a_k\otimes b_k )$ is compact for each $k$.
It is straightforward to check that $\varphi( a_k\otimes b_k )$ is soft whenever $a_k$ or $b_k$ is.
Hence, the elements $a_k$ and $b_k$ are compact.
This shows that the induced map $S_c\otimes_\CatPom T_c\to R_c$ is surjective, as desired.

Passing to Grothendieck completions, and using \eqref{pgr:AnswerTensCa:eq1} at the second step, we obtain a surjective map
\[
\Gr( \Cu(A)_c\otimes_{\CatPom}\Cu(B)_c )
\to \Gr( \Cu(A\otimes B)_c )
\cong K_0(A\otimes B).
\]
Since taking the Grothendieck completion commutes with tensor products (see \autoref{prp:GrTensorMon}), we have
\[
K_0(A)\otimes K_0(B)
\cong \Gr(\Cu(A)_c)\otimes_{\CatPom}\Gr(\Cu(B)_c)
\cong \Gr(\Cu(A)_c\otimes_{\CatPom}\Cu(B)_c).
\]
Thus, we have shown that the natural map
\[
K_0(A)\otimes K_0(B) \to K_0(A\otimes B)
\]
is surjective.

Let us further assume that $A$ is a \ca{} in the bootstrap class;
see \cite[V.1.5.4,~p.437]{Bla06OpAlgs}.
Then $A$ and $B$ satisfy the `K\"{u}nneth formula for tensor products in $K$-theory';
see \cite[Theorem~V.1.5.10,~p.440]{Bla06OpAlgs}.
This means that there is a short exact sequence
\[
0 \to \bigoplus_{i=0,1} K_i(A)\otimes K_i(B)
\to K_0(A\otimes B)
\to \bigoplus_{i=0,1} \Tor^\Z_1(K_i(A), K_{1-i}(B))
\to 0.
\]
Since our assumptions on $A$ and $B$ imply that the natural map from $K_0(A)\otimes K_0(B)$ to $K_0(A\otimes B)$ is surjective, we deduce from the K\"{u}nneth formula that
\[
K_1(A)\otimes K_1(B) = 0,\quad
\Tor^\Z_1(K_i(A), K_{1-i}(B))=0, \text{ for } i=0,1.
\]

In conclusion, we get that the map $\tau_{A,B}$ is not surjective whenever $A$ and $B$ are unital, simple, stably finite \ca{s} in the bootstrap class for which $K_1(A)\otimes K_1(B)\neq 0$ or for which $\Tor^\Z_1(K_i(A), K_{1-i}(B))\neq 0$ for $i=0$ or $i=1$.

On the other hand, $\tau_{A,B}$ is an isomorphism in the following cases:
\beginEnumStatements
\item
If $A$ or $B$ is an AF-algebra;
see \autoref{prp:tensProdAF}.
\item
If $A$ and $B$ are separable, and $B$ is simple, nuclear and purely infinite;
see \autoref{prp:tensCaWithSimplePI}.
\end{enumerate}
\end{pgr}

%==========================================================================================
\begin{prp}
\label{prp:tensProdAF}
Let $A$ and $B$ be \ca{s}.
Assume that at least one of the algebras is an AF-algebra.
Then the natural map
\[
\tau_{A,B}\colon\Cu(A)\otimes_{\CatCu}\Cu(B) \to \Cu(A\otimes B)
\]
is an isomorphism.
\end{prp}
\begin{proof}
Without loss of generality, we may assume that $A$ is an AF-algebra.
Choose an inductive system $((A_i)_{i\in I},(\varphi_{i,j})_{i,j\in I, i\leq j})$ of finite-dimensional \ca{s}  such that $A\cong\varinjlim A_i$.
For each $i$, there is $r_i\in\N$ such that $A_i$ is isomorphic to a direct sum of $r_i$ matrix algebras.
Then $A_i\otimes\K\cong\K^{r_i}$ and $\Cu(A_i)\cong\overline{\N}^{r_i}$.
Moreover, $A_i\otimes B\otimes\K \cong (B\otimes\K)^{r_i}$ and there are isomorphisms
\[
\Cu(A_i)\otimes_{\CatCu}\Cu(B)
\cong \overline{\N}^{r_i}\otimes\Cu(B)
\cong \Cu(B)^{r_i}
\cong \Cu(A_i\otimes B),
\]
for each $i$.

Since the maximal tensor product commutes with inductive limits of \ca{s} (see \cite[II.9.6.5,~p.200]{Bla06OpAlgs}), there is a natural isomorphism $\varinjlim (A_i\otimes B) \cong A\otimes B$.
Using \autoref{prp:functorCu} at the first and last step, and using \autoref{prp:tensLim} at the second step, we obtain
\begin{align*}
\Cu(A)\otimes_{\CatCu}\Cu(B)
&\cong \left( \CatCuLim \Cu(A_i) \right) \otimes_{\CatCu}\Cu(B) \\
&\cong \CatCuLim \left( \Cu(A_i) \otimes_{\CatCu}\Cu(B) \right) \\
&\cong \CatCuLim \Cu(A_i\otimes B)
\cong \Cu(A\otimes B),
\end{align*}
as desired.
\end{proof}

%==========================================================================================
\begin{pgr}
\label{pgr:tensFctl}
Recall that the set of functionals on a $\CatCu$-semigroup{} $S$ is defined as the set of generalized $\CatCu$-morphisms $S\to[0,\infty]$; see \autoref{pgr:fctl}.

Here, the product of two elements $x,y\in[0,\infty)$ is defined as the usual product of real numbers, and $0y=x0=0$ for all $x,y\in[0,\infty]$, and $\infty y = x\infty = \infty$ for all $x,y\in(0,\infty]$.
(This is the $\CatCu$-product on $[0,\infty]$ when considered with its structure as a $\CatCu$-semiring; see \autoref{exa:R}.)

Now, let $S$ and $T$ be $\CatCu$-semigroup{s}.
There is a natural map
\[
F(S)\times F(T) \to F(S\otimes_{\CatCu}T),
\]
defined as follows:
Given $\lambda\in F(S)$ and $\mu\in F(T)$, consider the map
\[
f\colon S\times T\to[0,\infty],\quad
(a,b)\mapsto \lambda(a)\mu(b),\quad
\txtFA a\in S, b\in T.
\]
It is easily checked that $f$ is a generalized $\CatCu$-bimorphism.
By \autoref{prp:tensProdCu}, $f$ induces a generalized $\CatCu$-morphism
\[
\tilde{f}\colon S\otimes_{\CatCu}T\to[0,\infty].
\]
This means $\tilde{f}$ is a functional on $S\otimes_{\CatCu}T$, satisfying $\tilde{f}(a\otimes b)=\lambda(a)\mu(b)$ for every $a\in S$ and $b\in T$.
\end{pgr}

%------------------------------------------------------------------------------------------
The following is a version of \cite[Theorem~4.1.10, p.69]{Ror02Classification} for $\CatCu$-semigroup{s}.

%==========================================================================================
\begin{prp}
\label{prp:tensprodTwoSimple}
Let $S$ and $T$ be simple, nonelementary $\CatCu$-semigroup{s} satisfying \axiomO{5} and \axiomO{6}.
\beginEnumStatements
\item
If $S$ and $T$ are stably finite, then so is $S\otimes_{\CatCu} T$.
\item
If $S$ or $T$ is not stably finite, then $S\otimes_{\CatCu} T \cong\{0,\infty\}$.
\end{enumerate}
\end{prp}
\begin{proof}
To show \enumStatement{1}, assume that $S$ and $T$ are simple and stably finite.
Set $R:=S\otimes_{\CatCu}T$.
There are nontrivial functionals $\lambda\in F(S)$ and $\mu\in F(T)$; see \autoref{prp:simpleSF}.
By \autoref{pgr:tensFctl}, $\lambda$ and $\mu$ induce a functional $\delta\in F(R)$ such that $\delta(a\otimes b)=\lambda(a)\mu(b)$ for every $a\in S$ and $b\in T$.
It is clear that $\delta$ is nontrivial, which implies that $R$ is stably finite.

To show \enumStatement{2}, we may assume without loss of generality that $T$ is not stably finite.
By \autoref{prp:tensLat}, $S\otimes_{\CatCu}T$ is simple.
Thus, there is a unique element $\infty\in S\otimes_{\CatCu}T$ such that $\infty=\sup_{n\in\N} nx$ for every nonzero $x\in S\otimes_{\CatCu}T$.
Let $a\in S$ and $b\in T$ be nonzero.
We will show $a\otimes b=\infty$.

Let $\infty_T$ denote the infinite element of $T$.
By \autoref{prp:simpleSF}, $\infty_T$ is compact.
Thus, there is $n\in\N$ such that $nb=\infty_T$.
By \cite[Proposition~5.2.1]{Rob13Cone}, here reproduced as \autoref{prp:GlimmHalving}, there is a nonzero element $x\in S$ such that $nx\leq a$.
It follows
\[
a\otimes b \geq (nx)\otimes b = x\otimes(nb) = x\otimes\infty_T = \infty.
\]
We also have $a\otimes b\leq \infty$, and therefore $a\otimes b=\infty$.
Since this holds for all nonzero $a$ and $b$, we get $S\otimes_{\CatCu} T \cong\{0,\infty\}$, as desired.
\end{proof}

%------------------------------------------------------------------------------------------
%==========================================================================================
%##########################################################################################
\chapter{\texorpdfstring{$\CatCu$}{Cu}-semirings and \texorpdfstring{$\CatCu$}{Cu}-semimodules}
\label{sec:CuSrgSmod}

%------------------------------------------------------------------------------------------
In \autoref{sec:solidCuSrg}, we introduce the concepts of $\CatCu$-semirings and their semimodules.
Natural examples are given by Cuntz semigroups of \ca{s} that are strongly self-absorbing and of \ca{s} that tensorially absorb such a \ca\ respectively; see \autoref{prp:semirgFromSSA}.

We say that a $\CatCu$-semiring $R$ is solid if the multiplication map $R\times R\to R$ induces an isomorphism $R\otimes_{\CatCu}R\xrightarrow{\cong} R$;
see \autoref{dfn:solidSemirg}.
This is analogous to the concept of solidity for rings as introduced in \cite[Definition~2.1;~2.4]{BouKan72Core}.
This property can also be interpreted as an algebraic analog of being strongly self-absorbing.
The Cuntz semigroup of every known strongly self-absorbing \ca{} is a solid $\CatCu$-semiring;
see \autoref{pgr:ssaSolid}.

Given a solid $\CatCu$-semiring $R$, we say that a $\CatCu$-semigroup $S$ has $R$-mul\-ti\-pli\-ca\-tion if it is a semimodule over $R$, in a suitable sense, and it is very interesting to study the class of such $\CatCu$-semigroups.
We show that every generalized $\CatCu$-morphism between $\CatCu$-semigroups with $R$-multiplication is automatically $R$-linear;
see \autoref{prp:solidTFAE}.
We deduce that a $\CatCu$-semigroup $S$ has at most one $R$-multipli\-ca\-tion.
In other words, either there is no way to give $S$ the structure of a semimodule over $R$, or there is a unique such structure.
This also means that, as we have already remarked in the Introduction, being a semimodule over $R$ is a \emph{property} of $S$, rather than an extra structure;
see \autoref{rmk:solidModuleIsProperty}.

It may therefore seem that $\CatCu$-semigroups with $R$-multiplication are rare.
However, we show in \autoref{prp:moduleTensorProd} that for every $\CatCu$-semigroup $S$, the tensor product $R\otimes_{\CatCu}S$ has $R$-multiplication.
We obtain that $S$ has $R$-multiplication if and only if $S$ is naturally isomorphic to $R\otimes_{\CatCu}S$;
see \autoref{prp:solidModuleTFAE}.

We refer to \autoref{sec:semirg} for a detailed study of the structure of $\CatCu$-semirings, including a complete classification of solid $\CatCu$-semirings in \autoref{sec:classificationSolid}.

In Sections~\ref{sec:pureInf} through \ref{sec:RMod}, we study $\CatCu$-semimodules over the following solid $\CatCu$-semirings:

(1)
If $A$ is a purely infinite, strongly self-absorbing \ca{}, for example the Cuntz algebra $\mathcal{O}_\infty$, then $\Cu(A)=\{0,\infty\}$.
In \autoref{sec:pureInf}, we study $\CatCu$-semigroups that have $\{0,\infty\}$-multiplication.
This can be considered as a theory of `purely infinite Cuntz semigroups'.

Indeed, it is clear that the Cuntz semigroup of every $\mathcal{O}_\infty$-stable \ca{} is a $\{0,\infty\}$-semimodule.
More generally, we show that a (not necessarily simple) \ca{} is purely infinite if and only if its Cuntz semigroup has $\{0,\infty\}$-multiplication.

(2)
The Jiang-Su algebra $\mathcal{Z}$ is a strongly self-absorbing \ca{} whose Cuntz semigroup is a solid $\CatCu$-semiring, denoted by $Z$.
In \autoref{sec:ZMod}, we study $\CatCu$-semigroups that have $Z$-multiplication.

The analogy between $\mathcal{Z}$-stable \ca{s} and Cuntz semigroups with $Z$-multiplication is however not as close as in the purely infinite case.
In one direction, we clearly have that the Cuntz semigroup of every $\mathcal{Z}$-stable \ca{} is a $Z$-semimodule.
However, the converse is not true in general.
Consider for example $C^*_\lambda(F_\infty)$, the reduced group \ca{} of the free group with infinitely many generators.
It is shown in \cite[Section~6.3]{Rob12LimitsNCCW} that $\Cu(C^*_\lambda(F_\infty))\cong\Cu(\mathcal{Z})$, which implies that the Cuntz semigroup of $C^*_\lambda(F_\infty)$ has $Z$-multiplication.
However, it was shown by Simon Wassermann that $C^*_\lambda(F_\infty)$ is tensorially prime, in particular $C^*_\lambda(F_\infty)\ncong C^*_\lambda(F_\infty)\otimes\mathcal{Z}$.

We show in \autoref{prp:ZModTFAE} that a $\CatCu$-semigroup has $Z$-multiplication if and only if it is almost unperforated and almost divisible.
On the other hand, it seems that the Cuntz semigroup of every $\mathcal{Z}$-stable \ca{} is even nearly unperforated;
see \autoref{conj:nearUnpCaZstable}.

(3)
Every strongly self-absorbing UHF-algebra is of the form $M_q$ for some supernatural number $q$ satisfying $q=q^2$ and $q\neq 1$.
We denote the Cuntz semigroup of $M_q$ by $R_q$, which is a solid $\CatCu$-semiring.
In \autoref{sec:RqMod}, we study $\CatCu$-semigroups that have $R_q$-multiplication.
This can be considered as a theory of `UHF-absorbing Cuntz semigroups'.
Given a $\CatCu$-semigroup $S$, we also think of $R_q\otimes_{\CatCu}S$ as a `rationalization' of $S$.

(4)
The Jacelon-Razak algebra $\mathcal{R}$ is a stably projectionless \ca{}, whence it does not satisfy the definition of a strongly self-absorbing \ca{} (which are required to be unital).
However, the Cuntz semigroup of $\mathcal{R}$ is $[0,\infty]$, which is a solid $\CatCu$-semiring.
Moreover, for every \ca{} $A$ we have $\Cu(\mathcal{R}\otimes A)\cong[0,\infty]\otimes_{\CatCu}\Cu(A)$.
In particular, if a \ca{} tensorially absorbs $\mathcal{R}$, then its Cuntz semigroup has $[0,\infty]$-multiplication.

In \autoref{sec:RqMod} we study $\CatCu$-semigroups that have $[0,\infty]$-multiplication.
This can be considered as a theory of `$\mathcal{R}$-absorbing Cuntz semigroups'.
Given a $\CatCu$-semigroup $S$, we also think of $[0,\infty]\otimes_{\CatCu}S$ as the `realification' of $S$, a term that was introduced by Robert.
\\

In the following table, we summarize some results of this chapter.
The middle column contains the characterizations when a $\CatCu$-semigroup $S$ has $R$-multiplication for the solid $\CatCu$-semiring listed in the left column.
The column on the right characterizes the effect that `stabilizing' with $R$ has on the order structure of the $\CatCu$-semigroup.

\vspace{5pt}

\noindent
{
\renewcommand{\arraystretch}{1.2}%
\begin{tabular}{|P{1.1cm}||P{5.2cm}|P{4.8cm}|}
\hline
$R$
& Characterization when $S$ is a $\CatCu$-semimodule over $R$.
& For $a,b\in S$, characterization when $1\otimes a\leq 1\otimes b$ in $R\otimes_{\CatCu}S$.
\\
\hline
\hline
$\{0,\infty\}$%
& $S$ is idempotent; \autoref{prp:PIModTFAE}.
& $a\varpropto^\ctsRel b$; \autoref{prp:MapToPIMod}.
\\%
\hline
$Z$
& $S$ is almost unperforated and almost divisible; \autoref{prp:ZModTFAE}.
& Unclear; \autoref{prbl:MapToZMod}.
\\
\hline
$R_q$
& $S$ is $q$-unperforated and $q$-divisible; \autoref{prp:RqModTFAE}.
& For each $a'\ll a$, there exists $n$ dividing $q$ such that $na'\leq nb$; \autoref{prp:MapToRqMod}.
\\
\hline
$[0,\infty]$
& $S$ is unperforated, divisible and every element is soft; \autoref{prp:RModTFAE}.
& $\hat{a}\leq\hat{b}$ in $\Lsc(F(S))$; \autoref{prp:MapToRMod}.
\\
\hline
\end{tabular}
}

%\vspace{5pt}
%------------------------------------------------------------------------------------------
%==========================================================================================
\section{Strongly self-absorbing \texorpdfstring{{C}*-algebras}{{C}*-algebras} and solid \texorpdfstring{$\CatCu$}{Cu}-semirings}
\label{sec:solidCuSrg}

%==========================================================================================
\begin{dfn}
\label{dfn:CuSemirg}
\index{terms}{Cu-semiring@$\CatCu$-semiring}
\index{terms}{Cu-product @$\CatCu$-product (of a $\CatCu$-semiring)}
A \emph{$\Cu$-semiring} is a $\CatCu$-semigroup $R$ together with a $\CatCu$-bimor\-phism, $(a,b)\mapsto ab$, and a distinguished element $1$ in $R$ such that for all $a,b,c\in R$ we have
\[
ab=ba,\quad a(bc)=(ab)c,\text{ and } \quad 1a=a=a1.
\]
The $\CatCu$-bimorphism $R\times R\to R$ is also called the \emph{$\CatCu$-product} of~$R$.
\end{dfn}

%==========================================================================================
\begin{rmks}
\label{rmk:CuSemirg}
(1)
A $\Cu$-semiring is a commutative, unital semiring (see \autoref{sec:poRg}) with a compatible partial order turning it into a $\omega$-continuous $\omega$-\dcpo\ in the sense of lattice theory;
see \autoref{rmk:catCu}.

(2)
It is natural to assume that a $\Cu$-semiring has no zero divisors.
Indeed, if $a$ and $b$ satisfy $ab=0$, then $st=0$ for all $s$ and $t$ with $s\leq\infty\cdot a$ and $t\leq\infty\cdot b$.
Thus, if $ab=0$, then the multiplication is trivial on the ideals generated by $a$ and $b$.
In particular, a simple $\Cu$-semiring with zero divisors is isomorphic to $\{0\}$.

(3)
Let $R$ be a $\Cu$-semiring.
Recall from \autoref{pgr:ideals} that an ideal $I$ in $R$ is an order-hereditary submonoid that is closed under passing to suprema of increasing sequences.
This notion of ideal is a-priori not related to the ring-theoretic notion of an ideal, which means that $ab$ is in $I$ for any $a\in I$ and $b\in R$.

If the unit of $R$ is full (that is, if it is also an order-unit), then every ideal $I$ is also a ring-theoretic ideal.
Indeed, given $a\in I$ and $b\in R$, we have $b\leq\infty\cdot 1$, and therefore $ab\leq a(\infty\cdot 1)=\sup k(a1)=\infty\cdot a\in I$.
Then $ab$ is in $I$ as desired.

We thank the referee for pointing out the following example of a $\CatCu$-semiring whose unit is not full.
Consider $R=[0,\infty]\times[0,\infty]$, with pointwise order and addition, and with multiplication given by $(a_1,a_2)(b_1,b_2) := (a_1b_2+a_2b_1,a_2b_2)$.
Then the multiplicative unit $(0,1)$ is clearly not an order-unit.
On the other hand, if $S$ is the Cuntz semigroup of a unital \ca{} $A$ which is a $\CatCu$-semiring and the class of the unit of $A$ acts as a unit for the product, then it will be full.
\end{rmks}

%------------------------------------------------------------------------------------------
The following definition is an adoption of the terminology introduced by Robert, \cite[Definition~3.1.2]{Rob13Cone}.

%==========================================================================================
\begin{dfn}
\label{dfn:CuSemimod}
\index{terms}{Cu-semimodule@$\CatCu$-semimodule}
\index{terms}{R-multiplication@$R$-multiplication}
Let $S$ be a $\CatCu$-semigroup{}, and let $R$ be a $\Cu$-semiring.
An \emph{$R$-mul\-ti\-pli\-ca\-tion} on $S$ is a $\CatCu$-bimorphism $R\times S \to S$, $(r,s)\mapsto rs$ such that for all $r_1,r_2\in R$ and $s\in S$, we have
\[
(r_1r_2)s=r_1(r_2s), \quad 1s=s.
\]
In this case we also say that $S$ is a \emph{$\Cu$-semimodule} over $R$.
\end{dfn}

%==========================================================================================
\begin{pgr}
\label{pgr:ssa}
One motivation for the definition of a $\Cu$-semiring comes from strongly self-absorbing \ca{s}.
Recall from \cite[Definition~1.3]{TomWin07ssa} that a unital \ca{} $D$ is \emph{strongly self-absorbing} if $D\ncong\C$ and if there exists an isomorphism $\psi\colon D\to D\otimes D$ such that $\psi$ is approximately unitarily equivalent to $\id_D\otimes 1_D$.
Every such algebra is simple, nuclear, and either purely infinite or stably finite with a unique tracial state; see \cite[1.6, 1.7]{TomWin07ssa}.
The Cuntz semigroup of a simple, purely infinite \ca\ is isomorphic to $\{0,\infty\}$.
%Thus, we will focus our attention on the stably finite case.

Another source of $\CatCu$-semimodules appears in group actions.
If $G$ is a compact group, then $\Cu(C^*(G))$ has a natural structure as a $\CatCu$-semiring.
An action of $G$ on a \ca{} $A$ induces a natural $\Cu(C^*(G))$-multiplication on $\Cu(A)$.
Such $\CatCu$-semimodules have been studied in \cite{GarSan15}.
\end{pgr}

%==========================================================================================
\begin{prp}
\label{prp:semirgFromSSA}
Let $D$ be a unital, separable, strongly self-absorbing \ca.
Then:
\beginEnumStatements
\item
The Cuntz semigroup $\Cu(D)$ has a natural $\CatCu$-product giving it the structure of a countably-based, simple $\Cu$-semiring satisfying \axiomO{5} and \axiomO{6}.
Moreover, if $D$ is stably finite then $\Cu(D)$ has a unique normalized functional.
\item
If $A$ is a $D$-absorbing \ca\ (that is, $A\cong A\otimes D$), then $\Cu(A)$ has a natural $\Cu(D)$-multiplication.
\end{enumerate}
\end{prp}
\begin{proof}
We use the symbol `$\approx$' to denote approximate unitary equivalence.
For positive elements in a \ca, this is a stronger equivalence relation than Cuntz equivalence.

\enumStatement{1}
By definition, we choose a $^*$-isomorphism $\psi\colon D\to D\otimes D$ such that $\psi\approx \id_D\otimes 1_D$.
Consider the natural $\Cu$-bimorphism $\Cu(D)\times\Cu(D)\to \Cu(D\otimes D)$ from \autoref{pgr:tensCa}.
Composed with $\Cu(\psi^{-1})$, this yields a $\Cu$-bimorphism
\[
\varphi\colon \Cu(D)\times \Cu(D)\to \Cu(D).
\]
We will show that $\varphi$ together with $1=[1_D]$ gives $\Cu(D)$ the structure of a $\Cu$-semiring.
We know from \cite[Corollary~1.11]{TomWin07ssa} that $D$ has approximately inner flip.
Thus, for any $x,y\in D_+$, we have $x\otimes y\approx y\otimes x$ in $D\otimes D$.
It follows
\[
\varphi([x],[y])=[\psi^{-1}(x\otimes y)]=[\psi^{-1}(y\otimes x)]=\varphi([y],[x]).
\]
The analogous computation holds for Cuntz classes of positive elements in $D\otimes\K$, which implies that $\varphi$ defines a commutative multiplication.

To show associativity of the product, consider positive elements $x,y$ and $z$.
Using the approximately inner flip, we get $x\otimes y\otimes z\approx y\otimes z\otimes x$.
Applying $\psi^{-1}\otimes \id_D$, it follows that $\psi^{-1}(x\otimes y)\otimes z \approx \psi^{-1}(y\otimes z)\otimes x$, and therefore
\begin{align*}
\varphi (\varphi([x],[y]),[z])
&= [\psi^{-1}(\psi^{-1}(x\otimes y)\otimes z)] \\
&= [\psi^{-1}(\psi^{-1}(y \otimes z)\otimes x)] \\
&= [\psi^{-1}(x\otimes \psi^{-1}(y\otimes z))] =\varphi([x],\varphi([y],[z])).
\end{align*}
The analogous computation in $D\otimes\K$ implies that the product is associative.

Using $\id_D\otimes 1_D\approx\psi$ and $\psi\approx 1_D\otimes\id_D$, we obtain
\[
\varphi([x],1)
=[\psi^{-1}(x\otimes 1)]
=[x]
=[\psi^{-1}(1\otimes x)]
=\varphi(1,[x]),
\]
for every positive $x$.
This finishes the proof that $\Cu(D)$ is a $\CatCu$-semiring.

The Cuntz semigroup of every separable \ca\ is countably-based, satisfies \axiomO{5} and \axiomO{6};
see \autoref{prp:CufromCAlg} and \autoref{prp:o5o6}.
Moreover, since $D$ is simple we get that $\Cu(D)$ is simple;
see \autoref{prp:simpleCa}.
If $D$ is stably finite, then it has a unique ($2$-quasi)tracial state, and then $\Cu(D)$ has a unique normalized functional by \autoref{prp:UniqueFctlFromCa}.

\enumStatement{2}
By \cite[Theorem~2.3]{TomWin07ssa} there is a ${}^*$-isomorphism $\phi\colon A\to D\otimes A$ such that $\phi\approx 1_D\otimes\id_A$ (note that the condition of $D$ being $\mathrm{K}_1$-injective is automatic by \cite[Theorem~3.1, Remark~3.3]{Win11ssaZstable}).
Arguing as in \enumStatement{1}, the natural map $D\times A\to D\otimes A$ induces a $\Cu$-bimorphism $\Cu(D)\times\Cu(A)\to \Cu(D\otimes A)$ and this, composed with $\Cu(\phi^{-1})$, yields a $\Cu$-bimorphism $\varphi\colon \Cu(D)\times \Cu(A)\to\Cu(A)$.

Given $x\in A_+$, we have $\phi(x)\approx 1_D\otimes x$ and therefore $\varphi([1_D],[x])=[x]$.
Given also $d_1,d_2\in D_+$, we have
\[
d_1\otimes d_2\otimes\phi(x) \approx d_1\otimes d_2\otimes 1_D\otimes x
\approx d_1\otimes 1_D\otimes d_2\otimes x
\]
in $D\otimes D\otimes D\otimes A$.
Applying $\phi^{-1}\circ (\psi^{-1}\otimes \phi^{-1})$ to the above relation, we get
\[
\phi^{-1}(\psi^{-1}(d_1\otimes d_2)\otimes x)
\approx \phi^{-1}(\psi^{-1}(d_1\otimes 1_D)\otimes \phi^{-1}(d_2\otimes x)).
\]
Therefore, $\varphi([d_1]\cdot[d_2],[x])=\varphi([d_1],\varphi([d_2],[x]))$.
The same computations hold for positive elements in the stabilizations, which implies that $\varphi$ defines a $\Cu(D)$-multiplication on $\Cu(A)$.
\end{proof}

%------------------------------------------------------------------------------------------
A ring $R$ is called \emph{solid} if the multiplication map induces an isomorphism $R\otimes R\cong R$; see \cite[Definition~2.1; 2.4]{BouKan72Core} where it is pointed out that solidity of the ring $R$ is equivalent to the requirement that $a\otimes 1=1\otimes a$ for every $a\in R$.\index{terms}{ring!solid}
Here, we use the usual tensor product of (discrete) groups and rings, and every ring is understood to be unital and commutative.
See \autoref{sec:poRg} for more details.

As pointed out in \cite{Gut13arXivSolid}, solid rings have also been called $T$-rings and \mbox{$\Z$-epi}\-morphs;
see \cite[Definition~1.6]{BowSch77Rings} and \cite{DicSte84Epimorphs}.
We define solid $\Cu$-semirings in analogy to solid rings.

%==========================================================================================
\begin{dfn}
\label{dfn:solidSemirg}
\index{terms}{Cu-semiring@$\CatCu$-semiring!solid}
A $\Cu$-semiring $R$ is \emph{solid} if the $\Cu$-bimorphism $\varphi\colon R\times R\to R$ defining the multiplication induces an isomorphism $R\otimes_{\CatCu} R\cong R$.
\end{dfn}

%------------------------------------------------------------------------------------------
The next result shows that for a $\Cu$-semiring there are many conditions equivalent to being solid.
This is analogous to the case for rings.
Indeed, for a ring, all of the conditions in \autoref{prp:solidTFAE}, when suitably interpreted, are equivalent to solidity of the ring.
This is known, and most of it is shown in the references mentioned in the paragraph before \autoref{dfn:solidSemirg}.

%==========================================================================================
\index{terms}{Cu-semiring@$\CatCu$-semiring!solid!characterization}
\begin{prp}
\label{prp:solidTFAE}
Let $R$ be a $\Cu$-semiring.
Then the following conditions are equivalent:
\beginEnumStatements
\item
The $\Cu$-semiring $R$ is solid.
\item
Whenever $S$ is a $\Cu$-semimodule over $R$, then the $R$-multiplication on $S$ induces an isomorphism $R\otimes_{\CatCu} S\cong S$.
\item
Whenever $S_1$ and $S_2$ are $\Cu$-semimodules over $R$, and $\tau\colon S_1\times S_2\to T$ is a generalized $\Cu$-bimorphism, then $\tau(ra_1,a_2)=\tau(a_1,ra_2)$ for all $r\in R$ and $a_i\in S_i$.
\item
Every generalized $\Cu$-morphism $S_1\to S_2$ between $\Cu$-semimodules $S_1$ and $S_2$ over $R$ is automatically $R$-linear.
\item
For all $a,b\in R$, we have $a\otimes b=b\otimes a$ in $R\otimes_{\CatCu} R$.
\item
For every $a\in R$, we have $a\otimes 1=1\otimes a$ in $R\otimes_{\CatCu} R$.
\end{enumerate}
\end{prp}
\begin{proof}
The implications `\enumStatement{2} $\Rightarrow$ \enumStatement{1} $\Rightarrow$ \enumStatement{5} $\Rightarrow$ \enumStatement{6}' are clear.

In order to prove that \enumStatement{6} implies \enumStatement{2}, let $S$ be a $\CatCu$-semimodule over $R$.
Given $r\in R$ and $s\in S$, we will show $r\otimes s=1\otimes rs$ in $R\otimes_{\CatCu} S$.
To that end, consider the map
\[
\tau_s\colon R\times R\to R\otimes_{\CatCu} S,\quad
\tau_s(a,b):=a\otimes bs,\quad
\txtFA a,b\in R.
\]
It is straightforward to check that this is a generalized $\CatCu$-bimorphism, whence there is a generalized $\CatCu$-morphism $\bar\tau_s\colon R\otimes_{\CatCu} R\to R\otimes_{\CatCu} S$ such that $\tau_s(a,b)=\bar\tau_s(a\otimes b)$.
Using the assumption at the third step, we obtain
\[
r\otimes s = \tau_s(r,1)=\bar\tau_s(r\otimes 1)=\bar\tau_s(1\otimes r)=\tau_s(1,r)=1\otimes rs.
\]

Let $\varphi\colon R\times S\to S$ be the $\Cu$-bimorphism defining the $R$-multiplication.
This induces a $\Cu$-morphism $\bar\varphi\colon R\otimes_{\CatCu} S\to S$.
We let $\psi\colon S\to R\otimes S$ be the generalized $\Cu$-morphism defined by $\psi(s)=1\otimes s$.
We clearly have $\bar\varphi\circ\psi=\id_S$.
On the other hand, for every $r\in R$ and $s\in S$, using the formula of the previous paragraph at the last step, we have
\[
\psi\circ\bar\varphi(r\otimes s)=\psi(rs)=1\otimes rs=r\otimes s.
\]
It follows that $\psi\circ\bar\varphi=\id_{R\otimes_{\CatCu} S}$, and so $R\otimes_{\CatCu} S\cong S$.

Next, we prove that \enumStatement{2} implies \enumStatement{3}.
It is enough to show that for every $r\in R$, $s_1\in S_1$ and $s_2\in S_2$, we have $rs_1\otimes s_2=s_1\otimes rs_2$ in $S_1\otimes_{\CatCu} S_2$.
To that end, we use the isomorphisms $S_1\cong S_1\otimes_{\CatCu} R$ and $R\otimes_{\CatCu} S_2\cong S_2$ given by assumption, and the natural isomorphism of associativity of the tensor product from \autoref{prp:tensProdAssoc}, to obtain the identifications shown in the following diagram:
\[
\xymatrix@R=10pt{
S_1\otimes_{\CatCu} S_2 \ar@{}[r]|-{\cong} \ar@{}[d]|{\inDown}
& (S_1\otimes_{\CatCu} R)\otimes_{\CatCu} S_2 \ar@{}[r]|{\cong} \ar@{}[d]|<<<{\inDown}
& S_1\otimes_{\CatCu} (R\otimes_{\CatCu} S_2) \ar@{}[r]|-{\cong} \ar@{}[d]|<<<{\inDown}
& S_1\otimes_{\CatCu} S_2 \ar@{}[d]|{\inDown} \\
rs_1\otimes s_2 \ar@{}[r]|<<<<<<<{\longleftrightarrow}
& \parbox[c]{2.3cm}{$(rs_1\otimes 1)\otimes s_2$ \\ $=(s_1\otimes r)\otimes s_2$} \ar@{}[r]|{\longleftrightarrow}
& \parbox[c]{2.3cm}{$s_1\otimes (1\otimes rs_2)$ \\ $=s_1\otimes (r\otimes s_2)$} \ar@{}[r]|>>>>>>{\longleftrightarrow}
& s_1\otimes rs_2.
}
\]

Next, to show that \enumStatement{3} implies \enumStatement{4}, let $S_1$ and $S_2$ be $\CatCu$-semimodules over $R$, and let $\alpha\colon S_1\to S_2$ be a generalized $\Cu$-morphism.
Consider the map
\[
\tau\colon R\times S_1\to S_2,\quad
\tau(r,s):=r\alpha(s),\quad
\txtFA r\in R,s\in S_1,
\]
which is easily seen to be a generalized $\Cu$-bimorphism.
We consider $R$ with the $R$-multiplication given by its $\Cu$-semiring structure.
Given $r\in R$ and $s\in S_1$, we use the assumption at the second step to obtain
\[
\alpha(rs)=\tau(1,rs)=\tau(r,s)=r\alpha(s).
\]
Thus, the map $\alpha$ is $R$-linear, as desired.

Finally, let us show that \enumStatement{4} implies \enumStatement{6}.
We endow the $\CatCu$-semigroup{} $R\otimes_{\CatCu} R$ with two (a priori different) $R$-multiplications induced by
\[
r\cdot(r_1\otimes r_2)=rr_1\otimes r_2,\quad
r\cdot(r_1\otimes r_2)=r_1\otimes rr_2,\quad
\txtFA r,r_1,r_2\in R.
\]
Now, we consider the identity $\Cu$-morphism $\id\colon R\otimes_{\CatCu} R\to R\otimes_{\CatCu} R$, but we equip the source and target with the two different $R$-multiplications.
By assumption, the map $\id_R$ is $R$-linear.
Then, given any $r\in R$, we compute the product of the element $1\otimes 1\in R\otimes_{\CatCu} R$ with $r$ using the different $R$-multiplications.
This gives $r1\otimes 1=1\otimes r1$, as desired.
\end{proof}

%==========================================================================================
\begin{exa}
\label{exa:R}
Consider the $\CatCu$-semigroup{} $[0,\infty]$.
We define the product of finite elements as for real numbers, set $\infty\cdot 0:=0$ and $\infty\cdot a:=\infty$ for every nonzero $a\in[0,\infty]$.
It is easy to check that this defines a $\CatCu$-product on $[0,\infty]$.

Thus, $[0,\infty]$ is a $\CatCu$-semiring.
Let us show that it is solid.
By \autoref{prp:solidTFAE}, it is enough to show that $1\otimes a$ equals $a\otimes 1$ for every $a\in[0,\infty]$.
Given $k,n\in\N$ with $n\neq 0$, we consider the element $\frac{k}{n}\in[0,\infty]$ and we compute
\[
1\otimes\tfrac{k}{n}
= \tfrac{n}{n}\otimes \tfrac{k}{n}
= kn\left(\tfrac{1}{n}\otimes \tfrac{1}{n}\right)
= \tfrac{k}{n}\otimes \tfrac{n}{n}
= \tfrac{k}{n}\otimes 1.
\]
Since rational elements are dense in $[0,\infty]$, we get that $[0,\infty]$ is solid.
\end{exa}

%==========================================================================================
\begin{cor}
\label{prp:solidModuleUnique}
Let $R$ be a solid $\CatCu$-semiring, and let $S$ be a $\CatCu$-semigroup.
Then any two $R$-multiplications on $S$ are equal.
\end{cor}
\begin{proof}
Consider the identity morphism $S\to S$, where the range and target are equipped with the two $R$-multiplications in question.
By \autoref{prp:solidTFAE}, this map is $R$-linear, which means exactly that the two $R$-multiplications are equal.
\end{proof}

%==========================================================================================
\begin{rmk}
\label{rmk:solidModuleIsProperty}
Let $R$ be a solid $\CatCu$-semiring, and let $S$ be a $\CatCu$-semigroup{}.
By \autoref{prp:solidModuleUnique}, $S$ has at most one $R$-multiplication.
Thus, having an $R$-multiplica\-tion is a \emph{property} rather than an additional \emph{structure}.
\end{rmk}

%==========================================================================================
\begin{lma}
\label{prp:moduleTensorProd}
Let $R$ be a $\Cu$-semiring, and let $S$ and $T$ be $\CatCu$-semigroups.
Assume that $S$ has an $R$-multiplication.
Then $S\otimes_{\CatCu} T$ also has an $R$-multiplication that satisfies
\[
r(s\otimes t)=(rs)\otimes t,
\]
for all $r\in R$, $s\in S$ and $t\in T$.
\end{lma}
\begin{proof}
Given $r\in R$, consider the map
\[
\alpha_r\colon S\times T\to S\otimes_{\CatCu} T,\quad
(s,t)\mapsto rs\otimes t,\quad
\txtFA s\in S, t\in T.
\]
This is a generalized $\CatCu$-bimorphism, which therefore induces a generalized $\CatCu$-morphism $\bar\alpha_r\colon S\otimes_{\CatCu} T\to S\otimes_{\CatCu} T$ with $\alpha_r(s\otimes t)=(rs)\otimes t$.
Using the universal properties of the tensor product (see \autoref{prp:tensProdCu}), one shows that the map
\[
R\times(S\otimes_{\CatCu} T)\to S\otimes_{\CatCu} T,\quad
(r,x)\mapsto \bar\alpha_r(x),\quad
\txtFA r\in R, x\in S\otimes_{\CatCu} T,
\]
is a $\CatCu$-bimorphism defining an $R$-multiplication on $S\otimes_\CatCu T$.
\end{proof}

%==========================================================================================
\begin{lma}
\label{prp:constructingSolidMod}
Let $R$ be a solid $\Cu$-semiring, let $S$ be a $\CatCu$-semigroup, and let $\varphi\colon R\times S\to S$ be a generalized $\CatCu$-bimorphism.
Assume that $\varphi(1_R,a)=a$ for every $a\in S$.
Then $\varphi$ defines an $R$-multiplication on $S$.
Thus, $\varphi$ is a $\CatCu$-bimorphism satisfying
\[
\varphi(r_1r_2,a)=\varphi(r_1,\varphi(r_2,a)),
\]
for every $r_1,r_2\in R$ and $a\in S$.
\end{lma}
\begin{proof}
Let $\varphi\colon R\times S\to S$ be a map as in the statement.
Let $\bar{\varphi}\colon R\otimes_{\CatCu} S\to S$ be the generalized $\CatCu$-morphism induced by $\varphi$.
We will show that $\bar{\varphi}$ is a $\CatCu$-isomorphism.
Consider the map
\[
\rho\colon S\to R\otimes_{\CatCu} S,\quad
\rho(a)=1_R\otimes a,\quad
\txtFA a\in S.
\]
It is clear that $\rho$ is a generalized $\CatCu$-morphism.
Using the assumption on $\varphi$ at the third step, we deduce for each $a\in S$ that
\[
\bar{\varphi}\circ\rho(a)
=\bar{\varphi}(1_R\otimes a)
=\varphi(1_R,a)
=a.
\]
Thus, $\bar{\varphi}\circ\rho=\id_S$.
For the converse, consider the generalized $\CatCu$-morphism $\rho\circ\bar{\varphi}\colon R\otimes_{\CatCu} S\to R\otimes_{\CatCu} S$.
For each $r\in R$ and $a\in S$ we have
\[
\rho\circ\bar{\varphi}(r\otimes a) = 1_R\otimes\varphi(r,a).
\]
By \autoref{prp:moduleTensorProd}, $R\otimes_\CatCu S$ has an $R$-multiplication such that $r_1(r_2\otimes a)=(r_1r_2)\otimes a$ for all $r_1,r_2\in R$ and $a\in S$.
It follows from \autoref{prp:solidTFAE} that $\rho\circ\bar{\varphi}$ is $R$-linear.
Using this at the second step, we obtain for each $r\in R$ and $a\in S$ that
\begin{align*}
\rho\circ\bar{\varphi}(r\otimes a)
= \rho\circ\bar{\varphi}(r\cdot (1_R\otimes a))
= r\cdot \left( \rho\circ\bar{\varphi}(1_R\otimes a) \right)
= r\cdot (1_R\otimes a)
= r\otimes a.
\end{align*}
This shows that $\rho\circ\bar{\varphi}$ is the identity map on simple tensors in $R\otimes_{\CatCu} S$.
It follows $\rho\circ\bar{\varphi}=\id_{R\otimes_\CatCu S}$.

In general, every $\CatPom$-isomorphism between $\CatCu$-semigroups automatically preserves the way-below relation and suprema of increasing sequences, since these notions are completely encoded in the order structure.
We clearly have that $\bar{\varphi}$ is an isomorphism of \pom{s}.
Therefore it is also a $\CatCu$-isomorphism.
It follows that $\varphi$ is a $\CatCu$-bimorphism.

Let $r_1,r_2\in R$ and $a\in S$.
Since $\rho\circ\bar{\varphi}=\id_{R\otimes_\CatCu S}$, we have $r_2\otimes a = 1_R \otimes \varphi(r_2,a)$.
Using the $R$-multiplication of $R\otimes_\CatCu S$ to multiply by $r_1$, we deduce
\[
(r_1r_2)\otimes a
= r_1\cdot (r_2\otimes a)
= r_1\cdot(1_R \otimes \varphi(r_2,a))
= r_1 \otimes \varphi(r_2,a).
\]
Applying $\bar{\varphi}$, we have
\[
\varphi(r_1r_2,a)
=\bar{\varphi}((r_1r_2)\otimes a)
=\bar{\varphi}(r_1 \otimes \varphi(r_2,a))
=\varphi(r_1,\varphi(r_2,a)),
\]
as desired.
\end{proof}

%==========================================================================================
\begin{thm}
\label{prp:solidModuleTFAE}
Let $R$ be a solid $\CatCu$-semiring, and let $S$ be a $\CatCu$-semigroup{}.
Then the following are equivalent:
\beginEnumStatements
\item
The $\CatCu$-semigroup{} $S$ has an $R$-multiplication.
\item
There exists a $\CatCu$-isomorphism between $R\otimes_{\CatCu} S$ and $S$.
\item
The map $S\to R\otimes_\CatCu S$ given by $a\mapsto 1_R\otimes a$ is a $\CatCu$-isomorphism.
\end{enumerate}
\end{thm}
\begin{proof}
It is clear that \enumStatement{3} implies \enumStatement{2}.
It follows from \autoref{prp:moduleTensorProd} that \enumStatement{2} implies \enumStatement{1}.
To show that \enumStatement{1} implies \enumStatement{3}, assume that $S$ has an $R$-multiplication $\varphi\colon R\times S\to S$.
By \autoref{prp:solidTFAE}, the induced $\CatCu$-morphism $\bar{\varphi}\colon R\otimes_{\CatCu} S\to S$ is an isomorphism.
It is straightforward to check that the inverse of $\bar{\varphi}$ sends $a$ in $S$ to $1_R\otimes a$.
\end{proof}

%------------------------------------------------------------------------------------------
The following result should be compared with an analogous result for strongly self-absorbing \ca{s} in \cite[Proposition~5.12]{TomWin07ssa}.

%==========================================================================================
\begin{prp}
\label{prp:solidMorImpliesAbsorb}
Let $R$ be a solid $\Cu$-semiring, and let $S$ be another $\Cu$-semiring.
Then the following are equivalent:
\beginEnumStatements
\item
There is a $\CatCu$-isomorphism between $R\otimes_{\CatCu} S$ and $S$.
\item
There exists a unital, multiplicative, generalized $\Cu$-morphism $R\to S$.
\item
There exists a unital, generalized $\Cu$-morphism $R\to S$.
\end{enumerate}
Moreover, if a map as in \enumStatement{3} exists, then it is unique (and therefore automatically multiplicative).
Furthermore, if $1_S$ is a compact element, then any map as in \enumStatement{3} is automatically a $\CatCu$-morphism.
\end{prp}
\begin{proof}
It is clear that \enumStatement{2} implies \enumStatement{3}.
Let $\alpha\colon R\to S$ be a unital, generalized $\Cu$-morphism.
This induces a map
\[
\varphi_\alpha\colon R\times S\to S,\quad
(r,a)\mapsto \alpha(r)\cdot a,\quad
\txtFA r\in R, a\in S.
\]
It follows easily from the properties of $\alpha$ that $\varphi_\alpha$ is a generalized $\CatCu$-bimorphism satisfying $\varphi_\alpha(1_R,a)=a$ for every $a\in S$.
Then it follows from \autoref{prp:constructingSolidMod} that $\varphi_\alpha$ is a $\CatCu$-bimorphism defining an $R$-multiplication on $S$.

Thus, using \autoref{prp:solidModuleTFAE}, we obtain that \enumStatement{3} implies \enumStatement{1}.
Moreover, by \autoref{prp:solidModuleUnique}, any two $R$-multiplications on $S$ are equal.
This implies that a map satisfying \enumStatement{3} is unique (if it exists).

Let us show that \enumStatement{1} implies \enumStatement{2}.
By \autoref{prp:moduleTensorProd}, the semigroup $R\otimes_\CatCu S$ has an $R$-multiplication.
Therefore, by assumption, $S$ has an $R$-multiplication $\varphi\colon R\times S\to S$.
Consider the map
\[
\alpha\colon R\to S,\quad
r\mapsto \varphi(r,1_S),\quad
\txtFA r\in R.
\]
It is clear that $\alpha$ is a unital, generalized $\CatCu$-morphism.
In order to show that $\alpha$ is multiplicative, we consider the following map
\[
\psi\colon R\times S\to S,\quad
(r,a)\mapsto \varphi(r,1_S)a,\quad
\txtFA r\in R, a\in S.
\]
It is easy to see that $\psi$ is a generalized $\CatCu$-bimorphism satisfying $\psi(1_R,a)=a$ for every $a\in S$.
By \autoref{prp:constructingSolidMod}, we have $\psi=\varphi$.
Using this at the third step, we deduce
\begin{align*}
\alpha(r_1r_2)
=\varphi(r_1r_2,1_S)
&=\varphi(r_1,\varphi(r_2,1_S)) \\
&=\psi(r_1,\varphi(r_2,1_S))
=\varphi(r_1,1_S)\varphi(r_2,1_S)
=\alpha(r_1)\alpha(r_2).
\end{align*}
Thus, $\alpha$ is multiplicative, as desired.

Finally, if $1_S$ is compact, then it is also clear from the definition that $\alpha$ is a $\CatCu$-morphism.
\end{proof}

\vspace{5pt}
%------------------------------------------------------------------------------------------
%==========================================================================================
\section[Cuntz semigroups of purely infinite {C}*-algebras]{Cuntz semigroups of purely infinite \texorpdfstring{{C}*-algebras}{{C}*-algebras}}
\label{sec:pureInf}

%------------------------------------------------------------------------------------------
In this section, we study $\CatCu$-semigroups that are semimodules over the $\CatCu$-sem\-i\-ring $\{0,\infty\}$.
If $A$ is a purely infinite, strongly self-absorbing \ca{} (for example $\mathcal{O}_2$, $\mathcal{O}_\infty$, or the tensor product of $\mathcal{O}_\infty$ with a UHF-algebra of infinite type), then $\Cu(A)\cong\{0,\infty\}$.

In \autoref{prp:PIModTFAE}, we characterize the $\CatCu$-semimodules over $\{0,\infty\}$ as the $\CatCu$-sem\-i\-groups that are idempotent.
We show that the tensor product of a given $\CatCu$-semigroup $S$ with $\{0,\infty\}$ is canonically isomorphic to $\Lat_{\mathrm{f}}(S)$, the semigroup of singly-generated ideals in $S$;
see \autoref{prp:tensWithInfty}.

In \autoref{prp:pureInfCa}, we apply our results to Cuntz semigroups of \ca{s} by showing that a (not necessarily simple) \ca{} $A$ is purely infinite if and only if
\[
\Cu(A)\cong \{0,\infty\}\otimes_{\CatCu}\Cu(A).
\]
We deduce that for every separable \ca{} $A$, there are natural isomorphisms of the following $\CatCu$-semigroups:
\[
\Cu(\mathcal{O}_\infty\otimes A)
\cong \Lat(A)
\cong \Lat(\Cu(A))
\cong \{0,\infty\}\otimes_{\CatCu}\Cu(A);
\]
see \autoref{prp:tensCaWithInfty}.

%==========================================================================================
The following is easy to prove and hence we omit the details:

\begin{lma}
\label{prp:PureInfSolid}
The $\CatCu$-semiring $\{0,\infty\}$ is solid.
\end{lma}

%------------------------------------------------------------------------------------------
Recall that a commutative semigroup $S$ is called \emph{idempotent} if each of its elements is idempotent, that is, if $2a=a$ for every $a\in S$. \index{terms}{semigroup!idempotent}
In the literature, an idempotent, commutative semigroup $S$ is also called a \emph{commutative band}, or a \emph{semilattice} (with `join' in the semilattice corresponding to addition in the semigroup).

%==========================================================================================
\begin{thm}
\label{prp:PIModTFAE}
Let $S$ be a $\CatCu$-semigroup.
Then the following are equivalent:
\beginEnumStatements
\item
We have $S\cong \{0,\infty\}\otimes_{\CatCu} S$.
\item
The $\CatCu$-semigroup $S$ has a $\{0,\infty\}$-multiplication.
\item
The semigroup $S$ is idempotent.
\end{enumerate}
\end{thm}
\begin{proof}
Since $\{0,\infty\}$ is a solid $\CatCu$-semiring, the equivalence of \enumStatement{1} and \enumStatement{2} follows from \autoref{prp:solidModuleTFAE}.
To show that \enumStatement{2} implies \enumStatement{3}, let $a$ be an element of $S$.
Since $\infty$ is the unit of the $\CatCu$-semiring $\{0,\infty\}$ and since $2\infty=\infty$ in $\{0,\infty\}$, we obtain
\[
a = \infty\cdot a = (2\infty)\cdot a = 2(\infty\cdot a) = 2a,
\]
as desired.

Next, to prove that \enumStatement{3} implies \enumStatement{2}, assume that $S$ is an idempotent $\CatCu$-semigroup.
Consider the map
\[
\varphi\colon\{0,\infty\}\times S\to S,\quad
(0,a)\mapsto 0,\quad
(\infty,a)\mapsto a,\quad
\txtFA a\in S.
\]
We have that $\varphi$ is a generalized $\CatCu$-morphism in the second variable.
It is also clear that $\varphi$ preserves zero, order and suprema of increasing sequences (there are no nontrivial ones) in the first variable.
Using that $S$ is idempotent, it follows that $\varphi$ is additive in the first variable.
By \autoref{prp:bimorCu}, we have that $\varphi$ is a generalized $\CatCu$-bimorphism.
Then we obtain from \autoref{prp:constructingSolidMod} that $S$ has $\{0,\infty\}$-multiplication, as desired.
\end{proof}

%------------------------------------------------------------------------------------------
Recall that, for a given $\CatCu$-semigroup $S$, we denote by $\Lat_{\mathrm{f}}(S)$ the $\CatCu$-sem\-i\-group of singly-generated ideals in $S$, as considered in \autoref{prp:LatCu}.

%==========================================================================================
\begin{prp}
\label{prp:tensWithInfty}
Let $S$ be a $\CatCu$-semigroup.
Then there is a natural $\CatCu$-isomor\-phism
\[
\{0,\infty\}\otimes_{\CatCu}S\cong\Lat_{\mathrm{f}}(S)
\]
identifying $\infty\otimes a$ with $\Idl(a)$.
\end{prp}
\begin{proof}
Let $S$ be a $\CatCu$-semigroup, and let $\varphi\colon \{0,\infty\}\times S\to \{0,\infty\}\otimes_{\CatCu}S$ denote the universal $\CatCu$-bimorphism.
Consider the map
\[
\tau\colon \{0,\infty\}\times S\to\Lat_{\mathrm{f}}(S),\quad
\tau(0,a)=0,\quad
\tau(\infty,a)=\Idl(a),\quad
\txtFA a\in S.
\]
It follows from \autoref{prp:LatCu} that $\tau$ is $\CatCu$-bimorphism.
Then there is a $\CatCu$-morphism
\[
\tilde{\tau} \colon \{0,\infty\}\otimes_{\CatCu}S\to\Lat_{\mathrm{f}}(S),
\]
such that $\tau=\tilde{\tau}\circ\varphi$.
It is clear that $\tilde{\tau}$ is a surjective $\CatCu$-morphism.

As shown in \autoref{pgr:Lat}, every ideal $I$ in $\Lat_{\mathrm{f}}(S)$ contains a largest element, denoted by $\bigvee I$.
We may therefore define a map
\[
\psi \colon \Lat_{\mathrm{f}}(S) \to \{0,\infty\}\otimes_{\CatCu}S,\quad
I\mapsto \infty\otimes\left( \bigvee I \right),\quad
\txtFA I\in\Lat_{\mathrm{f}}(S).
\]
It is easy to see that $\psi$ is a generalized $\CatCu$-morphism.
To see that $\psi$ preserves the way-below relation, let $I,J\in\Lat_{\mathrm{f}}(S)$ satisfy $I\ll J$.
Set $a:=\bigvee I$ and $b:=\bigvee J$.
Let $(b_n)_n$ be a rapidly increasing sequence in $S$ with supremum $b$.
Then $J = \sup_n \Idl(b_n)$ in $\Lat_{\mathrm{f}}(S)$.
By assumption, there exists $n$ such that $I\subset\Idl(b_n)$.
Using this at the second step, and using that $\infty\ll\infty$ and $b_n\ll b$ at the third step, we deduce
\[
\psi(I) = \infty\otimes a \leq \infty\otimes b_n \ll \infty\otimes b = \psi(J).
\]

Given $a\in S$, we clearly have
\[
\psi\circ\tilde{\tau}\circ\varphi(0,a) = 0 = \varphi(0,a).
\]
We can also deduce
\[
\psi\circ\tilde{\tau}\circ\varphi(\infty,a)
= \infty\otimes\left( \bigvee \Idl(a) \right)
= \infty\otimes(\infty\cdot a)
= \infty\otimes a
= \varphi(\infty,a).
\]
Thus, we have shown that $\psi\circ\tilde{\tau}\circ\varphi=\varphi$, which implies that $\psi\circ\tilde{\tau}$ is the identity on $\{0,\infty\}\otimes_{\CatCu}S$.
Therefore, the map $\tilde{\tau}$ is an order isomorphism, and therefore an isomorphism in $\CatCu$.
\end{proof}

%==========================================================================================
\begin{cor}
Let $S$ be a simple, nonzero $\CatCu$-semigroup{}.
Then
\[
\{0,\infty\}\otimes_{\CatCu}S\cong\{0,\infty\}.
\]
\end{cor}

%==========================================================================================
\begin{cor}
\label{prp:tensLat}
Let $S$ and $T$ be $\CatCu$-semigroup{s}.
Then there is a natural isomorphism
\[
\Lat_{\mathrm{f}}(S\otimes_{\CatCu} T) \cong \Lat_{\mathrm{f}}(S)\otimes_{\CatCu} \Lat_{\mathrm{f}}(T).
\]
\end{cor}
\begin{proof}
It is clear that $\{0,\infty\}\cong \{0,\infty\}\otimes_{\CatCu} \{0,\infty\}$.
Using \autoref{prp:tensWithInfty} at the first and last step, and using the associativity and symmetry of the tensor product (see \autoref{prp:tensProdAssoc} and \autoref{pgr:monoidalCu}) at the second step, we obtain
\begin{align*}
\Lat_{\mathrm{f}}(S\otimes_{\CatCu} T)
&\cong \{0,\infty\}\otimes_{\CatCu}(S\otimes_{\CatCu} T) \\
&\cong (\{0,\infty\}\otimes_{\CatCu}S) \otimes_{\CatCu} (\{0,\infty\}\otimes_{\CatCu} T)
\cong \Lat_{\mathrm{f}}(S)\otimes_{\CatCu} \Lat_{\mathrm{f}}(T),
\end{align*}
as desired.
\end{proof}

%==========================================================================================
\begin{thm}
\label{prp:MapToPIMod}
Let $S$ be a $\CatCu$-semigroup, and let $a,b\in S$.
Then the following are equivalent:
\beginEnumStatements
\item
We have $1\otimes a\leq 1\otimes b$ in $\{0,\infty\}\otimes_{\CatCu} S$, where $1=\infty$ is the unit of $\{0,\infty\}$.
\item
We have $a\varpropto^\ctsRel b$, that is, $a\leq\infty\cdot b$.
\end{enumerate}
\end{thm}
\begin{proof}
Let $a,b\in S$.
By \autoref{prp:tensWithInfty}, there is an isomorphism between $\{0,\infty\}\otimes_{\CatCu}S$ and $\Lat_{\mathrm{f}}(S)$ that identifies the simple tensor $1 \otimes a$ with $\Idl(a)$.
Then we have $1\otimes a\leq 1\otimes b$ in $\{0,\infty\}\otimes_{\CatCu}S$ if and only if $\Idl(a)\subset\Idl(b)$ in $\Lat_{\mathrm{f}}(S)$, which in turn happens if and only if $a\varpropto^\ctsRel b$, as desired.
\end{proof}

%==========================================================================================
\begin{prp}
\label{prp:axiomTensProdPI}
Let $S$ be a $\CatCu$-semigroup.
Then $\{0,\infty\}\otimes_{\CatCu}S$ is unperforated, divisible and satisfies \axiomO{5}.
Moreover, if $S$ satisfies \axiomO{6}, then so does $\{0,\infty\}\otimes_{\CatCu}S$.
\end{prp}
\begin{proof}
Set $T:=\{0,\infty\}\otimes_{\CatCu}S$, which by \autoref{prp:PIModTFAE} is an idempotent $\CatCu$-semigroup.
To show that it is unperforated, let $a,b\in T$ satisfy $na\leq nb$ for some $n\in\N_+$.
Since $a=na$ and $b=nb$, we immediately get $a\leq b$.
Similarly, $T$ is divisible.
The statements about \axiomO{5} and \axiomO{6} follow directly by combining \autoref{prp:tensWithInfty} with \autoref{prp:LatCu}.
\end{proof}

%------------------------------------------------------------------------------------------
A not necessarily simple \ca{} $A$ is \emph{purely infinite} \index{terms}{C*-algebra@\ca{}!purely infinite}  \index{terms}{purely infinite \ca{}} if $A$ has no characters and if for any positive elements $x$ and $y$ in $A$ we have $x\precsim y$ whenever $x$ is contained in the ideal of $A$ generated by $y$; see \cite[Definition~4.1]{KirRor00PureInf}, see also \cite[p.450ff]{Bla06OpAlgs}.
By \cite[Theorem~4.23]{KirRor00PureInf}, if $A$ is purely infinite, then so is $A\otimes\K$.

A nonzero element $x$ in $A_+$ is \emph{properly infinite}\index{terms}{properly infinite} if $x\oplus x\precsim x$, considered in $M_2(A)$.
If we denote by $[x]$ the class of $x$ in $\Cu(A)$, then $x$ is properly infinite if and only if $2[x]=[x]$ and $[x]\neq 0$.
By \cite[Theorem~4.16]{KirRor00PureInf}, a \ca{} is purely infinite if and only each of its nonzero positive elements is properly infinite.
Using \autoref{prp:PIModTFAE}, we may reformulate the result of Kirchberg and R{\o}rdam as follows:

%==========================================================================================
\begin{prp}
\label{prp:pureInfCa}
Let $A$ be a \ca.
Then $A$ is purely infinite if and only if $\Cu(A)\cong \{0,\infty\}\otimes_{\CatCu}\Cu(A)$.
\end{prp}

%------------------------------------------------------------------------------------------
It follows from Propositions~\ref{prp:axiomTensProdPI} and~\ref{prp:pureInfCa} that the Cuntz semigroup of every purely infinite \ca{} is unperforated.
This verifies \autoref{conj:nearUnpCaZstable} for the class for purely infinite \ca{s}.

%==========================================================================================
\begin{cor}
\label{prp:pureInfCanearUnp}
Let $A$ be a purely infinite \ca.
Then $\Cu(A)$ is nearly unperforated.
\end{cor}

%==========================================================================================
\begin{pgr}
\label{pgr:LatTens}
Let $A$ and $B$ be \ca{s}, and let $A_0\subset A$ and $B_0\subset B$ be sub-\ca{s}.
Then the natural map between the algebraic tensor products, $A_0\odot B_0\to A\odot B$ induces an isometric embedding $A_0\tensMin B_0 \subset A\tensMin B$;
see \cite[II.9.6.2, p.199]{Bla06OpAlgs}.
Given ideals $I\in\Lat(A)$ and $J\in\Lat(B)$, it is easy to see that $I\tensMin J$ is an ideal in $A\tensMin B$.
Moreover, if $I\in\Lat_{\mathrm{f}}(A)$ and $J\in\Lat_{\mathrm{f}}(B)$, then $I\tensMin J\in\Lat_{\mathrm{f}}(A\tensMin B)$.
Indeed, if $x$ and $y$ are positive, full elements in $I$ and $J$, respectively, then $x\otimes y$ is a positive, full element in $I\tensMin J$.
We leave it to the reader to check that the map
\[
\Lat_{\mathrm{f}}(A)\times\Lat_{\mathrm{f}}(B)\to\Lat_{\mathrm{f}}(A\tensMin B),\quad
(I,J)\mapsto I\tensMin J,
\]
is a $\CatCu$-bimorphism.
It induces a $\CatCu$-morphism
\[
\psi_{A,B}\colon\Lat_{\mathrm{f}}(A)\otimes_\CatCu\Lat_{\mathrm{f}}(B)\to\Lat_{\mathrm{f}}(A\tensMin B).
\]
\end{pgr}

%==========================================================================================
\begin{prbl}
\label{prbl:LatTens}
Let $A$ and $B$ be \ca{s}.
When is the map $\psi_{A,B}$ an isomorphism?
When is it surjective?
When is it an order-embedding?
\end{prbl}

%------------------------------------------------------------------------------------------
We will give partial answers to this problem in \autoref{prp:Latf}.
We expect that $\psi_{A,B}$ is always an order-embedding.

%==========================================================================================
\begin{lma}
\label{prp:idealsTens}
Let $A$ and $B$ be \ca{s}, let $I_1,I_2$ be ideals in $\Lat(A)$, and let $J_1,J_2$ be ideals in $\Lat(B)$.
Then:
\beginEnumStatements
\item
We have $(I_1\tensMin J_1)\cap(I_2\tensMin J_2) = (I_1\cap I_2)\tensMin(J_1\cap J_2)$.
\item
We have $I_1\tensMin J_1 \subset I_2\tensMin J_2$ if and only if $I_1=0$, or $J_1=0$, or $I_1\subset I_2$ and $J_1\subset J_2$.
\end{enumerate}
\end{lma}
\begin{proof}
\enumStatement{1}.
The inclusion `$\supset$' is clear.
Let us show the converse.
In general, for ideals $K$ and $L$ in a \ca{}, we have $K\cap L = KL = \left\{ab : a\in K, b\in L\right\}$.
Now, given $x\in I_1\odot J_1$ and $y\in I_2\odot J_2$ in the respective algebraic tensor products, one easily verifies that $xy$ is in $(I_1I_2)\odot(J_1J_2)$.
Passing to closures, we obtain
\begin{align*}
(I_1\tensMin J_1)\cap(I_2\tensMin J_2)
&= (I_1\tensMin J_1)(I_2\tensMin J_2) \\
&\subset (I_1I_2)\tensMin(J_1J_2)
= (I_1\cap I_2)\tensMin(J_1\cap J_2).
\end{align*}

\enumStatement{2}.
The `if' implication is clear.
To show the converse, assume $I_1\tensMin J_1 \subset I_2\tensMin J_2$.
A minimal tensor product of \ca{s} is only zero if each of the factors is zero.
Thus we may assume that $I_1,J_1,I_2$, and $J_2$ are nonzero.

Using \enumStatement{1} at the second step, we obtain
\begin{align}
\label{prp:idealsTens:eq1}
I_1\tensMin J_1
= (I_1\tensMin J_1)\cap(I_2\tensMin J_2)
= (I_1\cap I_2)\tensMin(J_1\cap J_2).
\end{align}
To show $I_1\cap I_2=I_1$, let us assume the converse, that is, assume there is $c\in I_1$ with $c\notin I_1\cap I_2$.
Using the Hahn-Banach theorem, choose a state $\varphi$ on $A$ with $\varphi(I_1\cap I_2)=\{0\}$ and $\varphi(c)=1$.
Consider the slice map $R_\varphi\colon A\tensMin B\to B$, which satisfies $R_\varphi(a\otimes b)=\varphi(a)b$, for $a\in A$ and $b\in B$;
see \cite[II.9.7.1., p.203]{Bla06OpAlgs}.
Using linearity and continuity of $R_\varphi$, we obtain $R_\varphi(x)=0$ for all $x\in (I_1\cap I_2)\tensMin(J_1\cap J_2)$.
In particular, by \eqref{prp:idealsTens:eq1}, we have $R_\varphi(c\otimes b)=0$ for all $b\in J_1$.
On the other hand, we have $R_\varphi(c\otimes b)=b$ for all $b\in B$.
Since $J_1\neq\{0\}$, this is a contradiction, which implies $I_1\cap I_2=I_1$ and hence $I_1\subset I_2$.
Analogously, we deduce $J_1\subset J_2$.
\end{proof}

%==========================================================================================
\begin{rmk}
\label{rmk:idealsTens}
Let $A$ and $B$ be \ca{s}, and let $\psi_{A,B}$ be the $\CatCu$-mor\-phism as defined in \autoref{pgr:LatTens}.
Then \autoref{prp:idealsTens}(2) shows that $\psi_{A,B}$ determines the order of simple tensors in $\Lat_{\mathrm{f}}(A)\otimes_\CatCu\Lat_{\mathrm{f}}(B)$.
More precisely, given $I_1,I_2\in\Lat_{\mathrm{f}}(A)$, and $J_1,J_2\in\Lat_{\mathrm{f}}(B)$, the following are equivalent:
\beginEnumStatements
\item
We have $I_1\otimes J_1 \leq I_2\otimes J_2$ in $\Lat_{\mathrm{f}}(A)\otimes_\CatCu\Lat_{\mathrm{f}}(B)$.
%\item
%We have $I_1\tensMin J_1 \subset I_2\tensMin J_2$ as subsets of $A\tensMin B$.
\item
We have $\psi_{A,B}(I_1\otimes J_1) \leq \psi_{A,B}(I_2\otimes J_2)$ in $\Lat_{\mathrm{f}}(A\tensMin B)$.
\end{enumerate}
\end{rmk}

%==========================================================================================
\begin{prp}
\label{prp:Latf}
Let $A$ and $B$ be separable \ca{s}, and let $\psi_{A,B}$ be the $\CatCu$-morphism as defined in \autoref{pgr:LatTens}.
Assume $B$ is simple.
Then:
\beginEnumStatements
\item
The map $\psi_{A,B}$ is an order-embedding.
\item
The map $\psi_{A,B}$ is surjective if and only if the (equivalent) conditions of Proposition~2.16 in \cite{BlaKir04PureInf} are satisfied.
In particular, this is the case if $A$ or $B$ is exact;
see \cite[Proposition~2.17]{BlaKir04PureInf}.
\end{enumerate}
\end{prp}
\begin{proof}
\enumStatement{1}.
Since $B$ is simple, we have $\Lat_{\mathrm{f}}(B)\cong\{0,\infty\}$.
Therefore, every element in $\Lat_{\mathrm{f}}(A)\otimes_\CatCu\Lat_{\mathrm{f}}(B)$ is a simple tensor, and thus the result follows since $\psi_{A,B}$ determines the order of simple tensors;
see \autoref{rmk:idealsTens}.

\enumStatement{2}.
Given $K\in\Lat(A\tensMin B)$, it is straightforward to see that $K$ lies in the image of the map $\psi_{A,B}$ if and only if $K$ is the supremum of the ideals $I\tensMin J\subset K$ with $I\in\Lat(A)$ and $J\in\Lat(B)$.
This is equivalent to condition~(ii) in \cite[Proposition~2.16]{BlaKir04PureInf}.
\end{proof}

%==========================================================================================
\begin{cor}
\label{prp:tensCaWithSimplePI}
Let $A$ and $B$ be separable \ca{s}.
Assume $B$ is simple, nuclear and purely infinite.
Then the natural map $\tau_{A,B}$ is an isomorphism.
Thus, there are natural isomorphisms:
\[
\Cu(A\otimes B)
\cong \Cu(A)\otimes_{\CatCu}\Cu(B)
\cong \Cu(A)\otimes_{\CatCu}\{0,\infty\}.
\]
\end{cor}
\begin{proof}
We have $\Cu(B)\cong\{0,\infty\}$, and therefore $\Cu(A)\otimes_\CatCu\Cu(B)\cong\Lat(A)$ by \autoref{prp:tensWithInfty}.
Kirchberg's $\mathcal{O}_\infty$-absorption theorem (\cite[Theorem~7.2.6, p.113]{Ror02Classification}) implies $B\cong B\otimes\mathcal{O}_\infty$, whence $A\tensMin B$ is purely infinite.
Therefore $\Cu(A\tensMin B)\cong\Lat(A\tensMin B)$, and the result follows from \autoref{prp:Latf}.
\end{proof}

%==========================================================================================
\begin{cor}
\label{prp:tensCaWithInfty}
Let $A$ be a separable \ca{}.
Then there are natural isomorphisms between the following $\CatCu$-semigroups:
\[
\Cu(\mathcal{O}_\infty\otimes A)
\cong \Lat(A)
\cong \Lat(\Cu(A))
\cong \{0,\infty\}\otimes_{\CatCu}\Cu(A).
\]
\end{cor}

%==========================================================================================
\begin{rmk}
Note that it follows from our observations that a \ca{} $A$ is purely infinite if and only if $\Cu(A)\cong\Cu(\mathcal{O}_{\infty}\otimes A)$.
\end{rmk}

%==========================================================================================
\begin{pgr}
Let $A$ be a \ca{}.
Consider the map
\[
I_A\colon\Cu(A)\to\Lat_{\mathrm{f}}(A),
\]
that sends the class of $x\in (A\otimes\K)_+$ to the ideal generated by $x$.
(We implicitly identify $\Lat_{\mathrm{f}}(A)\cong\Lat_{\mathrm{f}}(A\otimes\K)$.)
The map $I_A$ agrees with the composition of the map $\Cu(A)\to\Lat_{\mathrm{f}}(\Cu(A))$ from \autoref{prp:LatCu} with the isomorphism $\Lat_{\mathrm{f}}(\Cu(A))\cong\Lat_{\mathrm{f}}(A)$ from \autoref{prp:LatCa}.
By \autoref{prp:pureInfCa}, the \ca{} $A$ is purely infinite if and only if $I_A$ is an isomorphism.

Now, let $A$ and $B$ be \ca{s} and assume that $A$ or $B$ is purely infinite.
In \cite[Question~5.12]{KirRor00PureInf}, Kirchberg and R{\o}rdam ask whether it follows that $A\tensMin B$ is purely infinite.
Consider the following commutative diagram:
\[
\xymatrix@R=15pt{
\Cu(A)\otimes_\CatCu\Cu(B) \ar[r]^-{\tau_{A,B}^\txtMin} \ar[d]^{I_A\otimes I_B}
& \Cu(A\tensMin B) \ar[d]^{I_{A\tensMin B}} \\
\Lat_{\mathrm{f}}(A)\otimes_\CatCu\Lat_{\mathrm{f}}(B) \ar[r]^-{\psi_{A,B}}
& \Lat_{\mathrm{f}}(A\tensMin B).
}
\]
Using that $A$ or $B$ is purely infinite, it is straightforward to check that $I_A\otimes I_B$ is an isomorphism.
Moreover, $A\tensMin B$ is purely infinite if and only if $I_{A\tensMin B}$ is an isomorphism.

If $\tau_{A,B}^\txtMin$ is an isomorphism, then $A\tensMin B$ is purely infinite.
Conversely, if $A\tensMin B$ is purely infinite, then $\tau_{A,B}^\txtMin$ is an isomorphism if and only if $\psi_{A,B}$ is.
This connects Problems~\ref{prbl:tensCa}, \ref{prbl:LatTens} and \cite[Question~5.12]{KirRor00PureInf}.
\end{pgr}

\vspace{5pt}
%------------------------------------------------------------------------------------------
%==========================================================================================
\section{Almost unperforated and almost divisible \texorpdfstring{$\CatCu$}{Cu}-semigroups}
\label{sec:ZMod}

%------------------------------------------------------------------------------------------
In this section we study $\CatCu$-semigroups that are semimodules over the $\CatCu$-semiring of the Jiang-Su algebra $\mathcal{Z}$.
We use $Z$ to denote $\Cu(\mathcal{Z})$ and we begin by showing that $Z$ is a solid $\CatCu$-semiring;
see \autoref{prp:ZSolid}.
The main result of this section is \autoref{prp:ZModTFAE}, where we characterize the $\CatCu$-semimodules over $Z$ as the $\CatCu$-semigroups that are almost unperforated and almost divisible.
This can be interpreted as a verification of the $\CatCu$-semigroup version of the Toms-Winter conjecture;
see \autoref{rmk:TWConjecture}.

\vspace{5pt}
%------------------------------------------------------------------------------------------
Let $A$ be a \ca.
Recall that $V(A)$ denotes the Murray-von Neumann semigroup of equivalence classes of projections in matrices over $A$.
We use $\QT_2(A)$ to denote the set of $2$-quasitraces on $A$;
see \autoref{sec:fctl}.
By a famous result of Haagerup, \cite{Haa14arXivQuasitraces}, every $2$-quasitrace on a unital, exact \ca{} is a trace.
We let $\Lsc(\QT_2(A))$ denote the set of lower-semicontinuous linear functions from the cone $\QT_2(A)$ to $[0,\infty]$.
The $\CatCu$-semigroup $L(\QT_2(A))$ is defined as a certain subset of $\Lsc(\QT_2(A))$;
see \autoref{pgr:fctl} and \cite{EllRobSan11Cone} for more details.
If $A$ is simple and unital, then $L_b(\QT_2(A))$ is defined as the elements in $L(\QT_2(A))$ that are bounded by a finite multiple of the function $\hat{1}\in L(\QT_2(A))$ associated to the unit of $A$.

The following result is the combination of work of many people and has appeared in several (partial) versions in the literature.
In the formulation presented here, it can be found as Corollary~6.8 and Remark~6.9 in \cite{EllRobSan11Cone}.
Equivalent results and previous partial results can be found in Theorems~4.4 and~6.5 in \cite{PerTom07Recasting}, Theorem~2.6 in \cite{BroTom07ThreeAppl}, Theorems~6.2 and~6.3 in \cite{AntBosPer11CompletionsCu}, Theorem~5.27 in \cite{AraPerTom11Cu}, and Theorem~5.5 in \cite{BroPerTom08CuElliottConj}.

%==========================================================================================
\begin{prp}[{A number of people}]
\label{prp:CuSimpleZstable}
Let $A$ be a unital, separable, simple, finite, $\mathcal{Z}$-stable \ca{}.
Then the (pre)completed Cuntz semigroup of $A$ can be computed as:
\[
W(A) \cong V(A)^\times \sqcup L_b(\QT_2(A)),\quad
\Cu(A) \cong V(A)^\times \sqcup L(\QT_2(A)).
\]
In particular, if $A$ is exact and has a unique tracial state, then
\[
W(A) \cong V(A) \sqcup (0,\infty),\quad
\Cu(A) \cong V(A) \sqcup (0,\infty].
\]
\end{prp}

%==========================================================================================
\begin{pgr}
\label{pgr:Z}\index{terms}{Cuntz semigroup!of Jiang-Su algebra}
Now let $Z$ be the (completed) Cuntz semigroup of the Jiang-Su algebra $\mathcal{Z}$.
Using \autoref{prp:CuSimpleZstable}, we can compute $Z$ as
\[
Z=\N\sqcup(0,\infty],
\]
where the elements of $\N\subset Z$ are compact and the elements of $(0,\infty]\subset Z$ are soft;
for the concrete case of the Cuntz semigroup of the Jiang-Su algebra, this computation has also appeared in \cite[Theorem~3.1]{PerTom07Recasting}. (We are assuming here that the set $\N$ contains $0$.)

Using this decomposition into two parts, the addition, multiplication and order on $Z$ are defined as usual in each of the parts;
see also \autoref{exa:R}.
Given a compact, nonzero element $n\in\N\subset Z$, we let $n'\in(0,\infty]$ denote the associated soft element with the same number.
Then, given $n\in\N$ and $a\in(0,\infty]$, we define $n+a=n'+a$ and $na=n'a$.
Thus, the soft part of $Z$ is additively and multiplicatively absorbing.
For a compact element $n\in Z$ and a soft element $a\in Z$ we have $n\leq a$ if and only if $n'< a$; and we have $a\leq n$ if and only if $a\leq n'$.
\end{pgr}

%==========================================================================================
\begin{prp}
\label{prp:ZSolid}
The $\CatCu$-semiring $Z=\N\sqcup(0,\infty]$ is solid.
\end{prp}
\begin{proof}
By \autoref{prp:solidTFAE}, it is enough to show that $1\otimes a=a\otimes 1$ for every $a\in Z$.
This follows easily for compact elements in $Z$, since they are multiples of the unit.
In the other case, if $a\in Z$ is a soft element, we can use the same argument that was used in \autoref{exa:R} to show that $[0,\infty]$ is solid.
\end{proof}

%==========================================================================================
\begin{dfn}
\label{dfn:almDiv}
\index{terms}{element!almost $k$-divisible}
\index{terms}{element!almost divisible}
\index{terms}{almost divisible}
\index{terms}{positively ordered monoid!almost divisible}
Let $S$ be a $\CatCu$-semigroup, $a$ in $S$, and $k\in\N_+$.
We say that $a$ is \emph{almost $k$-divisible} if for any $a'\in S$ with $a'\ll a$ there exists $x\in S$ such that
\[
kx\leq a,\quad
a'\leq (k+1)x.
\]
We say that $a$ is \emph{almost divisible} if it is almost $k$-divisible for every $k\in\N_+$.
We say that $S$ is \emph{almost divisible} if each of its elements is.
\end{dfn}

%==========================================================================================
\begin{rmk}
\label{rmk:almDiv}
Divisibility properties of $\CatCu$-semigroups are studied in detail in \cite{RobRor13Divisibility}.
The above definition of `almost divisibility' for $\CatCu$-semigroups can be found as Definition~3.1 in connection with Remark~3.13 in \cite{RobRor13Divisibility}.

The Cuntz semigroup $\Cu(A)$ of a \ca{} $A$ has sometimes been called almost divisible provided that for all $k\in\N$ and $a\in\Cu(A)$, there exists $x\in\Cu(A)$ such that $kx\leq a\leq (k+1)x$.
This is the natural definition of \emph{almost divisibility} for \pom{s}.
However, for $\CatCu$-semigroups this is in general a stronger form of divisibility than the one in \autoref{dfn:almDiv}.

As suggested by L.~Robert, we use the notion of almost divisibility for $\CatCu$-semigroups as in \autoref{dfn:almDiv} since this is the natural property that is preserved by quotients, limits and ultraproducts of \ca{s} (see \cite{RobTik14} and \cite{RobRor13Divisibility}).
Moreover, as we show below (\autoref{prp:alternativealmDiv}), in the presence of almost unperforation, both definitions agree for $\CatCu$-semigroups.
Therefore, whenever a $\CatCu$-semigroup is almost unperforated, we will make no distinction on which form of divisibility is being used.

Note also that we introduce the notion of `almost $k$-divisibility' for each $k$ only for technical reasons.
This should not be confused with the notion of `$k$-almost divisibility' as in \cite{RobTik14}, \cite{RobRor13Divisibility} or \cite{Win12NuclDimZstable}.
\end{rmk}

%==========================================================================================
\begin{lma}
Let $S$ be an algebraic $\CatCu$-semigroup.
Then $S$ is almost divisible if and only if the subsemigroup $S_c$ of compact elements is almost divisible as a \pom.
\end{lma}
\begin{proof}
Suppose that $S$ is almost divisible, $a$ is in $S_c$ and $k\in\N$.
As $a\ll a$, there is $x\in S$ such that
\[
kx\leq a,\quad a\leq (k+1)x.
\]
Since $S$ is algebraic and $a$ is compact, we can find $x'\in S_c$ such that $x'\leq x$ and $a\leq(k+1)x'$.
Then
\[
kx'\leq kx\leq a\leq (k+1)x',
\]
which shows that $x'$ has the desired properties.

Conversely, assume that $S_c$ is almost divisible as a \pom.
Let $a'$ and $a$ be elements in $S$ satisfying $a'\ll a$, and let $k\in\N$.
Using that $S$ is algebraic there exists a compact element $b$ such that $a'\ll b\ll a$.
Now by assumption there is $x$ in $S_c$ such that
\[
kx\leq b\leq (k+1)x.
\]
It follows
\[
kx\leq b \leq a,\quad
a'\leq b\leq (k+1)x,
\]
which shows that $x$ has the desired properties.
\end{proof}

%==========================================================================================
\begin{prp}
\label{prp:alternativealmDiv}
Let $S$ be an almost unperforated $\CatCu$-semigroup.
If $S$ is almost divisible, then for all $a\in S$ and $k\in\N$ there exists $x\in S$ such that
\[
kx\leq a\leq (k+1)x.
\]
\end{prp}
\begin{proof}
Let $a$ and $k$ be as in the statement.
For each $n\in\N$, choose numbers $p_n,q_n\in\N_+$ such that
\[
k+1>\tfrac{p_n}{q_n}>\tfrac{p_{n+1}+1}{q_{n+1}}>k.
\]
Choose a rapidly increasing sequence $(a_n)_n$ in $S$ such that $a=\sup_n a_n$.
Using that $S$ is almost divisible, for each $n\in\N$ we can choose an element $x_n$ in $S$ such that
\begin{align}
\label{prp:alternativealmDiv:eq1}
p_n x_n\leq q_n a_n,\quad
q_n a_{n-1}\leq (p_n+1)x_n.
\end{align}
It follows
\[
q_{n+1} p_n x_n\leq q_{n+1} q_n a_n\leq q_n (p_{n+1}+1) x_{n+1},
\]
which by almost unperforation implies $x_n\leq x_{n+1}$.
Set $x :=\sup_n x_n$.

For each $n$, it follows from \eqref{prp:alternativealmDiv:eq1} that
\[
kp_n x_n\leq kq_n a_n,\quad
(k+1)q_n a_{n-1}\leq (k+1)(p_n+1)x_n.
\]
By almost unperforation, we obtain $k x_n\leq a_n$ and $a_{n-1}\leq (k+1)x_n$.
Passing to suprema, we obtain the desired inequalities.
\end{proof}

%------------------------------------------------------------------------------------------
It is easily checked that $Z$ is almost unperforated and almost divisible.
It then follows from the next result that every $\CatCu$-semigroup{} with $Z$-multiplication is almost divisible and almost unperforated.
The converse is shown in \autoref{prp:ZModTFAE}.
The following lemma is also used in the proof of \autoref{thm:simpleSemirg}, which is why we formulate it more general than needed in this section.

%==========================================================================================
\begin{lma}
\label{prp:ZModAlmUnpDiv}
Let $S$ be a $\Cu$-semimodule over a $\Cu$-semiring $R$.
Assume that the unit element of $R$ is almost divisible.
Then $S$ is almost unperforated and almost divisible.
In particular, $R$ itself is almost unperforated and almost divisible.
\end{lma}
\begin{proof}
Choose a rapidly increasing sequence $(u_n)_n$ in $R$ such that $1_R=\sup_n u_n$.
First, we show that $S$ is almost divisible.
Let $a$ be in $S$, $k\in \N$, and let $a'\in S$ satisfy $a'\ll a$.
Since $a=\sup_n u_n a$, there exists $N\geq 1$ such that $a'\leq u_Na$.
Then, since $1_R$ is almost divisible, we choose $x\in R$ such that
\[
kx\leq 1_R,\quad
u_N\leq (k+1)x.
\]
It follows
\[
k(xa)\leq a,\quad
a'\leq u_Na\leq (k+1)(xa).
\]
This shows that $a$ is almost divisible.
Since $a$ was arbitrary, we obtain that $S$ is almost divisible.

To show that $S$ is almost unperforated, let $a,b\in S$ satisfy $(k+1)a\leq kb$ for some $k\in\N$.
Let $a'\in S$ satisfy $a'\ll a$.
As above, there exists $N\in\N$ such that $a'\leq u_Na$ and we can choose $x\in R$ such that $kx\leq 1_R$ and $u_N\leq (k+1)x$.
Then
\[
a'\leq u_Na\leq x(k+1)a\leq xkb\leq b.
\]
Since this holds for every $a'$ satisfying $a'\ll a$, we obtain $a\leq b$, as desired.
\end{proof}

%------------------------------------------------------------------------------------------
To prepare the proof of \autoref{prp:ZModTFAE}, we first provide some results that are also of interest in themselves.
We need to introduce some notation. Let $S$ be an almost unperforated $\CatCu$-semigroup.
Given $a\in S$ and $k,n\in\N_+$, we set
\[
\mu((k,n),a) :=\left\{ x\in S : nx\leq ka \leq (n+1)x \right\}.
\]
Note that an element $a\in S$ is almost divisible if and only if $\mu((k,n),a)$ is nonempty for every $k,n\in\N_+$.
We interpret $\mu((k,n),a)$ as the interval $\left[\tfrac{k}{n+1}a,\tfrac{k}{n}a\right]$.
With this idea in mind, the following lemma asserts that for almost unperforated semigroups these `intervals' respect the order and way-below relation.

%==========================================================================================
\begin{lma}
\label{lem:mult}
Let $S$ be an almost unperforated $\CatCu$-semigroup{}, $a,b\in S$, and  $k_1,n_1,k_2,n_2\in\N_+$ such that $\tfrac{k_1}{n_1}<\tfrac{k_2}{n_2+1}$.
\beginEnumStatements
\item
If $a\leq b$, then $x\leq y$ for every $x\in\mu((k_1,n_1),a)$ and $y\in\mu((k_2,n_2),b)$.
\item
If $a\ll b$, then $x\ll y$ for every $x\in\mu((k_1,n_1),a)$ and $y\in\mu((k_2,n_2),b)$.
\end{enumerate}
\end{lma}
\begin{proof}
Let $k_1,n_1,k_2,n_2$ be as in the statement, and let $x\in\mu((k_1,n_1),a)$ and $y\in\mu((k_2,n_2),b)$.
Then
\[
n_1x \leq k_1a,\quad
k_2 b \leq (n_2+1)y.
\]
Multiplying the first inequality by $k_2$ and the second by $k_1$, we obtain
\[
k_2n_1x \leq k_2k_1a,\quad
k_1k_2b \leq k_1(n_2+1)y.
\]
If $a\leq b$, then it follows $k_2n_1x\leq k_1(n_2+1)y$.
Since $k_1(n_2+1)<k_2n_1$ and $S$ is almost unperforated, we obtain $x\leq y$.

In the second case, assuming $a\ll b$, it follows $k_2n_1x\ll k_1(n_2+1)y$.
Choose $y'$ such that $y'\ll y$ and $k_2n_1x\ll k_1(n_2+1)y'$.
As in the first case, we obtain $x\leq y'$.
Then $x\ll y$, as desired.
\end{proof}

%==========================================================================================
\begin{prp}
\label{prp:zembed}
Let $S$ be an almost unperforated $\CatCu$-semigroup{} and let $a$ in $S$.
Then there exists a generalized $\CatCu$-morphism $\alpha_a\colon Z\to S$ with $\alpha_a(1)=a$ if and only if $a$ is almost divisible.

If the map $\alpha_a$ exists, then it is unique.
Moreover, it is a $\CatCu$-morphism if and only if $a$ is compact.
\end{prp}
\begin{proof}
If there exists a generalized $\CatCu$-morphism $\alpha_a\colon Z\to S$ with $\alpha_a(1)=a$, then for each $n\in\N_+$ we have
\[
n\cdot\alpha_a(\tfrac{1}{n})=\alpha_a(1')\leq \alpha_a(1)=a\leq \alpha_a(\tfrac{n+1}{n})=(n+1)\alpha_a(\tfrac{1}{n}),
\]
which shows that $a$ is almost divisible.

For the converse, let $a\in S$ be an almost divisible element.
We define a map $\alpha_a\colon Z\to S$ by considering the decomposition $Z=\N\sqcup(0,\infty]$.
For $n\in\N\subset Z$, we set $\alpha_a(n) :=na=a+\stackrel{n}{\ldots}+a$.
For $t\in(0,\infty]\subset Z$, we set
\[
\alpha_a(t)
:=\sup \left\{ x\in\mu((k,n),a) : \tfrac{k}{n}<t \right\}.
\]
We first prove that this supremum exists, by finding an increasing cofinal sequence.
Choose numbers $k_d,n_d\in\N_+$ for $d\in\N_+$ such that
\[
\tfrac{k_d}{n_d}<\tfrac{k_{d+1}}{n_{d+1}+1},
\quad\text{ and }\quad
\sup_d\tfrac{k_d}{n_d}=t.
\]
Since $a$ is almost divisible, for each $d$ we can choose an element $x_d\in\mu((k_d,n_d),a)$.
By \autoref{lem:mult}, $(x_d)_d$ is an increasing sequence in $S$.
Moreover, it is easily checked that for every $\tfrac{k}{n}\in\Q_+$ with $\tfrac{k}{n}<t$ and for every $x\in\mu((k,n),a)$ there is an index $d$ such that $x\leq x_d$.
It follows $\alpha_a(t)=\sup_d x_d$, which exists by \axiomO{1}.

To show uniqueness, let $\beta\colon Z\to S$ be a generalized $\CatCu$-morphism with $\beta(1)=a$.
It is clear that $\beta(n)=\alpha_a(n)$ for all elements $n\in\N\subset Z$.
Consider now $t\in (0,\infty]\subset Z$.
As above, for each $d\in\N_+$ choose numbers $k_d,n_d\in\N_+$ and an element $x_d$ such that
\[
\tfrac{k_d}{n_d}<\tfrac{k_{d+1}}{n_{d+1}+1},\quad
\sup_d\tfrac{k_d}{n_d}=t,
\text{ and }
x_d\text{ is in }\mu((k_d,n_d),a).
\]
It is easy to see that $\beta(\tfrac{k_d}{n_d})$ is in $\mu((k_d,n_d),a)$ for each $d\in\N_+$.
By \autoref{lem:mult}, we deduce for each $d$ that
\[
x_d \leq \beta(\tfrac{k_{d+1}}{n_{d+1}}) \leq x_{d+2}.
\]
Using this at the second step, and that $\beta$ preserves suprema of increasing sequences at the first step, it follows
\[
\beta(t)
=\sup_d \beta(\tfrac{k_d}{n_d})
=\sup_d x_d
=\alpha_a(t).
\]
It is left to the reader to check that $\alpha_a$ preserves the zero element, the order, and suprema of increasing sequences.
It remains to prove that $\alpha_a$ is additive.
This is clear for sums of elements in $\N\subset Z$.

Let $t_1,t_2$ be in $(0,\infty]\subset Z$.
For each $d\in\N_+$, choose numbers $k_d^{(1)},k_d^{(2)},n_d\in\N_+$ such that for $i=1,2$:
\[
\tfrac{k_d^{(i)}}{n_d}<\tfrac{k_{d+1}^{(i)}}{n_{d+1}+1},
\quad\text{ and }\quad
\sup_d\tfrac{k_d^{(i)}}{n_d}=t_i.
\]
For each $d$ and $i=1,2$, choose $x_n^{(i)}\in\mu((k_d^{(i)},n_d),a)$.
Then $\alpha_a(t_i)=\sup_d x_d^{(i)}$ for $i=1,2$.
Moreover, we get
\[
\tfrac{k_d^{(1)}+k_d^{(2)}}{n_d}
<\tfrac{k_{d+1}^{(1)}+k_{d+1}^{(2)}}{n_{d+1}+1},
\quad\text{ and }\quad
\sup_d\tfrac{k_d^{(1)}+k_d^{(2)}}{n_d}=t_1+t_2.
\]
Thus, for any sequence of elements $(y_d)_d$ with $y_d\in\mu((k_d^{(1)}+k_d^{(1)},n_d),a)$ we will get $\alpha_a(t_1+t_2)=\sup_d y_d$.
However, it is easily seen that $x_d^{(1)}+x_d^{(2)}$ belongs to $\mu((k_d^{(1)}+k_d^{(2)},n_d),a)$.
Using \axiomO{4} at the second step, we obtain
\[
\alpha_a(t_1+t_2)=\sup_d (x_d^{(1)}+x_d^{(2)})
=\sup_d x_d^{(1)} + \sup_d x_d^{(2)}
=\alpha_a(t_1)+\alpha_a(t_2).
\]
It is left to the reader to show that $\alpha_a$ preserves the sum of an element in $\N\subset Z$ with an element in $(0,\infty]\subset Z$.

Finally, let us show that $\alpha_a$ is a $\CatCu$-morphism if and only if $a$ is compact.
Assume first that $\alpha_a$ preserves the way-below relation.
Since the unit element of $Z$ is compact, we obtain
\[
a=\alpha_a(1)\ll\alpha_a(1)=a.
\]
For the converse, assume that $a$ is compact.
We need to show that $x\ll y$ implies $\alpha_a(x)\ll\alpha_a(y)$, for any $x,y\in Z$.
This is clear if $x$ or $y$ is an element in $\N\subset Z$.

Assume that $x,y$ are in $(0,\infty]\subset Z$.
Without loss of generality we may assume $x<y$.
Choose $\tfrac{k}{n}\in\Q_+$ and elements $u\in\mu((k,n+2),a)$, $v\in\mu((k,n),a)$ such that
\[
x\leq\tfrac{k}{n+3},\quad
\tfrac{k}{n}< y.
\]
Since $a\ll a$, it follows from \autoref{lem:mult} that $u\ll v$.
It also follows from \autoref{lem:mult} and the definition of $\alpha_a$ that $\alpha_a(x)\leq u$ and $v\leq\alpha_a(y)$.
Therefore
\[
\alpha_a(x)
\leq u \ll v
\leq\alpha_a(y),
\]
as desired.
\end{proof}

%==========================================================================================
\begin{thm}
\label{prp:ZModTFAE}
Let $S$ be a $\CatCu$-semigroup.
Then the following are equivalent:
\beginEnumStatements
\item
We have $S\cong Z\otimes_{\CatCu} S$.
\item
The semigroup $S$ has $Z$-multiplication.
\item
The semigroup $S$ is almost unperforated and almost divisible.
\end{enumerate}
\end{thm}
\begin{proof}
By \autoref{prp:ZSolid}, the $\CatCu$-semiring $Z$ is solid.
Therefore, the equivalence between \enumStatement{1} and \enumStatement{2} follows from \autoref{prp:solidModuleTFAE}.
Since the unit of $Z$ is almost divisible, we obtain from \autoref{prp:ZModAlmUnpDiv} that every $\CatCu$-semigroup with $Z$-multiplication is almost unperforated and almost divisible.
This shows that \enumStatement{2} implies \enumStatement{3}.

To show that \enumStatement{3} implies \enumStatement{2}, suppose that $S$ is almost unperforated and almost divisible.
Using \autoref{prp:zembed}, we define $\alpha\colon Z\times S \to S$ by $\alpha(z,a)=\alpha_a(z)$ for each $a\in S$ and $z\in Z$.
We claim that $\alpha$ is a $\Cu$-bimorphism.

By \autoref{prp:zembed}, $\alpha(\freeVar,a)$ is a generalized $\CatCu$-morphism for each $a\in S$.
For the other variable, it is also clear that $\alpha(n,\freeVar)$ is a $\CatCu$-morphism for each $n\in\N\subset Z$.
Let $t$ be in $(0,\infty]\subset Z$.
To show that $\alpha(t,\freeVar)$ preserves order, let $a,b\in S$ satisfy $a\leq b$.
By definition, we have
\[
\alpha(t,a)
= \sup \left\{ x\in\mu((k,n),a) : \tfrac{k}{n}<t \right\}.
\]
Thus, given any $k,n\in\N_+$ satisfying $\tfrac{k}{n}<t$ and given any element $x\in\mu((k,n),a)$, we need to show that $x\leq\alpha(t,b)$.
Choose $k',n'\in\N_+$ and an element $y$ such that
\[
\tfrac{k}{n}<\tfrac{k'}{n'+1},\
\tfrac{k'}{n'}<t,\quad
y\in\mu((k',n'),b).
\]
By \autoref{lem:mult}, we have $x\leq y$.
Therefore
\[
x \leq \sup \left\{ z\in\mu((c,d),b) : \tfrac{c}{d}<t \right\}
= \alpha(t,b),
\]
from which we deduce $\alpha(t,a)\leq\alpha(t,b)$, as desired.

To show additivity in the second variable, let $z$ be in $Z$ and $a,b$  in $S$.
Consider the following maps from $Z$ to $S$ given by
\[
\alpha_{a+b}=(x\mapsto \alpha(x,a+b)),\quad
\alpha_a+\alpha_b=(x\mapsto \alpha(x,a)+\alpha(x,b)),\quad
\txtFA x\in Z.
\]
It is clear that both maps are generalized $\CatCu$-morphisms that send the unit of $Z$ to the element $a+b$.
By \autoref{prp:zembed}, the map with this property is unique, and therefore $\alpha(z,a+b)=\alpha(z,a)+\alpha(z,b)$.
Analogously, one proves that $\alpha(z,\sup_na_n)=\sup_n\alpha(z,a_n)$ for every $z\in Z$ and every increasing sequence $(a_n)_n$ in $S$.

Thus, $\alpha\colon Z\times S\to S$ is a generalized $\CatCu$-bimorphism.
It is clear that $\alpha(1,a)=a$ for every $a\in S$ and it is straightforward to check that
\begin{align}\tag{*}
\label{thm:zmult:1}
\alpha(z_1z_2,a)=\alpha(z_1,\alpha(z_2,a)),
\end{align}
for every $z_1,z_2\in Z$ and $a\in S$.

It remains to show that for any $t_1,t_2\in Z$ with $t_1\ll t_2$ and for any $a,b\in S$ with $a\ll b$, we have $\alpha(t_1,a)\ll \alpha(t_2,b)$.
This is clear if $t_1$ or $t_2$ is an element in $\N\subset Z$.
Thus, we consider the case that $t_1,t_2$ are in $(0,\infty]$, and without loss of generality we may assume $t_1<t_2$.
Then $t_1$ is necessarily finite, and $1_Z\leq t_1^{-1}t_2$ in $Z$.
In order to show $\alpha(t_1,a)\ll \alpha(t_2,b)$, let $(x_n)_n$ be an increasing sequence in $S$ with $\alpha(t_2,b)\leq \sup_n x_n$.
Using this at the fifth step, and using \eqref{thm:zmult:1} at the fourth step, we obtain
\begin{align*}
a \ll b
=\alpha(1_Z,b)
&\leq\alpha(t_1^{-1}t_2,b) \\
&=\alpha(t_1^{-1},\alpha(t_2,b)) \\
&\leq\alpha(t_1^{-1},\sup_n x_n)
=\sup_n (\alpha(t_1^{-1},x_n)).
\end{align*}
Therefore, there exists an index $n_0$ such that $a\leq \alpha(t_1^{-1},x_{n_0})$.
Then
\[
\alpha(t_1,a)
\leq\alpha(t_1,\alpha(t_1^{-1},x_{n_0}))
=\alpha(1',x_{n_0})
\leq\alpha(1_Z,x_{n_0})
=x_{n_0}.
\]
Hence $\alpha(t_1,a)\ll \alpha(t_2,b)$, as desired.
This finishes the proof that $S$ has $Z$-multiplication.
\end{proof}

%==========================================================================================
\begin{rmk}
\label{rmk:TWConjecture}\index{terms}{Regularity Conjecture}\index{terms}{Toms-Winter Conjecture}
The Toms-Winter conjecture (see \cite[Remarks~3.5]{TomWin09Villadsen} and
\cite[Conjecture~0.1]{Win12NuclDimZstable}) predicts that for every unital, separable, simple, nonelementary, nuclear \ca{} $A$, the following conditions are equivalent:
\beginEnumStatements
\item
The algebra $A$ is $\mathcal{Z}$-stable, that is, we have $A\cong \mathcal{Z}\otimes A$.
\item
The Cuntz semigroup $\Cu(A)$ is almost unperforated.
\item
The algebra $A$ has finite nuclear dimension.
\end{enumerate}

We can interpret \autoref{prp:ZModTFAE} as the verification of the $\CatCu$-semigroup version of a part of the Toms-Winter conjecture.
The analog of `$\mathcal{Z}$-stability' for a $\CatCu$-semigroup $S$ is the property that $S\cong Z\otimes_{\CatCu}S$, which is \enumStatement{1} of \autoref{prp:ZModTFAE}.
The second condition of the Toms-Winter conjecture is already formulated for $\CatCu$-semigroups.
However, in \autoref{prp:ZModTFAE}\enumStatement{3} we not only require the $\CatCu$-semigroup to be almost unperforated but also almost divisible.
We remark that not every Cuntz semigroup of a simple \ca{} is almost divisible;
see \cite{DadHirTomWin09ZNotEmb}.
On the other hand, it seems possible that the Cuntz semigroup of a simple \ca{} is automatically almost divisible whenever it is almost unperforated.
Indeed, if the Toms-Winter conjecture holds true, then this would be a consequence for at least the class of nuclear \ca{s}.

It is not clear what the analog of condition \enumStatement{3} of the Toms-Winter conjecture for $\CatCu$-semigroups should be.
This would entail to define nuclearity and dimension concepts for $\CatCu$-semigroups, which is not pursued here.
\end{rmk}

%------------------------------------------------------------------------------------------
The following problem asks if there is an analog of Theorems~\ref{prp:MapToPIMod}, \ref{prp:MapToRqMod} and \ref{prp:MapToRMod} for tensor products with $Z$.

%==========================================================================================
\begin{prbl}
\label{prbl:MapToZMod}
Let $S$ be a $\CatCu$-semigroup, and let $a,b$ be in $S$.
Characterize when $1\otimes a \leq 1\otimes b$ in $Z\otimes_{\CatCu}S$.
\end{prbl}

%==========================================================================================
\begin{prbl}
\label{prbl:axiomTensProdZ}
When does axiom~\axiomO{5}, \axiomO{6} or weak cancellation pass from a $\CatCu$-semigroup $S$ to the tensor product $Z\otimes_{\CatCu} S$?
\end{prbl}

%==========================================================================================
\begin{pgr}
\label{pgr:axiomTensProdZ}
In general, axiom~\axiomO{5} does not pass to tensor products with $Z$;
see \autoref{prp:tensornotO5}.
We have that $Z$ satisfies \axiomO{5}, \axiomO{6} and weak cancellation itself.
Therefore, if $S$ is an inductive limit of simplicial $\CatCu$-semigroups, then $Z\otimes_{\CatCu}S$ satisfies the three axioms as well;
see \autoref{prp:tensLimSimplicial}.

It seems likely that \autoref{prbl:axiomTensProdZ} has a positive answer if $S$ is assumed to be algebraic.
\end{pgr}

%------------------------------------------------------------------------------------------
We end this section with some structure results about $\CatCu$-semigroup{s} with $Z$-multi\-plication.
For the next result, recall that $1'\in Z$ denotes the soft `one'.

%==========================================================================================
\begin{prp}
\label{prp:softZmult}
Let $S$ be a $\CatCu$-semigroup{} with $Z$-multipli\-ca\-tion.
Then:
\beginEnumStatements
\item
An element $a\in S$ is soft if and only if $a=1'a$.
\item
For every functional $\lambda\in F(S)$ and every $a\in S$, we have $\lambda(a)=\lambda(1'a)$.
\end{enumerate}
\end{prp}
\begin{proof}
First, in order to verify \enumStatement{1}, let $a$ be in $S$.
To prove that $1'a$ is soft, let $a'\in S$ satisfy $a'\ll 1'a$.
We need to show that $a'<_s 1'a$.
Consider the increasing sequence $(\tfrac{k-1}{k})_k$ of noncompact elements in $(0,\infty]\subset Z$.
Since $1'=\sup_k\tfrac{k-1}{k}$ in $Z$, we get
\[
a'\ll 1'a=\sup_k\tfrac{k-1}{k}a.
\]
Thus, there exists $n\in\N$ such that $a'\leq\tfrac{n-1}{n}a$.
It is easy to verify that $\tfrac{n-1}{n}a<_s 1'a$.
It follows $a'<_s 1'a$, as desired

Conversely, assume that $a\in S$ is a soft element.
It is clear that $1'a\leq a$.
To show the converse inequality, it is enough to show that $a'\leq 1'a$ for every $a'\in S$ satisfying $a'\ll a$.
Given such $a'$, it follows from softness of $a$ that there exists $k\in\N$ such that $(k+1)a'\leq ka$.
Using this at the third step, we obtain
\[
a'\leq\tfrac{k+1}{k}a'
=\tfrac{1}{k}((k+1)a')
\leq\tfrac{1}{k}(ka)
=1'a,
\]
as desired.

To show \enumStatement{2}, let $\lambda$ be a functional in $F(S)$ and let $a$ be in $S$.
The $\CatCu$-semigroup $[0,\infty]$ has a $Z$-multiplication.
Since $\lambda$ is a generalized $\CatCu$-morphism from $S$ to $[0,\infty]$, it follows from \autoref{prp:solidTFAE} that $\lambda$ is $Z$-linear.
Using this at the first step, and using that every element of $[0,\infty]$ is soft, we deduce
\[
\lambda(1's)
=1'\lambda(s)
=\lambda(s),
\]
as desired.
\end{proof}

%==========================================================================================
\begin{rmk}
Let $S$ be a simple, nonelementary, stably finite $\CatCu$-semigroup with $Z$-multiplication.
Then $S$ is almost unperforated and therefore the conditions in \autoref{prp:predecessors} are equivalent. It follows from \autoref{prp:softZmult} that the predecessor of any compact element $p\in S$ is given as $1'p$.
\end{rmk}

%------------------------------------------------------------------------------------------
The following result provides a partial answer to \autoref{prbl:pncInCu}.

%==========================================================================================
\begin{prp}
\label{prp:softInCuZmult}
Let $S$ be a $\CatCu$-semigroup{} with $Z$-multipli\-ca\-tion.
Then the subsemigroup $S_\soft$ of soft elements is a $\CatCu$-semigroup{}.
If $S$ satisfies \axiomO{5} (respectively \axiomO{6}, weak cancellation), then so does $S_\soft$.
\end{prp}
\begin{proof}
By \autoref{prp:softAbsorbing}, $S_\soft$ is a subsemigroup of $S$ that is closed under suprema of increasing sequences.
This shows that $S_\soft$ satisfies \axiomO{1}.

Claim 1:
For every $a\in S_\soft$ there exists an increasing sequence $(a_k)_k$ in $S_\soft$ such that $a=\sup_ka_k$ and such that $a_k\ll a_{k+1}$ in $S$ for each $k$.

To show this claim, let $a$ be in $S_\soft$.
Since $S$ satisfies \axiomO{2} we choose a rapidly increasing sequence $(s_k)_k$ in $S$ such that $a=\sup_k s_k$.
For each $k\in\N_+$, set
\[
a_k :=\tfrac{k-1}{k}s_k.
\]
It is easy to check that $1'a_k=a_k$, which by \autoref{prp:softZmult} implies that $a_k$ belongs to $S_\soft$.
Moreover, for each $k$ we have $\tfrac{k-1}{k}\ll\tfrac{k}{k+1}$ in $Z$ and $s_k\ll s_{k+1}$ in $S$.
Since the $Z$-multiplication on $S$ is given by a $\CatCu$-bimorphism, we obtain
\[
a_k =\tfrac{k-1}{k}s_k \ll \tfrac{k}{k+1}s_{k+1} = a_{k+1},
\]
for each $k$.
It is clear that $a=\sup_ka_k$, which finishes the proof of the claim.

By \autoref{lem:density}, for every  $a,b\in S_\soft$ we have $a\ll b$ in $S$ if and only if $a\ll b$ in $S_\soft$.
Together with claim~1, this verifies \axiomO{2} for $S_\soft$.
Then axioms \axiomO{3} and \axiomO{4} for $S_\soft$ follow from their counterparts in $S$.

Next, assume that $S$ satisfies \axiomO{5}.
In order to show that $S_\soft$ satisfies \axiomO{5}, let $a',a,b',b,c\in S_\soft$ satisfy
\[
a+b\leq c,\quad
a'\ll a,\quad
b'\ll b.
\]
Using that $S$ satisfies \axiomO{5} choose $x\in S$ such that
\[
a'+x\leq c\leq a+x,\quad
b'\leq x.
\]
Set $y :=1'x$, which by \autoref{prp:softZmult} is an element in $S_\soft$.
We claim that $y$ has the desired properties to verify \axiomO{5} for $S_\soft$.
Indeed, using \autoref{prp:softZmult} again, we have $a'=1'a'$, $a=1'a$, $c=1'c$ and $b'=1'b'$.
Therefore
\[
a'+y
=1'(a'+x)\leq 1'c = c
\leq 1'(a+x)
=a+y,\quad
b'=1'(b')\leq 1'(x)=y,
\]
as desired.

In the same way, one shows that $S_\soft$ inherits \axiomO{6} from $S$.
Finally, it is straightforward to check that $S_\soft$ is weakly cancellative whenever $S$ is.
\end{proof}

\vspace{5pt}
%------------------------------------------------------------------------------------------
%==========================================================================================
\section{The rationalization of a semigroup}
\label{sec:RqMod}
\index{terms}{rationalization}

%------------------------------------------------------------------------------------------
In this section, we study $\CatCu$-semigroups that are semimodules over the $\CatCu$-semiring of a strongly self-absorbing UHF-algebra.
Given a supernatural number $q$ satisfying $q^2=q$ and $q\neq 1$, we let $M_q$ be the associated UHF-algebra;
see \autoref{pgr:Mq}.
We use $R_q$ to denote $\Cu(M_q)$.
In \autoref{prp:RqSolid}, we show that $R_q$ is a solid $\CatCu$-semiring.

In \autoref{dfn:qUnpDiv}, we recall the natural notions of $q$-unperforation and $q$-divisibility for semigroups.
The main result of this section is \autoref{prp:RqModTFAE}, where we characterize the $\CatCu$-semimodules over $R_q$ as the $\CatCu$-semigroups that are $q$-unperforated and $q$-divisible.

In \autoref{prp:tensRqCa}, we apply the results to the Cuntz semigroup of a \ca{} $A$.
In particular, we obtain that $\Cu(M_q\otimes A)\cong \Cu(A)$ if and only if $\Cu(A)$ is $q$-divisible and $q$-unperforated.
We also deduce that the Cuntz semigroup of a \ca{} $A$ is nearly unperforated whenever $A$ tensorially absorbs a strongly self-absorbing UHF-algebra;
see \autoref{prp:UHFStableNearUnp}.
This verifies \autoref{conj:nearUnpCaZstable} in that case.

%==========================================================================================
\begin{pgr}
\label{pgr:Rq}
\index{terms}{supernatural number}
A \emph{supernatural number}  $q$ is a formal product
\[
q=\prod_{k\in\N} p_k^{n_k},
\]
where $p_0,p_1,p_2,\ldots$ is an enumeration of all prime numbers and where each $n_k$ is a number in $\{0,1,2,\ldots,\infty\}$ that denotes the multiplicity with which the prime $p_k$ occurs in $q$.
By definition, zero is not a supernatural number.
Given supernatural numbers $q=\prod_k p_k^{m_k}$ and $r=\prod_k p_k^{n_k}$, their (formal) product is given by $qr=\prod_k p_k^{m_k+n_k}$. Analogously one can naturally define the product of infinitely many (super)natural numbers $\prod_{n\in \N}q_n$ in the obvious way. If $q=q^2$, then each $n_k$ is either $0$ or $\infty$.

We identify the nonzero natural numbers with the supernatural numbers of the form $\prod_{k\in\N} p_k^{n_k}$ where $\sum_{k\in\N} n_k <\infty$.
In particular, the number `one' is the supernatural number $\prod_{k\in\N} p_k^{n_k}$ where each $n_k$ is zero.

Let $q$ be a supernatural number satisfying $q=q^2$.
We write $\Z\left[\tfrac{1}{q}\right]$ for the ring obtained by inverting in $\Z$ all primes that divide $q$, that is:
\[
\Z\left[\tfrac{1}{q}\right]
=\Z\left[\left\{ \tfrac{1}{p} : p \text{ prime}, p|q \right\}\right].
\]
Then we let $K_q$ denote the subsemiring of nonnegative numbers in $\Z\left[\tfrac{1}{q}\right]$, that is:
\[
K_q=\Q_+\cap\Z\left[\tfrac{1}{q}\right].
\]
For example, we have
\[
K_1=\N=\{0,1,2,\ldots\},\quad
K_{2^\infty}=\Q_+\cap\Z\left[\tfrac{1}{2}\right]=\N\left[\tfrac{1}{2}\right].
\]
Then $K_q$ is a unital subsemiring of $\Q_+$, and all unital subsemirings of $\Q_+$ arise this way.
\index{symbols}{K$_q$@$K_q$}

For the rest of the paragraph, we fix a supernatural number $q$ satisfying $q=q^2$ and $q\neq 1$.
We equip $K_q$ with the natural algebraic order.
Recall that, for a \pom{} $M$, we denote by $\Cu(M)$ the $\CatCu$-completion of the $\CatPreW$-semigroup $(M,\leq)$;
see \autoref{pgr:algebraicSemigp}.
Then we define
\[
R_q = \Cu(K_q).
\]
\index{symbols}{R$_q$@$R_q$}
It follows from the results about algebraic $\CatCu$-semigroups in \autoref{sec:algebraicSemigp} that $R_q$ is a weakly cancellative $\CatCu$-semigroup satisfying \axiomO{5} and \axiomO{6}, and whose submonoid of compact elements is canonically identified with $K_q$.
It is then straightforward to check that there is a decomposition of $R_q$ as
\[
R_q = K_q \sqcup (0,\infty],
\]
where $K_q\subset R_q$ are the compact elements in $R_q$, and where $(0,\infty]\subset R_q$ is the submonoid of nonzero soft elements in $R_q$.

Using that $K_q$ is a semiring, we can define a product on $R_q$.
In \autoref{pgr:CuCompletionSrg}, this construction will be carried out in greater generality.
Here, we only consider the concrete case of $R_q$.

The order and semiring-structure of $R_q$ are so that the inclusion of $K_q$ in $R_q$ and the inclusion of $(0,\infty]$ in $R_q$ are order-embeddings and semiring-homomorphisms.
We let $\iota\colon K_q\to [0,\infty]$ be the natural inclusion map.
Let $a\in K_q$ and $t\in(0,\infty]$.
Then their sum in $R_q$ is given as $a+t=\iota(a)+t$, an element in $(0,\infty]$.
If $a=0$, then $at=0\in K_q$.
If $a$ is nonzero, then the product of $a$ and $t$ in $R_q$ is given as $at=\iota(a)t\in(0,\infty]$.
Moreover, we have $a\leq t$ in $R_q$ if and only if $\iota(a)<t$, and we have $t\leq a$ in $R_q$ if and only if $t\leq\iota(a)$.

Thus, the submonoid of soft elements in $R_q$ is additively and multiplicatively absorbing.
It is straightforward to check directly that the product on $R_q$ is a $\CatCu$-bimorphism and that the unit element of $K_q$ is also a unit for $R_q$.
This gives $R_q$ the structure of a $\CatCu$-semiring.
\end{pgr}

%==========================================================================================
\begin{pgr}
\label{pgr:Mq}
\index{terms}{UHF-algebra!infinite type}
\index{symbols}{M$_q$@$M_q$ \quad (UHF-algebra)}
Given a supernatural number $q$, one associates a UHF-algebra $M_q$ as follows:
If $q$ is finite, then $M_q$ denotes the \ca{} of $q$ by $q$ matrices.
If $q$ is infinite, then we choose a sequence $n_0,n_1,n_2,\ldots$ of prime numbers such that $q$ is equal to the product $\prod_{k=0}^\infty n_k$.
Then we set
\[
M_q := \bigotimes_{k=0}^\infty M_{n_k}.
\]
The isomorphism type of $M_q$ does not depend on the choice of the sequence $(n_k)_k$.

Let $q_1$ and $q_2$ be supernatural numbers.
Then $q_1=q_2$ if and only if $M_{q_1}\cong M_{q_2}$.
Moreover,
\[
M_{q_1}\otimes M_{q_2} \cong M_{q_1q_2}.
\]

The UHF-algebra $M_q$ is said to be of \emph{infinite type} if $M_q\cong M_q\otimes M_q$ and $M_q\neq\C$.
Equivalently, we have $q=q^2$ and $q\neq 1$.
It is known that every UHF-algebra of infinite type is strongly self-absorbing;
see \cite{TomWin07ssa}.
\end{pgr}

%==========================================================================================
\index{terms}{Cuntz semigroup!of UFH-algebra}
\begin{prp}
\label{prp:RqFromUHF}
Let $q$ be a supernatural number satisfying $q=q^2$ and $q\neq 1$.
Then
\[
R_q \cong \Cu(M_q).
\]
\end{prp}
\begin{proof}
It is well-known that $M_q$ is a unital, separable, simple, $\mathcal{Z}$-stable \ca{} with stable rank one and unique tracial state.
The $K_0$-group of $M_q$ is isomorphic to $\Z\left[\tfrac{1}{q}\right]$.
Since $M_q$ has stable rank one, the positive part of the ordered $K_0$-group is naturally isomorphic with $V(M_q)$.
We therefore have $V(M_q)= \Q_+ \cap \Z\left[\tfrac{1}{q}\right] = K_q$.
Then it follows from \autoref{prp:CuSimpleZstable} that
\[
\Cu(M_q) \cong V(M_q) \sqcup (0,\infty] \cong K_q \sqcup (0,\infty] = R_q,
\]
as desired.
\end{proof}

%==========================================================================================
\begin{prp}
\label{prp:RqSolid}
Let $q$ be a supernatural number satisfying $q=q^2$ and $q\neq 1$.
Then $R_q$ is a solid $\CatCu$-semiring.
\end{prp}
\begin{proof}
By \autoref{prp:solidTFAE}, it is enough to show that $1\otimes a=a\otimes 1$ for every $a\in R_q$.
If $a$ is a nonzero, compact element in $R_q$, then there are $k,n\in\N_+$ such that $n|q$ and $a=\tfrac{k}{n}$.
It follows
\[
1\otimes a
= \tfrac{n}{n}\otimes \tfrac{k}{n}
= (nk) \tfrac{1}{n}\otimes \tfrac{1}{n}
= \tfrac{k}{n}\otimes \tfrac{n}{n}
= a\otimes 1.
\]
For a soft element in $R_q$, one can apply the same argument that was used in \autoref{exa:R} to show that $[0,\infty]$ is solid.
\end{proof}

We remark that a more general result will be proved in \autoref{prp:solidSrgCu}.

%------------------------------------------------------------------------------------------
The following result follows by combining the observations in \autoref{pgr:Mq} with \autoref{prp:RqFromUHF} and \autoref{prp:tensProdAF}.

%==========================================================================================
\begin{prp}
\label{prp:tensProdRqs}
Let $q$ and $r$ be supernatural numbers satisfying $q=q^2\neq 1$ and $r=r^2\neq 1$.
Then $R_q\otimes_{\CatCu} R_r\cong R_{qr}$.
\end{prp}

If $q$ is a (nontrivial) supernatural number of infinite type, and $S$ is a $\CatCu$-semigroup, we will refer to $R_q\otimes_{\CatCu}S$ as a \emph{rationalization} of $S$.
\index{terms}{rationalization!Cu-semigroup@$\CatCu$-semigroup}

%------------------------------------------------------------------------------------------
The concepts of $n$-unperforation and $n$-divisibility of a positively ordered mo\-noid are well-known for a natural number $n$.
The following definition is a straightforward generalization to supernatural numbers.

%==========================================================================================
\begin{dfn}
\label{dfn:qUnpDiv}
\index{terms}{positively ordered monoid!q-unperforated@$q$-unperforated}
\index{terms}{positively ordered monoid!q-divisible@$q$-divisible}
\index{terms}{unperforated!q@$q$-}
\index{terms}{divisible!q@$q$-}
Let $S$ be a \pom, and let $q$ be a supernatural number.
We say that $S$ is \emph{$q$-unperforated} if for every finite number $n$ dividing $q$, and for any elements $a,b\in S$, we have $a\leq b$ whenever $na\leq nb$.

We say that $S$ is \emph{$q$-divisible} if for every finite number $n$ dividing $q$, and for every $a\in S$, there exists $x\in S$ such that $a=nx$.
\end{dfn}

%==========================================================================================
\begin{rmks}
\label{rmk:qUnpDiv}
Let $S$ be a \pom, and let $n$ be a nonzero natural number.
We let $\mu_n\colon S\to S$ be the map that multiplies each element in $S$ by $n$.

(1)
The monoid $S$ is $n$-divisible if and only if $\mu_n$ is surjective.

(2)
The monoid $S$ is $n$-unperforated if and only if $\mu_n$ is an order-embedding.

(3)
Let $q$ be a supernatural number, and let $q^\infty$ denote its infinite product with itself.
Then $S$ is $q$-divisible if and only if $S$ is $q^\infty$-divisible.
Similarly, $S$ is $q$-unperforated if and only if $S$ is $q^\infty$-unperforated.

(4)
Let $q_\infty$ be the largest supernatural number, for which each prime has infinite multiplicity.
A \pom{} $S$ is \emph{divisible} if it is $n$-divisible for every $n\in\N_+$, which is equivalent to being $q_\infty$-divisible.
Similarly, $S$ is \emph{unperforated} if it is $n$-unperforated for every $n\in\N_+$, or, equivalently, if if is $q_\infty$-unperforated.
\end{rmks}

%==========================================================================================
\begin{lma}
\label{prp:qUnpNearUnp}
Let $S$ be a \pom, and let $q$ be a supernatural number with $q\neq 1$.

(1)
If $S$ is $q$-divisible, then for any $a\in S$ and $n\geq 1$, there exists $x\in S$ such that $nx\leq a\leq (n+1)x$. In particular, $S$ is almost divisible.

(2)
If $S$ is $q$-unperforated, then $S$ is nearly unperforated and therefore also almost unperforated.
\end{lma}
\begin{proof}
To show the first statement, assume that $S$ is a $q$-divisible, \pom.
Let $a$ be in $S$ and let $n$ be in $\N_+$.
We need to find $x\in S$ such that $nx\leq a\leq(n+1)x$.

Choose a number $d\geq 2$ that divides $q$.
Since the set $\left\{ \tfrac{r}{d^k} : r,k\in \N_+ \right\}$ is dense in $\Q_+$, we can find $r$ and $k$ in $\N_+$ such that $\tfrac{1}{n+1}<\tfrac{r}{d^k}<\tfrac{1}{n}$.
Since $S$ is $d$-divisible, there exists $x\in S$ such that $d^kx=a$.
Then
\[
n(rx)\leq d^k x = a \leq (n+1)(rx),
\]
which shows that the element $rx$ has the desired properties.

To prove the second statement, assume that $S$ is a $q$-unperforated, \pom.
Choose a number $d\geq 2$ that divides $q$.
To show that $S$ is nearly unperforated, let $a,b\in S$ satisfy $a\leq_p b$.
This means that there exists $n_0\in\N$ such that $na\leq nb$ for all $n\in\N$ with $n\geq n_0$.
Choose $k\in\N_+$ such that $d^k\geq n_0$.
Then $d^ka\leq d^kb$.
As observed in \autoref{rmk:qUnpDiv}, we have that $S$ is $d^k$-unperforated.
Thus, we obtain $a\leq b$, as desired.

We have seen in \autoref{prp:nearUnperfImplications} that near unperforation implies almost unperforation in general.
\end{proof}

%==========================================================================================
\begin{lma}
\label{prp:RqModUnpDiv}
Let $S$ be a $\Cu$-semimodule over a $\Cu$-semiring $R$, and let $q$ be a supernatural number with $q\neq 1$.
Assume that the unit element of $R$ is $q$-divisible.
Then $S$ is $q$-unperforated and $q$-divisible.
In particular, $R$ itself is $q$-unperforated and $q$-divisible.
\end{lma}
\begin{proof}
The proof is analogous to that of \autoref{prp:ZModAlmUnpDiv} and is left to the reader.
\end{proof}

%==========================================================================================
\begin{thm}
\label{prp:RqModTFAE}
Let $S$ be a $\CatCu$-semigroup, and let $q$ be a supernatural number satisfying $q=q^2$ and $q\neq 1$.
Then the following are equivalent:
\beginEnumStatements
\item
We have $S\cong R_q\otimes_{\CatCu} S$.
\item
The $\CatCu$-semigroup $S$ has $R_q$-multiplication.
\item
The $\CatCu$-semigroup $S$ is $q$-divisible and $q$-unperforated.
\end{enumerate}
\end{thm}
\begin{proof}
By \autoref{prp:RqSolid}, the $\CatCu$-semiring $R_q$ is solid.
Therefore, the equivalence between \enumStatement{1} and \enumStatement{2} follows from \autoref{prp:solidModuleTFAE}.
The unit of $R_q$ is clearly $q$-divisible.
Therefore, it follows from \autoref{prp:RqModUnpDiv} that \enumStatement{2} implies \enumStatement{3}.

Finally, to show that \enumStatement{3} implies \enumStatement{2}, suppose that $S$ is $n$-divisible and $n$-unperforated for every $n\in\N_+$ that divides $q$.
It follows from \autoref{prp:qUnpNearUnp} and \autoref{prp:ZModTFAE} that $S$ has $Z$-multiplication.
By \autoref{prp:constructingSolidMod}, it is enough to define a generalized $\Cu$-bimorphism
\[
\varphi\colon R_q\times S\to S
\]
such that $\varphi(1,a)=a$ for each $a\in S$.
Recall that $R_q=K_q\sqcup (0,\infty]$, where $K_q$ is a unital subsemiring of $\Q_+$.
For $r\in (0,\infty]\subset R_q$ we use the $Z$-multiplication on $S$ to define $\varphi(r,\freeVar)$.
Let $r$ be in $K_q$.
Then there exist unique coprime integers $n,k\in\N_+$ such that $r=\tfrac{k}{n}$ and $n$ divides $q$.
Consider the map $\mu_n\colon S\to S$ that multiplies each element in $S$ by $n$.
Since $S$ is $n$-divisible and $n$-unperforated, the map $\mu_n$ is a $\CatPom$-isomorphism and therefore a $\CatCu$-isomorphism;
see \autoref{rmk:qUnpDiv}.
Given $a\in S$, we set
\[
\varphi(r,a) := k \mu_n^{-1}(a).
\]
It is now straightforward to check that $\varphi$ is a $\CatCu$-bimorphism.
It is also clear that $\varphi(1,a)=a$ for each $a\in S$.
Therefore, we may apply \autoref{prp:constructingSolidMod} to deduce that $S$ has $R_q$-multiplication.
\end{proof}

%==========================================================================================
\begin{thm}
\label{prp:MapToRqMod}
Let $S$ be a $\CatCu$-semigroup, let $a,b$ be elements in $S$, and let $q$ be a supernatural number satisfying $q=q^2$ and $q\neq 1$.
Then the following are equivalent:
\beginEnumStatements
\item
We have $1\otimes a\leq 1\otimes b$ in $R_q\otimes_{\CatCu} S$.
\item
For each $a'\in S$ satisfying $a'\ll a$, there exists $n\in\N_+$ dividing $q$ such that $na'\leq nb$ in $S$.
\end{enumerate}
\end{thm}
\begin{proof}
First, let us show that \enumStatement{1} implies \enumStatement{2}.
By definition, $R_q$ is the $\Cu$-completion of the (algebraically ordered) $\CatW$-semigroup $(K_q,\leq)$.
By \autoref{thm:tensProdCompl}, we have
\[
R_q\otimes_{\CatCu} S
= \gamma\left( K_q \otimes_{\CatPreW} S \right).
\]
Let $\alpha\colon K_q \otimes_{\CatPreW} S \to R_q\otimes_{\CatCu} S$ denote the universal $\CatW$-morphism of the $\Cu$-completion.
The underlying \pom{} of $K_q \otimes_{\CatPreW} S$ is
\[
K_q \otimes_{\CatPom} S
= \N\left[\tfrac{1}{q}\right] \otimes_{\CatPom} S.
\]
Then, given elements $x,y\in S$, it is easy to see that $1\otimes x\leq 1\otimes y$ in $K_q \otimes_{\CatPom} S$ if and only if there exists a natural number $n$ dividing $q$ such that $nx\leq ny$ in $S$.
Now, let $a,b\in S$ satisfy $1\otimes a\leq 1\otimes b$ in $R_q\otimes_{\CatCu} S$, and let $a'\in S$ satisfy $a'\ll a$.
Using at the second step that the unit of $R_q$ is a compact element, it follows
\[
\alpha(1\otimes a') = 1\otimes a' \ll 1\otimes b = \alpha(1\otimes b)
\]
in $R_q\otimes_{\CatCu} S$.
By properties of the $\Cu$-completion, we deduce $1\otimes a'\prec 1\otimes b$ in $K_q \otimes_{\CatPreW} S$.
Hence, $1\otimes a'\leq 1\otimes b$ in $K_q \otimes_{\CatPom} S$.
As observed above, this implies that there exists $n\in\N_+$ dividing $q$ such that $na'\leq nb$ in $S$.
This verifies \enumStatement{2}.

Next, let us show that \enumStatement{2} implies \enumStatement{1}.
Choose a rapidly increasing sequence $(a_k)_k$ in $S$ such that $a=\sup_k a_k$.
By assumption, for each $k$ there exists $n_k\in\N_+$ that divides $q$ and such that $n_ka_k\leq n_kb$.
Since $n_k$ is invertible in $R_q$, we deduce
\[
1\otimes a_k
= \left(n_k\tfrac{1}{n_k}\right)\otimes a_k
= \tfrac{1}{n_k} \otimes \left(n_k a_k\right)
\leq \tfrac{1}{n_k} \otimes \left(n_k b\right)
= \left(n_k\tfrac{1}{n_k}\right)\otimes b
= 1 \otimes b.
\]
Since this holds for each $k$, and since $1\otimes a=\sup_k \left( 1\otimes a_k \right)$ in $R_q\otimes_{\CatCu} S$, we obtain $1\otimes a\leq 1\otimes b$, as desired.
\end{proof}

%==========================================================================================
\begin{prp}
\label{prp:RqLimSimplicial}
Let $q$ be a supernatural number satisfying $q=q^2$ and $q\neq 1$.
Let $(d_k)_{k\in\N}$ be a sequence of natural numbers such that $q=\prod_{k\in\N}d_k$.
Then $R_q$ is isomorphic to the limit of the following inductive system of simplicial $\CatCu$-semigroups:
\[
\overline{\N} \xrightarrow{d_0} \overline{\N} \xrightarrow{d_1} \overline{\N} \xrightarrow{d_2} \ldots.
\]
Consequently, if we are given a $\CatCu$-semigroup $S$, then $R_q\otimes_{\CatCu}S$ is isomorphic to the limit of the inductive system
\[
S \xrightarrow{d_0} S \xrightarrow{d_1} S \xrightarrow{d_2} \ldots
\]
\end{prp}
\begin{proof}
Consider the following inductive system, where the map at the $k$-th step is multiplication by $d_k$:
\[
\N \xrightarrow{d_0} \N \xrightarrow{d_1} \N \xrightarrow{d_2} \ldots.
\]
It is straightforward to check that the inductive limit of this system in $\CatPom$ is $\N\left[\tfrac{1}{q}\right]$.
If we endow $\N$ and $\N\left[\tfrac{1}{q}\right]$ with auxiliary relations equal to their partial order, then we also have
\[
\N\left[\tfrac{1}{q}\right]
\cong \CatWLim \left( \N \xrightarrow{d_0} \N \xrightarrow{d_1} \N \xrightarrow{d_2} \ldots \right).
\]
Applying the reflection functor $\gamma\colon\CatPreW\to\CatCu$, which is a continuous functor, and using also \autoref{prp:limitsCu}, we obtain
\begin{align*}
R_q \cong \gamma\left(\N\left[\tfrac{1}{q}\right]\right)
&\cong \gamma\left( \CatWLim \left( \N \xrightarrow{d_0} \N \xrightarrow{d_1} \N \xrightarrow{d_2} \ldots \right) \right) \\
&\cong \CatCuLim \left( \overline{\N} \xrightarrow{d_0} \overline{\N} \xrightarrow{d_1} \overline{\N} \xrightarrow{d_2} \ldots \right),
\end{align*}
as desired.
The result for $R_q\otimes_{\CatCu}S$ follows from the limit presentation for $R_q$ in combination with \autoref{prp:tensLim}.
\end{proof}

%==========================================================================================
\begin{cor}
\label{prp:axiomTensProdRq}
Let $S$ be a $\CatCu$-semigroup, and let $q$ be a supernatural number satisfying $q=q^2$ and $q\neq 1$.
If $S$ satisfies \axiomO{5} (respectively \axiomO{6}, weak cancellation), then so does $S\otimes R_q$.
\end{cor}
\begin{proof}
By \autoref{prp:tensLimSimplicial}, each of the axioms~\axiomO{5}, \axiomO{6} and weak cancellation is preserved by taking the tensor product with a $\CatCu$-semigroup that is an inductive limit of simplicial $\CatCu$-semigroups.
Therefore, the result follows from \autoref{prp:RqLimSimplicial}.
\end{proof}

%==========================================================================================
\begin{prp}
\label{prp:tensRqCa}
Let $A$ be a \ca{}, and let $q$ be a supernatural number satisfying $q=q^2$ and $q\neq 1$.
Then there are natural isomorphisms
\[
\Cu(M_q\otimes A)
\cong \Cu(M_q)\otimes_{\Cu}\Cu(A)
\cong R_q\otimes_{\Cu} \Cu(A).
\]
In particular, we have $\Cu(A\otimes M_q)\cong \Cu(A)$ if and only if $\Cu(A)$ is $q$-unperforated and $q$-divisible.
\end{prp}
\begin{proof}
The isomorphism on the left follows from \autoref{prp:tensProdAF} since $M_q$ is an AF-algebra.
By \autoref{prp:RqFromUHF}, we have $\Cu(M_q)\cong R_q$, which gives the isomorphism on the right.
\end{proof}

%==========================================================================================
\begin{cor}
\label{prp:UHFStableNearUnp}
Let $A$ be a \ca{}.
If $A$ tensorially absorbs a UHF algebra of infinite type, then $\Cu(A)$ is nearly unperforated.
\end{cor}
\begin{proof}
Let $q$ be a supernatural number such that $q^2=q$ and $q\neq 1$ and $A\cong M_q\otimes A$.
By \autoref{prp:tensRqCa}, we have that $\Cu(A)$ is $q$-unperforated.
Then it follows from \autoref{prp:qUnpNearUnp} that $\Cu(A)$ is nearly unperforated.
\end{proof}

\vspace{5pt}
%------------------------------------------------------------------------------------------
%==========================================================================================
\section{The realification of a semigroup}
\label{sec:RMod}

%------------------------------------------------------------------------------------------
In this section, we study $\CatCu$-semigroups that are semimodules over the $\CatCu$-semiring $[0,\infty]$.
We have already shown in \autoref{exa:R} that $[0,\infty]$ is a solid $\CatCu$-semiring.
It is also known that $[0,\infty]$ is the Cuntz semigroup of a \ca{}, called the Jacelon-Razak algebra $\mathcal{R}$;
see \autoref{rmk:Razac}.

In \autoref{prp:RModTFAE}, we characterize the $\CatCu$-semimodules over $[0,\infty]$ as the $\CatCu$-semigroups that are unperforated, divisible and that contain only soft elements.
We observe in \autoref{rmk:realMult} that a $\CatCu$-semigroup has $[0,\infty]$-multiplication if and only if it has `real multiplication' in the sense of Robert, \cite{Rob13Cone}.
Given a $\CatCu$-semigroup $S$, Robert defines a `realification' $S_R$, which is a $\CatCu$-semigroup with real multiplication satisfying a natural universal property.
In \cite[Remark~3.1.5]{Rob13Cone}, Robert suggests that the realification of a $\CatCu$-semigroup can be considered as the tensor product of $S$ with $[0,\infty]$.
We verify this in \autoref{prp:RModRealification}.

%==========================================================================================
\begin{rmk}
\label{rmk:Razac}
\index{terms}{Cuntz semigroup!of Jacelon-Razak algebra}
\index{symbols}{R@$\mathcal{R}$ \quad (Jacelon-Razak algebra)}
The $\CatCu$-semiring $[0,\infty]$ is the Cuntz semigroup of the stably projectionless \ca\ known as the Jacelon-Razak algebra.
This algebra has been studied in \cite{Jac13Projectionless}, where it is denoted by $\mathcal{W}$.
Following Robert, we denote the Jacelon-Razak algebra by $\mathcal{R}$;
see \cite{Rob13Cone}.

Using the result in \cite{Jac13Projectionless}, the Cuntz semigroup of $\mathcal{R}$ was computed by Robert, \cite[\S~5]{Rob13Cone} as
\[
\Cu(\mathcal{R}) \cong [0,\infty].
\]
\end{rmk}

%==========================================================================================
\begin{rmk}
\label{rmk:realMult}
\index{terms}{real multiplication}
Let $S$ be a $\CatCu$-semigroup.
In \cite[Definition~3.1.2]{Rob13Cone}, Robert defines $S$ to have \emph{real multiplication} if there exists a map
\[
(0,\infty]\times S\to S,\quad
(t,a)\mapsto t\cdot a,\quad
\txtFA t\in(0,\infty], a\in S,
\]
that preserves addition, order and suprema of increasing sequences in each variable, and such that $1\cdot a=a$ for every $a\in S$.
It is clear that such a map extends uniquely to a generalized $\CatCu$-bi\-morphism
\[
\varphi\colon[0,\infty]\times S\to S,
\]
satisfying $\varphi(1,a)=a$ for each $a\in S$.
As observed in \autoref{exa:R}, the semiring $[0,\infty]$ is a solid $\CatCu$-semiring.
Thus, we may apply \autoref{prp:constructingSolidMod} to deduce that $S$ has a $[0,\infty]$-multiplication in the sense of \autoref{dfn:CuSemimod}.

To summarize, a $\CatCu$-semigroup has real multiplication in the sense of Robert if and only if it is a $\CatCu$-semimodule over the solid $\CatCu$-semiring $[0,\infty]$.
\end{rmk}

%==========================================================================================
\begin{lma}
\label{prp:RModViaSoft}
Let $S$ be a $\CatCu$-semigroup with $Z$-multiplication.
Then the map
\[
\varphi\colon[0,\infty]\times S \to S,\quad
(t,a)\mapsto t\cdot a,\quad
\txtFA t\in [0,\infty]\subset Z, a\in S,
\]
is a $\CatCu$-bimorphism that induces an isomorphism
\[
\bar{\varphi}\colon[0,\infty]\otimes_{\CatCu}S \xrightarrow{\cong} S_{\soft}.
\]
\end{lma}
\begin{proof}
The map $\varphi$ is a restriction of the $\CatCu$-bimorphism defining the $Z$-multiplication.
Therefore, $\varphi$ is a $\CatCu$-bimorphism.
By \autoref{prp:softInCuZmult}, the submonoid $S_\soft$ of soft elements in $S$ is a $\CatCu$-semigroup.
By \autoref{prp:softZmult}, an element $a$ in $S$ is soft if and only if $1'a=a$, where $1'$ denotes the `one' in the submonoid $[0,\infty]=Z_\soft\subset Z$ of soft elements in $Z$.
Therefore, $\bar{\varphi}$ maps into $S_\soft$.
We define a map, which will turn out to be the inverse of $\bar{\varphi}$, as follows:
\[
\psi\colon S_\soft \to [0,\infty]\otimes_{\CatCu} S,\quad
a\mapsto 1'\otimes a,\quad
\txtFA a\in S_\soft.
\]
It is easy to see that $\psi$ is a generalized $\CatCu$-morphism and that $\bar{\varphi}\circ\psi$ is the identity on $S_\soft$.
Let $t$ be in $[0,\infty]$ and let $a$ be in $S$.
Using that $[0,\infty]$ and $S$ have $Z$-multiplication, it follows from \autoref{prp:moduleTensorProd} (and also \autoref{prp:solidModuleUnique}) that $1'\otimes (t\cdot a)=t\otimes a$ in $[0,\infty]\otimes_{\CatCu}S$.
Using this at the third step, we deduce
\[
\psi\circ\varphi(t,a)
= \psi(t\cdot a)
= 1'\otimes (t\cdot a)
= t\otimes a.
\]
This implies that $\psi\circ\bar{\varphi}$ is the identity on $[0,\infty]\otimes_{\CatCu}S$, and hence $\bar{\varphi}$ is an isomorphism, as desired.
\end{proof}

%==========================================================================================
\begin{thm}
\label{prp:RModTFAE}
Let $S$ be a $\CatCu$-semigroup.
Then the following are equivalent:
\beginEnumStatements
\item
We have $S\cong [0,\infty]\otimes_{\CatCu} S$.
\item
The semigroup $S$ has $[0,\infty]$-multiplication.
\item
The semigroup $S$ is almost unperforated and almost divisible, and every element of $S$ is soft.
\item
The semigroup $S$ is unperforated and divisible, and every element of $S$ is soft.
\end{enumerate}
\end{thm}
\begin{proof}
Since $[0,\infty]$ is a solid $\CatCu$-semiring, the equivalence between \enumStatement{1} and \enumStatement{2} follows from \autoref{prp:solidModuleTFAE}.
It is clear that \enumStatement{4} implies \enumStatement{3}.
Let us show that \enumStatement{2} implies \enumStatement{4}.
To show that $S$ is unperforated, let $a,b\in S$ be such that $na\leq nb$ for some $n\in\N_+$.
Since in $[0,\infty]$ we have $1=\tfrac{1}{n}n$, we obtain
\[
a =
(\tfrac{1}{n}n)\cdot a
= \tfrac{1}{n}\cdot (na)
\leq \tfrac{1}{n}\cdot (nb)
= b.
\]
It is also clear that $S$ is divisible.

Next, let us show that \enumStatement{3} implies \enumStatement{1}.
By \autoref{prp:ZModTFAE}, we have that $S$ has $Z$-multiplication.
Using \autoref{prp:RModViaSoft} to obtain the first isomorphism, and using the assumption for the second equality, we obtain
\[
[0,\infty]\otimes_{\CatCu} S\cong S_\soft = S,
\]
as desired.
\end{proof}

%==========================================================================================
\begin{lma}
\label{prp:RModViaRq}
Let $S$ be a $\CatCu$-semigroup.
Then there is a natural isomorphism
\[
[0,\infty]\otimes_{\CatCu}S \cong (R_{2^\infty}\otimes_{\CatCu} S)_{\soft}.
\]
\end{lma}
\begin{proof}
Since $R_{2^\infty}$ has $Z$-multiplication, it follows from \autoref{prp:moduleTensorProd} that $R_{2^\infty}\otimes_{\CatCu} S$ has $Z$-multiplication.
Then, using \autoref{prp:RModViaSoft} to obtain the last isomorphism, and using that $[0,\infty] \cong [0,\infty]\otimes_{\CatCu}R_{2^\infty}$ at the first step (which follows from a second usage of \autoref{prp:RModViaSoft}), we obtain
\begin{align*}
[0,\infty]\otimes_{\CatCu}S
&\cong \left([0,\infty]\otimes_{\CatCu}R_{2^\infty}\right)\otimes_{\CatCu} S \\
&\cong [0,\infty]\otimes_{\CatCu}\left( R_{2^\infty}\otimes_{\CatCu} S\right)
\cong (R_{2^\infty}\otimes_{\CatCu} S)_{\soft},
\end{align*}
as desired.
\end{proof}

%==========================================================================================
\begin{prp}
\label{prp:axiomTensProdR}
Let $S$ be a $\CatCu$-semigroup.
If $S$ satisfies \axiomO{5} (respectively \axiomO{6}, weak cancellation), then so does $[0,\infty]\otimes_{\CatCu}S$.
\end{prp}
\begin{proof}
Assume that $S$ is a $\CatCu$-semigroup satisfying \axiomO{5}.
By \autoref{prp:axiomTensProdRq}, the tensor product $R_{2^\infty}\otimes_{\CatCu} S$ satisfies \axiomO{5}.
Since $R_{2^\infty}\otimes_{\CatCu} S$ has $Z$-multiplication, it follows from \autoref{prp:softInCuZmult} that the subsemigroup of soft elements in the tensor product $R_{2^\infty}\otimes_{\CatCu} S$ is a $\CatCu$-semigroup satisfying \axiomO{5}.
Now the desired result follows from \autoref{prp:RModViaRq}.
It is proved analogously that axiom \axiomO{6} and weak cancellation pass from $S$ to $[0,\infty]\otimes_{\CatCu}S$.
\end{proof}

%==========================================================================================
\begin{pgr}
\label{pgr:realification}
\index{terms}{realification}
\index{symbols}{S$_R$@$S_R$ \quad (realification)}
Let $S$ be a $\CatCu$-semigroup.
Recall from \autoref{pgr:fctl}, that there is a natural map
\[
S \to \Lsc(F(S)),\quad
a\mapsto\hat{a},\quad
\txtFA a\in S,
\]
where $F(S)$ is the cone of functionals on $S$.

In \cite[\S~3.1]{Rob13Cone}, Robert defines the \emph{realification} of $S$ as the smallest subsemigroup of $\Lsc(F(S))$ that is closed under passing to suprema of increasing sequences, and which contains all elements of the form $\tfrac{1}{n}\hat{a}$ for some $n\in\N_+$ and some $a\in S$.
We denote the realification of $S$ by $S_R$.
In \cite[Proposition~3.1.1]{Rob13Cone}, it is shown that $S_R$ is a $\CatCu$-semigroup.
Moreover, if $S$ satisfies \axiomO{5'}, the original version of the almost algebraic order axiom (see \autoref{rmk:addAxioms}(1)), then so does $S_R$.
In \autoref{prp:RModRealification}, we show that $S_R$ is naturally isomorphic to $[0,\infty]\otimes_{\CatCu}S$.
Then it follows from \autoref{prp:axiomTensProdR} that $S_R$ satisfies \axiomO{5} whenever $S$ does.
\end{pgr}

%==========================================================================================
\begin{lma}
\label{prp:fctlRMod}
Let $S$ be a $\CatCu$-semigroup.
Consider the map $\vartheta\colon S\to [0,\infty]\otimes_{\CatCu} S$ that sends $a$ in $S$ to $1\otimes a$.
Then, given a functional $\lambda\in F([0,\infty]\otimes_{\CatCu} S)$, the composition $\lambda\circ\vartheta$ is a functional in $F(S)$.
Moreover, the assignment
\[
\theta\colon F([0,\infty]\otimes_{\CatCu}S)\to F(S),\quad
\lambda\mapsto\lambda\circ\vartheta,
\]
is an isomorphism of topological cones.
\end{lma}
\begin{proof}
We first define a map that will turn out to be the inverse of $\theta$.
Given a functional $\mu$ in $F(S)$, consider the map
\[
[0,\infty]\times S \to [0,\infty],\quad
(t,a)\mapsto t\cdot\mu(a)\quad
\txtFA t\in [0,\infty], a\in S.
\]
It is straightforward to check that this is a generalized $\CatCu$-bimorphism, which therefore induces a generalized $\CatCu$-morphism
\[
\tilde{\mu}\colon [0,\infty]\otimes_{\CatCu}S \to [0,\infty].
\]
This means that $\tilde{\mu}$ is a functional in $F([0,\infty]\otimes_{\CatCu}S)$ such that $\tilde{\mu}(t\otimes a)=t\cdot\mu(a)$ for each $t\in[0,\infty]$ and $a\in S$.
This defines a map
\[
\psi\colon F(S)\to F([0,\infty]\otimes_{\CatCu}S),\quad
\mu\mapsto\tilde{\mu},\quad
\txtFA \mu\in F(S).
\]
Given $\mu$ in $F(S)$ and $a\in S$, we deduce
\[
\theta\circ\psi(\mu)(a)
= \theta(\tilde{\mu})(a)
= \tilde{\mu}(1\otimes a)
= 1\cdot\mu(a)
= \mu(a).
\]
Thus, $\theta\circ\psi$ is the identity on $F(S)$.
Conversely, let $\lambda$ be a functional in $F([0,\infty]\otimes_{\CatCu}S)$.
Since $[0,\infty]$ is solid, it follows from \autoref{prp:solidTFAE} that $\lambda$ is automatically $[0,\infty]$-linear.
Using this at the last step, we deduce for each $t\in[0,\infty]$ and $a\in S$ that
\[
\psi\circ\theta(\lambda)(t\otimes a)
= t\cdot\theta(\lambda)(a)
= t\lambda(1\otimes a)
= \lambda(t\otimes a).
\]
It follows that $\psi\circ\theta$ is the identity on $F([0,\infty]\otimes_{\CatCu}S)$.
It is straightforward to check that $\theta$ and $\psi$ are continuous and linear, which shows the desired result.
\end{proof}

%==========================================================================================
\begin{prp}
\label{prp:RModRealification}
Let $S$ be a $\CatCu$-semigroup.
Then the natural map
\[
\varphi\colon[0,\infty]\times S \to S_R \subset \Lsc(F(S)),\quad
(t,a)\mapsto t\cdot\hat{a},\quad
\txtFA t\in[0,\infty], a\in S,
\]
is a $\CatCu$-bimorphism that induces an isomorphism
\[
\bar{\varphi}\colon[0,\infty]\otimes_{\CatCu}S \xrightarrow{\cong} S_R.
\]
\end{prp}
\begin{proof}
It is straightforward to check that $\varphi$ is a generalized $\CatCu$-bimorphism.
By universal properties of the tensor product, the induced map $\bar{\varphi}$ is a generalized $\CatCu$-morphism.
It follows easily from the definition of $S_R$ that $\bar{\varphi}$ is surjective.

Next, we show that $\bar{\varphi}$ is an order-embedding.
Consider the isomorphism $\theta\colon F([0,\infty]\otimes_{\CatCu}S)\to F(S)$ from \autoref{prp:fctlRMod}.
It induces an isomorphism of \pom{s}
\[
\theta^\ast \colon \Lsc(F(S)) \to \Lsc(F([0,\infty]\otimes_{\CatCu}S)),\quad
f\mapsto f\circ\theta.
\]
We let
\[
\Gamma\colon [0,\infty]\otimes_{\CatCu}S \to \Lsc(F([0,\infty]\otimes_{\CatCu}S)),
\]
denote the canonical map that sends $x\in [0,\infty]\otimes_{\CatCu}S$ to $\hat{x}$.
Then, for $t\in[0,\infty]$, $a\in S$ and $\lambda\in F([0,\infty]\otimes_{\CatCu}S)$, we have
\[
\Gamma(t\otimes a)(\lambda)
= \lambda(t\otimes a)
= t\cdot \theta(\lambda)(a)
= \varphi(t, a)(\theta(\lambda))
= (\theta^\ast\circ\bar{\varphi})(t\otimes a)(\lambda).
\]
This implies that $\Gamma=\theta^\ast\circ\bar{\varphi}$.
The situation is shown in the following commutative diagram:
\[
\xymatrix{
[0,\infty]\otimes_{\CatCu}S \ar[r]^{\bar{\varphi}} \ar[dr]^{\Gamma}
& \Lsc(F(S)) \ar[d]^{\theta^\ast}_{\cong} \\
& \Lsc(F([0,\infty]\otimes_{\CatCu}S))
}
\]
It follows from \autoref{prp:RModTFAE} that the $\CatCu$-semigroup $[0,\infty]\otimes_{\CatCu}S$ is almost unperforated and that each of its elements is soft.
Then \autoref{prp:soft_comparison} implies that the map $\Gamma$ is an order-embedding.
Since $\theta^\ast$ is an order isomorphism, it follows that $\bar{\varphi}$ is an order-embedding.
Thus, $\bar{\varphi}$ is an isomorphism of \pom{s}, and consequently a $\CatCu$-isomorphism.
\end{proof}

%==========================================================================================
\begin{cor}
\label{prp:tensProdWithR}
Let $S$ be a $\CatCu$-semigroup.
Then there are canonical isomorphisms
\[
[0,\infty]\otimes_{\CatCu}S \cong S_R \cong (Z\otimes_{\CatCu} S)_{\soft}.
\]
In particular, if $S$ has $Z$-multiplication, then there are isomorphisms
\[
[0,\infty]\otimes_{\CatCu}S \cong S_R \cong S_{\soft}.
\]
\end{cor}
\begin{proof}
This follows by combining \autoref{prp:RModViaSoft} with \autoref{prp:RModRealification}. (Notice that $[0,\infty]\otimes_{\CatCu}Z\cong Z_{\soft}=[0,\infty]$.)
\end{proof}

%==========================================================================================
\begin{thm}
\label{prp:MapToRMod}
Let $S$ be a $\CatCu$-semigroup, and let $a,b$ be elements in $S$.
Then the following are equivalent:
\beginEnumStatements
\item
We have $1\otimes a\leq 1\otimes b$ in $[0,\infty]\otimes_{\CatCu}S$.
\item
We have $\hat{a}\leq\hat{b}$ in $\Lsc(F(S))$.
\item
For every $a'\in S$ satisfying $a'\ll a$, and every $\varepsilon>0$, there exist $k,n\in\N_+$ such
that $(1-\varepsilon)<\tfrac{k}{n}$ and $ka'\leq nb$ in $S$.
\end{enumerate}
\end{thm}
\begin{proof}
As shown in \autoref{prp:RModRealification}, there is an isomorphism between the tensor product $[0,\infty]\otimes_{\CatCu}S$ and $S_R$ that identifies $1\otimes a$ with $\hat{a}$ and $1\otimes b$ with $\hat{b}$.
This shows the equivalence between  \enumStatement{1} and \enumStatement{2}.
The equivalence between statement \enumStatement{2} and \enumStatement{3} is shown in \autoref{prp:comparison_hat}.
\end{proof}

%------------------------------------------------------------------------------------------
Let $A$ be a \ca.
It is shown in \cite[Theorem~5.1.2]{Rob13Cone} that there is a natural isomorphism between $\Cu(\mathcal{R}\otimes A)$ and $\Cu(A)_R$.
Using \autoref{prp:RModRealification} we can rephrase the result of Robert as follows:

%==========================================================================================
\begin{prp}
\label{prp:tensRCa}
Let $A$ be a \ca{}.
Then there are natural isomorphisms
\[
\Cu(\mathcal{R}\otimes A)
\cong\Cu(\mathcal{R})\otimes_{\Cu}\Cu(A)
\cong[0,\infty]\otimes_{\Cu}\Cu(A).
\]
In particular, we have $\Cu(\mathcal{R}\otimes A)\cong \Cu(A)$ if and only if $\Cu(A)$ is unperforated, divisible and each element in $\Cu(A)$ is soft (or equivalently, purely noncompact).
\end{prp}

\vspace{5pt}
%------------------------------------------------------------------------------------------
%==========================================================================================
\section{Examples and Applications}

%==========================================================================================
\begin{pgr}
\label{pgr:ssaSolid}
Let $D$ be a unital, separable, strongly self-absorbing \ca{}.
As mentioned before in \autoref{prp:semirgFromSSA}, it is known that $D$ is nuclear and simple and that $D$ is either purely infinite, in which case $\Cu(D)\cong\{0,\infty\}$, or that $D$ is stably finite with unique tracial state.
The only known examples of purely infinite, strongly self-absorbing \ca{s} are the Cuntz algebras $\mathcal{O}_\infty$ and $\mathcal{O}_2$, and the tensor products of $\mathcal{O}_\infty$ with a UHF-algebra of infinite type.
It follows from the Kirchberg-Phillips classification theorem, \cite[Theorem~8.4.1, p.128]{Ror02Classification}, that these are the only purely infinite, strongly self-absorbing \ca{s} satisfying the Universal Coefficient Theorem (UCT).

Let us assume that $D$ is stably finite.
In that case, the only known examples are the Jiang-Su algebra $\mathcal{Z}$ and the UHF-algebras of infinite type $M_q$;
see \autoref{pgr:Mq}.
Each of these algebras satisfies the UCT.
We also have that $D$ is $\mathcal{Z}$-stable, by \cite[Theorem~3.1]{Win11ssaZstable};
see also \autoref{prp:ssaZstable}.
Therefore, it follows from \autoref{prp:CuSimpleZstable} that the Cuntz semigroup of $D$ can be computed as
\[
\Cu(D) \cong V(D) \sqcup (0,\infty].
\]
By \cite[Theorem~6.7]{Ror04StableRealRankZ}, $D$ has stable rank one.
Therefore, $V(D)$ is a cancellative, algebraically ordered monoid that is isomorphic to the positive part of $K_0(D)$.

We conclude that the only known Cuntz semigroups realized by stably finite, strongly self-absorbing \ca{s} are the following:
\[
Z = \N \sqcup (0,\infty] = \Cu(\mathcal{Z}),
\quad
R_q = \N\left[\tfrac{1}{q}\right] \sqcup (0,\infty] = \Cu(M_q).
\]
It follows from the $K$-theory computations in \cite[Proposition~5.1]{TomWin07ssa} that the Cuntz semigroups $Z$ and $R_q$ are the only Cuntz semigroups of stably finite, strongly self-absorbing \ca{s} that satisfy the UCT.

We have seen that $Z$ and $R_q$ (and also $\{0,\infty\}$) are solid $\CatCu$-semirings;
see for example \autoref{prp:ZSolid} and \autoref{prp:RqSolid}.
Therefore, the Cuntz semigroup of every strongly self-absorbing \ca{} satisfying the UCT is a solid $\CatCu$-semiring.

It is an open problem whether every nuclear \ca{} satisfies the UCT.
It is also unclear if there exist strongly self-absorbing \ca{s} that do not satisfy the UCT.
More modestly, we ask the following question:
\end{pgr}

%==========================================================================================
\begin{prbl}
\label{prbl:ssaSolid}
Given a strongly self-absorbing \ca{} $D$, is the Cuntz semiring $\Cu(D)$ a solid $\CatCu$-semiring?
\end{prbl}

%------------------------------------------------------------------------------------------
As noted above, the answer is `yes' for every strongly self-absorbing \ca{} satisfying the UCT.
In \autoref{sec:classificationSolid}, we provide a complete classification of solid $\CatCu$-semirings.
We remark that even when excluding $\CatCu$-semirings that are elementary or have noncompact unit, there exist solid $\CatCu$-semirings that are not the Cuntz semigroup of any known strongly self-absorbing \ca{};
see \autoref{thm:solidSemirgClassification}.

%------------------------------------------------------------------------------------------
We can summarize \autoref{prp:tensCaWithInfty}, \autoref{prp:tensRqCa} and \autoref{prp:tensRCa} as follows:

%==========================================================================================
\begin{prp}
\label{prp:tensSaCa}
Let $A$ and $D$ be \ca{s}.
Assume that either $D=\mathcal{R}$ or that $D$ is a (unital) separable, strongly self-absorbing \ca{} satisfying the UCT but not equal to $\mathcal{Z}$.
Then there is a natural isomorphism
\[
\Cu(D\otimes A)
\cong\Cu(D)\otimes_{\CatCu} \Cu(A).
\]
\end{prp}

%==========================================================================================
\begin{pgr}
\label{pgr:strictComp}
\index{terms}{strict comparison}
Recall that a \ca\ $A$ is said to have \emph{strict comparison of positive elements} if for any positive elements $x,y\in M_\infty(A)$ the following holds:
If $x$ is contained in $\overline{AyA}$, the closed two-sided ideal generated by $y$, and if $d(x)<d(y)$ for every dimension function $d$ on $A$ satisfying $d(y)=1$, then $x\precsim y$ ($x$ is Cuntz subequivalent to $y$).
It was shown by R{\o}rdam that a \ca\ $A$ has strict comparison of positive elements if and only if its pre-completed Cuntz semigroup $W(A)$ is almost unperforated, \cite[Proposition~3.2]{Ror04StableRealRankZ}; see also \cite[Lemma~5.7]{AraPerTom11Cu}.
It is easy to see that $W(A)$ is almost unperforated if and only if $\Cu(A)$ is.

For a unital, simple \ca\ $A$, and positive elements $x,y\in M_\infty(A)$, the condition that $x$ belongs to $\overline{AyA}$ is automatically satisfied whenever $y$ is nonzero, and it is moreover enough to consider lower-semicontinuous dimension functions.
In this form, the notion of strict comparison of positive elements in simple \ca{s} was introduced by Blackadar as the `Fundamental Comparability Question (FCQ4)' in \cite[\S~6.4.7]{Bla88Comparison}.

In \cite[Question~5.3]{RorWin10ZRevisited}, R{\o}rdam and Winter ask whether the Jiang-Su algebra unitally embeds into any unital \ca{} $A$ such that the class of the unit is almost divisible in $\W(A)$.
It is easy to see that the class of the unit is almost divisible in $W(A)$ if and only if it is in $\Cu(A)$.
Thus, the implication `\enumStatement{2}$\Rightarrow$\enumStatement{4}' of the following \autoref{prp:embedZCa} provides a positive answer to the question of  R{\o}rdam and Winter for \ca{s} that have stable rank one and strict comparison of positive elements.
\end{pgr}

%==========================================================================================
\begin{prp}
\label{prp:embedZCa}
Let $A$ be a unital \ca\ with stable rank one and with strict comparison of positive elements.
Then the following statements are equivalent:
\beginEnumStatements
\item
For each $n\in\N$, there exists a unital \starHom\ $Z_{n,n+1}\to A$, where $Z_{n,n+1}$ is the dimension drop algebra.
\item
The element $[1_A]$ is almost divisible in $\Cu(A)$.
\item
There exists a $\CatCu$-morphism $Z\to\Cu(A)$ that sends $1_Z$ to $[1_A]$.
\item
There exists a unital \starHom\ $\mathcal{Z}\to A$.
\end{enumerate}
\end{prp}
\begin{proof}
The equivalence between \enumStatement{1} and \enumStatement{2} is shown in \cite[Proposition~5.1]{RorWin10ZRevisited}.
The equivalence between \enumStatement{2} and \enumStatement{3} follows from \autoref{prp:zembed}.
Finally, it is clear that \enumStatement{4} implies \enumStatement{3}.
The converse follows from Theorems~1.0.1 and 3.2.2 in \cite{Rob12LimitsNCCW}.
\end{proof}

%==========================================================================================
\begin{pgr}
Let $S$ be a simple, stably finite $\CatCu$-semigroup{} satisfying \axiomO{5}.
Recall that we denote by $S_c$ (respectively $S_c^\times$) the subsemigroup of (nonzero) compact elements in $S$.
As shown in \autoref{prp:softNoncpctSimple}, an element $a\in S$ is soft whenever it is not compact.
The only element that is both compact and soft is the zero element.
Thus, there is a natural decomposition
\[
S=S_c^\times\sqcup S_\soft.
\]

If we additionally assume that $S$ has $Z$-multiplication, then we can apply \autoref{prp:tensProdWithR} to compute $S_\soft$ as $S_R$, the realification of $S$.
We obtain
\[
S\cong S_c^\times\sqcup S_R.
\]

Since $S$ also satisfies \axiomO{5}, it follows from \cite[Theorem~3.2.1]{Rob13Cone}  that $S_R$ is isomorphic to $L(F(S))$.
The point is that in this case the semigroup $L(F(S))$ only depends on $F(S)$, the cone of functionals on $S$.
On the other hand, it is not clear if $S_R$ only depends on $F(S)$ in general;
see \autoref{prbl:LFS}.

We summarize our representation result for simple $\CatCu$-semigroups with $Z$-mul\-ti\-pli\-ca\-tion in the following Theorem.
For Cuntz semigroups of \ca{s}, the analogous result has appeared in \cite[Theorem~2.6]{BroTom07ThreeAppl} and \cite[Theorem~6.3]{AntBosPer11CompletionsCu}.
\end{pgr}

%==========================================================================================
\begin{thm}
\label{thm:isothm}
Let $S$ be a simple, stably finite $\CatCu$-semigroup that satisfies \axiomO{5} and has $Z$-multipli\-ca\-tion.
Then the soft part of $S$ is isomorphic to $L(F(S))$.
Thus, there is a natural isomorphism
\[
S\cong S_c^\times \sqcup L(F(S)).
\]
\end{thm}

%------------------------------------------------------------------------------------------
Let $S$ be a $\CatCu$-semigroup.
By definition, $S_R$ and $L(F(S))$ are submonoids of $\Lsc(F(S))$.
It follows from \cite[Proposition~3.1.6]{Rob13Cone} that $S_R$ is a subset of $L(F(S))$.
(The implicit assumption of \axiomO{5} in \cite{Rob13Cone} is not needed in the proof of \cite[Proposition~3.1.6]{Rob13Cone}.)

It is natural to ask whether $S_R$ is in fact equal to $L(F(S))$.
Under the assumption of \axiomO{5} this is indeed the case;
see \cite[Theorem~3.2.1]{Rob13Cone}.
Thus, we ask if the result of Robert holds without the assumption of \axiomO{5}.

%==========================================================================================
\begin{prbl}
\label{prbl:LFS}
Let $S$ be a $\CatCu$-semigroup.
Is it true that $S_R=L(F(S))$?
\end{prbl}

%------------------------------------------------------------------------------------------
%==========================================================================================
%##########################################################################################
\chapter{Structure of \texorpdfstring{$\CatCu$}{Cu}-semirings}
\label{sec:semirg}

%------------------------------------------------------------------------------------------
In this chapter, we study the structure of certain classes of $\CatCu$-semirings that satisfy \axiomO{5}.
The main result is \autoref{thm:simpleSemirg}, where we show that every simple, nonelementary $\CatCu$-semiring with \axiomO{5} is automatically almost unperforated and almost divisible.
Together with \autoref{prp:ZModTFAE}, we obtain that every simple, nonelementary $\CatCu$-semiring has $Z$-multiplication, which can be interpreted as the $\CatCu$-semiring version of Winter's result that every strongly self-absorbing \ca{} is $\mathcal{Z}$-stable;
see \autoref{prp:simpleSemirgZmult}.
We also use our findings to give an alternative proof of Winter's result;
see \autoref{prp:ssaZstable}.

In \autoref{sec:algebraicSemirg}, we study algebraic $\CatCu$-semirings.
We establish an equivalence between the category of weakly cancellative, algebraic $\Cu$-semirings and the category of directed, \por{s};
see \autoref{prp:equivalenceSrgAlgCu}.
We also give several characterizations when a simple $\Cu$-semiring with unique normalized functional is algebraic;
see \autoref{prp:imagelambda}.

In \autoref{sec:classificationSolid}, we analyse the structure of solid $\CatCu$-semirings.
The main result is \autoref{thm:solidSemirgClassification}, where we classify all solid, nonelementary $\CatCu$-semirings satisfying~\axiomO{5}.

%------------------------------------------------------------------------------------------
%==========================================================================================
\section{Simple \texorpdfstring{$\Cu$}{Cu}-semirings}
\label{sec:SimpleSrg}

%------------------------------------------------------------------------------------------
Recall that a simple $\Cu$-semigroup is \emph{elementary} if it contains a minimal nonzero element;
see \autoref{pgr:elementarySemigr}.

%==========================================================================================
\begin{lma}
\label{prp:simpleSemirgElmtry}
Let $R$ be a simple $\Cu$-semiring satisfying \axiomO{5}.
Then:
\beginEnumStatements
\item
If $R$ is nonelementary, then there exist nonzero $c,d\in R$ with $c+d\leq 1$.
\item
Assume $R\ncong\{0\}$.
Then $R$ is elementary if and only if the unit $1_R$ is a minimal nonzero element.
\item
If $1_R$ is properly infinite (that is, $2\cdot 1_R=1_R$), then $R\cong\{0,\infty\}$ or $R\cong\{0\}$.
\end{enumerate}
\end{lma}
\begin{proof}
We first assume $1_R\neq 0_R$ and that $1_R$ is not a minimal nonzero element.
Choose a nonzero $x\in R$ with $x<1$ and a nonzero $c\in R$ satisfying $c\ll x$.
Using that $R$ satisfies \axiomO{5}, we can choose $d\in R$ such that
\[
c+d \leq 1 \leq x+d.
\]
Since $x\neq 1$, $d$ is nonzero.
This immediately implies statement \enumStatement{1}.

To show \enumStatement{2}, note that the assumption $R\ncong\{0\}$ implies $1_R\neq 0_R$.
It follows that $R$ has no zero divisors;
see \autoref{rmk:CuSemirg}.
If $1_R$ is a minimal, nonzero element, then $R$ is elementary by definition.
Conversely, assume $1_R$ is not minimal.
As shown at the beginning of the proof, we can choose nonzero elements $c$ and $d$ such that $c+d\leq 1$.

In order to prove that $R$ is nonelementary, assume $a$ is a nonzero element in $R$.
Consider the elements $ca$ and $da$, which are nonzero since $R$ has no zero divisors.
Moreover, $ca+da\leq a$.
If $ca\neq a$ or $da\neq a$, then $a$ is not a minimal, nonzero element.
Otherwise, we can deduce $a=2a$ and hence $a=\infty$.
Since $1\leq \infty$ and $1$ is assumed not to be a minimal, nonzero element, we obtain that $a$ is not a minimal, nonzero element in either case.
Thus, $R$ is nonelementary.

Finally, statement \enumStatement{3} is easily verified.
\end{proof}

%==========================================================================================
\begin{exa}
\label{exa:elmtrySemirg}
The elementary $\Cu$-semigroups $\overline{\N}$ and $\elmtrySgp{k}=\{0,1,2,\ldots,k,\infty\}$ have natural (and unique) $\CatCu$-products giving them the structure of solid $\CatCu$-semirings.
These are the only simple, nonzero, elementary $\CatCu$-semirings satisfying \axiomO{5} and \axiomO{6};
see \autoref{pgr:elementarySemigr}.

Without \axiomO{6}, there are other examples of simple, elementary $\CatCu$-semirings:
Consider $S=\{0,1,1',2,3,4,\ldots,\infty\}$, with addition and multiplication among the un-apostrophized elements as usual and such that $1'+k=1+k$ and $k\cdot 1'=k$ for each $k\in\N$.
The elements $1$ and $1'$ are incomparable.
\end{exa}

%------------------------------------------------------------------------------------------
To prove \autoref{thm:simpleSemirg}, we need several lemmas.
Given a $\CatCu$-semigroup $S$ and $k\in\N_+$, recall that an element $a$ in $S$ is \emph{almost $k$-divisible} if for every $a'\ll a$ there exists $x\in S$ such that $kx\leq a$ and $a'\leq (k+1)x$;
see \autoref{dfn:almDiv}.
If this holds for every $k\in\N_+$, we say that the element is \emph{almost divisible}.
Moreover, $S$ is \emph{almost divisible} if each of its elements is.

%==========================================================================================
\begin{lma}
\label{lma:simpleSemirg1}
Let $R$ be a simple $\Cu$-semiring.
Let $c,d\in R$ be nonzero elements satisfying $c+d\leq 1$.
Then for every $x\in R$ with $x\ll\infty$ there exists $n\in\N$ such that $xc^n\leq 1$.
\end{lma}
\begin{proof}
We inductively show
\begin{align}
\label{lma:simpleSemirg1:eq1}
[(k+1)d]c^k\leq 1,
\end{align}
for every $k\in\N$.
For $k=0$ this is clear.
For the induction step, assume \eqref{lma:simpleSemirg1:eq1} is satisfied for some $k\in\N$.
Multiplying both sides of \eqref{lma:simpleSemirg1:eq1} by $c$, we obtain
\[
[(k+1)d]c^{k+1}\leq c.
\]
Using this at the second step, we deduce
\[
[(k+2)d]c^{k+1}
\leq [(k+1)d]c^{k+1}+d
\leq c+d
\leq 1.
\]
Since $R$ is simple, and since $d\neq 0$ and $x\ll\infty$, we can find $n\in\N$ such that $x\leq (n+1)d$.
Then $xc^n\leq [(n+1)d]c^n\leq 1$, as desired.
\end{proof}

%==========================================================================================
\begin{lma}
\label{lma:simpleSemirg2}
Let $R$ be a simple $\Cu$-semiring, and let $k\in\N_+$.
If $c+d\leq 1$ for some nonzero elements $c,d\in R$, then there exists a nonzero element $a\in R$ such that $ka\ll (k+1)a\leq 1$.
\end{lma}
\begin{proof}
Let $c,d$ and $k$ be as in the statement.
We may assume $c\ll\infty$ (by replacing $c$ by some nonzero $c'\in R$ satisfying $c'\ll c$, if necessary).
We construct the element $a$ in two steps.

Step 1:
Let $c_1$ be a nonzero element in $R$ satisfying $c_1\ll c$.
Since $c\ll\infty$, we can choose $n\in\N$ satisfying $c\leq nc_1$.
By \autoref{lma:simpleSemirg1}, we can find $m\in\N$ such that
\[
[k(k+1)n^2]c^m\leq 1.
\]
Choose elements $c_i$ in $R$ for $i=2,\ldots,m$ with
\[
c_1\ll c_2\ll\ldots\ll c_m \ll c.
\]
Set $b :=c_1c_2\cdots c_m$ (the product).
Using at the second step that compact containment is preserved under multiplication, we obtain
\[
b
= [c_1c_2\cdots c_{m-1}]c_m
\ll [c_2c_3\cdots c_m]c
\leq [c_2c_3\cdots c_m]nc_1
= nb.
\]
Moreover, we have $b\leq c^m$, and therefore $k(k+1)n^2 b\leq 1$.

Step 2:
Choose elements $b_i$ in $R$ for $i=1,\ldots,kn$ such that
\[
b\ll b_1\ll b_2\ll\ldots\ll b_{kn} \ll nb.
\]
Set $a :=b_1+\ldots+b_{kn}$.
Using at the second step that compact containment is preserved under addition, we deduce
\[
a
= [b_1+b_2+\ldots+b_{kn-1}]+b_{kn}
\ll [b_2+b_3+\ldots+b_{kn}]+nb
\leq a+nb.
\]
Multiplying this inequality by $k$, and using $knb\leq a$ at the last step, we obtain
\[
ka
\ll ka+knb
\leq (k+1)a.
\]
Moreover, we have $a\leq kn(nb)$, and thus $(k+1)a\leq (k+1)kn^2b\leq 1$,
as desired.
\end{proof}

%==========================================================================================
\begin{lma}
\label{lma:simpleSemirg3}
Let $R$ be a simple $\Cu$-semiring satisfying \axiomO{5}, and let $k\in\N_+$.
Suppose that there exists a nonzero element $a\in R$ such that $2ka\ll (2k+1)a\leq 1$.
Then there exists $c\in R$ such that $kc\leq 1\leq (k+1)c$.
\end{lma}
\begin{proof}
Using that $R$ satisfies \axiomO{5} we can choose $t$ in $R$ such that
\[
2ka+t\leq 1 \leq (2k+1)a+t.
\]
We think of this inequality as `$2ka\leq (1-t) \leq (2k+1)a$'.
Then we want to multiply by $(1-t)^{-1}$, which we think of as $(1+t+t^2+\ldots)$.
To make this precise, set
\[
z := 1+t+t^2+\ldots = \sup_{n\in\N} \left(1+t+t^2+\ldots+t^n\right).
\]
We construct $c$ such that $kc\leq 1\leq (k+1)c$ in two steps.

Step 1:
We show $2kaz \leq 1$.
To obtain this, we first show that if $x,y\in R$ satisfy $x+y\leq 1$, then
\[
x\cdot\left(\sum_{n=0}^\infty y^n\right)\leq 1.
\]
Indeed, multiplying the inequality $x+y\leq 1$ by $y$, we obtain $xy+y^2\leq y$.
Using this at the second step, we deduce
\[
x(1+y)+y^2 = x+(xy+y^2)\leq x+y \leq 1.
\]
Inductively, we obtain
\[
x(1+y+y^2+\ldots+y^n)+y^{n+1}\leq 1,
\]
for all $n\in\N_+$, and therefore $x\cdot\left(\sum_{n=0}^\infty y^n\right)\leq 1$, as desired.
Applying this to $2ka+t\leq 1$, we deduce $2kaz\leq 1$.

Step 2:
We show $1\leq (2k+2)az$.
Choose a rapidly increasing sequence $(w_r)_{r\in\N}$ in $R$ satisfying $1=\sup_r w_r$.
(If the unit element is compact, the following argument can be simplified.)
For each fixed $r\in\N$, since $w_r\ll 1 \leq (2k+1)a+t$, we can choose $t_r$ in $R$ such that
\[
t_r\ll t,\quad
w_r\leq (2k+1)a+t_r.
\]
We have $1\leq (2k+1)a+t$.
Multiplying this inequality by $t_r$, we obtain
\[
t_r \leq (2k+1)at_r +tt_r \leq (2k+1)at + tt_r.
\]
Since $w_r\leq (2k+1)a+t_r$, it follows
\[
w_r \leq (2k+1)a(1+t)+tt_r.
\]
Inductively, we obtain
\[
w_r \leq (2k+1)a(1+t+\ldots+t^n)+t^nt_r,
\]
for all $n\in\N$.
Since $t_r\ll\infty$ and $a\neq 0$, we can choose $m_r\in\N$ such that $t_r\leq(m_r+1)a$.
Then
\[
t^{m_r}t_r \leq t^{m_r}(m_r+1)a =a(t^{m_r}+\ldots+t^{m_r})\leq a(1+\dots+t^{m_r}),
\]
This implies
\begin{align*}
w_r
&\leq (2k+1)a(1+t+\ldots+t^{m_r})+t^{m_r}t_r \\
&\leq (2k+2)a(1+t+\dots+t^{m_r}) \\
&\leq (2k+2)a(1+t+\dots+t^{m_r}+\ldots) \\
&=(2k+2)az.
\end{align*}
Since this holds for all $r\in\N$ and since $1=\sup_r w_r$, we deduce $1\leq (2k+2)az$.
Then $k(2az)\leq 1\leq (k+1)(2az)$, which finishes the proof setting $c=2az$.
\end{proof}

%------------------------------------------------------------------------------------------
The following theorem is the main structure result for simple $\CatCu$-semirings.

%==========================================================================================
\begin{thm}
\label{thm:simpleSemirg}
Let $R$ be a simple, nonelementary $\Cu$-semiring satisfying \axiomO{5}.
Then $R$ is almost unperforated and almost divisible.
\end{thm}
\begin{proof}
By \autoref{prp:ZModAlmUnpDiv}, it is enough to show that the unit of $R$ is almost divisible.
Let $k$ be a natural number with $k\geq 1$.
By \autoref{prp:simpleSemirgElmtry} we can choose nonzero elements $c,d\in R$ satisfying $c+d\leq 1$.
Thus, we may apply \autoref{lma:simpleSemirg2} (for $2k$) to obtain a nonzero element $a\in R$ such that $2ka\ll (2k+1)a\leq 1$.
Now it follows from \autoref{lma:simpleSemirg3} that the unit of $R$ is almost $k$-divisible.
\end{proof}

%------------------------------------------------------------------------------------------
Combining the above result with \autoref{prp:ZModTFAE}, we obtain the $\CatCu$-semigroup{} version of Winter's result that strongly self-absorbing \ca{s} are $\mathcal{Z}$-stable, \cite[Theorem~3.1]{Win11ssaZstable}. An actual alternative proof will be obtained as a consequence.

%==========================================================================================
\begin{cor}
\label{prp:simpleSemirgZmult}
Let $R$ be a simple, nonelementary $\Cu$-semiring satisfying axiom~\axiomO{5}.
Then $R$ has $Z$-multiplication and is therefore `$Z$-stable' in the sense that $R\cong Z\otimes_{\CatCu}R$.
\end{cor}

%==========================================================================================
\begin{prp}[{Winter, \cite[Theorem~3.1]{Win11ssaZstable}}]
\label{prp:ssaZstable}
Let $D$ be a unital, separable, strongly self-ab\-sorb\-ing \ca.
Then $D$ is $\mathcal{Z}$-stable, that is, $D\cong\mathcal{Z}\otimes D$.
\end{prp}
\begin{proof}
The famous $\mathcal{O}_\infty$-absorption theorem states that every unital, separable, nuclear, purely infinite, simple \ca{} $A$ satisfies $\mathcal{O}_\infty\otimes A\cong A$;
see \cite[Theorem~3.15]{KirPhi00Crelle1}.
Thus, if $D$ is purely infinite, then it is $\mathcal{O}_\infty$-stable and therefore also $\mathcal{Z}$-stable.

Assume now that $D$ is stably finite.
Let $R$ be the Cuntz semigroup of $D$.
By \autoref{prp:semirgFromSSA}, $R$ is a simple $\Cu$-semiring satisfying \axiomO{5}.
Since $R$ is also nonelementary, we obtain from \autoref{thm:simpleSemirg} that $R$ is almost unperforated and almost divisible.

This implies that $D$ has strict comparison of positive elements, and that the class of the unit $1_D$ in $R$ is almost $k$-divisible, for any $k\in\N_+$.
Note also that the state space of $D$ is nonempty (as $D$ has a unique tracial state).
With these assumptions we can apply \cite[Theorem~3.6]{DadTom10Ranks} to deduce that there exists a unital \starHom{} from the dimension drop algebra $Z_{k,k+1}$ to $D$.
Since the Jiang-Su algebra $\mathcal{Z}$ is an inductive limit of dimension drop algebras $Z_{k,k+1}$, it follows from \cite[Proposition~2.2]{TomWin08ASH} that $D$ is $\mathcal{Z}$-stable, as desired.
\end{proof}

%==========================================================================================
\begin{cor}
\label{prp:simpleSemirgCompacts}
Let $R$ be a simple $\Cu$-semiring satisfying \axiomO{5}.
Then:
\beginEnumStatements
\item
If $R$ is nonelementary, then $R$ is stably finite.
\item
Either the unit $1_R$ is compact, or $R$ contains no nonzero compact elements.
\end{enumerate}
\end{cor}
\begin{proof}
To show \enumStatement{1}, let $R$ be a simple, nonelementary $\Cu$-semiring satisfying \axiomO{5}.
By \autoref{thm:simpleSemirg},  $R$ is almost divisible.
We now proceed similarly to \autoref{prp:tensprodTwoSimple}.

Assume that $R$ is not stably finite.
As shown in \autoref{prp:simpleSF}, it follows that $\infty$ is a compact element.
Since $R$ is nonelementary and $1_R$ is nonzero, we can find $k\in\N$ such that $k 1_R=\infty$.
Given a nonzero $a\in R$, let us show $a=\infty$.
Since $R$ is almost divisible, we can choose $t\in R$ with $kt\leq a\leq (k+1)t$.
This implies that $t$ is nonzero.
Then
\[
\infty=\infty\cdot t = k 1_R \cdot t\leq a\leq\infty.
\]
Thus, we have shown $a=\infty$ for every nonzero $a\in R$.
This implies $R\cong\{0,\infty\}$, which is a contradiction since $R$ was assumed to be nonelementary.

Let us show \enumStatement{2}.
The statement is clearly true if $R\cong\{0\}$ or if $1_R$ is compact.
Thus we may assume from now on $R\ncong\{0\}$ and that $1_R$ is not compact.

We claim that $R$ is nonelementary, and hence stably finite by statement \enumStatement{1}. Indeed, assume the opposite.
Then, by \autoref{prp:simpleSemirgElmtry}, $1_R$ is a minimal, nonzero element.
This implies that $1_R$ is compact, a contradiction.

Hence $R$ is nonelementary, and therefore stably finite by statement \enumStatement{1}.
Then, by \autoref{prp:softNoncpctSimple}, every nonzero element of $R$ is either soft or compact.
Thus, the unit element $1_R$ is soft.
It follows from \autoref{prp:simpleSemirgZmult} that $R$ has $Z$-mul\-ti\-pli\-ca\-tion.
By \autoref{prp:softZmult}, an element $a$ in $R$ is soft if and only if $a=1'_Za$, where $1_Z'$ is the `soft' unit of $Z$.
We deduce $1_R=1'_Z1_R$.
Given a nonzero element $a$ in $R$, we obtain
\[
a = 1_R a = 1'_Z 1_R a = 1'_Za.
\]
Using \autoref{prp:softNoncpctSimple} again, it follows that $a$ is noncompact, as desired.
\end{proof}

%------------------------------------------------------------------------------------------
Next, we study the multiplicativity of normalized functionals on $\Cu$-semirings.
The results are inspired by \cite[Corollary~3]{Han13SimpleDimGps}.
Given a $\CatCu$-semiring $R$, recall that $F_1(R)$ denotes the functionals of $R$ that are normalized at $1$, the unit element.

The main application is \autoref{prp:multiplicativeUniqueFctl}, where we consider $\CatCu$-semirings with a unique normalized functional.
As we will see in \autoref{sec:classificationSolid}, in particular \autoref{prp:propertiesSolid}, this class includes all (stably finite) solid $\Cu$-semirings.
It also includes the Cuntz semigroups of stably finite, strongly self-absorbing \ca{s};
see \autoref{prp:semirgFromSSA}.

%==========================================================================================
\begin{lma}
\label{prp:FctlExtrMult}
Let $R$ be simple, nonelementary $\Cu$-semiring satisfying \axiomO{5}.
Then $\lambda\in F_1(R)$ is multiplicative whenever it is an extreme point of $F_1(R)$.
\end{lma}
\begin{proof}
To reach a contradiction, assume $\lambda(ab)\neq\lambda(a)\lambda(b)$ for some $a,b\in R$.
By Corollaries~\ref{prp:simpleSemirgCompacts} and~\ref{prp:simpleSemirgZmult}, we know that $R$ is stably finite and that $R$ has $Z$-multiplication.
By \autoref{prp:softZmult}, we have $\lambda(x)=\lambda(1'x)$ for every $x\in R$.
Thus, we may assume that $a$ and $b$ are soft elements, by replacing $a$ with $1'a$ and by replacing $b$ with $1'b$, if necessary.

Since $a$ is the supremum of a rapidly increasing sequence and since functionals preserve suprema of increasing sequences, we may also assume $a\ll\infty$.
Choose $n\in\N$ such that $a\ll n1'$.

Since $R$ satisfies \axiomO{5}, we may apply \cite[Theorem~3.2.1]{Rob13Cone} to deduce $R_\soft\cong L(F(R))$.
By the definition of $L(F(R))$, we can find a sequence $(x_k)_k$ in $L(F(R))$ such that $x_k\lhd x_{k+1}$ for each $k\in\N$ and such that $a=\sup_k x_k$.
Since
\[
\sup_k\lambda(x_kb)=\lambda(ab)
\neq\lambda(a)\lambda(b) = \sup_k\lambda(x_k)\lambda(b),
\]
we can choose $k\in\N$ with $\lambda(x_kb)\neq\lambda(x_k)\lambda(b)$.
Set $x:=x_k$.

We have $x\lhd x_{k+1}\ll n1'$.
By \cite[Lemma~3.3.2]{Rob13Cone}, we can choose $y\in L(F(R))\cong R_\soft$ satisfying $x+y=n1'$.
Without loss of generality, we may assume $x,y\neq 0$.
Then we can consider the maps $\lambda_i\colon R\to[0,\infty]$ given by
\[
\lambda_0(s) =\lambda(x)^{-1}\lambda(xs),\quad
\lambda_1(s) =\lambda(y)^{-1}\lambda(ys),
\]
for $s\in R$.
It is easy to check that $\lambda_0$ and $\lambda_1$ are functionals on $R$ and that $\lambda_0(b)\neq\lambda_1(b)$.
Since
\[
\lambda=\tfrac{\lambda(x)}{n}\lambda_0+\tfrac{\lambda(y)}{n}\lambda_1,
\]
we have shown that $\lambda$ is not an extreme point of $F_1(R)$, as desired.
\end{proof}

%==========================================================================================
\begin{prp}
\label{prp:multiplicativeUniqueFctl}
Let $R$ be a simple, nonelementary $\Cu$-semiring satisfying \axiomO{5}.
Assume that $R$ has a unique functional $\lambda$ that is normalized at $1$.
Then $\lambda$ is multiplicative.
\end{prp}
\begin{proof}
This follows directly from \autoref{prp:FctlExtrMult}, since $F_1(R)=\{\lambda\}$ and hence $\lambda$ is an extreme point of $F_1(R)$.
\end{proof}

%==========================================================================================
\begin{cor}
\label{prp:uniquenessR}
Let $R$ be a simple $\Cu$-semiring satisfying \axiomO{5}.
Assume that $R$ has a unique normalized functional.
Then $R\cong [0,\infty]$, as $\CatCu$-semirings, if and only if $1_R$ is not compact.
\end{cor}
\begin{proof}
The unit of $[0,\infty]$ is clearly not compact.
Conversely, assume that $R$ satisfies the conditions of the statement and that $1_R$ is not compact.
Then $R_c=\{0\}$, by \autoref{prp:simpleSemirgCompacts}.
Using \autoref{prp:simpleSemirgElmtry} and \autoref{thm:simpleSemirg}, we deduce that $R$ is nonelementary, almost unperforated and almost divisible.
It follows from \autoref{thm:isothm} that $R\cong[0,\infty]$, as $\CatCu$-semigroups.
By rescaling, we can find an isomorphism $\varphi\colon R\to[0,\infty]$ with $\varphi(1_R)=1$.
Then $\varphi$ is a normalized functional and therefore automatically multiplicative by \autoref{prp:multiplicativeUniqueFctl}.
\end{proof}

%------------------------------------------------------------------------------------------
The requirement in the following result that the function $\widehat{1}\in\Lsc(F(R))$ be continuous is not very restrictive.
It is automatically satisfied if $1$ is a compact element or if $R$ has only finitely many extremal functionals.

%==========================================================================================
\begin{prp}
\label{prp:extremeFunctionals}
Let $R$ be simple, nonelementary $\Cu$-semiring satisfying \axiomO{5} and \axiomO{6}.
Assume that $\widehat{1}$ is continuous.
Then:
\beginEnumStatements
\item
A functional $\lambda\in F_1(R)$ is multiplicative if and only if it is an extreme point of $F_1(R)$.
\item
The space $F_1(R)$ is a Bauer simplex, that is, a Choquet simplex with closed extreme boundary.
\item
Every functional $\lambda\in F_1(R)$ satisfies $\lambda(ab)^2\leq\lambda(a^2)\lambda(b^2)$ for all $a,b\in R$.
In particular, we have $\lambda(a)^2\leq\lambda(a^2)$ for every $a\in R$.
\end{enumerate}
\end{prp}
\begin{proof}
The assumption that $\widehat{1}$ is continuous implies that $F_1(R)$ is a closed subset of $F(R)$.
(In fact, this is equivalent to $\widehat{1}$ being continuous.)
Since $F(R)$ is compact, it follows that $F_1(R)$ is a compact, convex set.
We denote the subset of its extreme points by $\partial F_1(R)$.

We note that $F_1(R)$ is a Choquet simplex.
This follows for instance from \cite[Proposition~3.2.3, Theorem~4.1.2]{Rob13Cone}.
Note first that $F_1(R)$ is a basis for
\[
F_0(R)=\left\{ \lambda\in F(R) : \lambda(a)<\infty \text{ for all } a\ll\infty \right\}.
\]
Let $V(F_0(R))$ denote the vector space of linear, real-valued, continuous functions on $F_0(R)$.
Then $F_0(R)$ is a lattice-ordered, strict, convex cone in the vector space of linear functionals on $V(F_0(R))$.

Let us show \enumStatement{3}.
Since $F_1(R)$ is a Choquet simplex, there is a measure $\mu$ on its extreme boundary $\partial F_1(R)$ such that
\begin{align}
\label{prp:extremeFunctionals:eq1}
\lambda(x)=\int_{\partial F_1(R)}\varphi(x)d\mu(\varphi),
\end{align}
for every element $x$ in $R$ for which $\widehat{x}$ is continuous.

We claim that \eqref{prp:extremeFunctionals:eq1} holds for every $x\in R$.
This is clear if $x$ is compact, since then $\widehat{x}$ is continuous.
If $x$ is soft, then it follows from \cite[Proposition~3.1.6]{Rob13Cone} that there is an increasing sequence $(x_k)_k$ in $R$ such that $x=\sup_k x_k$ and $\widehat{x_k}$ is continuous for each $k\in\N$.
Using the theorem of monotone convergence at the third step, we obtain
\begin{align*}
\lambda(x)
= \sup_k\lambda(x_k)
&= \sup_k \int_{\partial F_1(R)}\varphi(x_k)d\mu(\varphi) \\
&= \int_{\partial F_1(R)}\sup_k\varphi(x_k)d\mu(\varphi)
\,=\, \int_{\partial F_1(R)}\varphi(x)d\mu(\varphi),
\end{align*}
which verifies \eqref{prp:extremeFunctionals:eq1}.

Now, given $a$ and $b$ in $R$, we use the Cauchy-Schwarz inequality at the second step to deduce
\begin{align*}
\lambda(ab)^2
&=\left(\int_{\partial F_1(R)}\varphi(a)\varphi(b)d\mu(\varphi)\right)^2 \\
&\leq \int_{\partial F_1(R)}\varphi(a)^2d\mu(\varphi) \int_{\partial C}\varphi(b)^2d\mu(\varphi)
\,=\,\lambda(a^2)\lambda(b^2).
\end{align*}

Next, let us show \enumStatement{1}.
By \autoref{prp:FctlExtrMult}, every functional in $\partial F_1(R)$ is multiplicative.
To show the converse, assume $\lambda=\frac{1}{2}(\lambda_0+\lambda_1)$ for two different functionals $\lambda_0$ and $\lambda_1$.
Choose $a\in R$ satisfying $\lambda_0(a)\neq\lambda_1(a)$.
By switching the role of $\lambda_0$ and $\lambda_1$, if necessary, we may find $\varepsilon>0$ such that
\[
\lambda_0(a)=\lambda(a)-\varepsilon,\quad
\lambda_1(a)=\lambda(a)+\varepsilon.
\]
Then, using \enumStatement{3} at the second step, it follows
\[
\lambda(a^2)=\frac{1}{2}(\lambda_0(a^2)+\lambda_1(a^2))
\geq \frac{1}{2}(\lambda_0(a)^2+\lambda_1(a)^2)
=\lambda(a)^2+\varepsilon^2,
\]
which shows $\lambda(a^2)\neq\lambda(a)^2$ and thus $\lambda$ is not multiplicative.

Finally, it follows easily from \enumStatement{1} that $\partial F_1(R)$ is a closed subset of $F_1(R)$.
This verifies \enumStatement{2}.
\end{proof}

%\vspace{5pt}
%------------------------------------------------------------------------------------------
%==========================================================================================
\section{Algebraic \texorpdfstring{$\CatCu$}{Cu}-semirings}
\label{sec:algebraicSemirg}

%------------------------------------------------------------------------------------------
Recall from \autoref{dfn:algebraicSemigp} that a $\CatCu$-semigroup $S$ is \emph{algebraic} if every element in $S$ is the supremum of an increasing sequence of compact elements.
\index{terms}{Cu-semiring@$\CatCu$-semiring!algebraic}
\index{terms}{algebraic Cu-semiring@algebraic $\CatCu$-semiring}

%==========================================================================================
\begin{pgr}
\label{pgr:CuCompletionSrg}
Let $K$ be a cancellative, conical semiring.
We equip $K$ with the algebraic order.
Then $K$ is a \pom{} and we may apply the construction of \autoref{sec:algebraicSemigp} to the underlying additive monoid of $K$.
We denote by $S$ the resulting $\CatCu$-completion of $K$.
Then $S$ is an algebraic $\CatCu$-semigroup whose compact elements can be identified with $K$.
We therefore think of $K$ as a submonoid of $S$.

The multiplication on $K$ can be extended to $S$ as follows:
First, we define the product of an element in $K$ with an element in $S$.
Let $a\in K$ and $b\in S$.
Choose an increasing sequence $(b_k)_k$ in $K$ with $b=\sup_k b_k$.
Then the sequence $(ab_k)_k$ is increasing (in $K$) and we may set $ab := \sup_k (ab_k)$.
It is straightforward to check that this is independent of the choice of the sequence $(b_k)_k$.
Moreover, if $a'$ and $a$ in $K$ satisfy $a'\leq a$, then $a'b\leq ab$ for every $b\in S$.

Now we can define the product of two arbitrary elements $a$ and $b$ in $S$ as follows.
Choose an increasing sequence $(a_k)_k$ in $K$ with $a=\sup_k a_k$.
For each $k$, the product $a_kb$ is already defined.
Moreover, the resulting sequence $(a_kb)_k$ is increasing and we may set $ab := \sup_k (a_kb)$.
It is easy to check that this defines a $\CatCu$-product on $S$.
We denote the resulting $\CatCu$-semiring by $\Cu(K)$.

By \autoref{prp:propertiesAlgebraic}, $\Cu(K)$ is a weakly cancellative, algebraic $\CatCu$-semiring satisfying \axiomO{5}.
Moreover, there is a natural isomorphism between $K$ and the semiring of compact elements in $\Cu(K)$.

Given cancellative, conical semirings $K$ and $L$, it is clear that every semiring homomorphism $K\to L$ induces a multiplicative $\CatCu$-morphism from $\Cu(K)$ to $\Cu(L)$.
This defines a functor from the category $\CatCon\CatSrg_\txtCanc$ of cancellative, conical semirings to the category of algebraic $\CatCu$-semirings.

Conversely, for every $\CatCu$-semiring $S$ with compact unit, the compact elements in $S$ form a subsemiring.
The assignment $S\mapsto S_c$ can be extended to a functor from the category of algebraic $\CatCu$-semirings to the category of conical semirings.

Let $S$ be a weakly cancellative, algebraic $\Cu$-semiring satisfying \axiomO{5}.
By \autoref{prp:propertiesAlgebraic}, the subset $S_c$ is a cancellative, conical, algebraically ordered semiring.
Moreover, the $\CatCu$-semiring $S$ is naturally isomorphic to $\Cu(S_c)$.
\end{pgr}

%==========================================================================================
\begin{prp}
\label{prp:equivalenceSrgAlgCu}
The functors from Paragraphs \ref{pgr:CuCompletionSrg} and \ref{pgr:equivalenceSrgPoRg} establish equivalences between the following categories:
\beginEnumStatements
\item
Directed, \por{s}, together with ring homomorphisms.
\item
Cancellative, conical semirings, together with semiring homomor\-phisms.
\item
Weakly cancellative, algebraic $\Cu$-semirings satisfying \axiomO{5}, together with multiplicative $\Cu$-morphisms.
\end{enumerate}
\end{prp}

%------------------------------------------------------------------------------------------
The following notion of weak divisibility was introduced in \cite[Definition~2.2]{OrtPerRor11CoronaRefinement} (see also \cite{PerRor04AFembedding} and \cite{AraGooPerSil10NonsimplePurelyInfinite}).
This property has also been called \emph{quasi-divisible} in \cite[Definition~3.2]{Weh98EmbeddingRefinement}.

%==========================================================================================
\begin{dfn}
\label{dfn:wkDivMon}
\index{terms}{monoid!weakly divisible}
\index{terms}{weakly divisible monoid}
A monoid $M$ is \emph{weakly divisible} if for every $s\in M$, there are $a, b\in M$ such that $s=2a+3b$.
\end{dfn}

%==========================================================================================
\begin{dfn}
\label{dfn:nonElemSrg}
\index{terms}{semiring!nonelementary}
\index{terms}{nonelementary semiring}
A conical semiring $R$ is \emph{nonelementary} if there exist nonzero elements $s,t\in R$ such that $1=s+t$.
\end{dfn}

%==========================================================================================
\begin{rmks}
\label{rmk:wkDivNonElemSrg}
(1)
Let $S$ be a conical semiring.
Then the underlying additive monoid of $S$ is weakly divisible if and only if there exist $s,t\in S$ with $1=2s+3t$.

(2)
Let $S$ be a conical semiring.
If we equip $S$ with its algebraic pre-order, then $S$ is nonelementary if and only if $S\neq\{0\}$ and the unit is not a minimal nonzero element.
\end{rmks}

%------------------------------------------------------------------------------------------
It is easily seen that every (nonzero) weakly divisible, conical semiring is nonelementary.
In the next result, we show that the converse holds for simple, conical semirings.
Part of the argument below is inspired by \cite{DadRor09ssa}.

%==========================================================================================
\begin{prp}
\label{prp:simpleSrgDiv}
Let $R$ be a nonelementary, conical semiring that is simple for its algebraic pre-order.
Then $R$ is weakly divisible.
\end{prp}
\begin{proof}
In this proof, we will write $\leq$ for the algebraic pre-order on $R$.
Then simplicity of $R$ means that for every $x,y\in R$ with $y\neq 0$ there exists $n\in\N$ such that $x\leq ny$.

We first observe that $R$ contains no zero divisors.
Indeed, assume nonzero elements $x,y\in R$ satisfy $xy=0$.
Since $R$ is simple, we can choose $x',y'\in R$ and $k,k\in\N$ such that $1+x'=kx$ and $1+y'=ly$.
This implies $1+x'+y'+x'y'=0$.
Since $R$ is conical, it follows that the unit of $R$ is zero, whence $R=\{0\}$, which contradicts that $R$ is nonelementary.

To prove weak divisibility of $R$, it is enough to show that there are exist $a,b\in R$ such that $1=2a+3b$.
Since $R$ is nonelementary we can find nonzero $s,t\in R$ satisfying $1=s+t$.
Then $1=s^2+t^2+2st$.
Set
\[
f:=s^2+t^2,\quad
e:=st.
\]
Note that both $e$ and $f$ are nonzero.
Then $1=f+2e$, which implies $f=f^2+2ef$.
It follows $1=f^2+2e(1+f)$.
Inductively, we obtain
\[
1=f^m+2e(1+f+\cdots+f^{m-1}),
\]
for each $m\in\N$.
By simplicity of $R$, we can choose $m\in\N_+$ with $f\leq me$.
Then
\[
f^m\leq mef^{m-1}\leq e(1+f+\cdots+f^{m-1}).
\]
Set $b:=f^m$ and $a':=e(1+f+\cdots+f^{m-1})$.
Then $1=b+2a'$.
Since $b\leq a'$, we can find $a\in R$ such that $b+a=a'$.
Then
\[
1=b+2a'=b+2b+2a=2a+3b,
\]
as desired.
\end{proof}

%==========================================================================================
\begin{lma}
\label{prp:wkDivNearUnperf}
Let $R$ be a weakly divisible semiring.
Let $M$ be a \pom\ that is a semimodule over $R$.
Then $M$ is nearly unperforated.

In particular, if $R$ has a compatible positive order, then it is nearly unperforated itself.
\end{lma}
\begin{proof}
By \autoref{prp:nearUnperfTFAE}, it is enough to show that $2a\leq 2b$ and $3a\leq 3b$ imply $a\leq b$, for any $a,b\in M$.
Let such $a$ and $b$ be given.
By weak divisibility of $R$, there are elements $s,t\in R$ such that $1=2s+3t$.
Then
\[
a = (2s+3t)a
= s(2a) + t(3a)
\leq s(2b) +t(3b)
= (2s+3t)b
= b,
\]
as desired.
\end{proof}

%==========================================================================================
\begin{prp}
\label{prp:cancSimpleSemirg}
Let $R$ be a nonelementary, conical semiring that is simple and stably finite for its algebraic (pre)order.
Then $R$ is cancellative.
\end{prp}
\begin{proof}
The assumptions imply that the algebraic pre-order of $R$ is antisymmetric;
see \autoref{rmk:preminSimpleSFPrePom}.
Thus, the underlying monoid of $R$ is a simple, stably finite, partially ordered monoid (with its algebraic order).
By \autoref{prp:simpleSrgDiv} and \autoref{prp:wkDivNearUnperf}, the semiring $R$ is nearly unperforated.
Then we may apply \autoref{prp:nearUnperfSimplePom} to deduce that $R$ is cancellative.
\end{proof}

%==========================================================================================
\begin{lma}
\label{lma:wkDivNonElemSrg}
Let $R$ be a simple algebraic $\CatCu$-semiring satisfying \axiomO{5}.
Then the following conditions are equivalent:
\beginEnumStatements
\item $R$ is nonelementary as a $\CatCu$-semiring.
\item $R_c$ is stably finite, $1_R$ is compact and $R_c$ is a nonelementary semiring.
\item $R_c$ is stably finite, $1_R$ is compact and not minimal.
\end{enumerate}
\end{lma}
\begin{proof}
It is clear that (2) and (3) are equivalent conditions.

Assume condition (1).
Then $R_c\neq \{0\}$ as $R$ is algebraic, and thus $1_R$ is compact by \autoref{prp:simpleSemirgCompacts}. Another use of \autoref{prp:simpleSemirgCompacts} shows that $R$ (hence also $R_c$) is stably finite. Since $R$ is nonelementary and algebraic we can choose a compact element $a\in R$ such that $a<1_R$, and by \axiomO{5} we obtain a nonzero (compact) element $b$ with $1_R=a+b$.
This shows that $R_c$ is a nonelementary (conical) semiring.

Assume now (2), and let us show (1).
We have $1_R=a+b$ for nonzero compact elements $a$ and $b$.
If $c$ is a minimal nonzero element in $R$, then $c\in R_c$ (since $R$ is algebraic).
Now $c=ca+cb$, with both $ca$ and $cb$ nonzero as $R$ does not have zero divisors.
Note that $ca<c$ because $R_c$ is stably finite.
This contradicts the minimality of $c$, and thus $R$ is nonelementary as a $\CatCu$-semiring.
\end{proof}

%==========================================================================================
\begin{cor}
\label{prp:cancSimpleAlgebraicSemirg}
Let $R$ be a simple, nonelementary, algebraic $\Cu$-semiring satisfying~\axiomO{5}.
Then $R$ is weakly divisible and weakly cancellative.
\end{cor}
\begin{proof}
%Let $K$ denote the subsemiring of compact elements in $R$.
First notice that the subsemiring $R_c$ of compact elements in $R$ is conical.
By \autoref{lma:wkDivNonElemSrg}, $R_c$ is a nonelementary semiring (with unit $1_R$) whose underlying additive monoid is stably finite.

Since $R$ satisfies \axiomO{5}, we may apply \autoref{prp:propertiesAlgebraic} to deduce that the partial order on $R$ induces the algebraic order on $R_c$.

Using that $R$ is simple and algebraic, it follows that $R_c$ is simple.
Then $R_c$ is weakly divisible, by \autoref{prp:simpleSrgDiv}. Now write $1_R=2a+3b$ for $a$ and $b$ in $R_c$, and then $c=2ac+3bc$ for any $c\in R$. This shows that $R$ is weakly divisible as well.
Moreover, $R_c$ is cancellative by \autoref{prp:cancSimpleSemirg}.
Using \autoref{prp:propertiesAlgebraic} again, we obtain that $R$ is weakly cancellative.
\end{proof}

%------------------------------------------------------------------------------------------
The following proposition is a $\CatCu$-semigroup version of results for \ca{s} that have appeared in \cite[Proposition~5.8]{TomWin07ssa} and \cite[Theorem~2.5]{DadRor09ssa}.

%==========================================================================================
\begin{prp}
\label{prp:imagelambda}
Let $R$ be a simple, nonelementary $\Cu$-semiring satisfying~\axiomO{5} and with a unique functional $\lambda$ that is normalized at $1_R$.
Then the following conditions are equivalent:
\beginEnumStatements
\item
There exists a compact element $p\in R$ with $0<p<1_R$.
\item
There exists a compact element $p\in R$ with $\lambda(p)\notin\N$.
\item
The set $\lambda(R_c)$ is dense in $\R_+$.
\item
The $\CatCu$-semiring $R$ is weakly divisible.
\item
The $\CatCu$-semiring $R$ is algebraic.
\end{enumerate}
\end{prp}
\begin{proof}
By \autoref{thm:simpleSemirg}, the $\Cu$-semiring $R$ is almost unperforated.
Then, by \autoref{prp:almUnp}, we have $a<b$ if and only if $\lambda(a)<\lambda(b)$, for any $a,b\in R$.

The implications `\enumStatement{4} $\Rightarrow$ \enumStatement{3} $\Rightarrow$ \enumStatement{2}' are clear.
By \autoref{prp:cancSimpleAlgebraicSemirg}, \enumStatement{5} implies \enumStatement{4}.
To show that \enumStatement{2} implies \enumStatement{1}, choose a compact element $a\in R$ satisfying $\lambda(a)\notin\N$.
Let $n\in\N$ such that $n<\lambda(a)<n+1$.
Then
\[
\lambda(n 1_R) = n < \lambda(a) < n+1 = \lambda((n+1) 1_R).
\]
As explained at the beginning of the proof, it follows $n 1_R<a<(n+1) 1_R$.
Since $R$ satisfies \axiomO{5} we can choose a compact element $p$ in $S$ such that $a+p=(n+1)1_R$.
Then $0<\lambda(p)<1$, which implies that $0<p<1_R$, as desired.

Next, we show that \enumStatement{3} implies \enumStatement{5}.
Let $a\in R$ be an element.
We need to show that $a$ is the supremum of an increasing sequence of compact elements.
This is clear if $a$ is compact itself.

Thus, we may assume that $a$ is noncompact.
By \autoref{prp:softNoncpctSimple}, we get that $a$ is soft.
By assumption, there is a sequence $(b_k)_k$ of compact elements such that $\lambda(b_k)_k$ is strictly increasing with $\lambda(a)=\sup_k\lambda(b_k)$.
As observed at the beginning of the proof, it follows that the sequence $(b_k)_k$ is increasing.

Set $b:=\sup_k b_k$.
Then $\lambda(b)=\lambda(a)$.
Since $R$ is stably finite, the element $b$ is noncompact and therefore soft.
Then \autoref{prp:soft_comparison} implies $a=b$.
Thus, $a$ is the supremum of the increasing sequence $(b_k)_k$ of compact elements, as desired.

Finally, let us show that \enumStatement{1} implies \enumStatement{3}.
By \autoref{prp:multiplicativeUniqueFctl}, the functional $\lambda$ is multiplicative.
Therefore $\lambda(R_c)$ is a subsemiring of $[0,\infty]$, which must be dense as it contains arbitrarily small elements.
\end{proof}

%==========================================================================================
\begin{cor}[{Dadarlat, R{\o}rdam, \cite[Theorem~2.5]{DadRor09ssa}}]
Let $D$ be a strongly self-absorbing \ca.
Then $D$ has real rank zero if and only if it contains a nontrivial projection.
\end{cor}

\vspace{5pt}
%------------------------------------------------------------------------------------------
%==========================================================================================
\section{Classification of solid \texorpdfstring{$\Cu$}{Cu}-semirings}
\label{sec:classificationSolid}

\index{terms}{Cu-semiring@$\CatCu$-semiring!solid}
\index{terms}{solid Cu-semiring@solid $\CatCu$-semiring}
%------------------------------------------------------------------------------------------
We now study the structure of general solid $\CatCu$-semirings.
The goal is the classification result in \autoref{thm:solidSemirgClassification}.

%==========================================================================================
\begin{thm}
\label{prp:propertiesSolid}
Let $R$ be a solid $\Cu$-semiring.
Then $R$ is simple.
Moreover, $R$ has at most one functional that is normalized at the unit element $1$.
%This functional is moreover multiplicative.
\end{thm}
\begin{proof}
We first show that $R$ is simple.
Let $I$ be an ideal in $R$. Define $\tau_I\colon R\times R\to\{0,\infty\}$ by
\[
\tau_I(a,b) =
\begin{cases}
0, & \text{if $a\in I$ or $b=0$} \\
\infty, & \text{otherwise,}
\end{cases}
\]
for all $a,b\in R$.
This is easily checked to be a generalized $\CatCu$-bimorphism.
Since $R$ is solid, the map~$\tau_I$ factors through multiplication in $R$.
This means that there exists a generalized $\CatCu$-morphism $\widetilde{\tau}_I\colon R\to \{0,\infty\}$ such that $\widetilde{\tau}_I(ab)=\tau_I(a,b)$ for all $a,b\in R$.

Consider the case that $I$ is the ideal generated by $1$.
Then
\[
0=\tau_I(1,a)=\widetilde{\tau}_I(a)=\tau_I(a,1),
\]
for all $a\in R$.
This implies that $I=R$ and thus $1$ is a full element.

Now let $J$ be an ideal in $R$ satisfying $J\neq R$.
Since $1$ is full, this implies $1\notin J$.
Let $a\in J$.
We deduce
\[
0=\tau_J(a,1)=\widetilde{\tau}_J(a)=\tau_J(1,a).
\]
This implies $a=0$, hence $J=\{0\}$.
Thus, we have shown that $R$ is simple.

To show that $R$ has at most one normalized functional, let $\lambda_1$ and $\lambda_2$ be functionals on $R$ satisfying $\lambda_1(1)=\lambda_2(1)=1$.
Consider the map
\[
\tau\colon R\times R\to [0,\infty],\quad
\tau(a,b)=\lambda_1(a)\lambda_2(b),
\txtFA a,b\in R.
\]
It is clear that $\tau$ is a generalized $\CatCu$-bimorphism.
Since $R$ is solid, there exists a generalized $\CatCu$-morphism $\widetilde{\tau}\colon R\to [0,\infty]$ such that $\widetilde{\tau}(ab)=\lambda_1(a)\lambda_2(b)$ for all $a,b\in R$.
Then we obtain
\[
\lambda_1(a)=\tau(a,1)=\widetilde{\tau}(a1)=\widetilde{\tau}(1a)=\tau(1,a)=\lambda_2(a),
\]
for all $a\in R$.
This shows $\lambda_1=\lambda_2$, as desired.
\end{proof}

%------------------------------------------------------------------------------------------
The aim of the following two results is to show that a nonelementary, solid $\CatCu$-semiring satisfying \axiomO{5} has a unique functional which is moreover multiplicative.

%==========================================================================================
\begin{lma}
\label{prp:existenceFctlZMult}
Let $S$ be an almost unperforated, almost divisible $\CatCu$-semigroup, and let $a\in S$.
Then there exists $\lambda\in F(S)$ with $\lambda(a)=1$ if and only if $a\neq 2a$.
\end{lma}
\begin{proof}
Clearly, if $\lambda(a)=1$ for some $\lambda\in F(S)$, then $a\neq 2a$.
For the converse, assume there exists no $\lambda\in F(S)$ such that $\lambda(a)=1$.
This implies $\lambda(a)\in\{0,\infty\}$ for all $\lambda\in F(S)$.
By \autoref{prp:ZModTFAE}, $S$ has $Z$-multiplication.
Consider the elements $1'a$ and $3'a$.
By \autoref{prp:softZmult}, both $1'a$ and $3'a$ are soft, and $\lambda(1'a)=\lambda(a)$ for all $\lambda\in F(S)$.
Then $\lambda(1'a)=\lambda(3'a)$ for all $\lambda\in F(S)$.
Since $S$ is almost unperforated, it follows from \autoref{prp:soft_comparison} that $1'a=3'a$.
We deduce
\[
2a \leq 3'a = 1'a \leq a,
\]
which shows $a=2a$, as desired.
\end{proof}

%==========================================================================================
\begin{cor}
\label{prp:solidZmult}
Let $R$ be a nonelementary, solid $\Cu$-semiring satisfying axiom~\axiomO{5}.
Then $R$ is simple, almost unperforated, almost divisible (hence $R\cong Z\otimes_{\CatCu}R$) and stably finite.
Moreover, there is a unique functional $\lambda\in F(R)$ satisfying $\lambda(1)=1$.
This functional is multiplicative.
\end{cor}
\begin{proof}
By \autoref{prp:propertiesSolid}, $R$ is simple.
It follows from \autoref{prp:simpleSemirgZmult} that $R$ satisfies $R\cong Z\otimes_{\CatCu}R$.
Also by \autoref{prp:propertiesSolid}, $R$ has at most one functional that is normalized at the unit element $1$.
We have $1\neq 2$, since otherwise $R$ is elementary;
see \autoref{prp:simpleSemirgElmtry}.
Since $R$ is almost unperforated and almost divisible, we may apply \autoref{prp:existenceFctlZMult} to deduce that there exists $\lambda\in F(R)$ satisfying $\lambda(1)=1$.
By \autoref{prp:multiplicativeUniqueFctl}, $\lambda$ is multiplicative.
\end{proof}

%==========================================================================================
\begin{pgr}
\label{pgr:solidRings}
Let us recall the classification of solid rings from \cite[Proposition~3.5]{BouKan72Core} and \cite[Proposition~1.10]{BowSch77Rings}.
Every unital subring of the rational numbers~$\Q$ is a (torsion-free) solid ring.
Conversely, every torsion-free, solid ring is isomorphic to a unital subring of~$\Q$.

Given a set of primes $P$, we let $\Z\left[P^{-1}\right]$ denote the subring of $\Q$ generated by $\Z$ and the numbers $\tfrac{1}{p}$ for every $p\in P$.
We associate to $P$ the supernatural number $n_P = \prod_{p\in P}p^\infty$.
Then, using the notation from \autoref{pgr:Rq},
\[
\Z\left[P^{-1}\right]
=\Z\left[\left\{ \tfrac{1}{p} : p \in P \right\}\right]
=\Z\left[\tfrac{1}{n_P}\right].
\]
Every unital subring of $\Q$ is of the form $\Z\left[P^{-1}\right]$ for some set of primes~$P$.

Given a ring $R$, we let $t(R)$ denote the torsion part of~$R$.
If $R$ is a solid ring, then $R/t(R)$ is a torsion-free, solid ring.
It is possible that $R/t(R)=\{0\}$, which happens precisely when the unit of $R$ is a torsion element.

Let us now assume that $R$ is a solid ring whose unit is not torsion.
Then $R/t(R)$ is a unital subring of $\Q$, and consequently there is a set of primes $P$ such that
\[
R/t(R)\cong\Z[P^{-1}].
\]
Furthermore, it is known that the order of every torsion element in $R$ is divisible in $R/t(R)$.
More precisely, it is shown in \cite[3.12]{BouKan72Core} that there is a subset $K\subset P$ and integers $e(p)$ for $p\in K$ such that
\[
t(R)\cong\bigoplus_{p\in K} \Z_{p^{e(p)}}.
\]
If $R$ is a solid ring whose unit is not torsion and $R/t(R)\cong \Z$, then $t(R)=\{0\}$ and hence $R\cong \Z$. Indeed, if $R$ has nonzero $p$-torsion elements for a prime $p$, then $p$ becomes invertible in $R/t(R)$, which is impossible.
\end{pgr}

%==========================================================================================
\begin{dfn}
\label{dfn:solidRgPositive}
Let $R$ be a solid ring whose unit is not torsion.
As explained in \autoref{pgr:solidRings}, there exists a canonical embedding of $R/t(R)$ into $\Q$.
Let $\lambda_0\colon R\to\Q$ be the ring homomorphism obtained by composing the quotient map $R\to R/t(R)$ with the embedding $R/t(R)\subset\Q$.
We define $R_+$ as the set
\[
R_+ := \left\{ r\in R : \lambda_0(r)>0 \right\} \cup \{0\}.
\]
\end{dfn}

%==========================================================================================
\begin{dfn}
\label{dfn:solidSrg}
\index{terms}{semiring!solid}
A semiring $R$ is \emph{solid} if for every $a\in R$ the equality
\[
a\otimes 1=1\otimes a
\]
holds in $R\otimes_{\CatSrg}R$.
\end{dfn}

%==========================================================================================
\begin{lma}
\label{prp:solidSrgRg}
(1)
Let $S$ be a solid semiring.
Then the Grothendieck completion $\Gr(S)$ is a solid ring.

(2)
Let $R$ be a solid ring whose unit is not a torsion element.
Then the subset $R_+$ from \autoref{dfn:solidRgPositive} is a unital, conical subsemiring of $R$.
Moreover, the semiring $R_+$ is cancellative, and solid in the sense of \autoref{dfn:solidSrg}, and the algebraic order of $R_+$ is almost unperforated.
\end{lma}
\begin{proof}
To show the first part of the statement, let $S$ be a solid semiring.
Let $R$ denote the Grothendieck completion of $S$, and let
\[
\delta\colon S\to R:=\Gr(S)
\]
denote the natural map.
To prove that $R$ is solid, let $a\in R$.
We need to show $a\otimes 1=1\otimes a$ in $R\otimes R$.

By properties of the Grothendieck completion, we can choose $x,y\in S$ such that $a=\delta(x)-\delta(y)$.
Using that $S$ is solid at the second step, we deduce
\[
a\otimes 1
=(\delta\otimes\delta)(x\otimes 1) - (\delta\otimes\delta)(y\otimes 1)
=(\delta\otimes\delta)(1\otimes x) - (\delta\otimes\delta)(1\otimes y)
=1\otimes a,
\]
as desired.

To show the second part of the statement, let $R$ be a solid ring with non-torsion unit.
It is straightforward to check that $R_+$ is a unital, conical subsemiring of $R$.
Let us show that $R_+$ is a solid semiring.
If $R\cong\Z$, then $R_+\cong\N$, which is obviously a solid semiring.
Therefore, we may assume $R\ncong\Z$.
Let us denote the inclusion of $R_+$ into $R$ by $\iota\colon R_+\to R$.
Let $\lambda_0\colon R\to\Q$ denote the canonical ring homomorphism introduced in \autoref{dfn:solidRgPositive}.
By abuse of notation, we denote the composition $\lambda_0\circ\iota\colon R_+\to\Q$ also by $\lambda_0$.

We endow $R_+$ with the algebraic order.
Then, for any $a,b\in R_+$ we have
\[
a<b\quad
\text{ if and only if }\quad
\lambda_0(a) < \lambda_0(b).
\]
Note that for every nonzero element $a$ in $R_+$, we have $\lambda_0(a)>0$.
It follows easily that $R_+$ is a simple semiring.
This implies that $R_+\otimes_{\CatSrg}R_+$ is a simple semiring, as well.
It is easy to see that $\lambda_0$ is a state when considering it as a map $\lambda_0\colon R_+\to\R$.
This induces a state $\lambda_0\otimes\lambda_0$ on $R_+\otimes_{\CatSrg}R_+$, such that $(\lambda_0\otimes\lambda_0)(x)>0$ for every nonzero element $x\in R_+\otimes_{\CatSrg}R_+$.
It follows that $R_+\otimes_{\CatSrg}R_+$ is conical and stably finite.

Since $R\ncong\Z$, we have $R/t(R)\cong\Z[P^{-1}]$ for some nonempty set of primes $P$.
Then it is easy to see that $R_+$ contains nonzero elements $c$ and $d$ such that $1=c+d$.
Note that the elements $1\otimes c$ and $1\otimes d$ in $R_+\otimes_{\CatSrg}R_+$ are nonzero.
Thus, the unit of $R_+\otimes_{\CatSrg}R_+$ is equal to $1\otimes c+1\otimes d$, the sum of two nonzero elements.
By \autoref{prp:cancSimpleSemirg}, the semiring $R_+\otimes_{\CatSrg}R_+$ is cancellative.
The following commutative diagram shows the (semi)rings and maps to be considered.
\[
\xymatrix@R=15pt@M+3pt{
R_+\otimes_{\CatSrg}R_+ \ar@{}[r]|{=}
& R_+\otimes_{\CatMon}R_+ \ar@{^{(}->}[r] \ar[d]^{\iota\otimes\iota}
& \Gr(R_+\otimes_{\CatMon}R_+) \ar[d]^{\cong} \\
R\otimes R \ar@{}[r]|{=}
& R\otimes_{\CatMon}R \ar@{}[r]|-{\cong}
& \Gr(R_+)\otimes_{\CatMon}\Gr(R_+) \\
}
\]
The tensor product of two (semi)rings is just the tensor product of the underlying monoids, equipped with a natural multiplication;
see \autoref{sec:poRg}.
We want to show that the map $\iota\otimes\iota$ is injective.
This does not follow directly from the injectivity of $\iota$, since in general the tensor product of two injective morphisms need not be injective again.
However, we have shown above that $R_+\otimes_{\CatSrg}R_+$ is cancellative.
Therefore, the map to the Grothendieck completion is injective, as indicated by the upper-right horizontal arrow in the diagram.

In general, if $M$ and $N$ are monoids, there is a natural isomorphism between $\Gr(M\otimes_{\CatMon}N)$, the Grothendieck completion of their tensor product, and $\Gr(M)\otimes_{\CatMon}\Gr(N)$, the tensor product of their respective Grothendieck com\-ple\-tions;
see \cite[Proposition~17]{Ful70Tensor}, and also \autoref{prp:GrTensorMon}.

It is clear that $R$ is canonically isomorphic to the Grothendieck completion of~$R_+$.
It follows from the commutativity of the above diagram that the map $\iota\otimes\iota$ is injective.

Now let $a\in R_+$ be given.
Using that $R$ is solid at the second step, we deduce
\[
(\iota\otimes\iota)(a\otimes 1)
= \iota(a)\otimes 1
= 1\otimes\iota(a)
= (\iota\otimes\iota)(1\otimes a),
\]
in $R\otimes R$.
Since the map $\iota\otimes\iota$ is injective, this implies
\[
a\otimes 1
=1\otimes a,
\]
in $R_+\otimes_{\CatSrg}R_+$, as desired.
It is left to the reader to check that the algebraic order of $R_+$ is almost unperforated.
\end{proof}

%==========================================================================================
\begin{rmk}
\label{rmk:solidRgOrdered}
Let $R$ be a solid ring whose unit is not a torsion element.
By \autoref{prp:solidSrgRg}, the subset $R_+$ of $R$ is a unital, conical, subsemiring.
It follows that $R$ has the structure of a \por{} with positive cone given by $R_+$;
see \autoref{pgr:equivalenceSrgPoRg}.
It is clear that $R$ is directed.
This means that every solid ring with non-torsion unit has a canonical structure as a directed, \por.
\end{rmk}

%==========================================================================================
\begin{lma}
\label{prp:solidElementsRational}
Let $R$ be a solid ring whose unit is not a torsion element, and let $a\in R_+$.
Then there exist $k,n\in\N$ such that $na=k1$.
\end{lma}
\begin{proof}
We may assume that $a$ is nonzero.
Let $\lambda_0\colon R\to\Q$ be the canonical ring homomorphism introduced in \autoref{dfn:solidRgPositive} such that
\[
R_+ = \left\{ r\in R : \lambda_0(r)>0 \right\} \cup \{0\}.
\]
Choose positive $k_0,n_0\in\N$ such that $\lambda_0(a)=\tfrac{k_0}{n_0}$.
Then
\[
\lambda_0(n_0 a - k_0 1)=0.
\]
Therefore, $n_0 a - k_0 1$ is a torsion element of $R$.
Let $m$ be its order.
Then
\[
(mn_0)a = (mk_0)1,
\]
which shows that $n:=mn_0$ and $k:=mk_0$ have the desired properties.
\end{proof}

%==========================================================================================
\begin{prp}
\label{prp:equivalenceSolidSrgRg}
There is a natural one-to-one correspondence between the following classes:
\beginEnumStatements
\item
Solid rings whose unit is not a torsion element.
\item
Solid, cancellative, conical semirings for which the algebraic order is almost unperforated.
\end{enumerate}
The correspondence is given by associating to a solid ring $R$ with non-torsion unit the solid semiring $R_+$ from \autoref{dfn:solidRgPositive}, and conversely by associating to a solid semiring $S$ its Grothendieck completion $\Gr(S)$.
\end{prp}
\begin{proof}
It is clear from \autoref{prp:solidSrgRg} that the assignments of the statement are well-defined.
Thus, it remains to show that the assignments are inverse to each other.

Given a solid ring $R$ with non-torsion unit, the inclusion $\iota\colon R_+\to R$ induces a map $\iota_*\colon\Gr(R_+)\to \Gr(R)=R$.
It is easy to verify that $\iota_*$ is an injective ring homomorphism.
To show that $\iota_*$ is surjective, let $r\in R$.
Let $\lambda_0\colon R\to\Q$ be the canonical ring homomorphism from \autoref{dfn:solidRgPositive}.
It is clear that $r$ belongs to the image of $\iota_*$ whenever $\lambda_0(r)\neq 0$.
If $\lambda_0(r)=0$, then $r\in t(R)$.
Then $r=(r+1)-1$, with $r+1, 1\in R_+$ (as the unit is not torsion), which shows that also in this case $r$ belongs to the image of $\iota_*$.

Conversely, let $S$ be a solid, cancellative, conical semiring for which the algebraic order is almost unperforated.
Let $R$ be the Grothendieck completion of $S$.
Since $S$ is cancellative, we can consider $S$ as a unital subsemiring of $R$.
We need to show that $S=R_+$.
As before, we denote by $\lambda_0\colon R\to\Q$ the canonical ring homomorphism.

In order to show $S\subset R_+$, we let $a\in S$.
We may assume $a\neq 0$.
Since $S$ is conical, it contains no torsion element.
Thus, we have $\lambda_0(a)\neq 0$.
To reach a contradiction, assume $\lambda_0(a)<0$.
Choose $k,n\in\N$ such that $\lambda_0(a)=-\tfrac{k}{n}$.
Then
\[
\lambda_0(na+k1)=0.
\]
Since $1$ is an element of $S$, we deduce $na+k1\in S$, and therefore $na+k1=0$.
This contradicts conicality of $S$.
Thus $\lambda(a)\geq 0$ for each $a\in S$.
It follows $S\subset R_+$, as desired.

To show $R_+\subset S$, let $a\in R_+$.
We may assume $a\neq 0$.
Since $R$ is the Grothendieck completion of $S$, we can choose $x,y\in S$ with $a+x=y$.
Then $y$ is nonzero, and we may clearly assume that $x$ is also nonzero.
Let us show $x<_s y$.
It follows from the proof of \autoref{prp:solidElementsRational} that we can find $n,k_1,k_2\in\N$ such that
\[
nx = k_11,\quad
ny = k_21.
\]
Since $a$ is nonzero, we have $k_1<k_2$.
Then $k_2nx=k_1k_21=k_1ny$, with $k_1n<k_2n$, and thus $x<_s y$.
Since $S$ is almost unperforated (with the algebraic order), we deduce $x\leq y$.
Using that the order of $S$ is algebraic and that $S$ is cancellative, we obtain $a\in S$, as desired.
\end{proof}

%==========================================================================================
\begin{rmk}
In \autoref{prp:equivalenceSrgPoRg}, we recall the natural one-to-one correspondence between directed, \por{s} and cancellative, conical se\-mi\-rings, given by assigning to a \por{} its positive cone, and conversely by associating to a conical semiring its Grothendieck completion.
Every solid ring whose unit is not a torsion element has a canonical structure as a directed \por;
see \autoref{rmk:solidRgOrdered}.
Then \autoref{prp:equivalenceSolidSrgRg} shows that the above correspondence restricts to a natural identification between directed, \por{s} coming from solid rings and cancellative, conical semirings that are solid and whose algebraic order is almost unperforated.
\end{rmk}

%==========================================================================================
\begin{lma}
\label{prp:solidSrgCu}
(1)
Let $K$ be a solid, cancellative, conical semiring.
Then its $\Cu$-completion $\Cu(K)$ as constructed in \autoref{pgr:CuCompletionSrg} is a solid $\CatCu$-semiring.

(2)
Let $R$ be a solid, nonelementary, algebraic $\CatCu$-semiring satisfying~\axiomO{5}.
Then the subsemiring of compact elements $R_c$ is a solid, nonelementary, cancellative, conical semiring for which the algebraic order is almost unperforated.
\end{lma}
\begin{proof}
Let $K$ be a cancellative, conical semiring.
Consider the tensor square $K\otimes_{\CatMon}K$ of $K$ in the category $\CatMon$ of monoids.
Equipped with the natural multiplication, the monoid $K\otimes_{\CatMon}K$ becomes a semiring, denoted by $K\otimes_{\CatSrg}K$, which is the tensor square of $K$ in the category $\CatSrg$ of (unital, commutative) semirings;
see \autoref{pgr:srg}.

As explained in \autoref{pgr:algebraicSemigp}, we obtain a $\CatW$-semigroup $(K,\leq)$ if we equip the monoid $K$ with the auxiliary relation that is equal to its partial order.
As shown in \autoref{pgr:CuCompletionSrg}, it follows that the $\CatCu$-completion of $(K,\leq)$ is a $\CatCu$-semiring which we denote by $\Cu(K)$.
We denote the universal $\CatW$-morphism to the $\CatCu$-completion by
\[
\alpha\colon K\to\Cu(K).
\]
Considering $K$ and $\Cu(K)$ as semirings, the map $\alpha$ is a semiring homomorphism and an order-embedding that identifies $K$ with the compact elements of $\Cu(K)$;
see \autoref{rmk:Cuification} and \autoref{prp:algebraicSemigp}.
We therefore think of $K$ as a subset of~$\Cu(K)$ and identify $a$ with $\alpha(a)$, for each $a\in K$.

The map
\[
K\times K \to \Cu(K)\otimes_{\CatCu}\Cu(K),\quad
(a,b)\mapsto a\otimes b,\quad
\txtFA a,b\in K,
\]
is a monoid bimorphism and therefore induces a monoid homomorphism
\[
\varphi\colon K\otimes_{\CatMon} K \to \Cu(K)\otimes_{\CatCu}\Cu(K)
\]
such that $\varphi(a\otimes b)=a\otimes b$ for each $a,b\in K$.

To show the first part of the statements, assume that $K$ is a solid, cancellative, conical semiring.
Let $a\in\Cu(K)$.
In order to prove that $\Cu(K)$ is a solid $\CatCu$-semiring, we need to show by \autoref{prp:solidTFAE} that $1\otimes a= a\otimes 1$ in $\Cu(K)\otimes_{\CatCu}\Cu(K)$.
Assume first that $a$ is a compact element.
Then $a$ is an element of $K$.
Using that $K$ is solid at the second step, we obtain
\[
1\otimes a
= \varphi(1\otimes a)
= \varphi(a\otimes 1)
= a\otimes 1.
\]
If $a$ is a not necessarily compact element, then we can choose an increasing sequence $(a_k)_k$ of compact elements in $\Cu(K)$ such that $a=\sup_k a_k$.
Then
\[
1\otimes a
= 1\otimes (\sup_ka_k)
= \sup_k(1\otimes a_k)
= \sup_k(a_k\otimes 1)
= a\otimes 1.
\]

In order to show the second part of the statement, let $R$ be a solid, nonelementary, algebraic $\CatCu$-semiring satisfying~\axiomO{5}.
By \autoref{prp:propertiesSolid}, $R$ is simple.
Then it follows from \autoref{prp:cancSimpleAlgebraicSemirg} that $R$ is weakly divisible and weakly cancellative.

We set $K:=R_c$, the subsemiring of compact elements, and we identify $R$ with $\Cu(K)$.
It follows from \autoref{prp:propertiesAlgebraic} that $K$ is a conical, cancellative semiring such that the order on $K$ induced by $R$ is the algebraic order.
We know from \autoref{lma:wkDivNonElemSrg} that $K$ is a stably finite nonelementary semiring.
Moreover, by \autoref{prp:solidZmult}, $R$ has a unique normalized functional, which we denote by $\lambda$.
Note that $\lambda(a)>0$ for every nonzero element of $R$.
The map $\lambda$ is a state on $K$, which induces a state $\lambda\otimes\lambda$ on $K\otimes_{\CatSrg}K$ with the property that $(\lambda\otimes\lambda)(x)>0$ for every nonzero element $x$ of $K\otimes_{\CatSrg}K$.
It follows that $K\otimes_{\CatSrg}K$ is stably finite. Since $K$ is a nonelementary semiring we see, as in the proof of \autoref{prp:solidSrgRg}, that $K\otimes_{\CatSrg}K$ is also nonelementary.
It is also straightforward to deduce that $K\otimes_{\CatSrg}K$ is simple and weakly divisible.
Therefore, by \autoref{prp:cancSimpleSemirg}, the semiring $K\otimes_{\CatSrg}K$ is cancellative.
It follows that its algebraic pre-order is partial.
Thus, we have shown that the natural quotient map
\[
K\otimes_{\CatMon}K \to K\otimes_{\CatPom}K
\]
is an isomorphism.

We obtain a $\CatW$-semigroup by equipping $K\otimes_{\CatMon}K$ with the auxiliary relation that is equal to its partial order.
The tensor square of $(K,\leq)$ in the category $\CatPreW$ is given as the tensor square in $\CatPom$ of the underlying monoids together with a naturally defined auxiliary relation;
see \autoref{dfn:tensProdAuxRel}.
It follows that there is a natural isomorphism
\[
(K\otimes_{\CatMon}K,\leq) \cong (K,\leq)\otimes_{\CatPreW}(K,\leq)
\]
Applying $\CatCu$-completions to both sides and using \autoref{thm:tensProdCompl}, we deduce that there is a natural isomorphism
\[
\Cu(K\otimes_{\CatSrg}K) \cong \Cu(K) \otimes_{\CatCu} \Cu(K).
\]

We denote the universal $\CatW$-morphism to the $\CatCu$-completion of $K\otimes_{\CatSrg}K$ by
\[
\beta\colon K\otimes_{\CatSrg}K\to\Cu(K\otimes_{\CatSrg}K).
\]
By \autoref{rmk:Cuification}, the map $\beta$ is an order-embedding.
In conclusion, the map $\varphi$ from the beginning of the proof is an order-embedding.

Then, if $a$ is an element of $K$, we have
\[
\beta(1\otimes a)
= 1\otimes \alpha(a)
= \alpha(a)\otimes 1
= \beta(a\otimes 1).
\]
Since $\beta$ is an oder-embedding, we obtain $1\otimes a=a\otimes 1$ in $K\otimes_{\CatSrg}K$.
Thus, $K$ is a solid semiring.
It is easy to check that $K$ is almost unperforated.
\end{proof}

%==========================================================================================
\begin{thm}
\label{prp:equivalenceSolidSrgRgCu}
There is a natural one-to-one correspondence between each of the following classes:
\beginEnumStatements
\item
Solid rings whose unit is not a torsion element and that are not isomorphic to $\Z$.
\item
Solid, nonelementary, cancellative, conical semirings for which the algebraic order is almost unperforated.
\item
Solid, nonelementary, algebraic $\CatCu$-semirings satisfying~\axiomO{5}.
\end{enumerate}
The correspondence between \enumStatement{1} and \enumStatement{2} is given by associating to a solid ring $R$ with non-torsion unit the solid semiring $R_+$ from \autoref{dfn:solidRgPositive}, and conversely by associating to a solid semiring $S$ its Grothendieck completion $\Gr(S)$.

The correspondence between \enumStatement{2} and \enumStatement{3} is given by associating to a solid semiring $K$ the $\CatCu$-semiring $\Cu(K)$ as constructed in \autoref{pgr:CuCompletionSrg}, and conversely by associating to a solid, algebraic $\CatCu$-semiring $S$ its subsemiring of compact elements.
\end{thm}
\begin{proof}
The correspondence between the classes \enumStatement{1} and \enumStatement{2} follows directly from \autoref{prp:equivalenceSolidSrgRg}.
Let us show the correspondence between the classes \enumStatement{2} and \enumStatement{3}. By \autoref{prp:solidSrgCu}, the assignments of the statement are well-defined; thus, it just remains to show that the assignments are inverse to each other.
This follows directly from \autoref{prp:algebraicSemigp}.
\end{proof}

%==========================================================================================
\index{terms}{Cu-semiring@$\CatCu$-semiring!solid!classification}
\begin{thm}
\label{thm:solidSemirgClassification}
Let $S$ be a nonzero solid $\CatCu$-semiring satisfying~\axiomO{5}.
If $S$ is nonelementary, then exactly one of the following statements holds:
\beginEnumStatements
\item
We have $S\cong[0,\infty]$.
\item
We have $S\cong Z$.
\item
There is a solid ring $R$ with non-torsion unit such that $R\ncong\Z$ and such that $S\cong \Cu(R_+)$.
\end{enumerate}
If $S$ is elementary and satisfies \axiomO{6}, then exactly one of the following conditions holds:
\beginEnumStatements
\setcounter{enumi}{3}
\item
We have $S\cong\N$.
\item
There is $k\in\N$ such that $S\cong\elmtrySgp{k}=\{0,1,2,\ldots,k,\infty\}$.
\end{enumerate}

The $\CatCu$-semiring $S$ is algebraic if and only if it satisfies \enumStatement{3}, \enumStatement{4} or \enumStatement{5}.
\end{thm}
\begin{proof}
We have observed that all $\CatCu$-semirings in statements \enumStatement{1}-\enumStatement{5} are solid.
It is also clear that a solid $\CatCu$-semiring can satisfy at most one of the statements.

Let $S$ be a solid $\CatCu$-semiring.
We will show that $S$ satisfies one of the statements.

Case 1:
Assume $S$ is nonelementary.
By \autoref{prp:solidZmult}, $S$ is simple, almost unperforated, almost divisible, and stably finite.
Moreover, there is a unique normalized functional $\lambda$ on $S$, which is automatically multiplicative.
We obtain from \autoref{thm:isothm} that there is a canonical decomposition
\[
S = S_c \sqcup (0,\infty].
\]
If $S$ contains no nonzero compact element, then $S\cong[0,\infty]$ by \autoref{prp:uniquenessR}.
Otherwise, by \autoref{prp:simpleSemirgCompacts}, the unit of $S$ is compact.

It follows from \autoref{prp:imagelambda} that either $\lambda(S_c)\subset\N$ or that $S$ is algebraic.
In the latter case, $S$ satisfies \enumStatement{3};
see \autoref{prp:equivalenceSolidSrgRgCu}.
Thus, let us assume $\lambda(S_c)\subset\N$.
Then we consider the map
\[
\alpha\colon S = S_c\sqcup(0,\infty] \to Z = \N\sqcup(0,\infty],
\]
which maps a compact element $r$ in $S$ to the compact element $\lambda(r)\in\N\subset Z$, and which maps a soft element in $S_\soft = (0,\infty]$ to the same in $Z_\soft = (0,\infty]$.
It is straightforward to check that $\alpha$ is a unital, multiplicative $\CatCu$-morphism.
It follows from \autoref{prp:solidMorImpliesAbsorb} that $Z\otimes_{\CatCu} S\cong Z$.
By \autoref{prp:solidZmult}, we also have $S\cong S\otimes_{\CatCu} Z$.
It follows $S\cong Z$, which shows that $S$ satisfies \enumStatement{2}.

Case 2:
Assume that $S$ is elementary and satisfies~\axiomO{6}.
Then $S$ satisfies \enumStatement{4} or \enumStatement{5};
see \autoref{exa:elmtrySemirg}.
\end{proof}

%==========================================================================================
\begin{rmk}
\label{rmk:solidSemirgClassification}
If $S$ is a nonelementary, solid $\CatCu$-semiring satisfying \axiomO{5}, then by the classification above, we see that \axiomO{6} is also satisfied. Indeed, only the case where $S=\Cu(R_+)$ for a solid ring $R$ with nontorsion unit needs verification. In this situation, by \autoref{prp:propertiesAlgebraic}, it is enough to show that $R_+$, endowed with the algebraic order, is a Riesz semigroup. This is easy to check once we note that $R/t(R)\cong\Z[P^{-1}]$ for a nonemtpy set of primes $P$.
\end{rmk}

%------------------------------------------------------------------------------------------
We end this section with a result about initial and terminal objects among solid $\CatCu$-semirings.
This can be considered as a $\CatCu$-semigroup{s} version of \cite[Corollary~3.2]{Win11ssaZstable}, which characterizes the Jiang-Su algebra $\mathcal{Z}$ as the initial object in the category of unital, strongly self-absorbing \ca{s} with \starHom s up to approximate unitary equivalence.

%==========================================================================================
\begin{lma}
\label{prp:CompactSolidElementsRational}
Let $S$ be a solid, nonelementary $\CatCu$-semiring satisfying \axiomO{5}, and let $\lambda$ be its unique normalized functional.
Then $\lambda(S_c)\subset\Q_+$.
\end{lma}
\begin{proof}
By \autoref{prp:imagelambda}, either $\lambda(S_c)\subset\N$ or $S$ is algebraic.
If $S$ is algebraic, by \autoref{prp:solidSrgCu} $S_c$ is a conical, cancellative, solid semiring for which the algebraic order is almost unperforated.
Now the result follows from \autoref{prp:equivalenceSolidSrgRg} and \autoref{prp:solidElementsRational}.
\end{proof}

%==========================================================================================
\begin{exa}
\label{exa:Q}
\index{symbols}{Q@$Q$}
Let $\Q$ be the solid ring of rational numbers.
We obtain a corresponding solid $\CatCu$-semiring, which we denote by $Q$.
Thus
\[
Q = \Cu(\Q_+) \cong \Q_+\sqcup (0,\infty].
\]
As we show below, $Q$ is the terminal object in a suitable category.
\end{exa}

%------------------------------------------------------------------------------------------
We remark that for every simple, nonelementary  $\CatCu$-semiring $S$ that satisfies \axiomO{5} and has a compact unit, there exists a multiplicative $\CatCu$-morphism from $Z$ to $S$.
This follows from \autoref{prp:simpleSemirgZmult}.
If $S$ has a unique normalized functional, then the map is unique.

%==========================================================================================
\begin{prp}
\label{prp:initialTerminal}
Let $S$ be a solid, nonelementary $\CatCu$-semiring satisfying~\axiomO{5}.
Assume that $1_s$ is compact.
Let $Q$ be the solid $\CatCu$-semiring from \autoref{exa:Q}.
Then there are unique unital, multiplicative $\Cu$-morphisms
\[
Z\to S\to Q.
\]
Thus, $Z\otimes_{\CatCu} S\cong S$ and $S\otimes_{\CatCu} Q\cong Q$.

This means that $Z$ and $Q$ are the initial and final objects of the category of considered $\CatCu$-semirings with unital, multiplicative $\CatCu$-morphisms.
\end{prp}
\begin{proof}
The existence and uniqueness of the map $Z\to S$ is observed in the paragraph before this proposition.
Let $S$ be as in the statement.
As in the beginning of the proof of \autoref{thm:solidSemirgClassification}, we obtain that there is a natural decomposition
\[
S = S_c \sqcup (0,\infty].
\]
By \autoref{prp:solidZmult} and \autoref{prp:CompactSolidElementsRational}, $S$ has a unique normalized functional $\lambda$, which is automatically multiplicative, and which satisfies $\lambda(S_c)\subset\Q_+$.
Thus, we may consider the map
\[
\alpha\colon S = S_c\sqcup(0,\infty] \to Q = \Q_+\sqcup(0,\infty],
\]
which maps a compact element $r$ in $S_c$ to the compact element $\lambda(r)\in\Q_+\subset Q$, and which maps a soft element in $S_\soft = (0,\infty]$ to the same in $Q_\soft = (0,\infty]$.
It is easy to see that $\alpha$ is a unital, multiplicative $\CatCu$-morphism, as desired.

It is left to the reader to show uniqueness of the map $S\to Q$.
The results about tensorial absorption follow from \autoref{prp:solidMorImpliesAbsorb}.
\end{proof}

%------------------------------------------------------------------------------------------
%==========================================================================================
%##########################################################################################
\chapter{Concluding remarks and Open Problems}
\label{sec:openproblems}

%------------------------------------------------------------------------------------------
In this chapter we list some problems that we believe to be open and that have appeared in the course of our investigations.

\begin{enumerate}[leftmargin=*]
\item
\autoref{prbl:tensCa}:
Let $A$ and $B$ be \ca{s}.
When are the natural $\CatCu$-morphisms
\begin{align*}
\tau_{A,B}^\txtMin\colon\Cu(A)\otimes_{\CatCu}\Cu(B) \to \Cu(A\tensMin B), \\
\tau_{A,B}^\txtMax\colon\Cu(A)\otimes_{\CatCu}\Cu(B) \to \Cu(A\tensMax B),
\end{align*}
from \autoref{pgr:tensCa} surjective, or order-embeddings, or isomorphisms?
More generally, what is the relation between $\Cu(A)\otimes_{\CatCu}\Cu(B)$ and $\Cu(A\otimes B)$?

In \autoref{pgr:AnswerTensCa}, we mention some partial result concerning this problem.
This problem asks for a more general formula of a K\"{u}nneth type flavor.
It looks plausible that such a formula will have to take the $K_1$-groups of the involved \ca{s} into account.
One possible invariant is $\Cu_{\T}(\freeVar)$, as introduced in
\cite{AntDadPerSan14RecoverElliott}, which for a \ca{} $A$ is defined as
\[
\Cu{(C(\T)\otimes A)}.
\]
In significant cases, this invariant records both the Cuntz semigroups of $A$ and the $K_1$-group of $A$.

%==========================================================================================
\vspace{5pt}
\item
\label{prbl:Cuclosed}
Is $\CatCu$ a closed category?

This is a natural question given that $\CatCu$ is a symmetric, monoidal category.
A positive answer to this problem would provide additional structure to the morphism sets in $\CatCu$, and this is a potentially useful tool in connection with the current development of a bivariant version of the Cuntz semigroup (see \cite{BosTorZac14BivarCu}).

%==========================================================================================
\vspace{5pt}
\item
\autoref{prbl:AxiomsTensProd}:
Given $\CatCu$-semigroups $S$ and $T$ that satisfy \axiomO{5} (respectively \axiomO{6}, weak cancellation).
When does $S\otimes_\CatCu T$ satisfy \axiomO{5} (respectively \axiomO{6}, weak cancellation)?

In \autoref{pgr:answersAxiomsTensProd}, we mention some partial result concerning this problem.
A particular variant of this problem is:

%==========================================================================================
\vspace{5pt}
\item
\autoref{prbl:axiomTensProdZ}:
When does axiom~\axiomO{5}, \axiomO{6} or weak cancellation pass from a $\CatCu$-semigroup $S$ to the tensor product $Z\otimes_{\CatCu} S$?

%==========================================================================================
\vspace{5pt}
\item
\label{prbl:lsctensor}
Let $X$ be a finite-dimensional, compact, Hausdorff space, and let $S$ be a $\CatCu$-semigroup.
It was proved in \cite[Theorem~5.15]{AntPerSan11PullbacksCu} that the semigroup of lower semicontinuous functions from $X$ to $S$, denoted by $\Lsc(X,S)$, is a $\CatCu$-semigroup.
Under which conditions on $S$ and $X$ does $\Lsc(X,S)$ satisfy \axiomO{5} (respectively \axiomO{6}, weak cancellation)?

We remark that if $\Lsc(X,S)$ satisfies \axiomO{5} (respectively \axiomO{6}, weak cancellation), then so does $S$.
The natural test case is $X=[0,1]$.
A positive answer seems likely if $S$ is algebraic.

%==========================================================================================
\vspace{5pt}
\item
Let $X$ be a finite-dimensional, compact, Hausdorff space, and let $S$ be a $\CatCu$-semigroup.
When does it hold that $\Lsc(X,S)=\Lsc(X,\overline{\N})\otimes_{\CatCu}S$?

We show in \autoref{cor:tensornotlsc} that this question has a negative answer for $X=[0,1]$ and $S=Z$.
On the other hand, a positive answer seems likely if $S$ is algebraic.

%==========================================================================================
\vspace{5pt}
\item
\label{prbl:rangequtn}
Range problem:
Under which conditions can a $\CatCu$-semigroup $S$ be realized as the Cuntz semigroup $\Cu(A)$ of a \ca{} $A$?

Necessarily, such a $\Cu$-semigroup satisfies \axiomO{5} and \axiomO{6}.
Thus, we are asking for additional conditions on $\CatCu$-semigroups beyond these two axioms which would guarantee that a $\CatCu$-semigroup is realized by a \ca{}.

As we have already mentioned, Robert showed in \cite{Rob13CuSpaces2D} that if $X$ is a compact Hausdorff space whose covering dimension is at least $3$, then there is no \ca{} $A$ with $\Cu(A)\cong \Lsc(X,\overline{\N})$.
It was shown by Bosa (private communication), that none of the elementary semigroups $\elmtrySgp{k}$ as described in \autoref{pgr:elementarySemigr} for $k\geq 1$ are realized as the Cuntz semigroup of a \ca{}.
(The $\CatCu$-semigroup $\elmtrySgp{0}=\{0,\infty\}$ is of course the Cuntz semigroup of any purely infinite simple \ca.)

%==========================================================================================
\vspace{5pt}
\item
\label{listPrbls:Zprime}
Consider the $\CatCu$-semigroup $Z'$ defined as follows:
\[
Z'=\{0,1,1',2,3,4,\ldots\}\sqcup (0,\infty],
\]
with addition as in $Z$, except that $k+1'=k+1$ for any $k\in\N_+$.
It is easy to check that $Z'$ is a simple, stably finite $\CatCu$-semigroup satisfying \axiomO{5} and \axiomO{6}, but it is not weakly cancellative as $1+1'=1'+1'$ but $1\neq 1'$.
It is also easy to prove that $Z'$ has $Z$-multiplication, and therefore $Z' \cong Z\otimes_{\CatCu}Z'$.

A particularly interesting instance of the range problem is the following:
Does there exist a (separable, unital, simple, stably finite) \ca{} $A$ such that~$\Cu(A)\cong Z'$?

Note that if such a \ca{} $A$ exists, then it is necessarily simple, not $\mathcal{Z}$-stable, and not nuclear.
For if $A$ is $\mathcal{Z}$-stable, then it has stable rank one and hence its Cuntz semigroup has weak cancellation.
Similarly, if $A$ is nuclear, then as $Z'$ has only one normalized functional, we would get that $A$ is monotracial.
In that situation, the solution of the Toms-Winter conjecture (see \cite{MatSat12StrComparison}) would imply that $A$ is $\mathcal{Z}$-stable, a contradiction.

The $\CatCu$-semigroup $Z'$ seems to be the simplest example that is not weakly cancellative and has $Z$-multiplication.
A more general question is then:

%==========================================================================================
\vspace{5pt}
\item
\label{prbl:Zprimeprime}
\label{listPrbls:RangeZmultNotWkCanc}
Does there exist a finite, simple \ca{} $A$, such that $\Cu(A)$ has $Z$-mul\-ti\-pli\-ca\-tion, but is not weakly cancellative?

Let $A$ be such a \ca{}.
If $A$ is nuclear, then we could deduce as above that $A$ is not $\mathcal{Z}$-stable although $\Cu(A)$ is almost unperforated.
Therefore, the Toms-Winter Conjecture predicts that $A$ cannot be nuclear.

It is natural to seek for additional axioms that allow us to rule out $Z'$ as a semigroup in a future reformulation of the category $\CatCu$.
We make this explicit with the following question.

%==========================================================================================
\vspace{5pt}
\item
\label{prbl:dichotomyCu}
Under what additional axioms (besides \axiomO{5} and \axiomO{6}) is a simple $\CatCu$-semi\-group with $Z$-multiplication necessarily weakly cancellative?

This question also refers to structural properties of $\CatCu$-semigroups.
In this direction, \autoref{conj:nearUnpCaZstable} bears repeating.
Let us recall from \autoref{dfn:nearUnperf} that a $\CatCu$-semigroup $S$ is nearly unperforated if and only if $a\leq b$ whenever $a\leq_p b$ for any $a$ and $b$ in $S$.
Equivalently, by \autoref{prp:nearUnperfTFAE}, we have $a\leq b$ whenever $2a\leq 2b$ and $3a\leq 3b$.

%==========================================================================================
\vspace{5pt}
\item
\autoref{conj:nearUnpCaZstable}:
Let $A$ be a $\mathcal{Z}$-stable \ca.
Then $\Cu(A)$ is nearly unperforated.

%==========================================================================================
\vspace{5pt}
\item
\autoref{prbl:axiomsExtension}:
Let $S$ be a $\CatCu$-semigroup, and let $I$ be an ideal in $S$.
Assume that $I$ and $S/I$ satisfy \axiomO{5} (respectively \axiomO{6}, weak cancellation).
Under what assumptions does this imply that $S$ itself satisfies the respective axiom?

%==========================================================================================
\vspace{5pt}
\item
\autoref{prbl:pncInCu}:
Given a $\CatCu$-semigroup $S$, is the subsemigroup $S_\soft$ of soft elements again a $\CatCu$-semigroup?
Does this hold under the additional assumption that $S$ satisfies \axiomO{5}?
If that is the case, does then $S_\soft$ satisfy \axiomO{5} as well?

%==========================================================================================
\vspace{5pt}
\item
\autoref{prbl:nearUnpFromAlmUnp}:
Let $S$ be an almost unperforated $\CatCu$-semigroup.
Which conditions are necessary and sufficient for $S$ to be nearly unperforated?
In particular, is it sufficient to assume that $S$ satisfies weak cancellation and \axiomO{5}?

%==========================================================================================
\vspace{5pt}
\item
\autoref{prbl:LatTens}:
Let $A$ and $B$ be \ca{s}.
When is the map $\psi_{A,B}$ from \autoref{pgr:LatTens} an isomorphism?
When is it surjective?
When is it an order embedding?

%==========================================================================================
\vspace{5pt}
\item
\autoref{prbl:MapToZMod}:
Let $S$ be a $\CatCu$-semigroup, and let $a,b\in S$.
Characterize when $1\otimes a \leq 1\otimes b$ in $Z\otimes_{\CatCu}S$.

%==========================================================================================
\vspace{5pt}
\item
\autoref{prbl:ssaSolid}:
Given a strongly self-absorbing \ca{} $D$, is the Cuntz semiring $\Cu(D)$ a solid $\CatCu$-semiring?

Should the answer to this problem be positive, our \autoref{thm:solidSemirgClassification} (see also \autoref{rmk:solidSemirgClassification}) would yield a complete list of the possible Cuntz semigroups for stably finite, strongly self-absorbing \ca{s}, and this would be valuable information towards finding a possible non-UCT example, if such exists.

%==========================================================================================
\vspace{5pt}
\item
\autoref{prbl:LFS}:
Let $S$ be a $\CatCu$-semigroup.
Is it true that $S_R=L(F(S))$?
\end{enumerate}

\appendix

%\include{Cu-TensProd_ChA}
%\include{Cu-TensProd_ChB}

%------------------------------------------------------------------------------------------
%==========================================================================================
%##########################################################################################
\chapter{Monoidal and enriched categories}
\label{sec:appendixMonoidal}
\setcounter{lma}{0}

%------------------------------------------------------------------------------------------
In this appendix, we will recall the basic theory of monoidal and enriched categories.
For details we refer the reader to \cite{Kel05EnrichedCat} and \cite{MacLan71Categories}.

%==========================================================================================
\begin{pgrChapter}[Monoidal categories]
\label{pgr:monoidalCat}
\index{terms}{monoidal category}
\index{terms}{monoidal category!symmetric}
A \emph{monoidal category} $\CatV$ consists of the following data:
An underlying category $\CatV_0$, a bifunctor
\[
\otimes\colon\CatV_0\times\CatV_0\to\CatV_0,
\]
a unit object $I$ in $\CatV_0$, and natural isomorphisms
\[
(X\otimes Y)\otimes Z\cong X\otimes(Y\otimes Z),\quad
I\otimes X\cong X,\quad
X\otimes I\cong X.
\]
whenever $X,Y$ and $Z$ are objects in $\CatV_0$. Moreover, certain coherence axioms need to be satisfied.

By an object in $\CatV$ we mean an object of the underlying category.
Given objects $X$ and $Y$ of $\CatV$, we let $\CatVMor_0(X,Y)$ denote the collection of $\CatV_0$-morphisms from $X$ to $Y$, and we will always assume that it is a set.
(This means that $\CatV_0$ is \emph{locally small}.)

The monoidal category $\CatV$ is \emph{symmetric} if for any objects $X$ and $Y$ in $\CatV$ there is a natural isomorphism
\[
X\otimes Y\cong Y\otimes X.
\]
\end{pgrChapter}

%==========================================================================================
\begin{pgrChapter}[Concrete monoidal categories]
\label{pgr:concreteCat}
\index{terms}{monoidal category!concrete}
Let $\CatV$ be a monoidal category with unit object $I$.
Since $\CatV$ is assumed to be locally small, the representable functor $\CatV_0(I,\freeVar)$ is a functor from $\CatV_0$ to the category of sets.
We denote this functor by
\[
V\colon\CatV_0\to\CatSet.
\]

We say that $\CatV$ is a \emph{concrete monoidal category} if $V$ is faithful.
In that case we can think of objects in $\CatV$ as sets with additional structure, and we can think of morphisms in $\CatV$ as maps preserving that structure.

Let $X$ be an object in $\CatV$.
Then an \emph{element $x$ of $X$} is an element in $V(X)$, that is, a $\CatV_0$-morphism $x\colon I\to X$.
We write $x\in X$ to denote that $x$ is an element of $X$.
This terminology is even used when $V$ is not necessarily faithful.

Let $X$ and $Y$ be objects in $\CatV$, and let $x\in X$ and $y\in Y$.
Then the composed morphism
\[
I \xrightarrow{\cong}
I\otimes I \xrightarrow{x\otimes y}
X\otimes Y,
\]
is an element of $X\otimes Y$, which we will also denote by $x\otimes y$.
\end{pgrChapter}

%==========================================================================================
\begin{exaChapter}
\label{exa:monoidalCat}
The category $\CatTop$ of topological spaces is a concrete, symmetric monoidal category.
The tensor product of two spaces is their Cartesian product with the product topology.
The unit object is the one-element space.
\end{exaChapter}

%==========================================================================================
\begin{pgrChapter}[Enriched categories]
\label{pgr:enrichedCat}
\index{terms}{enriched category}
\index{terms}{V-category@$\CatV$-category}
\index{terms}{V-functor@$\CatV$-functor}
\index{terms}{V-natural transformation@$\CatV$-natural transformation}
Let $\CatV$ be a monoidal category.
A \emph{$\CatV$-category} $\CatC$ (also called a category \emph{enriched over $\CatV$}) consists of the following data:
A collection of objects in $\CatC$; an object $\CatCMor(X,Y)$ in $\CatV$, for any $X,Y\in\CatC$, playing the role of the collection of morphisms from $X$ to $Y$;
for each object $X$ in $\CatC$, an element $\id_X$ in $\CatCMor(X,X)$,
called the identity on $X$ and playing the role of the identity morphism on $X$;
and for any objects $X,Y, Z$ in $\CatC$, a $\CatV_0$-morphism
\[
M_{X,Y,Z}\colon\CatCMor(Y,Z)\otimes\CatCMor(X,Y)\to\CatCMor(X,Z),
\]
implementing the composition of morphisms.
This structure is required to satisfy certain conditions expressing for example the associativity of composition of morphisms.

Recall that the element $\id_X$ in $\CatCMor(X,X)$ is, by definition, a $\CatV_0$-morphism
\[
\id_X\colon I\to\CatCMor(X,X).
\]

Given $\CatV$-categories $\CatC$ and $\CatD$, a \emph{$\CatV$-functor} $F$ from $\CatC$ to $\CatD$ consists of the following data:
An assignment $F$ from the objects of $\CatC$ to the objects of $\CatD$,
and for any objects $X,Y$ in $\CatC$ a $\CatV_0$-morphism
\[
F_{X,Y}\colon\CatCMor(X,Y)\to\CatDMor\left( \vphantom{A^A} F(X),F(Y) \right).
\]
It is required that certain diagrams, which for instance express compatibility with composition of morphisms, commute.
(Again, we refer to \cite{Kel05EnrichedCat} for details.)

Given $\CatV$-functors $F$ and $G$ from $\CatC$ to $\CatD$, a \emph{$\CatV$-natural transformation} from $F$ to $G$, denoted by $F\Rightarrow G$, is a collection of elements $\alpha_X$ in $\CatDMor(F(X),G(X))$,
indexed by the objects $X$ in $\CatC$, such that certain natural conditions are satisfied.
\end{pgrChapter}

%==========================================================================================
\begin{pgrChapter}[Concrete enriched categories]
\label{pgr:concreteEnrCat}
Let $\CatC$ be a category that is enriched over the concrete monoidal category $\CatV$.
We can use the faithful functor $V$ to associate to $\CatC$ an ordinary underlying category $\CatC_0$ as follows:
The objects of $\CatC_0$ are the same as the objects of $\CatC$;
the $\CatC_0$-morphisms between two objects $X$ and $Y$ in $\CatC_0$ are given as
\[
\CatC_0(X,Y) :=V\left( \vphantom{A^A} \CatC(X,Y) \right)
=\CatV_0\left( \vphantom{A^A} I,\CatC(X,Y) \right);
\]
for each object $X$ in $\CatC_0$, the identity morphism $\id_X$ in $\CatC_0$ is just the $\CatV_0$-morphism
\[
\id_X\colon I\to \CatC(X,X),
\]
considered as an element of $\CatC_0(X,X)$;
and for any $\CatC_0$-morphisms $f\in\CatC_0(X,Y)$ and $g\in\CatC_0(Y,Z)$, which by definition are $\CatV_0$-morphisms
\[
f\colon I\to\CatC(X,Y),\quad
g\colon I\to\CatC(Y,Z),
\]
we consider the composed $\CatV_0$-morphism
\[
I\xrightarrow{g\otimes f}\CatC(Y,Z)\otimes\CatC(X,Y)
\xrightarrow{M_{X,Y,Z}}\CatC(X,Z),
\]
which is an element of $\CatC_0(X,Z)$ defining the composition $g\circ f$ in $\CatC_0$.
Using the coherence axioms for monoidal and enriched structures, one can show that the laws of a category are fulfilled for $\CatC_0$.

We can then think of $\CatC(X,Y)$ as the set of morphisms $\CatC_0(X,Y)$ endowed with additional structure making it into an object in $\CatV$.
\end{pgrChapter}

%==========================================================================================
\begin{pgrChapter}[Closed Categories]
\label{pgr:closedCat}
\index{terms}{monoidal category!closed}
\index{terms}{internal hom-bifunctor}
A monoidal category $\CatV$ is \emph{closed} if, for each object $Y$ in $\CatV$, the functor
\[
\freeVar\otimes Y\colon\CatV_0\to\CatV_0
\]
has a right adjoint, which we will denote by $(\freeVar)^Y$.
Thus, in a closed, monoidal category $\CatV$, for any objects $X,Y$ and $Z$, there is a bijection (natural in $X$ and $Z$) between the following sets of morphisms:
\begin{align}
\label{enriched:innerHom}
\CatVMor_0(X\otimes Y,Z) \cong \CatVMor_0\left( X,Z^Y \right).
\end{align}

We let $\CatV_0^\op$ denote the opposite category of $\CatV_0$, that is, $\CatV_0^\op$ has the same objects as $\CatV_0$ and for any objects $X$ and $Y$ we identify $\CatVMor_0^\op(X,Y)$ with $\CatVMor_0(Y,X)$.
Given $\CatV_0$-morphisms $f\colon Y_2\to Y_1$ and $g\colon Z_1\to Z_2$, there is an induced $\CatV_0$-morphism $(f^\ast,g)\colon Z_1^{Y_1}\to Z_2^{Y_2}$.
This defines a bifunctor
\[
(\freeVar)^{(\freeVar)}\colon\CatV_0^\op\times\CatV_0\to\CatV_0,
\]
called the \emph{internal hom-bifunctor} of $\CatV$.
Given objects $X$ and $Y$, the object $Y^X$ encodes the set of morphisms $\CatVMor_0(X,Y)$ via the identifications
\[
V(Y^X)=\CatVMor_0(I,Y^X) \cong \CatVMor_0(I\otimes X,Y) \cong \CatVMor_0(X,Y).
\]
Therefore, the object $Y^X$ is an internal realization of the morphisms from $X$ to $Y$.

Let $Z$ and $Y$ be objects in $\CatV$.
The evaluation morphism
\[
\varepsilon_Y^Z\colon  Z^Y \otimes Y\to Z,
\]
is defined as the $\CatV_0$-morphism that corresponds to the identity morphism under the natural identification
\[
\CatVMor_0\left( Z^Y \otimes Y,Z \right)
\cong\CatVMor_0\left( Z^Y, Z^Y \right).
\]

Let $\CatV$ be a closed, symmetric, monoidal category.
Then $\CatV$ is enriched over itself.
Given objects $X$ and $Y$ in $\CatV$, we use the internal hom-bifunctor to define $\CatV(X,Y)=Y^X$;
for each object $X$ in $\CatV$, the identity morphism $\id_X\in X^X$ is defined as the $\CatV_0$-morphism $\id_X\colon I \to X^X$ corresponding to the identity morphism from $X$ to $X$ under the identification
\[
\CatVMor_0(I,X^X) \cong\CatVMor_0(I\otimes X,X)\cong\CatVMor_0(X,X);
\]
for any objects $X,Y$ and $Z$ in $\CatV$, the map
\[
M_{X,Y,Z}\colon Z^Y \otimes Y^X \to Z^X,
\]
implementing the composition of morphisms is the $\CatV_0$-morphism that under the identification
\[
\CatVMor_0\left( Z^Y \otimes Y^X, Z^X \right)
\cong\CatVMor_0\left( \left(Z^Y \otimes Y^X \right)\otimes X,Z \right),
\]
corresponds to the composition
\[
\left( Z^Y \otimes Y^X \right)\otimes X \cong
Z^Y \otimes\left( Y^X\otimes X \right)
\xrightarrow{\id_{Z^Y}\otimes\,\varepsilon_X^Y}
Z^Y \otimes Y \xrightarrow{\varepsilon_Y^Z}
Z.
\]
\end{pgrChapter}

%==========================================================================================
\begin{exasChapter}
\label{exa:closedCat}
\index{terms}{compactly generated space}
(1)
We will see in Paragraphs~\ref{pgr:monoidalMon} and \ref{pgr:monoidalPom} that the category $\CatMon$ of monoids and the category $\CatPom$ of \pom{s} are both closed, symmetric, monoidal categories.

(2)
The category $\CatTop$ is not closed.
However, it contains several full, symmetric, monoidal subcategories that are closed.

A Hausdorff, topological space $X$ is \emph{compactly generated} if a subset $M\subset X$ is closed whenever $M\cap K$ is closed for every compact subset $K$ of $X$.
The class of compactly generated, Hausdorff spaces contains for example all metric spaces and all locally compact, Hausdorff spaces.
We let $\CatCGHTop$ denote the full subcategory of $\CatTop$ consisting of compactly generated, Hausdorff spaces.
It is known that $\CatCGHTop$ is a reflective subcategory of $\CatTop$.
Given a topological space $X$, we let $X_k$ denote its reflection in $\CatCGHTop$.

Let $X$ and $Y$ be compactly generated, Hausdorff spaces.
Their Cartesian product need not be compactly generated.
Therefore, the `correct' tensor product of $X$ and $Y$ in the category $\CatCGHTop$ is $(X\times Y)_k$.
Consider the set $C(X,Y)$ of continuous maps from $X$ to $Y$ equipped with the compact open topology.
Again, $C(X,Y)$ need not be compactly generated.
Nevertheless, $\CatCGHTop$ is a closed category with internal hom-bifunctor given by
\[
\CatCGHTop(X,Y) = (C(X,Y))_k.
\]
Then, given spaces $X,Y$ and $Z$ in $\CatCGHTop$, we have a natural isomorphism (that is, a homeomorphism) between the following spaces:
\[
\CatCGHTop((X\times Y)_k, Z) \cong \CatCGHTop(X,\CatCGHTop(Y,Z)).
\]

(3)
It is shown in \cite{DubPor71ConvenientCatTopAlg} that the category $\CatCa$ of \ca{s} is enriched over $\CatCGHTop$.
For any \ca{s} $A$ and $B$, the set of \starHom{s} from $A$ to $B$ is denoted by $\CatCaMor(A,B)$ and it is endowed with the topology of pointwise convergence.
This means that a net $(\varphi_i)_i$ in $\CatCaMor(A,B)$ converges to a \starHom{} $\varphi\colon A\to B$ if and only if $\varinjlim_i \|\varphi_i(x)-\varphi(x)\|=0$ for each $x\in A$.
With this topology, $\CatCaMor(A,B)$ is a compactly generated, Hausdorff space.
\end{exasChapter}

%------------------------------------------------------------------------------------------
%==========================================================================================
%##########################################################################################
\chapter{Partially ordered monoids, groups and rings}
\label{sec:appendixPom}

%------------------------------------------------------------------------------------------
In this appendix we will only consider commutative structures.
Therefore, every monoid, semigroup and group is written with operation of addition, and multiplication in every semiring and ring is commutative.
\vspace{5pt}

%------------------------------------------------------------------------------------------
In \autoref{sec:mon}, we study the category $\CatMon$ of (abelian) monoids, and the full, reflective subcategories $\CatGp$ of groups, and $\CatCon$ of conical monoids.
We recall the construction of tensor products in $\CatMon$ and we show that $\CatMon$ is a closed, symmetric, monoidal category.

%------------------------------------------------------------------------------------------
In \autoref{sec:pom}, we study the category $\CatPrePom$ of positively pre-or\-dered mo\-noids and its full subcategory $\CatPom$ of \pom{s}.
We show that $\CatPom$ is a reflective subcategory of $\CatPrePom$.
Both categories have a tensor product giving them a closed, symmetric, monoidal structure.

We also show that $\CatCon$ can be identified with the full (and reflective) subcategory of $\CatPrePom$ consisting of algebraically pre-ordered monoids.

%------------------------------------------------------------------------------------------
In \autoref{sec:poGp}, we study the category $\CatPoGp$ of \pog{s}.
In particular, we recall the equivalence between the categories of directed, \pog{s} and the category of cancellative, conical monoids.

Finally, in \autoref{sec:poRg}, we study semirings and (partially ordered) rings.

For further details, the reader is referred to, for example, \cite{Ful70Tensor}, \cite{Gol99Semirings}, \cite{Goo86Poag}, \cite{Gri69TensorCommutative}, \cite{Weh96TensorInterpol}.

\vspace{5pt}
%------------------------------------------------------------------------------------------
%==========================================================================================
\section{The category \texorpdfstring{$\CatMon$}{Mon} of monoids}
\label{sec:mon}

%==========================================================================================
\begin{pgr}
\label{pgr:mon}
\index{terms}{monoid}
\index{symbols}{Mon@$\CatMon$ \quad (category of monoids)}
In this paper, by a \emph{monoid} we always mean an abelian monoid, written additively, and with zero element denoted by $0$.

Let $M,N$ and $R$ be monoids.
A \emph{monoid homomorphism} from $M$ to $R$ is a map $f\colon M\to R$ that preserves addition and the zero element.
We denote the collection of such maps by $\CatMonMor(M,R)$.
We let $\CatMon$ denote the category whose objects are all monoids and whose morphisms are monoid homomorphisms.

We can endow the set $\CatMonMor(M,R)$ with a natural monoid structure as follows.
Given $f,g\in\CatMonMor(M,R)$, we define their sum $f+g$ by pointwise addition, that is, we have $(f+g)(a)=f(a)+g(a)$ for each $a\in M$.
The zero element in $\CatMonMor(M,R)$ is given by the zero map that sends every $a$ in $M$ to the zero element in $R$.

It is not difficult to verify that the homomorphism-monoid $\CatMonMor(M,R)$ is functorial in both variables.
Therefore, we obtain a bifunctor
\[
\CatMonMor(\freeVar,\freeVar)\colon\CatMon^{\op}\times\CatMon\to\CatMon,
\]
called the \emph{internal hom-bifunctor} of $\CatMon$.

A \emph{monoid bimorphism} from $M\times N$ to $R$ is a map $f\colon M\times N\to R$ which is a monoid homomorphism in each variable.
In other words, for fixed $a\in M$ and $b\in N$, the maps $f_a\colon M\to R$ and $g_b\colon N\to R$ given (respectively) by $f_a(x)=f(x,b)$ and $g_b(y)=f(a,y)$, for $x\in M$ and $y\in N$, are monoid homomorphisms.
We denote the set of these monoid bimorphisms by~$\CatMonBimor(M\times N,R)$.
It becomes a monoid when endowed with pointwise addition.
As for the monoid morphisms, one can check that monoid bimorphisms are functorial in all entries.
We therefore have a multifunctor
\[
\CatMonBimor(\freeVar\times\freeVar,\freeVar)\colon\CatMon^{\op}\times\CatMon^{\op}\times\CatMon\to\CatMon.
\]
\end{pgr}

%------------------------------------------------------------------------------------------
Next, we recall the construction of tensor products in the category $\CatMon$.
The construction is based on the tensor product of abelian semigroups as studied by Grillet, \cite{Gri69TensorCommutative}.
The tensor product in $\CatMon$ has also been studied in \cite{Ful70Tensor}, where it is denoted by $\otimes_0$.

%==========================================================================================
\begin{pgr}
\label{pgr:tensorMonConstr}
\index{symbols}{$\otimes$ \quad (tensor product in $\CatMon$)}
\index{terms}{congruence relation}
Let $M$ and $N$ be monoids.
Consider the free abelian monoid $F=\N[M\times N]$ whose basis is the cartesian product~$M\times N$.
For $a\in M$ and $b\in N$ we let $a\odot b$ denote the element in~$F$ that takes value~$1$ at~$(a,b)$ and that takes value~$0$ elsewhere.
Then for every element~$f\in F$ there exist a finite index set $I$ and elements $a_i\in M$ and $b_i\in N$ for $i\in I$ such that
\[
f=\sum_{i\in I} a_i\odot b_i.
\]
We do not require that $a_i\neq a_j$ for distinct indices $i$ and $j$.
Moreover, the presentation of $f$ in this way is essentially unique (up to permutation of the index set).

Following the notation in \cite[Section~2]{Weh96TensorInterpol}, we define two binary relations $\rightarrow^0$ and~$\rightarrow$ on~$F$.
We also define a binary relation~$\cong_0$ on~$F$.
Let $f$ and $g$ be elements in $F$.
Then:
\begin{enumerate}
\item
We set $f\rightarrow^0g$ if and only if there exists a pair $(a,b)\in M\times N$, nonempty finite index sets $I$ and $J$, and elements $a_i\in M$ for $i\in I$, $b_j\in N$ for $j\in J$ such that
\[
a=\sum_{i\in I}a_i,\quad
b=\sum_{j\in J}b_j,\quad
f=a\odot b,\quad
g=\sum_{i\in I, j\in J}a_i\odot b_j.
\]
\item
We set $f\rightarrow g$ if and only if either $f=g=0$ or else there are $n\in\N$ and $f_k,g_k\in F$ for $k=0,\ldots,n$ such that
\[
f=\sum_{k=0}^nf_k,\quad
g=\sum_{k=0}^ng_k,\quad\text{ and }\quad
f_k\rightarrow^0 g_k, \text{ for each } k.
\]
\item
For any $a\in M$ and $b\in N$, we set $0\cong_0 a\odot 0$ and $0\cong_0 0\odot b$.
\end{enumerate}

A binary relation $\mathcal{R}$ on $M$ is called \emph{additive} if, for any elements $a,b,c,d\in M$, we have $(a+c,b+d)\in\mathcal{R}$ whenever $(a,b)\in\mathcal{R}$ and $(c,d)\in\mathcal{R}$.
As in \cite[Lemma~2.1]{Weh96TensorInterpol}, the relation $\rightarrow$ is reflexive and additive.

A \emph{congruence relation} on a monoid is an additive equivalence relation.
We let~$\cong$ be the congruence relation on~$F$ generated by~$\rightarrow$, $\leftarrow$, and $\cong_0$.
We set
\[
M\otimes N := F/\!\!\cong,
\]
which is the set of $\cong$-congruence classes in $F$.
It is easy to check that $M\otimes N$ is a monoid (see \cite{Ful70Tensor}).

Given $(a,b)\in M\times N$, we write $a\otimes b$ for the congruence class of $a\odot b$ in $M\otimes N$.
In particular, we have $a\otimes 0=0$ and $0\otimes b=0$ for every $a\in M$ and $b\in N$.
We define a map
\[
\omega\colon M\times N\to M\otimes N\,\quad
\omega(a,b)=a\otimes b,\quad
\txtFA a\in M, b\in N,
\]
which is easily seen to be a monoid bimorphism.
\end{pgr}

%==========================================================================================
\begin{rmk}
\label{rmk:tensorMon}
We denote the tensor product in $\CatMon$ by $\otimes$.
If we need to specify the category in which the tensor product is taken, we will also write $\otimes_{\CatMon}$.

As we will see below, the tensor product in $\CatMon$ restricts to the tensor product in the categories of conical monoids, groups, semirings and rings.
Therefore, the tensor product in $\CatMon$ seems most universal and that is why we will usually drop the subscript to shorten notation.
\end{rmk}

%==========================================================================================
\begin{prp}
\label{prp:tensorMon}
Let $M$ and $N$ be monoids.
Then the monoid $M\otimes N$ and the monoid bimorphism
\[
\omega\colon M\times N\to M\otimes N
\]
constructed in \autoref{pgr:tensorMonConstr} have the following universal property:

For every monoid $R$ and for every monoid bimorphism $f\colon M\times N\to R$, there exists a unique monoid homomorphism $\tilde{f}\colon M\otimes N\to R$ such that $\tilde{f}\circ\omega=f$.

Thus, the assignment $g \mapsto g\circ\omega$ defines a map
\[
\CatMonMor(M\otimes N,R)\to \CatMonBimor(M\times N,R),
\]
%\xrightarrow{\cong}
which is a monoid isomorphism when considering the (bi)morphism sets as monoids.
\end{prp}
\begin{proof}
Let $M$ and $N$ be monoids.
To check the universal property of $\omega$, let $R$ and $f$ be as in the statement.
Since $F=\N[M\times N]$ is the free abelian monoid on the set $M\times N$, there is a unique monoid homomorphism $f'\colon F\to R$ such that $f'(a\odot b)=f(a,b)$ for every $(a,b)\in M\times N$.
It is routine to verify that $f'$ is constant on the congruence classes of~$\cong$.
Therefore, $f'$ induces a map
\[
\tilde{f}\colon M\otimes N = F/\!\!\cong \to R.
\]
It is clear that $\tilde{f}$ is a monoid homomorphism.
The rest of the statement is easy to check.
\end{proof}

%==========================================================================================
\begin{pgr}
\label{pgr:monoidalMon}
Using that the monoid bimorphisms are functorial in the first two entries, the tensor product in $\CatMon$ induces a bifunctor
\[
\otimes\colon\CatMon\times\CatMon\to\CatMon.
\]
In \autoref{pgr:functorialityTensor}, we explain this in more detail in the setting of enriched categories.

We remark that the category $\CatMon$ is enriched over the symmetric, closed, monoidal category of sets.
With this viewpoint, the tensor product in $\CatMon$ fits into the framework developed in \autoref{sec:abstractTensProd}.

Let $M,N$ and $R$ be monoids.
It is easy to verify that there is a natural isomorphism
\[
M\otimes N \cong N\otimes M,
\]
identifying the simple tensor $a\otimes b$ with $b\otimes a$, for any $a\in M$ and $b\in N$.
Similarly, there is a natural isomorphism
\[
(M\otimes N)\otimes R
\cong M\otimes(N\otimes R),
\]
identifying the simple tensor $(a\otimes b)\otimes c$ with $a\otimes(b\otimes c)$ for $a\in M$, $b\in N$ and $c\in R$.

The monoid $\N$ acts as a unit for the tensor product, that is, there are natural isomorphisms
\[
\N\otimes M \cong M \cong M\otimes\N.
\]
One can show that this gives to the category $\CatMon$ the structure of a symmetric, monoidal category.

Given monoids $M$ and $N$, we have seen that $\CatMonMor(M,N)$ has a natural structure as a monoid.
Thus, we can consider $\CatMonMor(M,N)$ as an object in $\CatMon$.
It follows that the category $\CatMon$ has an internal hom-bifunctor, which is equal to the (given) hom-bifunctor $\CatMonMor(\freeVar,\freeVar)$.

Moreover, the category $\CatMon$ is closed. Given a monoid $N$, the internal hom-bifunctor $\CatMonMor(N,\freeVar)$ is a right adjoint to the functor $\freeVar\otimes_{\CatMon} N$.
Indeed, given monoids $M$ and $R$, there are natural isomorphisms between the following monoids:
\[
\CatMonBimor(M\times N,R)
\cong \CatMonMor(M\otimes N,R)
\cong \CatMonMor(M,\CatMonMor(N,R)).
\]
\end{pgr}

%------------------------------------------------------------------------------------------
Next, we consider two important subcategories of $\CatMon$:
The category $\CatGp$ of (abelian) groups, and the category $\CatCon$ of conical monoids.

%==========================================================================================
\begin{pgr}
\label{pgr:gp}
\index{symbols}{Gp@$\CatGp$ \quad (category of groups)}
We let $\CatGp$ denote the category of (abelian) groups.
A map between two groups is a group homomorphism if and only if it is a monoid homomorphism.
Therefore, by considering a group as a monoid, we think of $\CatGp$ as a full subcategory of $\CatMon$.

There are two important observations:
\beginEnumStatements
\item
The category $\CatGp$ is reflective in $\CatMon$.
\item
The category $\CatGp$ is closed under the tensor product in $\CatMon$.
\end{enumerate}

Indeed, given a monoid $M$, its reflection in $\CatGp$ is the Grothendieck completion $\Gr(M)$ of $M$.
This induces a reflection functor
\[
\Gr\colon\CatMon\to\CatGp.
\]
\index{symbols}{Gr(M)@$\Gr(M)$ \quad (Grothendieck completion)}
Given groups $G$ and $H$, their tensor product as monoids, $G\otimes_{\CatMon}H$, is in fact a group.
Indeed, for a pair $(a,b)\in G\times H$, the inverse of the simple tensor $a\otimes b$ is equal to $(-a)\otimes b$.

In general, the tensor product of a monoidal category induces a tensor product in every reflective subcategory.
Given two objects in the subcategory, their tensor product in the subcategory is the reflection of their tensor product in the containing monoidal category.

In the concrete case of $\CatGp$ and $\CatMon$, this agrees with the tensor product of $\CatGp$ constructed above.
Indeed, given groups $G$ and $H$, their (abstract) tensor product is defined as
\[
\Gr(G\otimes_{\CatMon}H),
\]
the Grothendieck completion of their tensor product as monoids.
However, as observed above, the monoid $G\otimes_{\CatMon}H$ is automatically a group and therefore
\[
G\otimes_{\CatGp}H = \Gr(G\otimes_{\CatMon}H) = G\otimes_{\CatMon}H.
\]
Hence, it is unambiguous to write $G\otimes H$ for the tensor product of $G$ and $H$.

More generally, given monoids $M$ and $N$, there is a natural isomorphism between $\Gr(M\otimes N)$, the Grothendieck completion of their tensor product as monoids, and $\Gr(M)\otimes\Gr(N)$, the tensor product of their respective Grothendieck com\-ple\-tions.
\end{pgr}

%==========================================================================================
\begin{prp}[{Fulp, \cite[Proposition~17]{Ful70Tensor}}]
\label{prp:GrTensorMon}
Let $M$ and $N$ be monoids.
Then
\[
\Gr(M\otimes N) \cong \Gr(M)\otimes \Gr(N).
\]
\end{prp}

%==========================================================================================
\begin{pgr}
\label{pgr:con}
\index{terms}{conical monoid}
\index{terms}{monoid!conical}
\index{symbols}{Con@$\CatCon$ \quad (category of conical monoids)}
A monoid $M$ is \emph{conical} if $a+b=0$ implies $a=b=0$ for any $a,b\in M$.
Equivalently, the subset $M^\times$ of nonzero elements is a subsemigroup.
This property has appeared in the literature under many different names;
see \cite[p.268]{Weh96TensorInterpol}.

We let $\CatCon$ denote the full subcategory of $\CatMon$ consisting of conical (abelian) monoids.
Analogous to the category of groups, we have the following facts:
\beginEnumStatements
\item
The category $\CatCon$ is reflective in $\CatMon$.
\item
The category $\CatCon$ is closed under the tensor product in $\CatMon$.
\end{enumerate}

Given a monoid $M$, we let $U(M)$ denote the subgroup of units, that is,
\[
U(M) = \left\{ a\in M : a+b=0 \text{ for some } b\in M \right\}.
\]
Then $M$ is conical if and only if $U(M)=\{0\}$.

We define a binary relation $\sim$ on $M$ by setting $a\sim b$ if and only if there exist $x,y\in U(M)$ such that $a+x=b+y$, for any $a,b\in M$.
It is easy to check that $\sim$ is a congruence relation on $M$.
We set
\[
M_{\CatCon} := M/\!\!\sim,
\]
the set of congruence classes in $M$.
Then $M_{\CatCon}$ is a conical monoid, which is the reflection of $M$ in $\CatCon$.
This induces a reflection functor from $\CatMon$ to $\CatCon$.

Given monoids $M$ and $N$, it is shown in \cite[Corollary~8]{Ful70Tensor} that
\[
U(M\otimes N) \cong U(M)\otimes U(N).
\]
Thus, if $M$ and $N$ are conical monoids, their tensor product in $\CatMon$ is also conical.
Therefore, it is unambiguous to write $M\otimes N$ for the tensor product of $M$ and $N$.
\end{pgr}

%==========================================================================================
\begin{pgr}
\label{pgr:tensorConConstr}
Let us recall a different construction of the tensor product of conical monoids, as considered by Wehrung in \cite{Weh96TensorInterpol}.
Let $M$ and $N$ be conical monoids.
Set $M^\times:=M\setminus\{0\}$.
Since $M$ is conical, $M^\times$ is a subsemigroup of $M$.
Analogously, $N^\times:=N\setminus\{0\}$ is a subsemigroup of $N$.

In \autoref{pgr:tensorMonConstr}, we considered the binary relations $\rightarrow^0$ and $\rightarrow$ on the free monoid $\N[M\times N]$.
Now, we consider the free monoid
\[
F=\N[M^\times \times N^\times].
\]
We can define binary relations $\rightarrow^0$ and $\rightarrow$ on $F$ as in \autoref{pgr:tensorMonConstr}, and this is in fact the original definition of $\rightarrow^0$ and $\rightarrow$ as in \cite[Section~2]{Weh96TensorInterpol}.

We let~$\cong$ be the congruence relation on~$F$ generated by~$\rightarrow$ and~$\leftarrow$. (Note that, since we are taking elements in $M^\times\times N^\times$, the binary relation $\cong_0$ does not make sense.)
Thus, for elements~$f$ and~$g$ in~$F$ we have $f\cong g$ if and only if there are $n\in\N$ and elements $f_k,f_k'\in F$ for $k=0,\ldots,n$ such that $f=f_0$, $f_n=g$ and $f_k \rightarrow f_k' \leftarrow f_{k+1}$ for each $k<n$:
\[
f=f_0 \rightarrow f_0' \leftarrow
f_1 \rightarrow f_1' \leftarrow \ldots
f_n \rightarrow f_n' = g.
\]
We set
\[
M\otimes_{\CatCon}N := F/\!\!\cong.
\]
It is clear that $M\otimes_{\CatCon}N$ is an abelian semigroup.
Using that $F$ is conical, it follows from the definition of~$\leftarrow$ and~$\rightarrow$ that~$f\to 0$ or~$0\to f$ implies $f=0$ for any $f\in F$.
Therefore, the congruence class of the element $0$ contains only $0$ itself.
It follows that $F\setminus\{0\}$ is a subsemigroup of $F$ that is closed under the congruence relation.
Thus, $M\otimes_{\CatCon}N$ is a conical monoid.

The natural map from $\N[M^\times \times N^\times]$ to $\N[M\times N]$ induces a map
\[
M\otimes_{\CatCon} N
=\N[M^\times \times N^\times]_{/\langle \leftarrow,\rightarrow\rangle}
\to \N[M\times N]_{/\langle \leftarrow,\rightarrow,\cong_0\rangle}
= M\otimes N,
\]
which is easily checked to be an isomorphism.
\end{pgr}

\vspace{5pt}
%------------------------------------------------------------------------------------------
%==========================================================================================
\addtocontents{toc}{\SkipTocEntry}
\section[The categories \texorpdfstring{$\CatPrePom$}{PrePOM} and \texorpdfstring{$\CatPom$}{POM}]{The categories \texorpdfstring{$\CatPrePom$}{PrePOM} and \texorpdfstring{$\CatPom$}{POM} of positively (pre)ordered monoids}
\addtocontents{toc}{\protect\contentsline{section}%
{\protect\tocsection{}{\thesection}%
{The categories PrePOM and POM}}%
{\thepage}{section.\thesection}}
\label{sec:pom}

%==========================================================================================
\begin{pgr}
\label{pgr:pom}
\index{terms}{partially ordered monoid}
\index{terms}{positively ordered monoid}
\index{terms}{monoid!partially ordered}
\index{terms}{monoid!positively ordered}
\index{terms}{Pom@$\CatPom$}
\index{symbols}{PrePom@$\CatPrePom$ \quad (category of positively pre-ordered monoids)}
\index{symbols}{Pom@$\CatPom$ \quad (category of positively ordered monoids)}
A \emph{partially ordered monoid} is a monoid $M$ with a partial order $\leq$ such that $a\leq b$ implies $a+c\leq b+c$, for any $a,b,c\in M$.
If, in addition, we have $0\leq a$ for every $a\in M$, then we call $M$ a \emph{\pom}.

If the order is not necessarily antisymmetric, we speak of a \emph{(positively) pre-ordered monoid}.
This terminology follows Wehrung, \cite{Weh92InjectivePOM}.
Note, however, that in \cite{Weh92InjectivePOM} a `positively ordered monoid' (abbreviated by P.O.M.\ there) is only assumed to be pre-ordered.
We include the assumption of antisymmetry in our definition of `\pom' since our focus is on partially ordered structures.

A \emph{$\CatPom$-morphism} is an order-preserving monoid homomorphism.
Let $\CatPrePom$ denote the category of \prePom{s} together with $\CatPom$-mor\-phisms.
We let $\CatPom$ be the full subcategory of \pom{s}.

Let $M, N$ and $R$ be \prePom{s}.
Given $\CatPom$-morphisms $f,g\colon M\to R$, we set $f\leq g$ if and only if $f(a)\leq g(a)$ for each $a\in M$.
This defines a positive pre-order on $\CatPomMor(M,R)$.
Together with pointwise addition, this endows the set $\CatPomMor(M,R)$ of $\CatPom$-morphisms with the structure of a \prePom.
If the order of $R$ is antisymmetric, then this is also the case for the order of $\CatPomMor(M,R)$.
One can extend the assignment $R\mapsto\CatPomMor(M,R)$ to the following functors:
\[
\CatPomMor(M,\freeVar)\colon\CatPrePom\to\CatPrePom,\quad
\CatPomMor(M,\freeVar)\colon\CatPom\to\CatPom.
\]
The functor on the left is the internal hom-bifunctor of $\CatPrePom$.
If $M$ is in $\CatPom$, then the functor on the right is the internal hom-bifunctor of $\CatPom$.

A \emph{$\CatPom$-bimorphism} from $M\times N$ to $R$ is a map $f\colon M\times N\to R$ that is a $\CatPom$-morphism in each variable.
We denote the set of such bimorphisms by $\CatPomBimor(M\times N,R)$.
It is a \prePom\ when endowed with pointwise order and addition.
If the order of $R$ is antisymmetric, then so is the order of $\CatPomBimor(M\times N,R)$.
One can extend the assignment $R\mapsto\CatPomBimor(M\times N,R)$ to the following functors:
\[
\CatPomBimor(M\times N,\freeVar)\colon\CatPrePom\to\CatPrePom,\quad
\CatPomBimor(M\times N,\freeVar)\colon\CatPom\to\CatPom.
\]

Since our focus is on the category $\CatPom$, we denote the (bi)morphisms in both $\CatPom$ and $\CatPrePom$ as $\CatPom$-(bi)morphisms.
\end{pgr}

%==========================================================================================
\begin{pgr}
\label{pgr:reflectPom-PrePom}
Let $M$ be a \prePom.
We define a relation on $M$ by setting $a\equiv b$ if and only if $a\leq b$ and $b\leq a$, for any $a,b\in M$.
Then $\equiv$ is a congruence relation.
We define $\mu(M)$ as $M/\!\!\equiv$ and we let $\beta\colon M\to\mu(M)$ denote the quotient map.
The pre-order on $M$ induces a partial order on $\mu(M)$.
This gives $\mu(M)$ the structure of a \pom\ and it follows that $\beta$ is a $\CatPom$-morphism.

The assignment $M\mapsto\mu(M)$ extends to a functor
\[
\mu\colon\CatPrePom\to\CatPom,
\]
which is left adjoint to the inclusion of $\CatPom$ in $\CatPrePom$.
More precisely, for any \pom\ $R$, the following universal properties hold:
\beginEnumStatements
\item
For every $\CatPom$-morphism $f\colon M\to R$, there is a unique $\CatPom$-morphism $\tilde{f}\colon \mu(M)\to R$ such that $\tilde{f}\circ\beta=f$.
\item
If $g_1,g_2\colon \mu(M)\to R$ are $\CatPom$-morphisms, then $g_1\leq g_2$ if and only if $g_1\circ\beta\leq g_2\circ\beta$.
\end{enumerate}
Thus, the assignment $g \mapsto g\circ\beta$ defines a map
\[
\CatPomMor(\mu(M),R) \to \CatPomMor(M,R),
\]
which is a $\CatPom$-isomorphism when considering the (bi)morphism sets as \pom{s}.
\end{pgr}
%\xrightarrow{\cong}

%==========================================================================================
\begin{prp}
\label{prp:reflectPom-PrePom}
The category $\CatPom$ is a full, reflective subcategory of the category $\CatPrePom$.
\end{prp}

%==========================================================================================
\begin{pgr}[Tensor product in $\CatPrePom$]
\label{pgr:tensorPrePomConstr}
\index{symbols}{$\otimes_{\CatPrePom}$}
Let $M$ and $N$ be \prePom s.
We first consider the tensor product of the underlying monoids as constructed in \autoref{pgr:tensorMonConstr}.
Set $F:=\N[M\times N]$.

We define binary relations on $F$ as follows:
\begin{enumerate}
\item
We set $f\leq^0g$ if and only if there are $(a,b)$ and $(\tilde{a},\tilde{b})$ in $M\times N$ such that
\[
a\leq\tilde{a},\quad
b\leq\tilde{b},\quad
f=a\odot b,\text{ and }
g=\tilde{a}\odot\tilde{b}.
\]
\item
We set $f\leq' g$ if and only if $f=0$ or if there are $n\in\N$ and $f_k,g_k\in F$ for $k=0,\ldots,n$ such that
\[
f_k\leq^0g_k \text{ for each } k,\quad
f=\sum_{k=0}^nf_k,\text{ and }
g=\sum_{k=0}^ng_k.
\]
\end{enumerate}
Recall that $\cong$ is the congruence relation on $F$ generated by $\leftarrow$, $\rightarrow$ and $\cong_0$.
We let $\leq$ be the transitive relation on $F$ generated by $\cong$ and $\leq'$.
Thus, for any $f,g\in F$ we have $f\leq g$ if and only if there is $n\in\N$ and elements $f_k,f_k'\in F$ for $k=0,\dots,n$ such that $f=f_0$, $g=f_n$ and $f_k \leq' f_k' \cong f_{k+1}$ for all $k<n$.

It is easy to see that $\leq$ is a positive pre-order on $F$.
This induces a positive pre-order on $M\otimes_{\CatMon}N = F/\!\!\cong$.
We denote the resulting positively pre-ordered monoid by $M\otimes_{\CatPrePom}N$.

By construction, the universal monoid-bimorphism
\[
\omega\colon M\times N\to M\otimes N
\]
is order-preserving in each variable.
We may therefore consider $\omega$ as a $\CatPom$-bimorphism from $M\times N$ to $M\otimes_{\CatPrePom}N$.
\end{pgr}

%==========================================================================================
\begin{prp}
\label{prp:tensorPom}
\index{symbols}{$\otimes_{\CatPom}$}
Let $M$ and $N$ be \prePom{s}.
Then the \prePom{} $M\otimes_{\CatPrePom} N$ and the $\CatPom$-bimorphism
\[
\omega\colon M\times N\to M\otimes_{\CatPrePom}N
\]
constructed in \autoref{pgr:tensorPrePomConstr} satisfy the following universal properties for each \prePom{} $R$:
\beginEnumStatements
\item
For every $\CatPom$-bimorphism $f\colon M\times N\to R$ there exists a unique $\CatPom$-morphism $\tilde{f}\colon M\otimes_{\CatPrePom} N\to R$ such that $\tilde{f}\circ\omega=f$.
\item
If $g_1,g_2\colon M\otimes_{\CatPrePom} N\to R$ are $\CatPom$-morphisms, then $g_1\leq g_2$ if and only if $g_1\circ\omega\leq g_2\circ\omega$.
\end{enumerate}
Thus, the assignment $g \mapsto g\circ\omega$ defines a map
\[
\CatPomMor(M\otimes_{\CatPrePom} N,R) \to \CatPomBimor(M\times N,R),
\]
which is a $\CatPom$-isomorphism when considering the (bi)morphism sets as \prePom s.
%\xrightarrow{\cong}

Moreover, the reflection $\mu\colon\CatPrePom\to\CatPom$ induces a tensor product in $\CatPom$.
More precisely, given \pom{s} $M$ and $N$, we set
\[
M\otimes_{\CatPom} N := \mu(M\otimes_{\CatPrePom}N).
\]
The composition $\beta\circ\omega\colon M\times N\to M\otimes_{\CatPom} N$ is a $\CatPom$-bimorphism which has the analogous universal properties of the tensor product in $\CatPom$.
\end{prp}
\begin{proof}
Let $M$ and $N$ be \prePom s.
To check the universal property of $\omega$, let $R$ and $f$ be as in the statement.
Since the underlying monoid of $M\otimes_{\CatPrePom} N$ is the tensor product in $\CatMon$, there is a unique monoid homomorphism $\tilde{f}\colon M\otimes_{\CatMon}N \to R$ such that $\tilde f\circ\omega=f$.
It follows from the definition of the pre-order on $M\otimes_{\CatMon}N$ that $\tilde{f}$ is order-preserving.
This proves \enumStatement{1}, and the statement \enumStatement{2} is left to the reader.

Now, let $M$ and $N$ be \pom{s}.
Define $M\otimes_{\CatPom} N$ as the reflection of $M\otimes_{\CatPrePom} N$ in $\CatPom$.
To show that this has the analogous universal properties, let $R$ be a \pom.
Given a $\CatPom$-morphism $f\colon M\otimes_{\CatPom} N\to R$, we consider the maps $f\circ\beta$ and $f\circ\beta\circ\omega$ shown in the following commutative diagram:
\[
\xymatrix@M+2pt{
{ M\otimes_{\CatPom} N } \ar@{=}[r]
& { \mu(M\otimes_{\CatPrePom} N) }  \ar[dr]^{f}
& { M\otimes_{\CatPrePom} N }  \ar[l]^{\beta} \ar[d]^{f\circ\beta}
& M\times N \ar[dl]^{f\circ\beta\circ\omega} \ar[l]^{\omega} \\
& & R
}
\]
It follows from \autoref{pgr:reflectPom-PrePom} and the universal property of the tensor product in $\CatPrePom$ that this induces bijective maps
\[
\CatPomMor(M\otimes_{\CatPom} N,R) \xrightarrow{\cong} \CatPomMor(M\otimes_{\CatPrePom} N,R) \xrightarrow{\cong} \CatPomBimor(M\times N,R),
\]
which are $\CatPom$-isomorphism when considering the (bi)morphism sets as \pom{s}.
\end{proof}

%==========================================================================================
\begin{pgr}
\label{pgr:monoidalPom}
Analogous to \autoref{pgr:monoidalMon}, we  one proves that $\CatPrePom$ and $\CatPom$ are closed, symmetric, monoidal categories.
\end{pgr}

%==========================================================================================
\begin{pgr}
\label{pgr:algOrder}
\index{terms}{algebraic pre-order}
\index{terms}{algebraic order}
Let us clarify the connection between the categories $\CatMon$, $\CatCon$, $\CatPrePom$ and $\CatPom$.
We have already observed that $\CatCon$ is a full, reflective subcategory of $\CatMon$ and that $\CatPom$ is a full, reflective subcategory of $\CatPrePom$.

To every \prePom{} we may associate its underlying additive monoid.
This induces the forgetful functor
\[
\mathfrak{F}\colon\CatPrePom\to\CatMon.
\]
Conversely, let $M$ be a monoid.
The \emph{algebraic pre-order} on $M$ is defined as follows:
Given $a,b\in M$, we set $a\leq_\alg b$ if and only if there exists $x\in M$ such that $a+x=b$.
It is clear that $\leq_\alg$ is a positive pre-order on $M$.
Given monoids $M$ and $N$, every monoid homomorphism $f\colon M\to N$ becomes a $\CatPom$-morphism when $M$ and $N$ are equipped with their respective algebraic pre-orders.
This defines a functor
\[
\mathfrak{A}\colon\CatMon\to\CatPrePom,
\]
which assigns to a monoid $M$ the \prePom{} $(M,\leq_\alg)$.

We say that a \prePom{} $M$ is \emph{algebraically pre-ordered}, or we simply say that $M$ is an \emph{algebraically pre-ordered monoid}, if it is equipped with the algebraic pre-order of the underlying monoid.
It is easy to check that $\mathfrak{A}$ is a fully faithful functor.
Thus, we may identify $\CatMon$ with the full subcategory of $\CatPrePom$ consisting of algebraically pre-ordered monoids.

Moreover, we have the following:
\beginEnumStatements
\item
The category $\CatMon$ is reflective in $\CatPrePom$.
\item
The property of being algebraically pre-ordered is closed under tensor products in $\CatPrePom$;
see \autoref{prp:tensPomPreservesAlgOrd}.
Thus, the category $\CatMon$ considered as a subcategory of $\CatPrePom$ is closed under the tensor product in $\CatPrePom$.
\end{enumerate}

To see that $\CatMon$ is a reflective subcategory of $\CatPrePom$, let us show that the forgetful functor $\mathfrak{F}$ is a left adjoint to the inclusion $\mathfrak{A}$.
Indeed, given a monoid $M$ and a \prePom{} $R$, observe that every monoid homomorphism $f\colon M\to\mathfrak{F}(R)$ is automatically order-preserving as a map from $M$ (with the algebraic pre-order) to $R$.
Thus, we have a natural bijection of the following morphism-sets:
\[
\CatMonMor(M,\mathfrak{F}R) \cong \CatPomMor(\mathfrak{A}(M),R).
\]
\end{pgr}

%==========================================================================================
\begin{prp}
\label{prp:tensPomPreservesAlgOrd}
Let $M$ and $N$ be algebraically pre-ordered monoids.
Then $M\otimes_{\CatPrePom}N$ is algebraically pre-ordered.
\end{prp}
\begin{proof}
We use the notation that was introduced in \autoref{pgr:tensorMonConstr} and \autoref{pgr:tensorPrePomConstr}.
Thus, we consider the monoid $F=\N[M\times N]$ and the congruence relation $\cong$ on~$F$ generated by $\leftarrow$, $\rightarrow$ and $\cong_0$.

Claim 1:
If $f,g\in F$ satisfy $f\leq^0 g$, there exists $h\in F$ such that $f+h\cong g$.
To prove the claim, assume that $f,g\in F$ satisfying $f\leq^0 g$ are given.
By definition there are $(a,b)$ and $(\tilde{a},\tilde{b})$ in $M\times N$ such that
\[
a\leq\tilde{a},\quad
b\leq\tilde{b},\quad
f=a\odot b,\text{ and }
g=\tilde{a}\odot\tilde{b}.
\]
Since $M$ and $N$ are algebraically pre-ordered we can choose $x\in M$ and $y\in N$ such that $a+x=\tilde{a}$ and $b+y=\tilde{b}$. Set $h:=x\odot b+\tilde{a}\odot y$.
Then
\[
f + h
= a\odot b + x\odot b + \tilde{a}\odot y
\cong \tilde{a}\odot b + \tilde{a}\odot y
\cong \tilde{a}\odot \tilde{b}
= g,
\]
as desired.

Claim 2:
If $f,g\in F$ satisfy $f\leq' g$, then there exists $h\in F$ such that $f+h\cong g$.
This follows from claim 1 since $\leq'$ is defined as the additive closure of $\leq^0$, and since $\cong$ is an additive relation.

Next, let us show that the pre-order of $M\otimes_{\CatPrePom}N$ is algebraic.
The underlying monoid of $M\otimes_{\CatPrePom}N$ is equal to $M\otimes_{\CatMon}N = F/\!\!\cong$.
The pre-order $\leq$ of $F$ is defined as the transitive relation generated by $\leq'$ and $\cong$.
This induces the pre-order of $M\otimes_{\CatPrePom}N$, which by abuse of notation is also denoted by $\leq$.
Now, let $x,y\in M\otimes_{\CatPrePom}N$ satisfy $x\leq y$.
Choose representatives $f$ and $g$ in $F$ such that $x=[f]$ and $y=[g]$.
Then $f\leq g$.
This means that there are $n\in\N$ and elements $f_k,f_k'\in F$ for $k\leq n$ such that
\[
f=f_0\leq' f_0' \cong f_1 \leq' f_1' \cong \ldots  \cong f_n \leq' f_n' = g.
\]
For each $k\leq n$ we have $f_k\leq' f_k'$.
By claim 2 we can choose $h_k\in F$ such that $f_k+h_k \cong f_k'$.
Then
\[
f_k+h_k \cong f_{k+1},
\]
for each $k\leq n$.
Set $h:=h_0+\ldots+h_n$.
It follows that
\[
f+h
= f_0 + \sum_{k=0}^n h_k
\cong f_1 + \sum_{k=1}^n h_k
\cong f_2 + \sum_{k=2}^n h_k
\cong \ldots
\cong g.
\]
Thus, $f+h\cong g$, which implies that $x+[h]=y$ in $M\otimes_{\CatPrePom}N$, as desired.
\end{proof}

%==========================================================================================
\begin{pgr}
The underlying monoid of a \pom{} is conical.
Therefore, the forgetful functor $\mathfrak{F}\colon\CatPrePom\to\CatMon$ considered in \autoref{pgr:algOrder} restricts to a functor
\[
\mathfrak{F}\colon\CatPom\to\CatCon,
\]
which, by abuse of notation, is also denoted by $\mathfrak{F}$.

However, the algebraic pre-order on a conical monoid is not necessarily antisymmetric.
Therefore, the functor $\mathfrak{A}\colon\CatMon\to\CatPrePom$ from \autoref{pgr:algOrder} does not restrict to a functor from $\CatMon$ to $\CatPom$.

We say that a \pom{} $M$ is \emph{algebraically ordered}, or we simply say that $M$ is an \emph{algebraically ordered monoid}, if it is equipped with the algebraic partial order of the underlying monoid.
\end{pgr}

\vspace{5pt}
%------------------------------------------------------------------------------------------
%==========================================================================================
\section{The category \texorpdfstring{$\CatPoGp$}{POGp} of partially ordered groups}
\label{sec:poGp}

%==========================================================================================
\begin{pgr}
\label{pgr:pog}
\index{terms}{partially ordered group}
\index{symbols}{PoGp@$\CatPoGp$ \quad (category of partially ordered groups)}
A \emph{\pog} is an (abelian) group with a partial order $\leq$ such that
for any group elements $a,b$ and $c$, $a\leq b$ implies $a+c\leq b+c$.
Given \pog{s} $G$ and $R$, a \emph{$\CatPoGp$-morphism} from $G$ to $R$ is an order-preserving group homomorphism.
We denote the collection of such maps by $\CatPoGpMor(G,R)$.
We let $\CatPoGp$ denote the category of \pog{s} and $\CatPoGp$-morphisms.

Given another \pog\ $H$, a \emph{$\CatPoGp$-bimorphism} from $G\times H$ to $R$ is a map $f\colon G\times H\to R$ that is a group homomorphism in each variable and such that whenever $a_1\leq a_2\in G$ and $b_1\leq b_2\in H$, then $f(a_1,b_1)\leq f(a_2,b_2)$.
We denote the collection of such bimorphisms by $\CatPoGpBimor(G\times H,R)$.

Note that for a $\CatPoGp$-bimorphism $f\colon G\times H\to R$ and a fixed $a\in G$, the map $
H\to R$, defined by the assignment $b\mapsto f(a,b)$, for $b\in H$,
is not necessarily order-preserving unless $a\geq 0$.
\end{pgr}

%==========================================================================================
\begin{pgr}
\label{pgr:equivalenceMonPoGp}
\index{terms}{positively ordered monoid!cancellative}
\index{terms}{partially ordered group!directed}
\index{terms}{positive cone}
\index{terms}{partially ordered group!directed}
\index{symbols}{G$_+$@$G_+$ \quad (positive cone)}
To clarify the connection between $\CatPoGp$ and $\CatMon$, we need to recall the following notion.
We say that a \pom\ $M$ has \emph{cancellation} (or that $M$ is \emph{cancellative}) if for any $a,b,x\in M$ we have $a+x\leq b+x$ implies $a\leq b$.
Let $\CatMon_\txtCanc$ denote the full subcategory of $\CatMon$ of cancellative monoids.
Similarly, we let $\CatCon_\txtCanc$ be the full subcategory of $\CatCon$ of cancellative, conical monoids.

Let $G$ be a \pog.
The \emph{positive cone} of $G$ is defined as
\[
G_+ = \left\{a\in G : 0\leq a \right\}.
\]
It is easy to check that $G_+$ satisfies
\[
G_+\cap(-G_+)=\{0\},\quad
G_++G_+\subset G_+.
\]
Therefore, $G_+$ is a conical submonoid of $G$.
Moreover, the order of $G$ induces the algebraic order on $G_+$.
Thus, considering $G_+$ as an algebraically ordered monoid, the inclusion of $G_+$ in $G$ is an order-embedding.

Now, let $G$ be a group and let $P$ be a conical submonoid of $G$.
This defines a partial order on $G$ by setting, for any $a,b\in G$, $a\leq b$ if and only if there exists $x\in P$ such that $a+x=b$.
For every group $G$ this establishes a natural one-to-one correspondence between:
\beginEnumStatements
\item
Partial orders on $G$ such that $(G,\leq)$ is a \pog.
\item
Conical submonoids of $G$.
\end{enumerate}
Since the positive cone of a \pog{} is automatically cancellative, there is a functor
\[
\mathfrak{P}\colon\CatPoGp\to\CatCon_\txtCanc,
\]
that assigns to a \pog{} $G$ its positive cone $G_+$.

Recall that a \pog{} $G$ is \emph{directed} if $G=G_+-G_+$.
We let $\CatPoGp_\txtDir$ denote the full subcategory of $\CatPoGp$ consisting of directed, \pog{s}.

Now, let $M$ be a monoid.
Consider the Grothendieck completion $\Gr(M)$ together with the universal map
\[
\delta\colon M\to\Gr(M).
\]
We have $\delta$ is injective if and only if $M$ is cancellative.
The image of $\delta$ is a submonoid of $\Gr(M)$.
If $M$ is conical then so is $\delta(M)$, whence it gives $\Gr(M)$ the structure of a \pog{} whose positive cone is $\delta(M)$.
Moreover, $\Gr(M)$ is directed.
This induces a functor
\[
\Gr\colon\CatCon\to\CatPoGp_\txtDir.
\]

A \pog{} $G$ is in the image of the functor $\Gr$ if and only if it is directed, in which case $G\cong \Gr(G_+)$.
Conversely, a conical monoid $M$ is in the image of the functor $\mathfrak{P}$ if and only if it is cancellative, in which case $M\cong \Gr(M)_+$.
We summarize this as follows:
\end{pgr}

%==========================================================================================
\begin{prp}
\label{prp:equivalenceMonPoGp}
The functors $\mathfrak{P}$ and $\Gr$ from \autoref{pgr:equivalenceMonPoGp} establish an equivalence between the following categories:
\beginEnumStatements
\item
The category $\CatPoGp_\txtDir$ of directed, \pog{s}.
\item
The category $\CatCon_\txtCanc$ of cancellative, conical monoids.
\end{enumerate}
\end{prp}

%==========================================================================================
\begin{pgr}
\label{pgr:tensorPoGp}
Let $G$ and $H$ be \pog{s}.
Consider the tensor product $G\otimes H$ of the underlying groups.
The map $G_+\times H_+\to G\otimes H$, that sends $(a,b)$ to $a\otimes b$, for $a\in G_+, b\in H_+$,
is a monoid bimorphism.
It therefore induces a monoid homomorphism
\[
\delta\colon G_+\otimes H_+ \to G\otimes H.
\]
Since $G_+$ and $H_+$ are conical, so is $G_+\otimes H_+$;
see \autoref{pgr:con}.

The image of the map $\delta$ is a conical submonoid of $G\otimes H$.
As explained in \autoref{pgr:equivalenceMonPoGp}, this induces a partial order on $G\otimes H$ whose positive cone is the image of $\delta$.
We denote the resulting \pog{} by $G\otimes_{\CatPoGp}H$, which is the tensor product of $G$ and $H$ in the category $\CatPoGp$;
see \cite[Proposition~1.1]{Weh96TensorInterpol}.

In \autoref{prp:equivalenceMonPoGp}, we have seen that the full subcategory $\CatPoGp_\txtDir$ of directed, \pog{s} is equivalent to $\CatCon_\txtCanc$.
This equivalence is also compatible with the tensor product.
Thus, given directed \pog{s} $G$ and $H$, we have
\[
G\otimes_{\CatPoGp}H \cong \Gr( G_+ \otimes H_+).
\]
Conversely, if $M$ and $N$ are cancellative, conical monoids, then the monoid
\[
\left( \Gr(M)\otimes_{\CatPoGp}\Gr(N)\right)_+
\]
is isomorphic to the reflection of $M\otimes N$ in the subcategory of cancellative, conical monoids.
A proof of these statements can be found in \cite[Proposition~1.2]{Weh96TensorInterpol}.

It is natural to ask whether the tensor product of two cancellative, conical monoids is again cancellative. As it turns out, the answer to this question is negative. We thank Fred Wehrung for showing us a counterexample.

In \cite[Examples 1.4 and 1.5]{Weh96TensorInterpol}, examples of partially ordered abelian groups with Riesz interpolation $G$ and $H$ such that $G\otimes_{\CatPoGp} H$ does not have interpolation are given. If we let $M=G_+$ and $N=H_+$, then $M$ and $N$ are conical, cancellative, monoids that satisfy the Riesz refinement property. It then follows from \cite[Theorem 2.9]{Weh96TensorInterpol} that $M\otimes N$ also satisfies the Riesz refinement property.
Since
\[
G\otimes_{\CatPoGp}H=\Gr(M)\otimes_{\CatPoGp}\Gr(N)\cong\Gr(M\otimes N)
\]
we conclude that $\Gr(M\otimes N)_+$ does not satisfy the Riesz refinement property, and hence $M\otimes N$ is not a cancellative monoid (as otherwise would be isomorphic to the positive cone of its Grothendieck completion).
\end{pgr}

\vspace{5pt}
%------------------------------------------------------------------------------------------
%==========================================================================================
\section{The category \texorpdfstring{$\CatPoRg$}{PORg} of partially ordered rings}
\label{sec:poRg}

%==========================================================================================
\begin{pgr}
\label{pgr:srg}
\index{terms}{semiring}
\index{terms}{ring}
A \emph{semiring} is a monoid $R$ together with a commutative and distributive multiplication with a unit element (which we denote by~$1$).
A semiring homomorphism is a multiplicative monoid homomorphism preserving the unit element.
We let $\CatSrg$ denote the category of semirings together with semiring homomorphisms.

Let $R$ and $S$ be semirings.
Consider the tensor product $R\otimes S$ of the underlying monoids.
Given $r_1,r_2\in R$ and $s_1,s_2\in S$, we define the product of simple tensors as
\[
(r_1\otimes s_1)(r_2\otimes s_2) = (r_1r_2)\otimes(s_1s_2).
\]
This extends to a well-defined, commutative multiplication on $R\otimes S$.
The element $1\otimes 1$ is a unit element.
It follows that $R\otimes S$ has a natural structure as a semiring.
If we want to stress that $R\otimes S$ has a semiring-structure, we write $R\otimes_{\CatSrg} S$.
This is the tensor product of $R$ and $S$ in the category of semirings.

All \emph{rings} will be unital and commutative.
Equivalently, a ring will be a semiring such that the underlying additive monoid is a group.
Given rings $R$ and $S$, the tensor product $R\otimes S$ of the underlying groups has a natural multiplication such that $R\otimes S$ is a ring.
This is the tensor product of $R$ and $S$ in the category of rings.
\end{pgr}

%==========================================================================================
\begin{pgr}
\label{pgr:por}
\index{terms}{partially ordered ring}
\index{terms}{ring!partially ordered}
\index{symbols}{PoRg@$\CatPoRg$ \quad (category of partially ordered rings)}
A \emph{\por} is a ring $R$ together with a partial order $\leq$ such that $0\leq 1$ and such that:
\beginEnumConditions
\item
If $a\leq b$, then $a+c\leq b+c$, for any $a,b,c\in R$.
\item
If $a\leq b$ and $0\leq c$, then $ac\leq bc$, for any $a,b,c\in R$.
\end{enumerate}
In particular, the underlying group of $R$ together with $\leq$ is a \pog.
This induces a forgetful functor
\[
\mathfrak{F}\colon\CatPoRg\to\CatPoGp.
\]
Let $R$ be a \por.
As for \pog s, we define the \emph{positive cone} of $R$ as
\[
R_+ =\left\{ a\in R : 0\leq a \right\}.
\]
As in the group case, we have $R_++R_+\subset R_+$ and $R_+\cap(-R_+)=\{0\}$.
Moreover, we have $R_+\cdot R_+\subset R_+$.
Thus, the positive cone is a cancellative, conical semiring.

Now let $R$ and $S$ be \por{s}.
Consider the tensor product $R\otimes S$ of the underlying rings.
As for \pog{s}, there is a natural partial order on $R\otimes S$ with positive cone given as the image of $R_+\otimes S_+$ in $R\otimes S$.
This partial order is compatible with the multiplication, and we denote the resulting \por\ by $R\otimes_{\CatPoRg}S$.
This is the tensor product of $R$ and $S$ in the category $\CatPoRg$.

Equivalently, we can consider $R\otimes_{\CatPoRg}S$ as the tensor product $R\otimes_{\CatPoGp}S$ of the underlying \pog{s} equipped with a product as in \autoref{pgr:srg}.
\end{pgr}

%==========================================================================================
\begin{pgr}
\label{pgr:equivalenceSrgPoRg}
Recall that $\CatSrg$ denotes the category of semirings.
We let $\CatCon\CatSrg$ denote the full subcategory of conical semirings.
Further, we let $\CatSrg_\txtCanc$ (respectively $\CatCon\CatSrg_\txtCanc$) denote the full subcategories of cancellative (conical) semirings.
We have a functor
\[
\mathfrak{P}\colon\CatPoRg\to\CatCon\CatSrg_\txtCanc,
\]
which assigns to a \por{} $R$ its positive cone $R_+$.

Conversely, it is easy to see that the Grothendieck completion of a semiring $S$ has a natural multiplication giving it the structure of a ring.
If $S$ is conical, the image of $S$ in $\Gr(S)$ defines a partial order on $\Gr(S)$.
Let $\CatPoRg_\txtDir$ be the full subcategory of $\CatPoRg$ consisting of directed \por{s}.
Then the Grothendieck completion induces a functor
\[
\Gr\colon\CatCon\CatSrg\to\CatPoRg_\txtDir.
\]
The situation is completely analogous to the connection between (cancellative) conical monoids and (directed) \pog{s} that was discussed in \autoref{pgr:equivalenceMonPoGp}.
Therefore, we have the following analog of \autoref{prp:equivalenceMonPoGp}:
\end{pgr}

%==========================================================================================
\begin{prp}
\label{prp:equivalenceSrgPoRg}
The functors $\mathfrak{P}$ and $\Gr$ from \autoref{pgr:equivalenceSrgPoRg} establish an equivalence between the following categories:
\beginEnumStatements
\item
The category $\CatPoRg_\txtDir$ of directed \por{s}.
\item
The category $\CatCon\CatSrg_\txtCanc$ of cancellative, conical semirings.
\end{enumerate}
\end{prp}

%==========================================================================================
\begin{pgr}
\label{pgr:tensorSrg}
In \autoref{prp:equivalenceSrgPoRg}, we have seen that the full subcategory $\CatPoRg_\txtDir$ of directed, \por{s} is equivalent to $\CatCon\CatSrg_\txtCanc$ of cancellative, conical semirings.
This equivalence is also compatible with the tensor product in the following sense:
Given directed \por{s} $R$ and $S$, we have
\[
R\otimes_{\CatPoRg}S \cong \Gr( R_+ \otimes S_+).
\]
Conversely, if $M$ and $N$ are cancellative, conical semirings, then the semiring
\[
\left( \Gr(M)\otimes_{\CatPoGp}\Gr(N)\right)_+
\]
is isomorphic to the reflection of $M\otimes N$ in the subcategory of cancellative, conical semirings.
\end{pgr}

%------------------------------------------------------------------------------------------
%==========================================================================================
%##########################################################################################
\backmatter

%\bibliographystyle{amsalphaMy}
%\bibliography{ReferencesMR}

%\printindex

\newcommand{\etalchar}[1]{$^{#1}$}
\providecommand{\bysame}{\leavevmode\hbox to3em{\hrulefill}\thinspace}
\providecommand{\MR}{\relax\ifhmode\unskip\space\fi MR }
\providecommand{\MRhref}[2]{%
  \href{http://www.ams.org/mathscinet-getitem?mr=#1}{#2}
}
\providecommand{\href}[2]{#2}

\renewcommand*{\indexname}{Index of Terms}
\begin{theindex}

  \item algebraic $\CatCu$-semigroup, 68
  \item algebraic $\CatCu$-semiring, 146
  \item algebraic order, 174
  \item algebraic pre-order, 174
  \item almost algebraic order, 31
  \item almost divisible, 119
  \item almost Riesz decomposition, 31
  \item almost unperforated, 53
  \item auxiliary relation, 11
    \subitem on $S\otimes_{\CatPom}T$, 90
    \subitem on (bi)morphism sets, 87

  \indexspace

  \item basis (of a $\CatPreW$-semigroup), 11

  \indexspace

  \item \ca{}
    \subitem local, 16
    \subitem pre-, 16
    \subitem purely infinite, 115
    \subitem simple, 43
    \subitem strongly self-absorbing, 2
  \item category
    \subitem $\CatCa$, 29
    \subitem $\CatCu$, 21
    \subitem $\CatLocCa$, 16
    \subitem $\CatPreW$, 12
    \subitem $\CatW$, 12
  \item compact containment, 11
  \item compactly generated space, 165
  \item completed Cuntz semigroup, 21
  \item congruence relation, 168
  \item conical monoid, 171
  \item countably-based, 11
  \item $\CatCu$-bimorphism, 94
    \subitem generalized, 94
  \item $\CatCu$-completion, 25
  \item $\CatCu$-morphism, 21
    \subitem generalized, 21
  \item $\CatCu$-product (of a $\CatCu$-semiring), 107
  \item $\CatCu$-semigroup, 21
    \subitem algebraic, 68
    \subitem elementary, 44
    \subitem residually stably finite, 59
    \subitem simple, 43
    \subitem simplicial, 72
    \subitem stably finite, 47
  \item $\CatCu$-semimodule, 107
  \item $\CatCu$-semiring, 107
    \subitem algebraic, 146
    \subitem solid, 109, 149
      \subsubitem characterization, 109
      \subsubitem classification, 156
  \item Cuntz equivalent, 17
  \item Cuntz semigroup
    \subitem completed, 28
    \subitem of Jacelon-Razak algebra, 132
    \subitem of Jiang-Su algebra, 118
    \subitem of UFH-algebra, 128
    \subitem pre-completed, 17
  \item Cuntz subequivalent, 17

  \indexspace

  \item dimension monoid, 72
  \item directed complete partially ordered set (dcpo), 22
  \item divisible
    \subitem $q$-, 129

  \indexspace

  \item element
    \subitem almost $k$-divisible, 119
    \subitem almost divisible, 119
    \subitem compact, 11
    \subitem finite, 47
    \subitem full, 35
    \subitem properly infinite, 47
    \subitem purely noncompact, 57
    \subitem soft, 53, 57
    \subitem weakly purely noncompact, 57
  \item Elliott's Classification Conjecture, 1
  \item enriched category, 163
  \item extended state, 45

  \indexspace

  \item finite element, 47
  \item functional, 45

  \indexspace

  \item ideal, 37
    \subitem lattice, 39
  \item inductive limits
    \subitem in $\CatCu$, 27
    \subitem in $\CatLocCa$, 19
    \subitem in $\CatPreW$, 15
    \subitem in $\CatW$, 16
  \item infinite element, 47
  \item internal hom-bifunctor, 164
  \item interval, 56
    \subitem countably generated, 69
    \subitem soft, 56

  \indexspace

  \item Jiang-Su algebra, 2, 73

  \indexspace

  \item local \ca, 16

  \indexspace

  \item monoid, 167
    \subitem conical, 171
    \subitem partially ordered, 172
    \subitem positively ordered, 172
    \subitem weakly divisible, 146
  \item monoidal category, 163
    \subitem closed, 164
    \subitem concrete, 163
    \subitem symmetric, 163

  \indexspace

  \item nearly unperforated, 73
  \item nearly unperforated Conjecture, 80
  \item no $K_1$-obstructions, 78
  \item nonelementary $\CatCu$-semigroup, 44
  \item nonelementary semiring, 147

  \indexspace

  \item \axiomO{1}-\axiomO{4}, 21
  \item \axiomO{5}, 31
  \item \axiomO{6}, 31
  \item order-hereditary, 37
  \item order-ideal, 37

  \indexspace

  \item partially ordered group, 176
    \subitem directed, 176
  \item partially ordered monoid, 172
  \item partially ordered ring, 178
  \item $\CatPom$, 172
  \item positive cone, 176
  \item positively ordered monoid, 172
    \subitem almost divisible, 119
    \subitem almost unperforated, 53
    \subitem cancellative, 69, 176
    \subitem nearly unperforated, 73
    \subitem preminimally ordered, 74
    \subitem $q$-divisible, 129
    \subitem $q$-unperforated, 129
    \subitem simple, 74
    \subitem simplicial, 72
    \subitem stably finite, 74
    \subitem unperforated, 72
    \subitem weakly separative, 74
  \item \preCa, 16
  \item predecessor, 65
  \item preminimally ordered, 74
  \item $\CatPreW$-semigroup, 12
  \item properly infinite, 115
  \item properly infinite element, 47
  \item property
    \subitem $(QQ)$, 59
    \subitem $(RQQ)$, 59
  \item purely infinite \ca{}, 115
  \item purely noncompact, 57

  \indexspace

  \item $R$-multiplication, 107
  \item Range problem, 5
  \item rapidly increasing sequence, 21
  \item rationalization, 126
    \subitem $\CatCu$-semigroup, 128
  \item real multiplication, 132
  \item realification, 134
  \item regular relation, 53
  \item Regularity Conjecture, 3, 124
  \item regularization of a relation, 53
  \item relation
    \subitem compact containment, 11
    \subitem Cuntz equivalence, 17
    \subitem Cuntz subequivalence, 17
    \subitem stable domination, 52
    \subitem way-below \seeonly{compact containment}, 11
  \item representable functor, 82
  \item Riesz decomposition property, 69
  \item Riesz interpolation property, 69
  \item Riesz refinement property, 69
  \item ring, 178
    \subitem partially ordered, 178
    \subitem solid, 109

  \indexspace

  \item Scott-topology, 22
  \item semigroup
    \subitem idempotent, 113
  \item semiring, 178
    \subitem nonelementary, 147
    \subitem solid, 151
  \item simple ($\CatCu$-semigroup), 43
  \item simple (positively ordered monoid), 74
  \item simplicial monoid, 72
  \item soft, 57
  \item soft element, 53
  \item soft interval, 56
  \item solid $\CatCu$-semiring, 149
  \item stably dominated, 52
  \item stably finite $\CatCu$-semigroup, 47
  \item stably finite positively pre-ordered monoid, 74
  \item state, 45
  \item strict comparison, 137
  \item strongly self-absorbing \ca{}, 2
  \item supernatural number, 126

  \indexspace

  \item tensor product, 84
    \subitem in $\CatCu$, 94
      \subsubitem associative, 96
      \subsubitem continuous, 96
    \subitem in $\CatPreW$, 90
      \subsubitem continuous, 96
    \subitem in $\CatW$, 95
  \item Toms-Winter Conjecture, 3, 124

  \indexspace

  \item UHF-algebra
    \subitem infinite type, 128
  \item unperforated, 72
    \subitem almost, 53
    \subitem nearly, 73
    \subitem $q$-, 129

  \indexspace

  \item $\CatV$-category, 163
  \item $\CatV$-functor, 163
  \item $\CatV$-natural transformation, 163

  \indexspace

  \item $\CatW$-bimorphism, 86
    \subitem generalized, 86
  \item $\CatW$-completion, 13
  \item $\omega$-continuous, 22
  \item $\CatW$-morphism
    \subitem generalized, 12
  \item $\CatW$-semigroup, 12
  \item \axiomW{1}-\axiomW{4}, 11
  \item \axiomW{5}, 32
  \item \axiomW{6}, 32
  \item way-below \seeonly{compact containment}, 11
  \item weak cancellation, 31
  \item weakly divisible monoid, 146
  \item weakly purely noncompact, 57

\end{theindex}

\renewcommand*{\indexname}{Index of Symbols}
\begin{theindex}

  \item $<_s$ \quad (stable domination relation), 52
  \item $<_s^\ctsRel$, 53
  \item $\CatLocCa$ \quad (category of local \ca{s}), 16
  \item $\gamma(S)$ \quad (reflection from $\CatPreW$ in $\CatCu$), 27
  \item $\leq_I$, 37
  \item $\leq_p$, 73
  \item $\lhd$, 46
  \item $\ll$ \quad(compact containment), 11
  \item $\otimes$ \quad (tensor product in $\CatMon$), 168
  \item $\otimes_{\CatCu}$, 94
  \item $\otimes_{\CatPom}$, 173
  \item $\otimes_{\CatPrePom}$, 173
  \item $\otimes_{\CatPreW}$, 91
  \item $\prec$ \quad(auxiliary relation), 11
  \item $\precsim$ \quad (Cuntz subequivalence), 17
  \item $\sim$ \quad (Cuntz equivalence), 17
  \item $\sim_I$, 37
  \item $\tau_{A,B}^\txtMax$, 101
  \item $\tau_{A,B}^\txtMin$, 101
  \item $\varpropto$, 52
  \item $\varpropto^\ctsRel$, 53

  \indexspace

  \item $\hat{a}$, 46
  \item $a^\prec$ \quad(set of $\prec$-predecessor of $a$), 11

  \indexspace

  \item $\CatCa$ \quad (category of \ca{s}), 29
  \item $\CatLocCaLim$ \quad (inductive limit in $\CatLocCa$), 19
  \item $\CatCon$ \quad (category of conical monoids), 171
  \item $\CatCu$ \quad (category of $\CatCu$-semigroups), 21
  \item $\Cu(A)$ \quad (completed Cuntz semigroup), 28
  \item $\CatCuMor(S,T)$ \quad ($\CatCu$-morphisms), 21
  \item $\CatCuGenMor{S,T}$ \quad (generalized $\CatCu$-morphisms), 21
  \item $\CatCuLim$ \quad (inductive limit in $\CatCu$), 27

  \indexspace

  \item $\elmtrySgp{k}$ \quad (elementary $\CatCu$-semigroup), 44

  \indexspace

  \item $F(S)$ \quad (functionals), 45

  \indexspace

  \item $G_+$ \quad (positive cone), 176
  \item $\CatGp$ \quad (category of groups), 170
  \item $\Gr(M)$ \quad (Grothendieck completion), 170

  \indexspace

  \item $\Idl(a)$ \quad (ideal generated by $a$), 40

  \indexspace

  \item $K_q$, 127

  \indexspace

  \item $L(F(S))$, 46
  \item $\Lat(S)$ \quad (lattice of ideals), 39
  \item $\Lat_{\mathrm{f}}(S)$ \quad (singly-generated ideals), 40
  \item $\Lsc(F(S))$, 46

  \indexspace

  \item $M_\infty(A)$, 16
  \item $M_q$ \quad (UHF-algebra), 128
  \item $\CatMon$ \quad (category of monoids), 167

  \indexspace

  \item $\CatPoGp$ \quad (category of partially ordered groups), 176
  \item $\CatPom$ \quad (category of positively ordered monoids), 172
  \item $\CatPoRg$ \quad (category of partially ordered rings), 178
  \item $\CatPrePom$ \quad (category of positively pre-ordered monoids),
		172
  \item $\CatPreW$ \quad (category of $\CatPreW$-semigroups), 12
  \item $\CatPreWLim$ \quad (inductive limit in $\CatPreW$), 15

  \indexspace

  \item $Q$, 157

  \indexspace

  \item $\mathcal{R}$ \quad (Jacelon-Razak algebra), 132
  \item $R_q$, 127
  \item $R^\ctsRel$ \quad (regularization of a relation), 53

  \indexspace

  \item $S_\pnc$ \quad (purely noncompact elements), 59
  \item $S_\soft$ \quad (soft elements), 57
  \item $S_\wpnc$ \quad (weakly purely noncompact elements), 59
  \item $S_R$ \quad (realification), 134
  \item $S/I$ \quad (quotient semigroup), 37

  \indexspace

  \item $\CatW$ \quad (category of $\CatW$-semigroups), 12
  \item $\W(A)$ \quad (precompleted Cuntz semigroup), 17
  \item $\CatWMor(S,T)$ \quad ($\CatW$-morphisms), 12
  \item $\CatWGenMor{S,T}$ \quad (generalized $\CatW$-morphisms), 12
  \item $\CatWLim$ \quad (inductive limit in $\CatW$), 16

  \indexspace

  \item $Z$ \quad (Cuntz semigroup of Jiang-Su algebra), 98
  \item $\mathcal{Z}$ \quad (Jiang-Su algebra), 73

\end{theindex}

%\printindex

%\Printindex{terms}{Index of Terms}
%\Printindex{symbols}{Index of Symbols}

\end{document}